%% file: main.tex
\documentclass[11pt]{article}
\usepackage{fullpage}
\usepackage[utf8]{inputenc}

\providecommand{\keywords}[1]
{
  \small	
  \textbf{\textit{Keywords}:} #1
}

\usepackage{graphicx}
\usepackage{epstopdf}

\RequirePackage{amsthm,amsmath,amsfonts,amssymb}
\RequirePackage[numbers]{natbib}
\RequirePackage[colorlinks,citecolor=blue,urlcolor=blue]{hyperref}
\RequirePackage{graphicx}

\usepackage{natbib}
\setcitestyle{authoryear, round}
\usepackage{mathtools}

\newtheorem{theorem}{Theorem}
\newtheorem{lemma}{Lemma}
\newtheorem{assumption}{Assumption}
\newtheorem{proposition}{Proposition} 
\theoremstyle{remark}
\newtheorem{definition}{Definition}
\newtheorem{corollary}{Corollary}
\newtheorem{remark}{Remark}
\input{myfiles.tex}

\title{On the Structural Dimension of Sliced Inverse Regression}

\author{
Dongming Huang\thanks{Department of Statistics and Data Science, National University of Singapore. Email: \texttt{stahd@nus.edu.sg}}, 
Songtao Tian~\thanks{Department of Mathematical Sciences, Tsinghua University. Email: \texttt{tst20@mails.tsinghua.edu.cn}}~\thanks{Co-first author},
Qian Lin\thanks{Center for Statistical Science \& Department of Industrial Engineering, Tsinghua University; Beijing Academy of Artificial Intelligence.  Email: \texttt{qianlin@tsinghua.edu.cn}}.
}
\date{}

\begin{document}
\maketitle
\begin{abstract}
In this work, we address the longstanding puzzle that Sliced Inverse Regression (SIR) often performs poorly for sufficient dimension reduction when the structural dimension $d$ (the dimension of the central space) exceeds 4. 
We first show that in the multiple index model $Y=f( \mathcal P \boldsymbol{X})+\epsilon$ where 
$\boldsymbol{X}$ is a $p$-standard normal vector, 
$\epsilon$ is an independent noise, 
and $\mathcal P$ is a projection operator from $\mathbb R^{p}$ to $\mathbb R^{d}$, if the link function $f$ follows the law of a Gaussian process, then with high probability, the $d$-th eigenvalue $\lambda_{d}$ 
 of $\mathrm{Cov}\left[\mathbb{E}(\boldsymbol{X}\mid Y)\right]$ satisfies $\lambda_{d}\leq C e^{-\theta d}$ for some positive constants $C$ and $\theta$. 
We then focus on the low signal regime where $\lambda_{d}$ can be arbitrarily small and not larger than $d^{-8.1}$, and prove that the minimax risk of estimating the central space is lower bounded by $\frac{dp}{n\lambda_{d}}$. 
Combining these two results, we provide a convincing explanation for the poor performance of SIR when $d$ is large, a phenomenon that has perplexed researchers for nearly three decades. The technical tools developed here may be of independent interest for studying other sufficient dimension reduction methods.
\end{abstract}


\input{content.tex}

\section*{Acknowledgements}
Huang's research is supported in part by NUS Start-up Grant A-0004824-00-0 and Singapore Ministry of Education AcRF Tier 1 Grant A-8000466-00-00. 
Lin's research is supported in part by the National Natural Science Foundation of China (Grant 92370122, Grant 11971257).

\bibliography{reference.bib}
\bibliographystyle{chicago}

\appendix
\input{Appendix.tex}

\end{document}

%% file: myfiles.tex
\usepackage{tabls}
\usepackage{url}
\usepackage{multicol} 
\usepackage{multirow} 

\usepackage{hyperref}

\usepackage{bm}
\usepackage{subcaption}
\usepackage[lined,boxed,ruled,linesnumbered]{algorithm2e}

\usepackage{amsfonts}

\usepackage{amsmath,amssymb}
\usepackage{tablists}

\usepackage{mathrsfs}
\usepackage{amsmath}
\usepackage{float}
\usepackage{amscd}
\usepackage{enumerate}
\usepackage{paralist}
\usepackage{graphicx}
\usepackage{mathtools}
\usepackage[all,cmtip]{xy}
\usepackage{kantlipsum}
\usepackage{tikz}
\usepackage{tikz-cd}
\allowdisplaybreaks

\def\R{\mathbb R}

\def\P{\mathbb P}
\def\S{\mathcal S}
\def\E{\mathbb E}

\def\col{\mathrm{col}}
\def\wh{\widehat}

\def\var{\mathrm{var}}
\def\mb{\mathbb}
\def\bs{\boldsymbol}
\def\ol{\overline}
\def\Cov{\mathrm{Cov}}
\def\mr{\mathrm}
\newcommand{\Tr}{{\text{Tr}}}
\newcommand{\supp}{\textrm{supp}}
\newcommand{\kl}{\textrm{KL}} 
\newcommand{\sgn}{\mathrm{sgn}}
\def\X{\mathcal{X}}
\newcommand{\wt}{\widetilde}
\newcommand{\ttE}{\mathtt{E}}
\newcommand\independent{\protect\mathpalette{\protect\independenT}{\perp}}
\def\independenT#1#2{\mathrel{\rlap{$#1#2$}\mkern2mu{#1#2}}}

\newcommand{\bO}{\boldsymbol{\Omega}}

\def\independenT#1#2{\mathrel{\rlap{$#1#2$}\mkern2mu{#1#2}}}
\newcommand{\indp}{\protect\mathpalette{\protect\independenT}{\perp}}
\newcommand{\bS}{\boldsymbol{\Sigma}}
\newcommand{\tp}{^\top}

\newcommand{\N}{\mathcal{N}}
\newcommand{\bbE}{\mathbb{E}}
\newcommand{\bbP}{\mathbb{P}}

\newcommand{\leqs}{\leqslant}
\newcommand{\geqs}{\geqslant}

\def\Ent{\mathrm{Ent}}

\newcommand{\bbR}{\mathbb{R}}
\newcommand{\va}{\boldsymbol{a}}

\newcommand{\ve}{\boldsymbol{e}}

\newcommand{\vl}{\boldsymbol{l}}
\newcommand{\vm}{\boldsymbol{m}}

\newcommand{\vt}{\boldsymbol{t}}
\newcommand{\vu}{\boldsymbol{u}}
\newcommand{\vv}{\boldsymbol{v}}

\newcommand{\vx}{\boldsymbol{x}}
\newcommand{\vy}{\boldsymbol{y}}
\newcommand{\vz}{\boldsymbol{z}}

\newcommand{\vA}{\boldsymbol{A}}
\newcommand{\vB}{\boldsymbol{B}}
\newcommand{\vC}{\boldsymbol{C}}
\newcommand{\vD}{\boldsymbol{D}}
\newcommand{\vE}{\boldsymbol{E}}
\newcommand{\vF}{\boldsymbol{F}}
\newcommand{\vG}{\boldsymbol{G}}

\newcommand{\vI  }{\mathbf{I}}
\newcommand{\vJ}{\boldsymbol{J}}
\newcommand{\vK}{\boldsymbol{K}}
\newcommand{\vL}{\boldsymbol{L}}
\newcommand{\vM}{\boldsymbol{M}}

\newcommand{\vQ}{\boldsymbol{Q}}

\newcommand{\vU}{\boldsymbol{U}}
\newcommand{\vV}{\boldsymbol{V}}
\newcommand{\vW}{\boldsymbol{W}}
\newcommand{\vX}{\boldsymbol{X}}

\newcommand{\vZ}{\boldsymbol{Z}}

\newcommand{\bbeta}{\boldsymbol{\beta}}

\newcommand{\bLambda}{\boldsymbol{\Lambda}}

\newcommand{\bxi}{\boldsymbol{\xi}}
\newcommand{\one}{\mbox{~$1$}}

\newcommand{\vmu}{{ \boldsymbol{\mu} }}

\def\diff{\,\textrm{d}\,}

\newcommand{\mc}{\mathcal}

\makeatletter

  \makeatletter
  \newcommand{\be}{%
  \begingroup
  \eqnarray%
   \@ifstar{\nonumber}{}%
  }

\makeatother


%% file: content.tex
\section{Introduction}\label{sec:introduction}

A subspace $\mathcal S\subset \mathbb R^{p}$ is a dimension reduction subspace of a joint distribution of $(\vX,Y)\in\mathbb R^{p}\times \mathbb R$ if $Y\independent \vX| P_{\mathcal S}\vX$, 
where $P_{\mc{S}}$ denotes the projection operator from $\mathbb{R}^{p}$ to the subspace $\mc{S}$.
The \textit{central space}, denoted by $\mathcal S_{Y\mid \vX}$, is the intersection of all dimension reduction subspaces.
The primary objective of sufficient dimension reduction (SDR) is to estimate the central space $\mathcal S_{Y \mid \vX}$. \text{~~~~~}

\cite{li1991sliced} considers the following multiple-index model that
\begin{align}\label{eq:multiple index model vector}
Y=f(\bbeta_1\tp\vX,\dots,\bbeta_d\tp\vX,\epsilon),\quad d < p
\end{align}
where $f:\R^{d+1}\to \R$ is an unknown nonparametric link function,  $\bbeta_{i}\in\R^p$ are the index vectors, and ${\epsilon}$ is a random noise independent of $\boldsymbol{X}$, and proposes \textit{sliced inverse regression} (SIR), the first SDR method, to estimate the central space. 
Note that the model for $(\vX,Y)$ in \eqref{eq:multiple index model vector} is semiparametric where $f$ is a infinite-dimensional nuisance parameter and the central space is $\mathcal S_{Y \mid \vX}=\mr{span}\{\bbeta_1,\dots,\bbeta_d\}$. 
Besides SIR, various SDR algorithms have been proposed in the literature, such as \textit{sliced average variance estimation} (SAVE, \cite{cook1991sliced}), \textit{directional regression} (DR, \cite{li2007directional}), and many others. 
These methods are widely applied and often serve as an intermediate step in modelling the relation between $Y$ and $\vX$ in various fields. 
However, in modern data analyses, the dimension of covariates $p$ can often be so large that traditional SDR methods do not apply, which compels us to develop high-dimensional SDR methods. 
In addition, \emph{the structural dimension} $d$ (i.e., the dimension of the central space $\mc S_{Y\mid \vX}$) may also be large, which presents further challenges. 
Understanding the theoretical limitations of existing SDR methods in these challenging scenarios is important for advancing this development.

The asymptotic properties of SIR have been of particular interest in recent decades, as SIR is considered one of the most popular SDR methods due to its simplicity and computational efficiency. 
When the dimension $p$ is either fixed or growing at a slower rate than the sample size $n$, researchers have extensively studied SIR's asymptotic properties in various settings. 
For example,   \cite{hsing1992asymptotic} established the root-$n$ consistency and asymptotic normality of SIR when the sample size in each slice (denoted by $c$) is $2$; 
\cite{zhu1995asymptotics} provided asymptotic results of SIR for an arbitrary fixed constant $c$ or a growing $c$ as the sample size $n$ increases; 
\cite{zhu2006sliced} discussed the condition for SIR to give a consistent estimator when $p=o(\sqrt n)$;  \cite{wu2011asymptotic} determined the convergence rate of SIR for sparse multiple-index models with $p=o(n/\log(n))$. 

It becomes more challenging to study the properties of SIR in the high-dimensional regime, where the dimension $p$ could be comparable to or even larger than $n$. 
\cite{lin2018consistency} studied the situation where $\delta:=\lim p/n$ is a constant and proved that the SIR estimator for the central space is consistent if and only if $\delta=0$. 
To deal with the challenges in high-dimensional settings, they proposed the DT-SIR method that is consistent under the sparsity assumption that the central space only depends on a small proportion of the predictors. 
We hereafter call by \emph{sparse SIR} any method that modifies the original SIR so as to estimate the central space under the sparsity assumption. 
 \cite{lin2021optimality} established the minimax rate optimality of sparse SIR over a large class of high-dimensional multiple-index models. 
 \cite{lin2019sparse} proposed the Lasso-SIR algorithm, which is computationally efficient and achieves the minimax rate.  
 In a different setting, \cite{tan2020sparse} 
 studied the minimax rates under various loss functions and proposed a computationally tractable adaptive estimation scheme for sparse SIR.

Despite these advancements, there is a noticeable gap between theory and application. 
It is widely observed that SIR and other SDR algorithms often fail to fully recover the central space when the structural dimension $d$ is larger than $4$ \citep{ferre1998determining} and most numerical experiments are reported only for $d\leq 4$ (see, for example, \cite{li1991sliced,cook1991sliced,zhu2006sliced,li2007sparse,ma2012semiparametric,jiang2014variable,yu2016marginal}). 
In contrast, theoretical research asserts that these algorithms can consistently estimate the central space for any given $d$. 
This discrepancy between empirical observations and theoretical assertions motivates our investigation into the impact of the structural dimension $d$ on estimation accuracy.

The theory established by \cite{lin2021optimality} suggests that for a fixed structural dimension $d$, the performance of SIR crucially depends on the ``generalized signal-to-noise ratio'' (gSNR) introduced by \cite{lin2019sparse}. 
Inspired by this result, we conducted extensive numerical experiments and observed that as $d$ increases, the gSNR tends to decay quickly (see Section~\ref{sec:simulation small gSNR} for details). 
In light of this phenomenon, if we can 
(1) quantitatively characterize the decay rate of the gSNR as $d$ increases, such as gSNR=$O_{P}(e^{-\theta d})$ for some positive constant $\theta$, and (2) extend the minimax results in \cite{lin2021optimality} to situations where the structure dimension $d$ is unbounded, then it becomes clear that to guarantee the risk of estimation is smaller than any prescribed constant $\varepsilon$, the required sample size must be larger than $\frac{d  p}{\varepsilon  \text{gSNR}}$. This implies that the number of required samples grows rapidly (or even exponentially) as $d$ increases. 
However, neither of these two statements is available in the literature. First, existing minimax theories are restricted because they have to assume that the structural dimension $d$ is either bounded or fixed \citep{lin2021optimality,tan2020sparse}. 
Second, few theoretical tools are available for analyzing the gSNR, so there is a need for novel proof techniques. 
This paper aims to rigorously implement the aforementioned strategy, so that we can offer a convincing explanation for the puzzling discrepancy between empirical observations and existing theories for SIR and related SDR algorithms.

\subsection{Major contributions}
In this article, we provide a rigorous theory to explain how the structural dimension $d$ affects the estimation accuracy of the central space $\mc S_{Y\mid \vX}$. Throughout the paper,  we assume that the structural dimension $d$ is known but we allow $d$ to be arbitrarily large (i.e., there is no constant upper bound on $d$).

Our first major contribution is the introduction of a relatively mild condition to replace the vague technical condition, the sliced stable condition (SSC), used in previous studies to obtain the concentration inequality for the SIR estimator \citep{lin2018consistency}. The new condition, called weak sliced stable condition (WSSC), guarantees all the results previously established under SSC (such as those in \cite{lin2018consistency,lin2019sparse,lin2021optimality}). We prove that WSSC holds under some mild conditions that are readily interpretable: 
$\sup_{\|\bbeta\|=1}\E[|\langle\bs X,\bbeta\rangle|^\ell]<\infty$ for some $\ell>2$, $Y$ is a continuous random variable, and $\bbE[\vX\mid Y=y]$ is a continuous function; see Theorem~\ref{thm:moment condition to sliced stable}. 
This relaxation allows us to investigate the theory of SIR under the least restrictive conditions found in the SIR literature, and it may inspire further theoretical investigation of other slicing-based SDR algorithms.

Our second major contribution is demonstrating that the generalized signal-to-noise ratio (gSNR), defined as the $d$-th largest eigenvalue $\lambda_{d}$ of $\Cov(\bbE[\vX|Y])$, decays quickly as the structural dimension $d$ increases. 
Extensive numerical experiments illustrate that the gSNR appears to decay exponentially with increasing $d$. 
We rigorously prove that with high probability, $\lambda_{d}=O(e^{-\theta d})$ for some positive constant $\theta$ under the assumption that the link function $f$ follows the law of a Gaussian Process; see Theorem~\ref{thm:GP-gSNR-fast-decay}. 
This property of the eigenvalue $\lambda_{d}$ of $\Cov(\bbE[\vX|Y])$ is closely related to the study of level sets of Gaussian processes \citep{azais2009level}. 
However, to the best of our knowledge, no result regarding the conditional mean $\bbE[\vX|Y=y]$ and its covariance matrix has been developed in the literature on level sets of Gaussian processes. 
Therefore, further development of the related theory might be of independent interest.

The third major contribution of this paper is establishing a sharp lower bound on the minimax risk of estimating the central space with a diverging structural dimension $d$ in the low gSNR regime. 
Specifically, for a large class of link functions where $\lambda_{d}$ can be arbitrarily small but no larger than $O(d^{-8.1})$, we show that the minimax lower bound on the risk of estimating the central space is given by $\frac{dp}{n\lambda_{d}}$; see Theorem~\ref{thm:lower-smallp}. 
Apart from the classical Fano's method, 
our proof involves a novel construction of nonlinear dependent distributions and a new sharp upper bound on the pairwise Kullback-Leibler (KL) divergence. 
Unlike most previous works on related problems such as linear regression, principal component analysis, and canonical correlation analysis \citep{raskutti2011minimax,cai2013sparse,gao2014minimax}, 
the KL-divergence in our problem does not have a closed-form expression because the relationship between $Y$ and $\vX$ is nonlinear and semiparametric (recall that $f$ is a nonparametric nuisance parameter and the estimation is about $\mr{span}\{\bbeta_1,\dots,\bbeta_d\}$). 
Furthermore, the proof strategy for bounded $d$ in \citet{lin2021optimality} does not apply here, since their upper bound on the KL-divergence fails to permit the desired minimax rate. 
What makes the current problem even more challenging is that the KL-divergence depends on the gSNR $\lambda_{d}$ in a complicated way since $\lambda_{d}$ is defined through the inverse regression (i.e., the conditional mean of $\vX$ given $Y$). 
With few existing results to utilize, we have to carefully analyze the KL-divergence and its derivatives with respect to the parameters of interest in order to establish a sharp bound. 
This effort turns out to be fruitful; 
combining the minimax lower bound with the exponential decay rate of $\lambda_{d}$, we come to the most important conclusion of this paper: 
\begin{quote}
\emph{As the structural dimension $d$ increases, the sample size required for consistent estimation of the central space typically grows at least at the order of $e^{\theta d}dp$.}
\end{quote} 
This explains why SIR and other SDR algorithms perform poorly in practice when $d$ is large.
Although this theory is developed by assuming $d$ is known, our conclusion here continues to hold because the estimation problem will be harder if $d$ is unknown.

In addition to the major contributions mentioned above, this work also presents some novel results. 
For example, we prove a universal upper bound on the gSNR, which is tight but only attained if the link function is deliberately constructed. 
In addition, we extend the results about minimax rates reported in \cite{lin2021optimality} to situations where the structural dimension $d$ is large and $\lambda_d$ is small and establish the rate optimality of SIR in low-dimensional settings and one of its sparse variants in high-dimensional settings.

\bigskip
\subsection{Organization of the paper} 

The paper is organized as follows. 
In Section \ref{sec:A brief review of SIR}, we briefly review the SIR procedure for estimating the central space. 
Sections \ref{sec: Key lemma under weak sliced stable condition} and \ref{sec: Weak sliced stable condition is very mild} introduce a new mild condition for analyzing SIR estimators. 
In Section~\ref{sec:small-gSNR}, we demonstrate that the gSNR decays fast as the structural dimension increases. 
Specifically, Section~\ref{sec:eigen-upper} provides a universal upper bound on the gSNR and Section~\ref{sec:gSNR-exponential-decay} shows that the gSNR decays exponentially fast if the link function $f$ is sampled from a Gaussian process. 
Section~\ref{sec:minimax lower bounds} presents our minimax lower bounds for estimating the central space when the structural dimension is large. 
In Section \ref{sec:minimax rate}, we establish upper bounds on the risk of SIR and sparse SIR, which match the respective minimax lower bounds under certain mild conditions. 
Numerical studies illustrating the phenomena revealed by our theory are presented in Section~\ref{sec:simulation}. 
Section~\ref{sec:discussions} discusses related problems and future directions.

\subsection{Notation}
For a matrix $\vV\in\R^{l\times m}$, we denote its column space, its $i$-th row, $j$-th column, and $k$-th singular value by $\col(\vV)$, $\vV_{i,*}$, $\vV_{*,j}$, and $\sigma_k(\vV)$, respectively.
We use $\mc{P}_{\vV}$  to denote the orthogonal projection w.r.t. the standard inner product $\langle\cdot,\cdot\rangle$ in Euclidean space onto $\col(\vV)$. 
For a square matrix $\vA\in\R^{m\times m}$, we denote the $i$-th largest eigenvalue and the trace of $\vA$ by $\lambda_i(\vA)$ and  $\Tr(\vA):=\sum_{i=1}^m \vA_{i,i}$, respectively. 
The Frobenius norm and the operator
norm ($2$-norm) of the matrix $\vV\in\R^{l\times m}$   are defined as $\|\vV\|_{F}:=\sqrt{\Tr(\vV^\top \vV)}$ and $\|\vV\|:=\sqrt{\lambda_1(\vV^\top \vV)}$, respectively.
For a vector $\vX$, denote by $X_k$ the $k$-th entry of $\vX$ and write
$\vX^{\otimes}=\vX\vX\tp$. 
For two  numbers $a$ and $b$, we use $a\vee b$ and $a\wedge b$ to denote $\max\{a,b\}$ and  $\min\{a,b\}$, respectively.
For a positive integer $c$, denote by $[c]$ the set $\{1,2,...,c\}$. 
For any positive integers $p$ and $d$ such that $p\geqslant d$, denote by $\mathbb{O}(p,d)$ the set of all $p\times d$ matrices with orthonormal columns. 
We use $\mb{S}^{p-1}$ to denote the ($p-1$)-sphere, i.e., $\mb{S}^{p-1}:=\{\vx\in\R^p:\|\vx\|=1\}$.

We use $C$, $C'$, $C_1$, and $C_2$ to denote generic absolute constants, though their actual values may vary from case to case.
For two sequences $a_{n}$ and $b_{n}$, we write $a_{n}\gtrsim b_{n}$ (resp. $a_{n}\lesssim b_{n}$) when there exists a positive constant $C$ such that $a_{n} \geqslant Cb_{n}$ (resp. $a_{n} \leqslant C'b_{n}$). 
If both $a_{n}\gtrsim b_{n}$ and $a_{n}\lesssim b_{n}$ hold, we write $a_{n}\asymp b_{n}$. We write $a_n=o(b_n)$ if $\lim_{n\to\infty}a_n/b_n=0$.

\section{Asymptotics of SIR under Weak Sliced Stable Condition}\label{sec:concentration under WSSC}

Recently, a series of studies \citep{lin2018consistency,lin2021optimality,lin2019sparse} have investigated the behavior of SIR on high-dimensional data. 
These studies showed that SIR can produce a consistent estimate of the central space if and only if $\lim \frac{p}{n}=0$ and SIR attains the minimax rate in various settings. 
One of the key technical tools developed in these studies is the `key lemma': a `deviation property' of the quantity $\bm{\beta}^{\top}\widehat{\bLambda}_{H}\bm{\beta}$ (see the definition in Equation~\eqref{eqn:lambda}) for any unit vector $\bbeta\in\R^p$. 
The `key lemma' relies heavily on the sliced stability condition (SSC), which is scarcely seen. In this section, we propose the weak sliced stable condition (WSSC), a mild condition that is easy to verify, and we investigate the asymptotic behavior of SIR under WSSC.

\subsection{A brief review of SIR}\label{sec:A brief review of SIR}
Suppose that we observed $n$ i.i.d. samples $\{(\vX_{i},Y_{i})\}_{i\in[n]}$ drawn from the joint distribution of $(\vX,Y)$ given by the following multiple-index model: 
\begin{align}\label{model:modified:multiple}
Y=f(\vB\tp\vX,\epsilon), \quad \vB\tp \bS \vB = \vI_d,
\end{align}
where $f$ is an unknown link function, $\vB=\left[\bbeta_1,\dots,\bbeta_d\right]\in\R^{ p\times d}$ is the index matrix, $\epsilon\sim N(0,1)$ is independent of $\vX$, and $\bS$ is the covariance matrix of $\vX$. 
Throughout the paper, we assume that $\E[\vX]=0$ without loss of generality. 
Though $\vB$ is not identifiable because of the unknown link function $f$, the central space $\S_{Y|\bs{X}}:=\col(\vB)$ can be estimated.

The SIR procedure for estimating $\col(\vB)$ can be briefly described as follows.
First of all, the samples $\{(\vX_{i},Y_{i})\}_{i\in[n]}$ are divided into $H$ equal-sized slices according to the order statistics $Y_{(i)}$; for simplicity, we assume that $n=cH$, where $c$ is a positive integer. 
Next, the data can be re-expressed as $\vX_{h,j}$ and $Y_{h,j}$, with $(h,j)$ as the double subscript, where $h$ denotes the order of the slice and $j$ the order of the sample in the $h$-th slice, i.e.,
\begin{equation}\label{eq:sliced-xy}
\vX_{h,j}=\vX_{(c(h-1)+j)}, \mbox{\quad \quad }Y_{h,j}=Y_{(c(h-1)+j)} .
\end{equation}
Here $\vX_{(k)}$ is the concomitant of $Y_{(k)}$ \citep{yang1977}. 
Let $\S_{h}$ be the $h$-th interval $(Y_{(h-1,c)},Y_{(h,c)}]$ for $h=2,\dots,H-1$, $\S_1=\{y\mid y\leqslant Y_{(1,c)}\}$, and $\S_{H}=\{y\mid y>Y_{(H-1,c)} \}$. 
Consequently, $\mathfrak{S}_{H}(n):=\{\S_{h}, h=1,..,H\}$ is a partition of $\R$ and is referred to as the \textit{sliced partition}. 
Denote the sample mean of $\vX$ in the $h$-th slice by  $
\overline{\vX}_{h,\cdot}$. 
The SIR algorithm estimates the candidate matrix $\bLambda:= \Cov(\bbE[\vX|Y])$ by 
\begin{equation}\label{eqn:lambda}
\widehat{\bLambda}_{H}=\frac{1}{H}\sum_{h=1}^H\ol{\vX}_{h,\cdot}\ol{\vX}_{h,\cdot}\tp.
\end{equation}
In \cite{lin2018consistency}, it is shown that when the ratio $p/n\to0$,
a consistent estimator of the central space $\col(\vB)$ is given by $\widehat{\bS}^{-1}\col(\widehat{\eta}_{H})$,  where $\widehat \bS$ is the 
sample covariance matrix of $\vX$ and $\widehat{\eta}_{H}$ is the matrix formed by the top $d$ eigenvectors of $\widehat{\bLambda}_{H}$. 
Alternatively, the central space $\col(\vB)$ could be estimated by $\col(\widehat{\vB}_{H})$, where $\wh\vB_H$ is defined by 
\begin{equation}\label{eqn:Bhat}
\begin{aligned}
\widehat{\vB}_{H}:= 
\arg\max_{\vB} ~~~\Tr(\vB\tp
\widehat{\bLambda}_{H}\vB) \quad \mbox{ s.t. } \vB\tp \widehat{\bS} \vB=\vI_d 
\end{aligned},
\end{equation} 
since $\col(\widehat{\eta}_{H})=\widehat{\bS}\col(\widehat{\vB}_{H})$ (see, e.g., Proposition~4 in \citet{li2007sparse}).

To ensure that SIR provides a consistent estimator of the central space, the relation $\mr{col}( \bLambda) =\bS\S_{Y|\bs{X}}$ is needed and the following two assumptions have been suggested in the literature \citep{li1991sliced,hsing1992asymptotic,zhu2006sliced,lin2018consistency}. 

\begin{assumption}\label{as:linearity and coverge} The joint distribution of $(\vX,Y)\in\bbR^{p}\times \bbR$ satisfies the following conditions:  
\begin{itemize}
 \item[\bf{i)}] \textit{Linearity condition:}
 	For any $\va\in\R^p$, the conditional expectation $\mathbb E\left[\langle \va,\vX\rangle \mid \vB^\top \vX\right]$ is linear in $\vB^\top \vX$.
 \item[\bf{ii)}]\textit{Coverage condition:} for some positive number $\lambda$, the eigenvalues of $\Cov(\bbE[\vX \mid Y])$ satisfy	
$$
\lambda\leqslant\lambda_{d}(\Cov(\bbE[\vX \mid Y])) \leqslant \lambda_{1}(\Cov(\bbE[\vX \mid Y]))\leqslant \lambda_{1}(\bS).
$$ 
\end{itemize}
\end{assumption}

The condition ($\bf i$) is a classic condition for SIR \citep{li1991sliced}. The condition ($\textit{\bf ii}$) is a refinement of the coverage condition in the SIR literature, as explained in Condition A2 of \citet{lin2019sparse}. 
We want to emphasize that the coverage condition is crucial for SDR methods that rely on the slicing procedure. 
As demonstrated by the minimax lower bounds (see Theorems \ref{thm:lower-smallp} and  \ref{thm:lower-largep}), if the eigenvalues of $\Cov(\bbE[\vX \mid Y])$ are too small, then no estimation method can accurately estimate the central space.

\bigskip

\subsection{Key lemma under weak sliced stable condition}\label{sec: Key lemma under weak sliced stable condition}
We begin by revisiting the sliced stable condition (SSC) used in the works of \cite{lin2018consistency,lin2021optimality,lin2019sparse} and then introduce an alternative condition. 
Throughout the paper, $\gamma$ is a fixed small positive constant. 

\begin{definition}[Sliced Stable Condition]\label{def:strong-sliced-stable}
	Let $Y\in\R$ be a random variable, $K$ a positive integer, and $\vartheta >0$ a constant. 
	A continuous curve $\boldsymbol{\kappa}(y): \R\to\R^p$ is said to be \textit{$( K, \vartheta)$-sliced stable} w.r.t. $Y$, if 
	for any $H\geqslant K$ and any  partition  $\mathcal{B}_{H}:=\{-\infty=a_0<a_1<\dots<a_{H-1}<a_H=\infty\}$ of $\R$ such that
	\begin{align}\label{eq:gamma partition}
	\frac{1-\gamma}{H} \leqslant \mathbb{P}(a_h\leqslant Y\leqslant a_{h+1})\leqslant \frac{1+\gamma}{H}, \qquad \forall h=0,1,\ldots, H-1,
	\end{align}
	it holds that
	\begin{align*}
	\frac{1}{H}\sum_{h=0}^{H-1}\mathrm{var}\left(\bbeta\tp\boldsymbol{\kappa}(Y)\big|  a_h\leqslant Y\leqslant a_{h+1} \right)\leqslant \frac{1}{H^{\vartheta}} \mathrm{var}\left(\bbeta\tp\boldsymbol{\kappa}(Y)\right)\quad(\forall\bbeta\in\mb{S}^{p-1}).
	\end{align*}
\end{definition}

\cite{lin2018consistency}  utilized this condition to establish the deviation properties of the eigenvalues, eigenvectors, and entries of $\widehat{\bLambda}_{H}$.  
Although they showed that the SSC is a mild condition, it is hard to verify whether the central curve  $\vm(y):=\mathbb E [ \vX\mid Y=y]$ of a given joint distribution $(\vX,Y)$ satisfies SSC.

To motivate a more manageable condition than SSC, we revisit an intuitive explanation of SSC: for any well-behaved continuous curve $\boldsymbol{\kappa}(y)$, if it is divided into $K$ pieces with roughly equal probability mass of the distribution of $Y$, then the average of the variances of $\boldsymbol{\kappa}(Y)$ in each piece  tends to $0$ as the slice number $H$ tends to $\infty$. 
The $(K,\vartheta)$-sliced stable condition requires the average of the variances to tend to $0$ at a certain rate (e.g., $H^{-\vartheta}$). 
This geometric explanation leads us to introduce the following definition of weak sliced stable condition (WSSC), which only requires the average of the variances to be sufficiently small.
\begin{definition}[Weak Sliced Stable Condition]\label{def:weak-sliced-stable}
	Let $Y\in\R$ be a random variable, $K$  a positive integer, and $\tau >1$ a constant. 
	A continuous curve $\boldsymbol{\kappa}(y): \R\to\R^p$ is said to satisfy the \textit{weak $( K, \tau)$-sliced stable condition} (weak $(K,\tau)$-SSC) w.r.t. $Y$, if 
	for any $H\geqslant K$ and  any partition  $\mathcal{B}_{H}$ of $\R$ such that
	$
	\frac{1-\gamma}{H} \leqslant \mathbb{P}(a_h\leqslant Y\leqslant a_{h+1})\leqslant \frac{1+\gamma}{H}, \forall h=0,1,\ldots, H-1,$
	it holds that
	\begin{align*}
	\frac{1}{H}\sum_{h=0}^{H-1}\mathrm{var}\left(\bbeta\tp\boldsymbol{\kappa}(Y)\big|   a_h\leqslant Y\leqslant a_{h+1} \right)\leqslant \frac{1}{\tau} \mathrm{var}\left(\bbeta\tp\boldsymbol{\kappa}(Y)\right)\quad(\forall\bbeta\in\mb{S}^{p-1}).
	\end{align*}
\end{definition}

In the following, we show that Lemma~1 in \citet{lin2018consistency}, referred to as the `key lemma' therein, still holds  if we replace  SSC with WSSC. Specifically, we can establish the
deviation properties of $\bbeta\tp\wh\bLambda_H\bbeta$ under WSSC if $H$ is sufficiently large.

We first recall the sub-Gaussian norm of a random variable $X$ is defined as  $\|X\|_{\psi_2}=\sup_{p \geq 1} p^{-1 / 2}\left(\mathbb{E}|X|^p\right)^{1 / p}$
and $X$ is said to be sub-Gaussian if its sub-Gaussian norm is finite. For a random vector $\vX$ in $\mathbb{R}^p$, its sub-Gaussian norm is defined as $\|\vX\|_{\psi_2}=\sup_{\boldsymbol{\beta} \in \mathbb{S}^{n-1}}\|\langle \vX, \boldsymbol{\beta}\rangle\|_{\psi_2}$, and it is said to be sub-Gaussian if its sub-Gaussian norm is finite. 

\begin{lemma}\label{lem:key lemma in Lin under moment condition}
Suppose the random vector $\vX$ is sub-Gaussian with sub-Gaussian norm bounded by $M$, and its covariance matrix $\bS$ satisfies $\|\bS\|\vee\|\bS^{-1}\| \leqslant M$, where $M$ is a positive constant. 
Suppose that Assumption \ref{as:linearity and coverge} holds

and that the central curve $\vm(y):=\mb E[\vX \mid Y=y]$ satisfies the weak $(K,\tau)$-SSC w.r.t. $Y$  with $\tau>16$.

If $\lambda\leqslant\lambda_{d}(\Cov(\bbE[\vX \mid Y])) \leqslant \lambda_{1}(\Cov(\bbE[\vX \mid Y]))\leqslant \kappa\lambda$ for some constant $\kappa>1$, 
	then there exist positive absolute constants $C,C_1,C_2$, and $C_3$  such that if $H\geqslant K\vee Cd$, then for any 
	$\nu\in(1,\tau/16]$, any unit vector $\bbeta\in\mr{col}(\bLambda)$ and any sufficiently large $n>1+4H/\gamma$, it holds that
	\begin{align*}
\bbP\left(\left|\bbeta^\top\left(\widehat{\bLambda}_{H}-\bLambda\right)\bbeta\right|\geqslant  \frac{1}{2\nu}\bbeta\tp\bLambda\bbeta\right)\leqslant C_{1} \exp\left(-C_{2}\frac{n\bbeta\tp\bLambda\bbeta}{H^2\nu^2}+C_{3}\log(nH)\right).
	\end{align*}
\end{lemma}

Compared with the `key lemma' in \cite{lin2018consistency}, the range of $\nu$ in Lemma~\ref{lem:key lemma in Lin under moment condition} is narrower, i.e., $\nu$ cannot go to infinity. 
The reason behind this difference is that, in the current lemma, we have relaxed the condition from SSC to WSSC. 
However, it is worth noting that the narrower range of $\nu$ makes no essential difference in the asymptotic theory developed for SIR, as we can establish convergence under WSSC in Section~\ref{sec:minimax rate}.

\subsection{Weak sliced stable condition is very mild}\label{sec: Weak sliced stable condition is very mild}
Though it is hard to verify whether the central curve $\vm(y)$ satisfies SSC, we can show that $\vm(y)$ satisfies WSSC under plausible assumptions. 

\begin{theorem}\label{thm:moment condition to sliced stable}
	Suppose that the  joint distribution of $(\vX,Y)\in \mathbb{R}^{p}\times\mb{R}$ satisfies the following  conditions:
\begin{itemize}
    \item[${\bf i)}$] for any  $\bbeta\in \mb S^{p-1}$,  $\mathbb{E}\left[|\langle\bbeta,\vX\rangle|^{\ell}\right]\leqslant c_{1}$ holds for absolute constants $\ell>2$ and $c_{1}>0$;
    \item[${\bf ii)}$]$Y$ is a continuous random variable;
\item[${\bf iii)}$]the central curve $\vm(y):=\bbE[\vX|Y=y]$ is continuous.
\end{itemize}
Then for any $\tau>1$, there exists an integer $K=K(\tau,d)\geqslant d$ such that $\vm(y)$ satisfies the weak $(K,\tau)$-SSC w.r.t. $Y$.
If we further assume that $\|\vm(y)-\vm(y')\|\leqslant C\zeta|y-y'|$ 
 for any $y,y'$ defined on $\R$ 
 and $C>0$ is a constant, then the WSSC coefficient $K$ can be as small as $\lceil K_0\zeta\rceil$ for some absolute constant $K_0$.

\end{theorem}

Theorem~\ref{thm:moment condition to sliced stable} provides a simple way to ensure the WSSC condition, for example, by utilizing the smoothness of the central curve $\vm(y)$. We discuss this point in the following remarks. 

\begin{remark}\label{remark:sufficient condition of K=Cd}

Since $\mr{dim}\{\mr{span}\{\vm(y):y\in \R\}\}=\mr{rank}(\mr{Cov}(\vm(Y)))=d$, one can always turn the last $p-d$ entries of $\vm(y)$ into zero through an orthogonal transformation. If we assume the $d$-dimensional curve $\vm(y)$ is $(Cd)$-Lipschitz for some constant $C$, we can take $\zeta$ to be $d$ in Theorem~\ref{thm:moment condition to sliced stable} and conclude that the WSSC holds with $K\asymp d$. 
We consider the $(Cd)$-Lipschitz condition on $\vm(\cdot)$ to be rather mild. 
First of all, it is less restrictive than the conditions in \cite{wu2011asymptotic} and \cite{tan2020sparse}, which require $\vm(\cdot)$ to take only a limited number of values. 
Furthermore, the Lipschitz condition is more suitable for studies with diverging $d$ than the previous conditions proposed in the literature, such as the smoothness condition \citep{hsing1992asymptotic} and the total variation condition \citep{zhu1995asymptotics,zhu2006sliced}, since these conditions view $d$ as fixed.

\end{remark}

\begin{remark}
    If the WSSC does not hold with $K$ sufficiently smaller than the sample size $n$, SIR may fail because it may require too many slices for $\widehat{\bLambda}_{H}$ to well approximate $\bLambda$. For example, Chapter 3.3.2 in \cite{huang2020reliable} constructs a population where the WSSC does not hold unless $K$ grows at the order of $2^d$. In this example, slicing methods perform poorly if $n$ is not much larger than $K$. 
\end{remark}

A result in \cite{lin2018consistency} guarantees that with high probability, the probability mass of the distribution of $Y$ in each slice $\S_{h}$ is roughly the same, i.e., \eqref{eq:gamma partition} holds for the sliced partition $\mathfrak{S}_{H}(n):=\{\S_{h}, h=1,..,H\}$. Combining this result with Theorem \ref{thm:moment condition to sliced stable} leads to the following corollary.
\begin{corollary}\label{cor:moment to WSSC}
Suppose that the conditions ${\bf i)}$ - ${\bf iii)}$ in Theorem \ref{thm:moment condition to sliced stable} hold. 
For any sufficiently large $H\geqslant K$, if $n>1+4H/\gamma$ is sufficiently large, then 
for the sliced partition $\mathfrak{S}_H(n)=\{\S_{h}, h=1,..,H\}$, the inequality 
\begin{align*}
\frac{1}{H}\sum_{h=1}^H\mathrm{var}\left(\bbeta\tp\vm(Y)\big|  Y\in \S_h \right)\leqslant \frac{1}{\tau} \mathrm{var}\left(\bbeta\tp\vm(Y)\right)\quad(\forall\bbeta\in\mb{S}^{p-1})
\end{align*}
holds with probability at least  $1-CH^2\sqrt{n+1}\exp\left(-\gamma^2(n+1)/32H^2\right)$ for some absolute constant  $C>0$.    
\end{corollary}

\section{Decay of gSNR with diverging $d$}\label{sec:small-gSNR}
\cite{lin2021optimality} have made a conjecture that the minimax optimal rate for estimating the central space under multiple-index models should be inversely proportional to the gSNR (i.e., the $d$-th largest eigenvalue of $\Cov\left(\bbE \left( \vX \mid  Y  \right)  \right)$). 
This indicates that a small gSNR would result in a large estimation error, so it is important to understand how the gSNR depends on the structural dimension $d$. 

In this section, we show that for a general multiple-index model, the gSNR always decreases as the structural dimension $d$ grows and often decays exponentially fast. 

\subsection{An upper bound on the gSNR}\label{sec:eigen-upper}

We first show that when the structural dimension $d$ increases, the gSNR must decrease at least at the rate of $\frac{\log d}{d}$.

\begin{theorem}\label{thm:max-smallest-eigen}
Assume $\vX\sim N(\bs{0}, \bs{I}_{p})$ and $Y$ is a random variable. 
If $\vm(y)$ satisfies the weak $(K,\tau)$-SSC w.r.t. $Y$ with $\tau>1+\gamma$ and $K\leq C_0 d$, then the $d$-th largest eigenvalue of $\Cov\left[ \E \left( \vX \mid  Y  \right) \right]$ (i.e., the gSNR) is no greater than $C \frac{\log (d)}{d}$, where $C$ is a constant that only depends on $\tau$, $\gamma$, and $C_0$. 
\end{theorem}

The rate of the upper bound in Theorem~\ref{thm:max-smallest-eigen} is tight, as we can construct a joint distribution of $(\vX,Y)$ where the gSNR is asymptotically $O( \frac{\log d}{d} )$. 
This construction is given in Section~\ref{app:eigen-upper} in the supplementary material. 
However, except for such a deliberate construction, the upper bound is unlikely to be attained by most distributions in real-world practice. 
As discussed in Section~\ref{sec:introduction}, we have conducted extensive numerical experiments and observed that the gSNR decays much faster as the structural dimension $d$ increases. We report some of these experiments in Section~\ref{sec:simulation small gSNR} and Appendix~\ref{app:assitional-simulation} in the supplementary material.

\subsection{Fast decay of the gSNR}\label{sec:gSNR-exponential-decay}

We show in this section that if we place a distribution on the space of link functions, then with high probability, the gSNR decays exponentially fast in $d$.  
Specifically, we use a Gaussian process (GP) to model the $d$-variate link function and show that with probability at least $1-\tilde{C}e^{-\theta d}$, the gSNR is bounded above by $O(e^{-\theta d})$ for some positive constants $\tilde{C}$ and $\theta$.

For a given domain $\mathcal{Z}$, a Gaussian process $W$ is a stochastic process such that for any $ m \in \mathbb{N}_+$ and any $z_1, \ldots, z_m \in \mathcal{Z},$, the finite-dimensional distribution of $\left(W\left(z_1\right), \ldots, W\left(z_m\right)\right)$ is multivariate normal. The mean and the covariance are defined by the mean function $\mu: \mathcal{Z} \rightarrow \mathbb{R}$ and the covariance kernel $K: \mathcal{Z} \times \mathcal{Z} \rightarrow \mathbb{R}$ such that 
$$
\mu(z)=\bbE[W(z)], \quad K(z_1, z_2)=\Cov(W(z_1), W(z_2)).
$$
See \cite{rasmussen2006} for a detailed introduction to GPs. 
GPs have often been employed as the prior measures for regression functions in nonparametric Bayesian inference; see, for example, \citep[Chapter 11]{ghosal2017fundamentals}. 
Here, we utilize the GP to postulate a probability measure on the class of link functions, under which most populations witness an exponential decay of the gSNR as the structural dimension $d$ increases.

Specifically, we assume the following model on $(\vX, Y)$: 
\begin{align}\label{eq: GP model}
\begin{split}
\vX& \sim N(0,\bs{I}_p);\\ 
\mc{M}_{GP}:Y&=f(\vB\tp\vX)+ \sigma \varepsilon,
\end{split}
\end{align}
where $\vB \in \mathbb{R}^{p\times d}$ has orthonormal columns, 
$\varepsilon$ is an independent bounded noise with a smooth density function $\varphi(\cdot)$, and $f$ is drawn from a $d$-variate GP with zero mean and squared exponential covariance function $K\left(\mathbf{z}_1, \mathbf{z}_2\right)=\exp \left(-\alpha\left\|\mathbf{z}_1-\mathbf{z}_2\right\|^2\right)$ for any $\mathbf{z}_i\in \mathbf{R}^{d}$, where $\alpha$ is a positive constant. 
Without loss of generality, we assume the density of the noise $\varphi(\cdot)\in C^{2}(\R)$ is supported on $(-1,1)$ with $\left\|\varphi\right\|_{\infty} \leqslant C_{0}$ and has a bounded gradient $\left\|\varphi^{\prime}\right\|_{\infty} \leqslant C_{1}$ for some constants $C_1$ and $C_0$. 

We have the following result regarding the gSNR for this model. 
\begin{theorem}\label{thm:GP-gSNR-fast-decay}
There exist two positive constants $\theta$ and $\tilde{C}$ depending only on $(\sigma, \alpha, C_0,C_1)$   
such that for any integer $d\geq 2$, with probability at least $1-\tilde{C} e^{-\theta d}$, the gSNR resulting from the model in \eqref{eq: GP model} satisfies that 
$$
\lambda_{d}( \Cov\left[ \bbE \left( \vX \mid  Y  \right) \right] )  \leq  \tilde{C} e^{-\theta d}.
$$
\end{theorem}

Theorem~\ref{thm:GP-gSNR-fast-decay} states that for most link functions drawn from the GP, the corresponding gSNRs will be bounded above by $\tilde{C}e^{-\theta d}$, which decays exponentially fast as the structural dimension $d$ increases. 
To the best of our knowledge, this is the first result to provide a theoretical justification of the fast decay phenomenon of the gSNR in the SDR literature.

The result in Theorem~\ref{thm:GP-gSNR-fast-decay} is closely related to level sets of Gaussian processes, since studying property of the eigenvalues of $\Cov(\bbE[\vX|Y])$ involves the conditional mean of $\vX$ given $Y=y$. The study of level sets of Gaussian processes, and random fields in general, is an important subject that remains largely unexplored \citep{azais2009level}. Most of the existing works focus on the study of the geometric measure of the level set (see, for example, \citet{benzaquen1982expected,armentano2019conditions,azais2020}) rather than the properties of the mean value, which is needed in proving Theorem~\ref{thm:GP-gSNR-fast-decay}. Our result may serve as a starting point and motivate future development of the related theory.

\section{Lower bound for sufficient dimension reduction}\label{sec:minimax lower bounds}

In this section, we establish a minimax lower bound for central space estimation with the structural dimension $d$ diverging in the low gSNR regime. 
This result clarifies that it is mainly the weakened signal strength that causes the observed difficulty in SDR problems when $d$ is large.

\subsection{Minimax lower bound}\label{sec:lower bound oracle risk}

To introduce our lower bound, we define a loss function for the central space estimation and a large class of joint distributions for $(\vX, Y)$. 

Let $\wh\vB$ be an estimate of $\vB$, whose columns form a basis of the central space $\mc S_{Y|\vX}$. 
Note that the parameter $\vB$ itself is not identifiable, while $\vB\vB^{\top} $ is identifiable. 
To evaluate the estimated central space $\col(\wh\vB)$, we consider the loss function
$L(\widehat{\vB}, \vB):=\|\widehat{\vB} \widehat{\vB}^{\top} -\vB\vB^{\top} \|_{\mathrm{F}}^{2}$, which we will refer to as the \textit{general loss}. 
This loss function is commonly used in the sufficient dimension reduction literature; for more details, see Section~1.1 of \citet{tan2020sparse}.

Our lower bound is related to the following class of distributions. 
\begin{definition}\label{def:class-lower-bound}
Define $\mathfrak{M}\left(p,d,\lambda\right)$ as the class of distributions for $(\vX,Y)$ that satisfy the followings:
\begin{enumerate}
\item The covariance matrix of $\vX$ is $\bS \in \mathbb{R}^{p\times p}$, such that $\|\bS\|\vee\|\bS^{-1}\|<\infty$. 

    \item $Y=f(\vB^{\top}\vX,\epsilon)$, where $\epsilon\sim N(0,1)$ is independent of $\vX$, $\vB\in \mathbb{R}^{p\times d}$ such that $\vB^{\top} \bS\vB=\vI_d$, and $f: \mathbb{R}^{d+1}\mapsto \mathbb{R}$. 
    \item Assumption~\ref{as:linearity and coverge} is satisfied. 
    
 \item   The central curve
		$\vm(y)=\bbE[\vX|Y= y]$ satisfies the weak $(K,32)$-SSC w.r.t. $Y$ where $K\geq 8d$. 
\end{enumerate}
\end{definition}

In the following, we do not require the population parameters $(p,d,\lambda)$ of the model class $\mathfrak{M}(p,d,\lambda)$ to be fixed; instead, they are allowed to depend on the sample size $n$, i.e., both $p$ and $d$ may grow and $\lambda$ may decay as $n$ increases. 
We are now ready to state the minimax lower bound for estimating the central space over the model class $\mathfrak{M}\left(p,d,\lambda\right)$ in terms of the triplet $(n,p,d)$ as well as $\lambda$. 
Throughout this section, the infimum $\inf_{\wh \vB}$ is taken over all estimators that depend on the sample $\{(\vX_{i},Y_{i})\}_{i=1}^{n}$, which consists of $n$ i.i.d. observations from $\mc M\in \mathfrak{M}(p,d,\lambda)$. In addition, the expectation $\bbE_{\mc M}$ is taken with respect to the randomness of the sample.

\begin{theorem}[Minimax lower bound]\label{thm:lower-smallp}
There exists a universal positive constant $c_{L}$, such that if  $\lambda\leqslant c_{L} d^{-8.1}$ and $2d<p$, 
then we have
	\begin{align*}
	\inf_{\widehat{\vB}}  \sup _{ \mathcal{M}\in \mathfrak{M}\left(p,d,\lambda\right)  }  
 \bbE_{\mathcal{M}} 
 L(\widehat{\vB}, \vB)
	\gtrsim \min\left(\frac{d(p-d)}{n \lambda}, 1 \right). 
	\end{align*}
 \end{theorem}

As proved in Section~\ref{sec:gSNR-exponential-decay}, the gSNR often decays at the order of $e^{-\theta d}$ as the structural dimension $d$ increases. 
For such a small gSNR, Theorem~\ref{thm:lower-smallp} states that to guarantee the risk for central space estimation is smaller than any prescribed small constant $\varepsilon$, the sample size required by any SDR algorithm must be at least $O\left(\frac{d  p e^{\theta d}}{\varepsilon }\right)$. 
This leads to an important conclusion in SDR: as $d$ increases, the sample size required for any SDR algorithm to work often grows exponentially in $d$. 
To the best of our knowledge, this is the first theoretically solid explanation for the poor performance of the SIR and other SDR algorithms when the structural dimension $d$ is large.

The proof of our minimax lower bound is based on a novel application of Fano's method (see, e.g., \citep{yu1997assouad}). 
Specifically, we have to find a collection of distributions in $\mathfrak{M}\left(p,d,\lambda\right)$ that are separated from each other in terms of $\vB$ while being close in terms of the Kullback–Leibler divergence (KL-divergence). 
As the main technical contribution of this paper, we provide a novel construction of such a collection of distributions and derive a sharp upper bound on the pairwise KL-divergence. 
The difficulty arises from the nonlinear relationship between $Y$ and $\vX$
and the semiparametric nature of the SDR problem since in the multiple-index model $Y=f(\vB^{\top}\vX,\epsilon)$, $f$ is a nonparametric nuisance parameter while the estimation is regarding $\vB$.

\begin{remark}\label{rem:other-SDR}

We briefly discuss whether the result in Theorem~\ref{thm:lower-smallp} can be generalized to describe the performance of other SDR methods. 
For the SAVE method, the space spanned by its candidate matrix includes the space spanned by SIR's candidate matrix under mild conditions \citep{cook2000identifying,ye2003using}, which suggests that the corresponding model class should include $\mathfrak{M}\left(p,d,\lambda\right)$ and thus the minimax lower bound in Theorem~\ref{thm:lower-smallp} also applies. Therefore, SAVE also experiences a performance decay when $d$ is large. 
However, significant differences arise in the theoretical development of other SDR methods such as the ordinary least squares estimates (OLS, \citet{li1989regression}), the principal Hessian directions (PHD, \citet{li1992principal}), and \textit{minimum average variance estimation} (MAVE, \cite{xia2009adaptive}). 
Crucially, these methods aim to estimate the central mean subspace, which is the smallest subspace $\mathcal{S}_m$ such that $\bbE(Y \mid X)=\bbE\left[Y \mid \mathcal{P}_{\mathcal{S}_m}\vX\right]$. Under mild conditions, $\mathcal{S}_m$ is a subspace of the central space $\mathcal{S}$ \citep[Theorem 8.1]{liSufficient2017} and its dimension could be smaller than $d$. 
Consequently, the development of the corresponding minimax theory will be quite different from the one in this work, and we left further exploration of other SDR methods to future studies.

\end{remark}

\subsection{Sketch of proof}\label{sec:sketch-lower-bound}
We will first describe how to construct a collection of distributions of $(\vX, Y)$ and then sketch the proof of the minimax lower bound based on this construction.
As a building block for the construction, we introduce a piecewise constant function as follows. 
\begin{definition}\label{def:fn-lower-median}
Let $m$ be the median of $\chi^2_{d}$ distribution. 
The function $\psi(s_{1},\ldots, s_d)$ is from $\R^{d}$ to $\{i\in \mathbb{Z}:|i|\leq d\}$

and satisfies:
\begin{enumerate}
  \item If $\sum_{i}s_i^2 \leqslant m$, suppose $|s_{i}|$ is uniquely the largest among $|s_1|,\ldots, |s_d|$, then $\psi(s_{1},\ldots, s_d)=\sgn(s_i) i$;
  \item if  $\sum_{i}s_i^2 > m$ or no absolute value is uniquely the largest, 
  then $\psi(s_{1},\ldots, s_d)=0$. 
\end{enumerate}
\end{definition}

For each $\vB\in \mathbb{O}(p,d)$, we construct the following joint distribution $\bbP_{\vB}$ of $(\vX,Y)$: 
\begin{align}\label{eq:joint-x-y-B}
\vX ~ & ~ \sim N(0, \bs{I}_{p}),\\ \nonumber
\vZ~ & ~ = 	\rho \vB^{\top} \vX + \sqrt{1-\rho^2} \xi, \quad \xi \sim N(0, \bs{I}_{d}),\\ \nonumber
W ~ & ~ = \psi (\vZ),\\\nonumber
Y~ & ~ = W + \eta, \quad \eta \sim \text{Unif}(-\sigma ,\sigma ), \nonumber 
\end{align}
where $\vX, \xi, \eta$ are independent of each other, and 
$\sigma\in (0,1/2]$  and $\rho\in (0,1)$ are fixed constants. For this distribution $\bbP_{\vB}$, we have the following important proposition that summarizes its properties.  
\begin{proposition}
For any $\lambda\leqs c_L d^{-8.1}$, we can choose $\rho>0$ such that $\mb{P}_{\vB}$ constructed in \eqref{eq:joint-x-y-B} satisfies the following two properties:  
\begin{itemize}
	\item[$\mathbf{(i)}$] 
 for any $\vB \in \mathbb{O}(p,d)$,  $Y$ can be represented as $f(\vB^{\top}\vX,\epsilon)$ where $\epsilon \sim N(0,1)$ and $\mb{P}_{\vB}$ belongs to  $\mathfrak{M}\left(p,d,\lambda\right)$; 
	\item[$\mathbf{(ii)}$] There is  some absolute constant $C$ such that 
 for any $\vB,\widetilde{\vB}\in \mathbb{O}(p,d)$, 
 \begin{equation}\label{eq:lower-KL-bound}
 \kl( \mb{P}_{\vB}, \mb{P}_{\widetilde{\vB}})\leqslant C\lambda\|\vB-\widetilde{\vB}\|_{F}^{2}. 
 \end{equation}
\end{itemize} 
\end{proposition}

We are now ready to present the proof sketch for the minimax lower bound. 
We can show that for any sufficiently small  $\varepsilon>0$ and any $\alpha \in (0,1)$, there is a subset $\Theta\subset \mathbb{O}(p,d)$ such that $|\Theta| \geqslant \left( C_{0}/\alpha  \right)^{d(p-d)}$ and $\|\vB-\widetilde\vB\|_{\mathrm{F}} \leqslant 2 \varepsilon$, $
\|\vB\vB^{\top} -\widetilde{\vB}\widetilde{\vB}^{\top}\|_F \geqslant  \alpha\varepsilon$ 
for any $\vB, \widetilde{\vB}\in \Theta$, where  $C_{0}$ is an absolute constant. 
Therefore, the class of distributions $\{\mb{P}_{\vB}:\vB\in \Theta\}$ are separated from each other in terms of $\vB\vB^{\top}$ and are close to each other in terms of KL-divergence, which allows the application of the classical Fano's method for obtaining minimax lower bounds.

Although we follow the standard Fano method framework, our proof is very different from all previous minimax lower bounds developed in related problems \citep{raskutti2011minimax,cai2013sparse,gao2014minimax,lin2021optimality}. 
The main technical difficulty in proving Theorem~\ref{thm:lower-smallp} lies in upper bounding the KL-divergence between any two joint distributions, $\bbP_{\vB}$ and $\bbP_{\widetilde{\vB}}$, sharply by $\lambda \|\vB-\widetilde{\vB}\|^{2}_{F}$. 
For the construction in \eqref{eq:joint-x-y-B}, the relationship between $Y$ and $\vX$ is nonlinear (otherwise $d=1$). Consequently, the KL-divergence does not have any closed-form expression, and it depends on the parameter $\lambda$ in a complicated way. 
Therefore, we cannot use the arguments for minimax lower bounds based on closed-form expressions as in the previous studies. 
Furthermore, if we instead make use of the basic property of KL-divergence that $ \kl (\bbP_{\vB}(\vX, Y), \bbP_{\widetilde{\vB}}(\vX, Y))\leqs\kl (\bbP_{\vB}(\vX, \vZ, Y), \bbP_{\widetilde{\vB}}(\vX, \vZ, Y))$, the bound is not as sharp as \eqref{eq:lower-KL-bound} and the resultant minimax lower bound will be at the order of $\frac{p}{n\lambda}$, which does not capture the linear dependence on the structural dimension $d$. 
To help readers further understand the differences between our proof of Theorem~\ref{thm:lower-smallp} and the proof of the most related result in \cite{lin2021optimality}, we provide a detailed comparison in Appendix~\ref{sec:comparison-2021}.

To prove the sharp upper bound in \eqref{eq:lower-KL-bound}, we have carefully analyzed the relationship between the KL-divergence and the difference $\Delta\vB=\tilde{\vB}-\vB$. 
By the construction in \eqref{eq:joint-x-y-B}, we have $\bbP_{\vB}( Y \mid \vX, W)=\bbP_{\widetilde{\vB}}(Y\mid \vX, W)$ a.s. and we can use basic properties of KL-divergence to show that
\begin{align*}
	\kl(\bbP_{\vB}, \bbP_{\widetilde{\vB}}) \leqslant  &\;  	\kl(\bbP_{\vB}, \bbP_{\widetilde{\vB}})  + \E_{(\vX, Y)\sim \bbP_{\vB} }\left(
	\kl( \bbP_{\vB}( W \mid \vX, Y),  \bbP_{\widetilde{\vB}}(W\mid \vX, Y) 	\right)  \nonumber \\
	=	&	\,  \kl (\bbP_{\vB}(\vX, W, Y), \bbP_{\widetilde{\vB}}(\vX, W, Y)) \nonumber \\
	=&\,  \kl (\bbP_{\vB}(\vX, W), \bbP_{\widetilde{\vB}}(\vX, W)). 
\end{align*} 
Denoting by $A_i$'s the inverse images of $i$ under the mapping $\psi$, we can show that 
\begin{align*}
\kl (\bbP_{\vB}(\vX, W), \bbP_{\widetilde{\vB}}(\vX, W)) =& \E_{\vB}\left[ \log \left( \frac{  \bbP_{\vB}(\vZ\in A_W \mid \vX) }{ \bbP_{\widetilde{\vB}}(\vZ\in A_W \mid \vX) } \right) \right]. 
\end{align*}
Note that the conditional probability of $\vZ$ given $\vX$ under $\bbP_{{\vB}}$  depends on $\vX$ through ${\vB}^{\top} \vX$. Therefore, it suffices to analyze the function $h({\vB}^{\top} \vx):=\log \bbP_{{\vB}}(\vZ\in A_W \mid \vX=\vx)$ for any $\vx\in \mathbb{R}^p$, because the log ratio in the above expression is $\left. h(\vB^{\top} \vx + \vt) \right|_{\vt={\Delta\vB}^{\top}\vx }^{\bs{0}}$. This analysis involves a careful derivation of the second-order Taylor expansion, upper bounds on the derivatives, and estimates of $d$-dimensional Gaussian measures uniformly for most $\vx$. 
To bound the remainder term in the Taylor expansion, the parameter $\lambda$ (or equivalently $\rho$ in \eqref{eq:joint-x-y-B}) has to be sufficiently small. 
In order to determine how small $\lambda$ needs to be, we meticulously derive accurate estimates of each term involved. The details of the proof can be found in Appendix~\ref{sec:proof lower bound given K}.

\begin{remark}\label{rem:gSNR-joint}
The gSNR of the distribution constructed by \eqref{eq:joint-x-y-B} can be explicitly computed as  
\begin{equation}\label{eq:gSNR-joint-x-y-B}
\lambda_{d}=
2 d^{-1} \rho^2\left(\bbE[ \max_{i\in[d]} \left( |Z_{i}|\right)  1_{\|\vZ\|^2 \leqslant m}]  \right)^{2},
\end{equation}
where $\vZ$ is a $d$-variate standard normal vector. 
The expression \eqref{eq:gSNR-joint-x-y-B} will allow researchers to conduct simulation experiments to numerically examine the dependence of the estimation error on the gSNR and the structural dimension. See Section~\ref{sec:simulation-dependence} for such an example. 
\end{remark}

\subsection{Lower bounds for sparse high-dimensional models}

In cases where the dimension $p$ of the predictor is comparable to or larger than the sample size $n$, the minimax lower bound developed in Theorem~\ref{thm:lower-smallp} suggests that there is no SDR method consistent for estimating the central space. 
To overcome the curse of dimensionality, many SDR methods have been proposed to make use of the sparsity assumption that only a small proportion of the underlying parameter values are nonzero. 
It is thus of particular interest to determine the minimax lower bound in the sparse high-dimensional setting.

Specifically, we impose the sparsity assumption on the index matrix $\vB$ as follows. 
Denote by $\supp(\vB)$ the support of $\vB$:
$$
\supp(\vB)=\{j\in [p]:  \|\vB_{j, *}\| >0 \}, 
$$
and by $\|\vB\|_{0}$ the number of non-zero rows of $\vB$, i.e., $\|\vB\|_{0} =  | \supp(\vB) |$. 
We assume that \textit{$\vB$ is $\ell_0$-sparse}, i.e., $\|\vB\|_{0}\leqslant s$ for some integer $s$. 
Accordingly, we refine the definition of the model class $\mathfrak{M}\left(p,d,\lambda\right)$ in Definition~\ref{def:class-lower-bound} to the following class of sparse models:
\begin{align}\label{model:high:dim}
\mathfrak{M}_{s}\left(p,d,\lambda\right)=\mathfrak{M}\left(p,d,\lambda\right)\cap \left\{\text{distribution of }(\vX,Y=f(\vB^{\top}\vX,\epsilon))~ \vline ~\|\vB\|_0\leqslant s\right\}. 
\end{align}

We consider the estimation of the central space while allowing the parameters $(p,s,d)$ to grow and $\lambda$ to decay as $n$ increases. 
Particularly, in the high-dimensional setting, $p$ might be much larger than $n$ while $d$ and $s$ might grow at a slow rate in $n$.
Extending the result in Theorem \ref{thm:lower-smallp}, we establish a lower bound for the minimax risk for the high-dimensional sparse model~\eqref{model:high:dim}. 

\begin{theorem}[Lower bound for sparse models] \label{thm:lower-largep}	
There exist a universal positive constant $c_{L}$ such that if 
$2d<s$ and $\lambda\leqslant c_{L} d^{-8.1}$, then we have
	\begin{align}\label{eqn:high-dim:rate}
	\inf_{\widehat{\vB}} \sup _{\mathcal{M}\in \mathfrak{M}_{s}\left(p,d,\lambda\right)  }
 L(\widehat{\vB}, \vB)
	\gtrsim \min \left(1, \frac{ds+s\log(ep/s)}{n\lambda}\right).
	\end{align}
\end{theorem}

The proof of Theorem~\ref{thm:lower-largep} is based on a modification of the construction in \eqref{eq:joint-x-y-B} and the inequality in \eqref{eq:lower-KL-bound}. The details can be found in Appendix~\ref{app:rate-largep}.

Similar to Theorem~\ref{thm:lower-smallp}, Theorem~\ref{thm:lower-largep} characterizes the difficulty in high-dimensional central space estimation with diverging $d$ for sparse multiple-index models. 
The lower bound contains two terms, namely $ds/(n\lambda)$ and $s\log(ep/s)/(n\lambda)$. 
The first term of the lower bound appears in the oracle case where $\supp(\vB)$ is known and the estimation error is no less than the lower bound given in Theorem~\ref{thm:lower-smallp} with $p$ substituted by $s$. 
The second term of the lower bound is induced by the ignorance about $\supp(\vB)$, which has to be estimated based on observed data. 
Both terms can be large if the value of $\lambda$ is small, which is particularly the case when $d$ is large, as demonstrated in Section~\ref{sec:small-gSNR}. 
This result suggests that even though the sparsity assumption substantially reduces the order of the model complexity (the effective number of unknown parameters of interest) from $dp$ to $ds+s\log(p/s)$, the gSNR still plays an important role in estimation accuracy.

\section{Minimax optimal rates with diverging $d$}\label{sec:minimax rate}
As a byproduct, the novel minimax lower bounds developed in Section~\ref{sec:minimax lower bounds} allow us to study the minimax optimal rates for central space estimation in settings where the structural dimension $d$ is unbounded. 
In this section, we illustrate how to achieve this goal by establishing some upper bounds that match the rates of the minimax lower bounds. 
To facilitate this illustration, we will introduce a few additional conditions on the population.

\subsection{Upper bound and optimal rate}
In order to obtain a tight upper bound that matches the lower bound in Theorem~\ref{thm:lower-smallp}, we consider the following subclass of $\mathfrak{M}\left(p,d,\lambda\right)$ with more specific conditions. 

\begin{definition}\label{def:class-upper-bound}
Given three constants $\kappa\geq 1$, $K_0\geq 8$, and $M\geq 1$, 
define $\overline{\mathfrak{M}}\left(p,d,\lambda\right)$ as the class of distributions for $(\vX,Y)$ that satisfy the following conditions:
\begin{enumerate}
\item $\vX\sim N(0,\bS)$, where the covariance matrix $\bS \in \mathbb{R}^{p\times p}$ satisfies
$\|\bS\|\vee\|\bS^{-1}\|\leqslant M$.
    \item $Y=f(\vB\tp\vX,\epsilon)$, where $\epsilon\sim N(0,1)$ is independent of $\vX$ and $\vB\in \mathbb{R}^{p\times d}$ such that $\vB\tp \bS\vB=\vI_d$.
    \item The eigenvalues of SIR's candidate matrix $\bLambda=\Cov(\bbE[\vX\mid Y])$ satisfy
$$
\lambda\leqslant\lambda_{d}(\bLambda) \leqslant \lambda_{1}(\bLambda)\leqslant \kappa \lambda.
$$
 \item The central curve
		$\vm(y)=\bbE[\vX| y]$ satisfies the weak $(K_0 d, 32\kappa)$-SSC w.r.t. $Y$.
\end{enumerate}
\end{definition}

\begin{remark}
We comment that $\overline{\mathfrak{M}}\left(p,d,\lambda\right)$ is fairly large to cover distributions of our interest. 
The normality assumption on $\vX$ aligns with those in related works on PCA \citep{cai2013sparse}, CCA \citep{gao2014minimax}, and SIR \citep{lin2019sparse,tan2020sparse}. This assumption can be relaxed to the linearity condition together with the assumption that $\vX$ is a sub-Gaussian vector as stated in Lemma \ref{lem:key lemma in Lin under moment condition}. 
This relaxation would involve a more complicated argument from \cite{lin2018consistency}, which makes the proof unnecessarily tedious. 
Since this relaxation is not the main focus of this paper, we impose the normality assumption here to simplify the exposition.

Compared to the coverage condition in Assumption~\ref{as:linearity and coverge}, we also require $\lambda_{1}(\bLambda)\leqslant \kappa \lambda$ so that the condition number of $\bLambda$ is bounded by $\kappa$. 
A similar condition is also assumed in  \cite{cai2013sparse} for PCA problems. 
If the distribution of $(\vX,Y)$ lies in $\overline{\mathfrak{M}}\left(p,d,\lambda\right)$, 
then we have $\vZ:=\vB\tp \vX \sim N(0,\vI_{d})$ and  $\bbE[\vX|Y]=\bbE[\bbE[\vX|\vB\tp \vX]|Y]=\bbE[\bS \vB\vB\tp \vX |Y]=\bS\vB\bbE[\vZ |Y]$ by the law of total expectation. 
Therefore, the WSSC for the central curve $\vm(y)=\bbE[\vX\mid Y=y]$ is equivalent to the WSSC for $\vm_{z}(y):=\bbE[\vZ|Y=y]$ because $\vm(y)=\bS\vB\vm_{z}(y)$. 
Note that $\vm_{z}(y)$ is a $d$-dimensional curve. 
From Theorem \ref{thm:moment condition to sliced stable}, it is clear that under the mild assumption that $Y$ is continuous and $\vm_{z}(y)$ is Lipschitz continuous with Lipschitz constant $O(d)$, $\vm_{z}(y)$ satisfies the WSSC, so does $\vm(y)$. 
Therefore, $\overline{\mathfrak{M}}\left(p,d,\lambda\right)$ is a fairly large class of distributions.
\end{remark}

From the definitions, it is clear that $\mathfrak{M}\left(p,d,\lambda\right)$ contains $\overline{\mathfrak{M}}\left(p,d,\lambda\right)$. 
A bound on the minimax risk for the class $\overline{\mathfrak{M}}\left(p,d,\lambda\right)$ is given by the following result. 

\begin{theorem}[Upper bound]
	\label{thm:risk:oracle:upper:d} 

There are two constants $C_1$ and $C_2$ depending only on $(\kappa, K_0,M)$ such that whenever $dp+d^2\left( \log\left(nd\right)+ d \right)  < C_1 n \lambda<C_1e^p$, the SIR estimator $\widehat{\vB}$ defined in \eqref{eqn:Bhat} satisfies 
	\begin{align}\label{eqn:SIR-upper}
	 \sup _{\mathcal{M}\in \overline{\mathfrak{M}}\left(p,d,\lambda\right)  }\bbE_{\mathcal{M}} 
  L(\widehat{\vB}, \vB)
  \leq C_2 \frac{dp }{n\lambda}.
	\end{align}
\end{theorem}

The proof of this upper bound is a nontrivial refinement of the proof in \cite{lin2021optimality}.

Apart from the relaxation that we only require the central curve to satisfy WSSC rather than SSC, we consider more general covariance matrices than those considered in the previous work. 
This improvement is substantial and cannot be achieved using the original proof. 
The detailed proof can be found in the supplementary material, where we provide a simplified proof for the case $\bS=\bs{I}_p$ in Appendix~\ref{app:sir:low-d} and the complete proof for the general case in Appendix~\ref{app:ub-ld-general}.

Note that the minimax lower bound in Theorem~\ref{thm:lower-smallp} also applies for the class $\overline{\mathfrak{M}}\left(p,d,\lambda\right)$: as revealed by the proof of Theorem~\ref{thm:lower-smallp}, the distributions constructed in \eqref{eq:joint-x-y-B} belong to $\overline{\mathfrak{M}}\left(p,d,\lambda\right)$, and thus the proof there still applies if $\mathfrak{M}\left(p,d,\lambda\right)$ is replaced by $\overline{\mathfrak{M}}\left(p,d,\lambda\right)$. 
Combining this minimax lower bound with the upper bound in Theorem~\ref{thm:risk:oracle:upper:d}, we conclude the following minimax optimal rate for estimating the central space over $\overline{\mathfrak{M}}\left(p,d,\lambda\right)$ when the structural dimension $d$ is allowed to grow.

\begin{theorem}[Minimax optimal rate] \label{thm:rate-smallp}

Suppose $(\kappa,K_0,M)$ are fixed. 
There are two positive constants $C_1$ and $c_{L}$ such that if $2d<p$, $\lambda\leqslant c_{L} d^{-8.1}$, and $dp+d^2\left( \log\left(nd\right)+ d \right)  < C_1 n \lambda<C_1e^p$, then we have
	\begin{align}
	\inf_{\widehat{\vB}} \sup _{\mathcal{M}\in \overline{\mathfrak{M}}\left(p,d,\lambda\right)  }
 \bbE_{\mathcal{M}}
 L(\widehat{\vB}, \vB)
	\asymp  \frac{d(p-d)}{n \lambda}. 
	\end{align}
\end{theorem}

Theorem~\ref{thm:rate-smallp} characterizes the dependence of the minimax rate for estimating the central space on the parameters $p$, $n$, $\lambda$, and the structural dimension $d$.  In particular, the minimax rate is inversely proportional to $\lambda$ and is linear w.r.t. the structural dimension $d$ (provided that $d$ is negligible relative to $p$). 
We provide a numerical illustration of these relationships in Section~\ref{sec:simulation-dependence}.

\subsection{Optimal rate in high-dimensions}

For high-dimensional sparse multiple-index models, we consider the following subclass:
\begin{align}\label{model:sparse-model}
\overline{\mathfrak{M}}_{s}\left(p,d,\lambda\right)=\overline{\mathfrak{M}}\left(p,d,\lambda\right)\cap \left\{\text{distribution of }(\vX,Y=f(\vB\tp\vX,\epsilon))~ \vline ~\|\vB\|_0\leqslant s\right\}, 
\end{align}
which is also contained by $\mathfrak{M}_{s}\left(p,d,\lambda\right)$ defined in \eqref{model:high:dim}. 
An upper bound on the minimax risk of estimating the central space over $\overline{\mathfrak{M}}_{s}\left(p,d,\lambda\right)$ is given in the following theorem. 
\begin{theorem}[Upper bound for sparse models]
	\label{thm:risk:sparse:upper:d}
 
 There are two positive constants $C_1$ and $C_2$ depending only on $(\kappa, K_0, M)$ such that 
	\begin{align}\label{eqn:rate:sparse:rsik:s=q}
	\inf_{\widehat{\vB}} \sup_{\mathcal{M}\in \overline{\mathfrak{M}}_{s}\left( p,d,\lambda\right) } \bbE_{\mathcal{M}} 
 L(\widehat{\vB}, \vB)
 \leq C_2 \frac{ds  + s\log(ep/s) }{n\lambda} 
	\end{align}
	holds wherever  $d^2\left( \log(nd)+ d \right)+s \log(p/s)+ds<C_1n\lambda< C_1e^{s}.$
\end{theorem}

In order to obtain the upper bound in Theorem \ref{thm:risk:sparse:upper:d}, we introduce an aggregation estimator following the ideas in \cite{cai2013sparse} and \cite{lin2021optimality}. 
This estimator is constructed through sample splitting and aggregation. 
For simplicity, assume that there are $n=2Hc$ samples for some positive integer $c$ and that these samples are divided into two equal-sized sets. 
We denote by $\bLambda^{(i)}_{H}$ $(i=1,2)$ the SIR estimates of $\bLambda=\Cov(\bbE[\vX|Y])$ using the $i$-th set of samples. 
Denote by $\wh\bS^{(1)}$ the sample covariance matrix based on the first set of samples. 
Let $\mathcal{L}( s)$ be the collection of all subsets of $[p]$ with cardinality $s$. 
The aggregation estimator $\widehat{\vB}$ is constructed as follows:
\begin{itemize}
 	\item[$\mathbf{(i)}$]  	For each $L \in \mathcal{L}( s)$, let 
	\begin{equation}\label{estimator:sample general covariance}
	\begin{aligned}
	&\widehat{\vB}_{L}:= 
	\arg\max_{\vB} ~~~\Tr(\vB\tp
	{\bLambda}_{H}^{(1)}\vB)\\
	&\mbox{ s.t. } \vB\tp \wh\bS^{(1)} \vB = \vI_d  \mbox{ and } \supp(\vB)\subset L.
	\end{aligned}
	\end{equation}
	
	\item[$\mathbf{(ii)}$]  
 Define the aggregation estimator $\widehat{\vB}$ to be $\widehat{\vB}_{L^{*}}$ where 
	\begin{align*}
	L^{*}:= 	\arg\max_{L\in \mathcal{L}(s)}~~~\Tr(\widehat{\vB}_{L}\tp
	{\bLambda}_{H}^{(2)}\widehat{\vB}_{L}).
	\end{align*}
\end{itemize}	
\smallskip
We show that the risk of the aggregation estimator is bounded by the right-hand side of \eqref{eqn:rate:sparse:rsik:s=q}; see Appendix~\ref{app:sparse:risk:proof} for the special case where $\bS$ is known to be $\bs{I}$ and Appendix~\ref{app:high dimension upper bound general cov} for general unknown $\bS$.

We highlight the difference between Theorem~\ref{thm:risk:sparse:upper:d} and the upper bound for high-dimensional sparse models in \cite{lin2021optimality}.

In Theorem~\ref{thm:risk:sparse:upper:d}, we do not require the covariance matrix $\bS$ to be known, nor do we assume $\bS$ has any sparsity. 
In contrast, the previous work requires the row-sparsity of $\bS$ to derive upper bounds, and the proof there cannot be extended to the case we considered here.

The rate given by Theorem~\ref{thm:risk:sparse:upper:d} is optimal. 
Since the distributions constructed for proving Theorem~\ref{thm:lower-largep} also belong to $\overline{\mathfrak{M}}_{s}\left(p,d,\lambda\right)$, and thus the minimax lower bound in Theorem~\ref{thm:lower-largep} is also a minimax lower bound for the class $\overline{\mathfrak{M}}_{s}\left(p,d,\lambda\right)$. 
The following theorem summarizes the minimax optimal rate for the high-dimensional sparse model. 

\begin{theorem}[Optimal rate of sparse models] \label{thm:rate-largep}		

Suppose $(\kappa,K_0,M)$ are fixed. 
There are positive constants $C_1$ and $c_{L}$ such that if $2d<s$, $\lambda\leqslant c_{L} d^{-8.1}$, and $d^2\left( \log(nd)+ d \right)+s \log(p/s)+ds< C_1 n\lambda< C_1e^{s}$, then we have
	\begin{align}
	\inf_{\widehat{\vB}} \sup _{\mathcal{M}\in \overline{\mathfrak{M}}_{s}\left(p,d,\lambda\right)  }
 \bbE_{\mathcal{M}}
 L(\widehat{\vB}, \vB)
	\asymp \frac{ds+s\log(ep/s)}{n\lambda}.
	\end{align}
\end{theorem}

\section{Simulation studies}\label{sec:simulation}
This section presents simulation studies that illustrate the phenomena revealed by our theory, specifically, the fast decay rate of gSNR indicated by Theorem~\ref{thm:GP-gSNR-fast-decay} and the optimal rate for estimating the central space established in Theorem~\ref{thm:rate-smallp}. 
Throughout this section, the estimation error is defined using the general loss $L(\widehat{\vB}, \vB)$ introduced in Section~\ref{sec:minimax lower bounds}.

\subsection{Small gSNR}\label{sec:simulation small gSNR}

In this subsection, we examine several numerical examples to demonstrate that the gSNR of a general multiple-index model tends to be extremely small when $d$ is large.

We first present two exemplary five-index models to illustrate that SIR does not perform well when $d=5$, a puzzling phenomenon that has been well-documented in the literature \citep{ferre1998determining, lin2021optimality}. 
We point out that this poor performance is due to the small gSNR in each example.
Subsequently, we explore a scenario where the link function is randomly sampled from a GP, and we find that the estimated gSNR appears to decay exponentially with
increasing $d$. 
The theoretical development of Theorem~\ref{thm:GP-gSNR-fast-decay} is inspired by these empirical findings.

\subsubsection{Synthetic experiments}\label{sec:Synthetic experiments}

In the following experiments with $p=15$ and $d=5$, SIR performs poorly, and the gSNR is very close to $0$. 
Suppose $\vX= (X_{1},...,X_{15})\tp\sim N(0,\vI_{15})$ and $\epsilon \sim N(0,1)$, and consider two different models for $Y$:
\begin{align*}
M_1&:Y=X_{1}+\exp(X_{2})+\log(|X_{3}+1|+1)+\sin(X_{4})+\arctan(X_{5})+0.01*\epsilon;\\
\mc M_2&:Y=X_1^3+\frac{X_2}{(1+X_3)^2}+ \operatorname{sgn}(X_4)\log(|X_5 +0.02 |+5)+0.01*\epsilon.
\end{align*}

In both $\mc M_1$ and $\mc M_2$, the columns of $\vB$ are $\ve_1,\dots,\ve_5$, the first five standard basis vectors in $\R^{15}$. 
Table~\ref{tab:poor performance of SIR} shows the estimation error of the SIR method.

The sample size $n$ ranges in $\{10^3,10^4,10^5,10^6,10^7,8\cdot10^7\}$ and the number of slices $H$ ranges in $\{2,5,10,20,50,100,200,500\}$. 
Each entry of the table is the averaged estimation error of SIR estimates for a given pair of $n$ and $H$ across $100$ replications. 
The results indicate that whatever $H$ is chosen, the estimation error decays very slowly as $n$ increases  and remains large even when $n$ reaches $8\cdot10^7$. 

\begin{table}[!htbp]\small
	\begin{tabular}{|c|c|c|c|c|c|c|c|c|c|}
		\hline
		& $n$   & $H=2$ & $H=5$ & $H=10$  &  $H=20$ & $H=50$ & $H=100$ & $H=200$ & $H=500$ \\
		\hline
		\hline
		\multirow{5}{*}{$\mc M_1$} 
		& $10^3$  &  {6.350} &  {4.303} &  {4.821} &  {4.850} &  {4.668} &  {4.644} &  {4.641}&  {4.602} \\ 
		& & {\footnotesize(0.2122)} & {\footnotesize(0.0699)} & {\footnotesize(0.0604)} & {\footnotesize(0.0507)} & {\footnotesize(0.0492)} & {\footnotesize(0.0482)} & {\footnotesize(0.0534)} & {\footnotesize(0.0511)} 
		\\
		& $10^4$  &  {5.762} &  {3.505} &  {4.468} &  {4.479} &  {4.438} &  {4.511} &  {4.434}&  {4.564} \\ 
		& & {\footnotesize(0.2012)} & {\footnotesize(0.0510)} & {\footnotesize(0.0510)} & {\footnotesize(0.0520)} & {\footnotesize(0.0505)} & {\footnotesize(0.0571)} & {\footnotesize(0.0530)} & {\footnotesize(0.0468)}
		\\ 
		& $10^5$  &  {5.578} &  {3.247} &  {3.726} &  {3.548} &  {3.620} &  {3.681} &  {3.644}&  {3.958} \\ 
		& & {\footnotesize(0.2242)} & {\footnotesize(0.0703)} & {\footnotesize(0.0497)} & {\footnotesize(0.0459)} & {\footnotesize(0.0437)} & {\footnotesize(0.0375)} & {\footnotesize(0.0432)} & {\footnotesize(0.0566)} 
		\\
		& $10^6$  &  {5.735} &  {2.492} &  {2.806} &  {2.712} &  {2.907} &  {2.914} &  {3.112}&  {3.216} \\ 
		& & {\footnotesize(0.2314)} & {\footnotesize(0.0665)} & {\footnotesize(0.0422)} & {\footnotesize(0.0445)} & {\footnotesize(0.0435)} & {\footnotesize(0.0468)} & {\footnotesize(0.0452)} & {\footnotesize(0.0451)} 
		\\
		& $10^7$  &  {5.924} &  {1.646} &  {1.929} &  {1.877} &  {1.915} &  {1.912} &  {1.956}&  {2.068} \\ 
		& & {\footnotesize(0.2509)} & {\footnotesize(0.0683)} & {\footnotesize(0.0293)} & {\footnotesize(0.0260)} & {\footnotesize(0.0263)} & {\footnotesize(0.0271)} & {\footnotesize(0.0207)} & {\footnotesize(0.0240)} 
		\\[1pt]
		& $8\cdot10^7$  &  {6.004} &  \textbf{1.062} &  {1.832} &  {1.851} &  {1.829} &  {1.835} &  \textbf{1.802}&  {1.819} \\ 
		& & {\footnotesize(0.2301)} & {\footnotesize(0.0818)} & {\footnotesize(0.0220)} & {\footnotesize(0.0202)} & {\footnotesize(0.0243)} & {\footnotesize(0.0235)} & {\footnotesize(0.0233)} & {\footnotesize(0.0251)} 
		\\[1pt]
		\hline
		\multirow{5}{*}{$\mc M_2$} 
		& $10^3$  &  {8.027} &  {3.749} &  {3.709} &  {3.601} &  {3.549} &  {3.514} &  {3.628}&  {3.972} \\ 
		& & {\footnotesize(0.2795)} & {\footnotesize(0.0681)} & {\footnotesize(0.0426)} & {\footnotesize(0.0468)} & {\footnotesize(0.0456)} & {\footnotesize(0.0376)} & {\footnotesize(0.0404)} & {\footnotesize(0.0493)} 
		\\
		& $10^4$  &   {7.274} &  {2.795} &  {2.780} &  {2.866} &  {3.024} &  {3.123} &  {3.192} &  {3.243}  \\ 
 	&  &\footnotesize(0.2617) & {\footnotesize(0.0683)} & \footnotesize(0.0411) & \footnotesize(0.0458) & \footnotesize(0.0483) & \footnotesize(0.0453) &
\footnotesize(0.0399) & \footnotesize(0.0414) 	\\ 
		& $10^5$  &  {7.214} &  {2.324} &  {1.984} &  {1.963} &  {1.962} &  {1.984} & { 2.096}&  {2.375} \\ 
		& & 
\footnotesize(0.2909) &  \footnotesize(0.0604) & \footnotesize( 0.0218) & \footnotesize(0.0248) & \footnotesize(0.0240) & \footnotesize(0.0274) &
\footnotesize(0.0213) & \footnotesize(0.0374)	\\
		& $10^6$  &  {6.553} &  {1.859} &  {1.844}&  {1.840} &  {1.865} &  {1.857} &  {1.863}&  {1.828} \\ 
		& & 
\footnotesize(0.3084) & \footnotesize(0.0312) & \footnotesize(0.0221) & \footnotesize(0.0194) & \footnotesize(0.0171) & \footnotesize(0.0182) &
\footnotesize(0.0216) & \footnotesize(0.0249) 
		\\
		& $10^7$  &  {5.445} &  {1.659} &  {1.747} &  {1.722} &  {1.759} &  {1.837} &  {1.792}&  {1.824} \\ 
		& &
\footnotesize(0.2813) & \footnotesize(0.0413) & \footnotesize(0.0258) & \footnotesize(0.0299) & \footnotesize(0.0246) & \footnotesize(0.0177) &
\footnotesize(0.0226) & \footnotesize(0.0220) 
		\\
		& $8\cdot10^7$  &  {5.152} &  {1.239} &  \textbf{0.688} &  \textbf{0.821} &  {1.025} &  {1.298} &  {1.397}&  {1.652} \\ 
		& &
\footnotesize(0.2724) & \footnotesize(0.0627) & \footnotesize(0.0355) & \footnotesize(0.0444) & \footnotesize(0.0506) & \footnotesize(0.0492) &
\footnotesize(0.0480) & \footnotesize(0.0407) 
		\\
  [1pt]
		\hline
	\end{tabular}\caption{Estimation error of SIR under models $\mc M_1$ and $\mc M_2$ with each $(n, H)$ combination. 
 The two best combinations are highlighted in bold. The average and the standard error (in parenthesis) are based on 100 replications. }\label{tab:poor performance of SIR}
\end{table}

To further elucidate the poor performance of SIR, we check whether the estimated directions of SIR, denoted by $\wh\bbeta_i$ $(i=1,\dots,5)$, lie in the central space $\mathcal{S}$, which is the space spanned by $\ve_1,\dots,\ve_5$ in our example. 
If an estimated direction lies outside $\mathcal{S}$, its last 10 entries will not be all zeros. 
As a quantitative measure, we compute the mean squared value of each of the last 10 entries of $\wh\bbeta_i$ $(i=1,\dots,5)$, based on $100$ replications. 
We focus on model $\mc M_1$ and consider the cases where $n=8\cdot10^7$ and $H\in\{5,200\}$---these two configurations have the smallest estimation error among all choices of $n$ and $H$ in Table~\ref{tab:poor performance of SIR}. 
The results are shown in Figure~\ref{ figure, estimation directions bad simple}, where we can see the last two estimated directions ($\wh\bbeta_4$ and $\wh\bbeta_5$) could lie substantially outside the central space $\mathcal{S}$. 
This observation implies that SIR fails to provide a reasonable estimate of the central space when $d=5$ even if the sample size is as large as $8\cdot10^{7}$ and $H$ enumerates all reasonable choices. 
This observation is consistent with the high loss presented in Table~\ref{tab:poor performance of SIR}. 

\begin{figure}[!htbp]
	\centering
	\begin{minipage}{0.5\textwidth}
		\includegraphics[width=\textwidth]{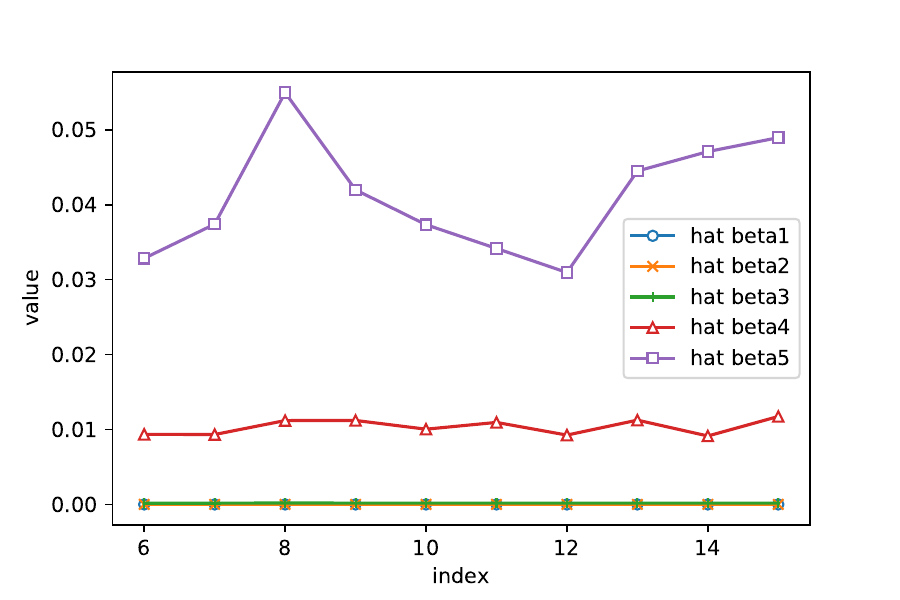}
	\end{minipage}\hfill
	\begin{minipage}{0.5\textwidth}
		\includegraphics[width=\textwidth]{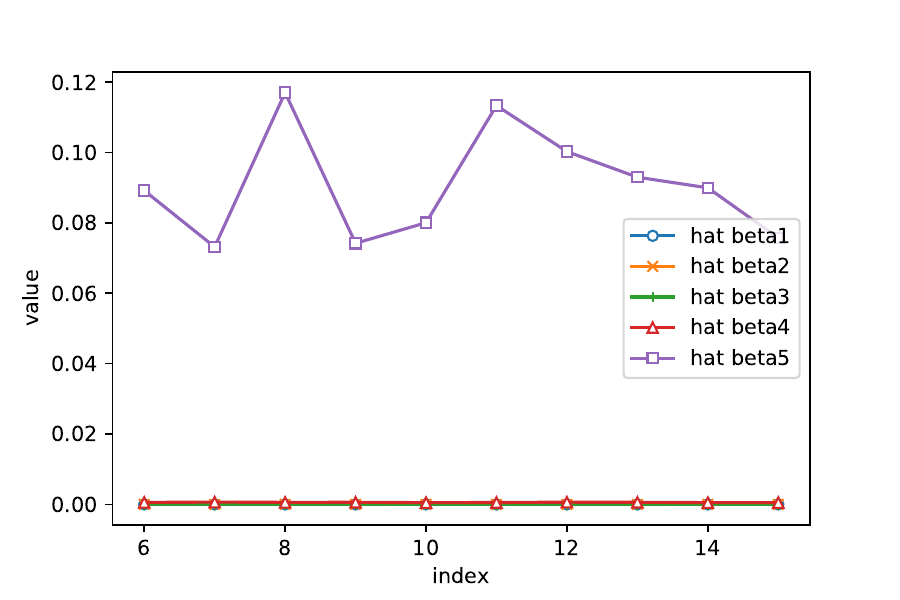}
	\end{minipage}\hfill
	\caption{Mean squared value of the last $10$ entries of estimated directions of SIR for model $\mc M_1$ with $n=8\cdot10^7$, $H=5$ (left) and $H=200$ (right).}
	\label{ figure, estimation directions bad simple}
\end{figure}

We proceed to examine the gSNR in these models by calculating the average of the logarithm of $\lambda_i(\wh\bLambda_H)$, the eigenvalues of the SIR estimate of $\Cov(\bbE(\vX\mid Y))$ with $n=8\cdot10^7$ (the largest sample size in this experiment). 
As seen in Table \ref{tab:eigvals}, the estimated gSNRs (i.e. $\lambda_5(\wh\bLambda_H)$) are very close to $0$ for both models $\mc M_1$ and $\mc M_2$.

\begin{table}[!htbp]
	\begin{tabular}{|c|c|c|c|c|c|c|c|c|c|}
		\hline
		 & index   & $H=2$ & $H=5$ & $H=10$  &  $H=20$ & $H=50$ & $H=100$ & $H=200$ & $H=500$ \\
		\hline
		\hline
		\multirow{5}{*}{$\mc M_1$} 
		& $1$  & -0.54 & -0.22 & -0.16 & -0.15 & -0.14 & -0.14 & -0.14& -0.14 \\ 
		
		& $2$  & -39.19 & -2.83 & -2.33 & -2.16 & -2.09 & -2.08 & -2.07& -2.07 \\ 
		
		& $3$  & -45.88 & -9.45 & -8.21 & -7.60 & -7.24 & -7.12 & -7.05& -7.00 \\ 
		
		& $4$  & -51.60 & -13.61 & -11.37 & -10.88 & -10.53 & -10.70 & -10.62& -10.48 \\ 
		
		& \textbf{gSNR}   & \textbf{-54.56} & \textbf{-41.31} & \textbf{-15.07} &  \textbf{-14.45} & \textbf{-13.71} & \textbf{-13.15} & \textbf{-12.56} & \textbf{-11.75} \\ 
		[1pt]
		\hline
		\multirow{5}{*}{$\mc M_2$}

& $1$  & -0.78 &-0.58 &-0.54 &-0.53 
&       -0.53 &-0.53 &-0.53 &-0.53\\
& $2$  &-40.11 &-1.72 &-1.43 &-1.39 
&       -1.38 &-1.37 &-1.37 &-1.37\\
& $3$  &-47.09 &-3.32 &-2.30 &-2.25 
&       -2.22 &-2.22 &-2.22 &-2.21 \\
& $4$  &-51.97 & -9.50 &-6.78  &-6.68 
&       -6.62 &-6.61 &-6.61 &-6.60\\
 & \textbf{gSNR}   &\textbf{-53.97}  &\textbf{-47.14} & \textbf{-14.52} & \textbf{-14.10} 
&      \textbf{-13.57}  &\textbf{-13.09} & \textbf{-12.54} &\textbf{-11.75}\\
		[1pt]
		\hline
	\end{tabular}
	\caption{Logarithm of the eigenvalues of $\wh\bLambda_H$ the SIR estimate of $\Cov(\bbE(\vX\mid Y))$ (averaged based on $100$ replications) under models $\mc M_1$ and $\mc M_2$. For each model, the $i$-th row corresponds to the $i$-th eigenvalue. 
 }\label{tab:eigvals}
\end{table}

\paragraph*{Gaussian process}
To further explore the decay of gSNR as $d$ increases, we study a general setting where the link function is a random continuous function sampled from a GP.

For each $d\in  \, \{1,2,3,4,5\}$, let $\vB=[\ve_1,\dots,\ve_d]$, where $\ve_i$ is the $i$-th standard basis vector of $\R^{15}$ and consider the following joint distribution of $(\vX,Y)$:
\begin{align}\label{eq: multiple index model GP}
\begin{split}
\vX&= (X_{1},...,X_{15})\tp\sim N(0,\vI_{15});\\
\mc M_3:Y&=f(\vB\tp\vX)+0.01*\epsilon,\quad\epsilon \sim N(0,1), 
\end{split}
\end{align}
where $f$ is a random function generated from the GP with mean function $\mu( \vx)=\bs{0}$ and covariance function $\Sigma(\vx,\vx^{\prime})=e^{-\frac{\| \vx- \vx^{\prime}\|^2}{2}}$.

We numerically explore the relationship between the gSNR for model $\mc M_3$ in \eqref{eq: multiple index model GP} and the structural dimension $d$. 
Specifically, for each value of $d$, we sample $f$ for $1,000$ times. 
For each sampled $f$, we draw a sample of $(\vX, Y)$ of size $n$ from model $\mc M_3$ and compute the estimated gSNR of $\mc M_3$ by $\lambda_{d}(\wh\bLambda_H)$, the $d$-th eigenvalue of the SIR estimate of $\Cov(\bbE(\vX\mid Y))$. 
Due to computational limitations, we set the maximum sample size at $50,000$. 
Here we present the result for $H=15$ (the results are not sensitive to the choice of $H$). Detailed sampling procedures and results for other values of $H$ can be found in Appendix~\ref{app:assitional-simulation}.

Figure \ref{figure, log plot of average gSNR} plots the average logarithm of the estimated gSNR over $1,000$ replications. 
We show the average  as a function of $n$ for various values of $d$ in the left subfigure and as a function of $d$ for various values of $n$ in the right subfigure. 
All of the associated standard errors are less than $0.005$.
Histograms of the estimated gSNR can be found in Appendix~\ref{app:Histogram of sample gSNR of GP}.
Based on the left subfigure, the estimated gSNR keeps decreasing as $n$ grows, indicating that it overestimates the true gSNR. 
The right subfigure shows that for a sufficiently large $n$, the estimated gSNR exhibits an exponential decay with respect to $d$ and becomes extremely small when $d=5$. 
The findings in this experiment reveal the rapid decay rate of the gSNR with increasing $d$ in realistic situations and motivate our theoretical development in Theorem~\ref{thm:GP-gSNR-fast-decay}.

\begin{figure}[H]
	\centering
	\begin{minipage}{0.5\textwidth}
		\includegraphics[width=\textwidth]{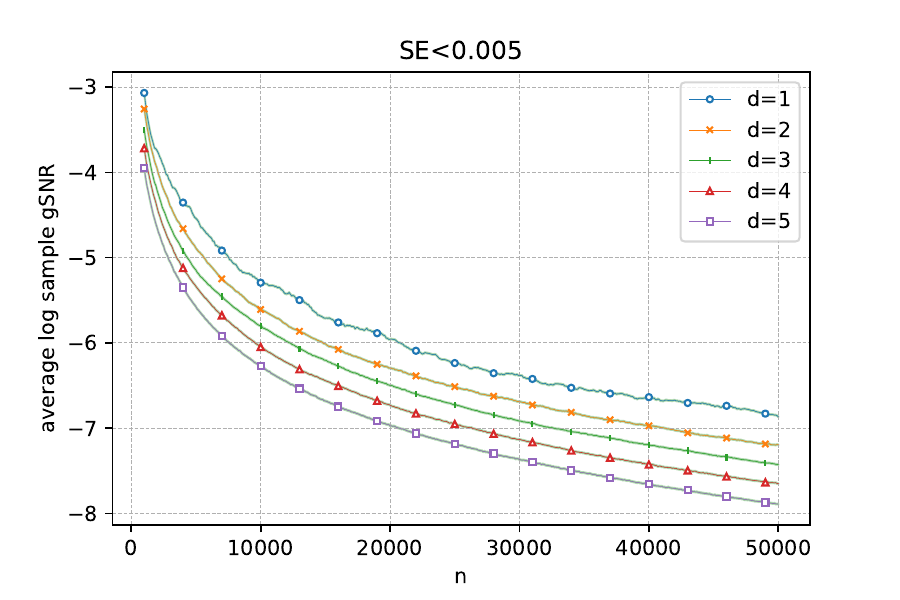}
	\end{minipage}\hfill
	\begin{minipage}{0.5\textwidth}
		\includegraphics[width=\textwidth]{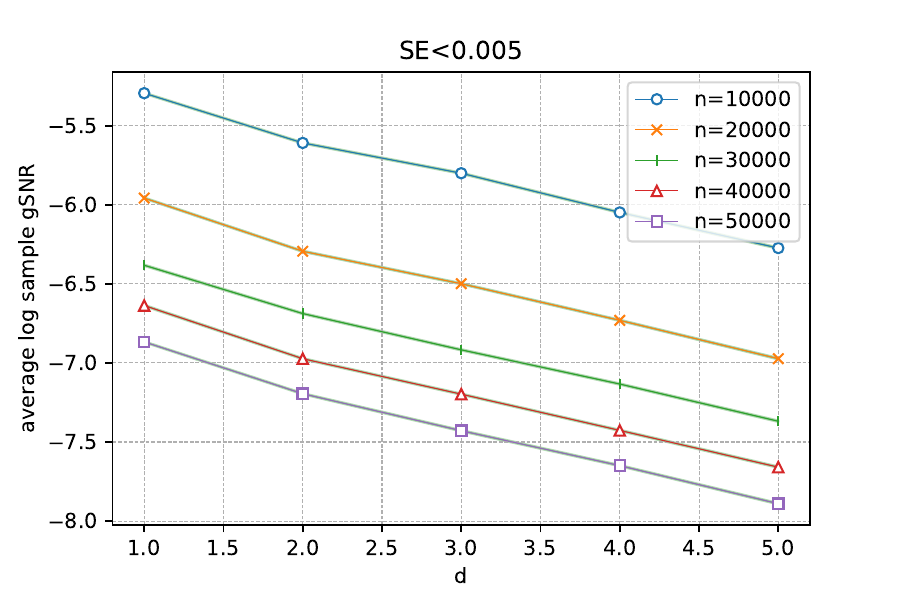}
	\end{minipage}\hfill
	\caption{Average logarithm of gSNR with increasing $n$ (left) and $d$ (right).}
	\label{figure, log plot of average gSNR}
\end{figure}

\subsection{Dependence of Estimation Error on $d$ and $\lambda$}\label{sec:simulation-dependence}
This subsection examines the dependence of the estimation error of SIR on the structural dimension $d$ and the gSNR $\lambda_d$. 
We observe that for various values of $d$ and $\lambda_d$, the averaged estimation error of SIR exhibits a linear relationship with $d$ and an inversely proportional relationship with $\lambda$. This is consistent with the theoretical result in Theorem~\ref{thm:rate-smallp}, where the minimax optimal rate for estimating the central space is achieved by SIR. 

In this experiment, we construct  $(\vX, Y)$ as per Equation~\eqref{eq:joint-x-y-B}, with $\vB$ chosen as  $\left[\bs{I}_{d}, \bs{0}_{d\times (p-d)} \right]\tp$, $\sigma=0.5$. 
We set $\rho=\theta ~\cdot~  \sqrt{ d } \left(\bbE[ \max |Z_{i}| 1_{\|\vZ\|^2 \leqslant m}]  \right)^{-1}$, where $\theta$ is a scaling factor such that $\rho$ lies in $(0,1)$, $\vZ$ is a $d$-variate standard normal random vector, and the expectation is computed numerically. 
Under such a specification of $\rho$, Remark~\ref{rem:gSNR-joint} implies two properties: 
\begin{enumerate}
    \item[1)] the gSNR remains constant if we fix the value of $\theta$, and 
    \item[2)] the gSNR  $\propto \theta^2$ if we fixed the value of $d$. 
\end{enumerate}
Therefore, we can illustrate the linear dependence of the estimation error on $d$ by varying the value of $d$ while fixing the triplets $(n,p,\theta)$. 
In addition, we can demonstrate that the estimation error is inversely proportional to $\lambda$ by varying the value of $\theta$ while fixing the triplets $(n,p,d)$.
The value of $d$ ranges in $\{2,4,6,8,10\}$, $\theta$ in $\{0.01,0.02,\dots,0.07\}$, and we fix $n=10^6$, $p=200$, and $H=1000$.

Figure~\ref{fig:error-with-increasing-d-good-example} shows the relationship between the estimation error of SIR and the varying factors $d$ and $\theta$ based on $100$ replications. 
In the left subplot, the solid line represents the averaged estimation error as $d$ increases (with a fixed $\theta=0.05$), which aligns perfectly with the dotted straight line fitted by least squares regression ($1-R^2<0.001$). The shaded areas represent the standard error associated with these estimates, which are all less than $0.003$. 
This plot indicates a linear dependence of the estimation loss on $d$. 
In the right subplot, the solid line plots the averaged logarithm of the estimation error against the logarithm of $\theta$ (with a fixed $d=10$). 
The dotted line is the straight line fitted by least squares regression, featuring a slope of $-2.021$ and $1-R^2<0.03$. 
This plot indicates that the estimation loss is inversely proportional to $\theta^2$, and thus to $\lambda$. 
The observations in this experiment are consistent with the theoretical result in Theorem~\ref{thm:rate-smallp}. 

\begin{figure}[H] 
	\centering
	\begin{minipage}{0.5\textwidth}
		\includegraphics[width=\textwidth]{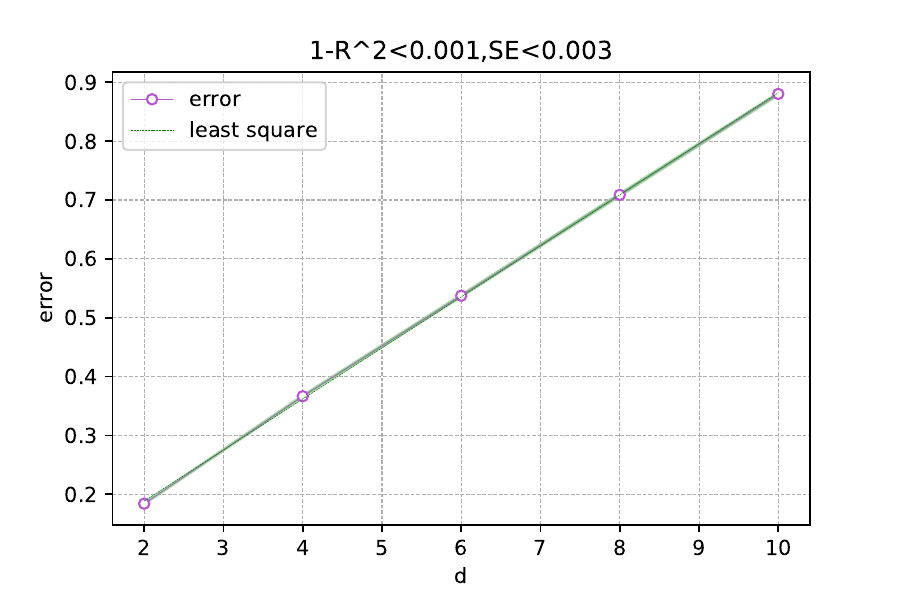}
	\end{minipage}\hfill
	\begin{minipage}{0.5\textwidth}
		\includegraphics[width=\textwidth]{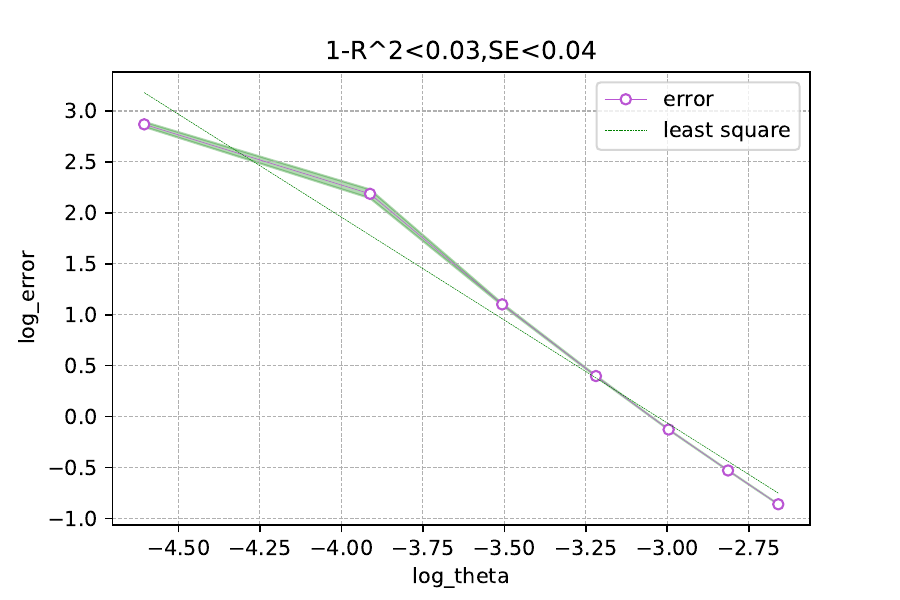}
	\end{minipage}\hfill
	\caption{Left: estimation error against $d$; Right: logarithm of estimation error against $\log(\theta)$.}
	\label{fig:error-with-increasing-d-good-example}
\end{figure}

\section{Discussions}
\label{sec:discussions}

In this paper, we study the performance of sliced inverse regression (SIR) for sufficient dimension reduction (SDR) under a multiple-index model where the structural dimension $d$ diverges. 
We demonstrate that the generalized signal-to-noise ratio (gSNR), defined as  $\lambda_{d}=\lambda_{d}(\Cov\left[\bbE(\boldsymbol{X}\mid Y)\right]$, tends to decay exponentially fast as $d$ grows, and we prove that the minimax risk for estimating the central space is lower bounded by $C\frac{dp}{n\lambda_d}$ for some constant $C>0$. 
Our results quantify the intrinsic difficulty of SDR in terms of $d$ and $\lambda_d$; in particular, to guarantee a prescribed accuracy level, the number of samples required grows rapidly with $d$. 
To our knowledge, this is the first theoretical explanation of the empirically observed poor performance of SIR when $d$ is large.

Our theory on the decay rate of gSNRs contains two parts. 
First, we prove that if the link function $f$ of the multiple-index model is sampled from a Gaussian process, then with high probability, the gSNR is bounded by $O(e^{-\theta d})$ for some $\theta>0$. 
Our proof involves a novel analysis of the conditional mean $\bbE(\vX\mid Y=y, f)$, which is related to the study of level sets of Gaussian processes. 
This result suggests that the gSNR generally tends to decay fast as $d$ grows, which has been empirically observed. 
Second, we prove that for any link function, the gSNR is bounded by $O(\log d/d)$, though this bound is rarely attained unless the link function is deliberately constructed.

Our minimax lower bounds are very different from existing ones. 
The result in \cite{lin2021optimality} has certain limitations since it requires the structural dimension $d$ to be bounded above and the gSNR to be bounded away from $0$. 
In contrast, our results not only allow $d$ to grow but also allow the gSNR to approach $0$. 
Besides, our results are also different from those of \cite{tan2020sparse}, which considered Gaussian mixture models and required $d$ to be fixed. 
Compared to previous results, our minimax lower bounds highlight the crucial roles of the structural dimension and the gSNR in SDR problems.

Our introduction of the weak sliced stable condition (WSSC) also contributes to the theoretical development of SDR. 
This condition simplifies our proofs for both the upper and lower bounds of the minimax rates presented in this paper. 
Furthermore, we expect that deriving the WSSC from the moment condition in Theorem~\ref{thm:moment condition to sliced stable} would inspire future research on the high-dimensional behavior of other SDR methods, such as SAVE, without the need for higher-order sliced stability conditions.  
For instance, if one intends to study the phase transition phenomenon of SAVE in high dimensions, as was done for SIR in \cite{lin2018consistency}, 
it might initially seem that a higher-order sliced stability conditions is indispensable. 
This speculation stems from the observation that in low-dimensional settings, the asymptotic theory of SAVE developed by \cite{li2007asymptotics} required several conditions similar to those proposed for SIR by \cite{hsing1992asymptotic,zhu1995asymptotics}, but with a higher order. 
However, our results shed light on another possibility: the higher-order sliced stability condition could potentially be replaced by the existence of $\ell>2$ moments of predictors.

Our findings raise several open questions. 
\begin{itemize}
    \item 
Our theory focuses on the SIR method, and our notion of gSNR is related to $\Cov\left(\bbE[\vX \mid Y]\right)$.
When studying other SDR methods, it might be possible to consider alternative definitions of gSNR in the corresponding contexts and establish the minimax optimality. 
In particular, if the conditional mean $\E(Y \mid \vX)$ is of interest, it is beneficial to consider estimating the central mean subspace rather than the central space since it provides a further reduction of dimension. 
As discussed in Remark~\ref{rem:other-SDR}, the theoretical analysis for estimating the central mean subspace is very different from the one in the current paper. 
To the best of our knowledge, there is no minimax theory developed for estimating the central mean subspace with a diverging structural dimension. 
Whether or not the theory developed in the current paper can be extended to the estimation of the central mean subspace is open for future exploration.

\item 
In our proof of the upper bound on the minimax risk for high-dimensional sparse models, we construct the aggregation estimator primarily for theoretical purposes rather than computational efficiency. 
From a practical perspective, it is beneficial to develop computationally efficient algorithms that can achieve or nearly achieve the optimal rate.

\item 
To reveal the intrinsic difficulty of SDR with a large structural dimension $d$, our theory presumes that $d$ is known, which allows us to focus on studying the impact of $d$ on the gSNR and the estimation error. 
In real practice, the value of $d$ has to be inferred from the data, which is known as order determination and has been widely studied in the literature (see \citet[Chapter 9]{liSufficient2017} and references therein). 
Our minimax lower bound continues to hold when $d$ is unknown, but it remains to be explored whether there is an even larger lower bound in this case.

\end{itemize}

%% file: Appendix.tex
\section{Proof of Lemma \ref{lem:key lemma in Lin under moment condition}}
\begin{proof}
	It is a direct corollary of Lemma 1 ({\it `key lemma'}) in \cite{lin2018consistency} by noticing that  factor $\frac{\gamma_3}{H^\nu}$ in the proof of Lemma 2(i) therein  corresponds to factor $\frac1\tau$ in Definition \ref{def:weak-sliced-stable}. 
\end{proof}

\section{Proof of Theorem \ref{thm:moment condition to sliced stable} and Corollary \ref{cor:moment to WSSC}}

We first introduce the following lemma, which shows that the  WSSC is easy to satisfy for general curves.
\begin{lemma}\label{lem:moment condition to sliced stable, continuous random variable}
	Let $Y\in\R$ be a random variable and $\boldsymbol{\kappa}:~\R\rightarrow \R^p$  be a nonzero continuous function satisfying 
	\begin{itemize}
		\item[(i)]$\sup_{\|\bbeta\|=1}\mathbb{E}[|\langle\boldsymbol{\kappa}(Y),\bbeta\rangle|^{\ell}]\leqslant c_{1}$ holds for some $\ell>2$ and $c_{1}>0$;
		\item[(ii)] $Y$ is a continuous random variable.
	\end{itemize}
	Then for any $\tau>1$,  there exists  a constant $K=K(\tau,d')\geqslant d'$ such that $\bs{\kappa}(y)$ satisfies the weak $(K,\tau)$-sliced stable condition w.r.t. $Y$ where $d':=\mr{dim}\{\mr{span}\{\bs\kappa(y):y\in \R \}\}$.

If we further assume that $\|\bs\kappa(y)-\bs\kappa(y')\|\leqslant C\zeta|y-y'|$ 
 for any $y,y'$ defined on $\R$ 
 and  $C>0$  a constant, then the WSSC coefficient $K= \lceil K_0\zeta\rceil$ for some absolute constant $K_0$.

\end{lemma}
By the density transformation formula of continuous random variables, it is easy to check that there exists a monotonic function $f$ such that the (probability density function) p.d.f. of $f(Y)$ is continuous and positive everywhere. Since any monotonic transformation $f(Y)$ keeps SIR procedures unchanged, (ii) in Lemma \ref{lem:moment condition to sliced stable, continuous random variable} can be replaced by
\begin{itemize}
    \item[(ii’)] The p.d.f. of $Y$ is continuous and  positive everywhere on $\R$.
\end{itemize}

\subsection{Proof of Lemma \ref{lem:moment condition to sliced stable, continuous random variable}}
To prove  Lemma \ref{lem:moment condition to sliced stable, continuous random variable}, we need some properties of   
 any  partition  $\mathcal{B}_{H}:=\{-\infty=a_0<a_1<\dots<a_{H-1}<a_H=\infty\}$ of $\R$. Before this, we introduce some important definitions.  For simplicity, we  let $U_h:=(a_{h-1},a_h]$.
\begin{definition}\label{def:gamma}
	Let $Y$ be a real random variable and $\gamma \in (0,1)$. A partition
	$\mathcal{B}_{H}:=\{U_h\}_{h=1}^H$
	of $\R$ is called a $\gamma$-partition w.r.t. $Y$ if 
	\[\frac{1-\gamma}{H}\leqslant \mathbb{P}(Y\in U_h)\leqslant\frac{1+\gamma}{H}\]
	for $h=1,\dots,H$.
\end{definition}
\begin{definition}\label{def:admissible} 
	Suppose $\delta>0$, $D\subset \mathbb{R}$ is  compact and $Y$ is a real random variable. 
	A  $\gamma$-partition  $\mathcal{B}_H$ w.r.t. $Y$ is $(\delta,D)$-admissible if there exists a compact set $D'\supset D$ with the following properties:
	\begin{itemize}
		\item [(1)]
		For any $U_h\in\mathcal{B}_H$ such that $U_h\bigcap D\neq\emptyset \empty$, then $U_h\subset D'$.
		\item[(2)]For any $U_h\in\mathcal{B}_H$ such that $U_h\subset D'$, then 
		$\mbox{diam}(U_h)<\delta$, where $\mbox{diam}(U_h)$ is the diameter of $U_h$.
	\end{itemize}
\end{definition}

Intuitively, when $\R$ is partitioned into sufficiently many intervals of the same probability mass of the distribution of $Y$, then the Euclidean lengths of these intervals within a specified compact set are all sufficiently small. Thus the partition  $\mathcal{B}_{H}$ is $(\delta,D)$-admissible. We shall illustrate this intuition by the following Lemma, which is an essential property of the  partition.
\begin{lemma}\label{lem:ddas}
	Suppose that $Y\in\R$ is a real random variable  and $p(y)$ is continuous and positive everywhere. Let $\{\mathcal{B}_{H}\}$ be a sequence of $\gamma$-partition.  For any $\delta>0$ and compact set $D$, there exists an $H_{0}=H_{0}(\gamma,\delta,D)\asymp \frac{1+\gamma}{\delta}$ such that for any $H>H_{0}$,  $\mathcal{B}_{H}$ is a $(\delta,D)$-admissible $\gamma$-partition.
\end{lemma}
\begin{proof}
	Let $D^{\prime}$ be a compact interval containing the set $\{x\mid \mbox{dist}(x,D)\leqslant 1/2\}$. Since $p(y)$ is positive and continuous over $D^{\prime}$, there exists a positive number $\mu$, such that for any interval $[a,b]\subset D^{\prime}$ satisfying $|b-a|\geqslant 1/2$, one has $\P(Y\in[a,b])\geqslant \mu$. Choose a real number $H_{1}>\frac{1+\gamma}{\mu}$. For any $H>H_{1}$, one has $\P(Y\in  U_h)< \mu$. Hence, if $  U_h\cap D\neq \emptyset$, we  must have $\mbox{diameter}( U_h)\leqslant 1/2$. This implies that $U_h\subset D^{\prime}$.  Let $\mu^{\prime}:=\min_{Y\in D^{\prime}} p(Y)$. As  $D^{\prime}$ is compact, our assumption about $p(Y)$ implies that $\mu^{\prime}$ is well-defined and $\mu^{\prime}>0$.
	Let $H_{2}\in \mb{Z}_{\geqslant1}$ satisfying $\frac{1+\gamma}{H_{2}\mu^{\prime}}<\delta$. For any $H>H_{2}$, let  $a_{h}<b_{h}$ be two endpoints of $U_h\subset D^{\prime}$, one has $\mu^{\prime}(b_{h}-a_{h})\leqslant \P(Y\in U_h)\leqslant \frac{1+\gamma}{H}$. Then we obtain
	$\mbox{diam}( U_h)=b_{h}-a_{h}\leqslant \frac{1+\gamma}{H\mu^{\prime}}< \delta$. Choose $H_{0}=\max\{H_{1},H_{2}\}$, then for any $H>H_{0}$, one has $\mathcal{B}_{H}$  is  $(\delta,D)$-admissible.
\end{proof}

\textbf{Proof of Lemma \ref{lem:moment condition to sliced stable, continuous random variable}:}
Now we are ready to prove Lemma \ref{lem:moment condition to sliced stable, continuous random variable} under conditions (i) and (ii'). We first prove that for any $\bbeta\in \mb{S}^{p-1}$, 
\begin{align}\label{eq:epsilon sliced stable initial version}
\frac{1}{H}\sum_{h:U_h\in\mathcal{B}_{H}}\mathrm{var}(\boldsymbol{\kappa}(\bbeta)\mid Y\in U_h)\leqslant\frac1\tau\lambda_{\min}^+(\mathrm{Cov}(\boldsymbol{\kappa}(Y)))
\end{align}
where $\boldsymbol{\kappa}(\bbeta):=  \bbeta^\top\boldsymbol{\kappa}(Y)$.

For any $\tau>0$, let us choose a compact set $D$ such that 
$\P(Y\in D^c)<\epsilon_{1}^{\frac{\ell}{\ell-2}}$,
where $\epsilon_{1}=\frac{(1-\gamma)\lambda_{\min}^+(\mathrm{Cov}(\boldsymbol{\kappa}(Y)))}{{\tau\sup_{\|\bbeta\|=1}\E[|\langle\boldsymbol{\kappa}(Y),\bbeta\rangle|^\ell]}^{2/{\ell}}}$. 
Let $\delta$ be some small positive constant, say, $0.2$.
By Lemma \ref{lem:ddas}, we know that there exists an $H_0$, such that for any $H>H_0$, the  partition $\mathcal{B}_{H}$ is a $(\delta, D)$ admissible $\gamma$-partition. This means 
there exists a compact set $D^{\prime}$,  such that if $U_h\cap D\neq\emptyset$, then  $U_h\subset D'$ and $\mbox{diam}(U_h)<\delta$.   Since $\boldsymbol{\kappa}$ is uniformly continuous on $D'$, for $\epsilon_{2}=\sqrt{\frac{\lambda_{\min}^+(\mathrm{Cov}(\boldsymbol{\kappa}(Y)))}{\tau}}$, there exists an $\epsilon'$ such that $\|\boldsymbol{\kappa}(y_{1})-\boldsymbol{\kappa}(y_{2})\|\leqslant \epsilon_{2}$ for any $y_{1},y_{2}\in D'$ satisfying that $|y_{1}-y_{2}|<\epsilon'$. Let $\delta'=\min\{\delta,\epsilon^{\prime}\}$.  By Lemma \ref{lem:ddas}, there exists an $H_0^{\prime}$, such that if $\mathcal{B}_{H}$ is a $\gamma$ partition, then it is $(\delta^{\prime}, D)$ admissible for any $H>H_0^{\prime}$.

In particular, when it holds that $\|\bs\kappa(y)-\bs\kappa(y')\|\leqslant C\zeta|y-y'|$, then we can ensure that $\epsilon'\asymp \frac1{\zeta}$ and $H_0'\asymp \zeta$.

 For such $(\delta^{\prime}, D)$ admissible partition $\mathcal{B}_{H}$, one has 
\begin{itemize}
	\item[1] For any $U_h\cap D\neq \emptyset$ and any  $\bbeta\in \mb{S}^{p-1}$, by intermediate value theorem, we know that there exists a $\xi\in U_h$, such that 
	\[\int_{U_h}\boldsymbol{\kappa}(\bbeta)p(y\mid Y\in U_h)dy=\bbeta^\top\boldsymbol{\kappa}(\xi).\]
	Thus, one has 
	\begin{align*}
	\mathrm{var}(\boldsymbol{\kappa}(\bbeta)\mid Y\in U_h)&=\int_{U_h}(\boldsymbol{\kappa}(\bbeta)-\bbeta^\top\boldsymbol{\kappa}(\xi))^2p(y\mid Y\in U_h)dy  
	\leqslant \epsilon_{2}^2=\frac{\lambda_{\min}^+(\mathrm{Cov}(\boldsymbol{\kappa}(Y)))}{\tau}.
	\end{align*}
	\item[2]For any $U_h\subset D^c$ and any  $\bbeta\in\mb{S}^{p-1}$, one has 
\begin{align*}	&\mathrm{var}(\boldsymbol{\kappa}(\bbeta)\mid Y\in U_h)\leqslant \int_{U_h}(\boldsymbol{\kappa}(\bbeta))^2p(y\mid Y\in U_h)dy\leqslant \frac{H}{1-\gamma}\int_{U_h}(\boldsymbol{\kappa}(\bbeta))^2p(y)dy
\end{align*}
because $\gamma\leqslant\frac{1}{\tau}$. 
Thus, 
\[
\frac{1}{H}\sum_{h:U_h\subset D^c}\mathrm{var}(\boldsymbol{\kappa}(\bbeta)\mid Y\in U_h)
\leqs \frac{1}{1-\gamma}\int_{D^c}(\boldsymbol{\kappa}(\bbeta))^2p(y)dy. 
\]

Let $f(y)=\kappa(\bbeta)^2$ and $q(y)=p(y) / \bbP(Y\in D^c)$. By Jensen's inequality, we have $\int_{D^c} f(y)q(y)dy\leqs 
 \left(\int_{D^c}f^{\ell/2}(y) q(y) dy\right)^{2/\ell}$. This can be written as 
 \begin{align*}	
 \int_{D^c}(\boldsymbol{\kappa}(\bbeta))^2p(y)dy \leqs  \left(\int_{D^c} (\kappa(\bbeta) )^{\ell} p(y) dy\right)^{2/\ell} \bbP(Y\in D^c)^{1-2/\ell}, 
\end{align*}
which is bounded by $$\sup_{\|\bbeta\|=1}\E[|\langle\boldsymbol{\kappa}(Y),\bbeta\rangle|^\ell]^{2/\ell}\epsilon_{1}=\frac{1}{\tau}\lambda_{\min}^+(\mathrm{Cov}(\boldsymbol{\kappa}(Y))).$$
\end{itemize}
Then
\begin{align*}
&\frac{1}{H}\sum_{h:U_h\in\mathcal{B}_{H}}\mathrm{var}(\boldsymbol{\kappa}(\bbeta)\mid Y\in U_h)=\frac{1}{H}\sum_{h:U_h\cap D\neq\emptyset}\mathrm{var}(\boldsymbol{\kappa}(\bbeta)\mid Y\in U_h) +\frac{1}{H}\sum_{h:U_h\subset D^c}\mathrm{var}(\boldsymbol{\kappa}(\bbeta)\mid Y\in U_h)\\
\leqslant& \frac{1}{H\tau}N_{1}\lambda_{\min}^+(\mathrm{Cov}(\boldsymbol{\kappa}(Y)))+\frac{1}{\tau}\lambda_{\min}^+(\mathrm{Cov}(\boldsymbol{\kappa}(Y)))\leqslant\frac{2}{\tau}\lambda_{\min}^+(\mathrm{Cov}(\boldsymbol{\kappa}(Y)))
\end{align*}
where $N_{1}$ is the number of $U_h$ such that $U_h\cap D\neq\emptyset$.
This completes the proof of \eqref{eq:epsilon sliced stable initial version}. 

Note that for any $\bbeta\in\mr{col}(\mathrm{Cov}(\boldsymbol{\kappa}(Y))),\lambda_{\min}^+(\mathrm{Cov}(\boldsymbol{\kappa}(Y)))\leqslant\bbeta\tp\mathrm{Cov}(\boldsymbol{\kappa}(Y))\bbeta$. In the case when $\mathrm{var}(\bbeta\tp\boldsymbol{\kappa}(Y))=0$, it holds that $\mathrm{var}(\bbeta\tp\boldsymbol{\kappa}(Y)|Y\in U_h)=0 (\forall U_h)$.  Thus, the proof is completed.

\qed

\subsection{Proof of Theorem \ref{thm:moment condition to sliced stable}}
\begin{proof}
Note that $\langle \vm(y),\bbeta\rangle = \E[\langle\vX,\bbeta\rangle|Y]$. By Jensen's inequality for conditional expectation, one has
	\begin{align*}	\bbE[|\E[\langle\vX,\bbeta\rangle|Y]|^\ell]\leqslant\E[\E[|\langle\vX,\bbeta\rangle|^\ell|Y]]=\E[|\langle\vX,\bbeta\rangle|^\ell]\leqslant c_1.
	\end{align*}
	The proof is then completed by Lemma \ref{lem:moment condition to sliced stable, continuous random variable}. 
\end{proof}
\subsection{Proof of Corollary \ref{cor:moment to WSSC}}
It is a direct corollary of Theorem \ref{thm:moment condition to sliced stable} by noticing that the following result.
	\begin{lemma}[Lemma 11 in \cite{lin2018supplement}]\label{lem:epfs}
		For any  sufficiently large $H,c$ and  $n>\frac{4H}{\gamma}+1$, $\mathfrak{S}_{H}(n)$ is a $\gamma$-partition with probability at least 
		\[1-CH^2\sqrt{n+1}\exp\left(-\frac{\gamma^2(n+1)}{32H^2}\right)\]
		for some absolute constant $C$.
	\end{lemma}

\section{Proofs of results in Theorem~\ref{thm:max-smallest-eigen}}\label{app:eigen-upper}

To illustrate the upper bound in Theorem~\ref{thm:max-smallest-eigen} is tight, consider a function $\psi^{0}(\bs{z}=(z_{1},\ldots, z_d))$ from $\R^{d}$ to $\{-d, \ldots, 0,  1, \ldots, d\}$ such that if $|z_{i}|$ is uniquely the largest among $|z_i|$'s, then $\psi^{0}(\bs{z})=\sgn(z_i) i$; otherwise, $\psi^{0}(\bs{z})=0$. 
Let $\vB=\left[\bs{e}_1,\ldots, \bs{e}_{d}\right]$ whose columns are the first $d$ standard basis vectors in $\R^{p}$. 
We construct the following joint distribution of $(\vX,Y)$:
\begin{align} \label{eq:joint-basic}
\vX ~ & ~ \sim N(0, \bs{I}_{p}),\\ \nonumber
Y~ & ~ = \psi^{0} (\vB\tp \vX) +  \eta, \quad \eta \sim \text{Unif}(-1/2,1/2), \nonumber
\end{align}
where $\vX$ and $\eta$ are independent. This distribution satisfies the WSSC and its gSNR decays to $0$ at the rate of $\frac{\log d}{d}$ as $d$ tends to $\infty$.

The gSNR of the distribution in \eqref{eq:joint-basic} equals to $d^{-1}\left( \bbE[\max |Z_i| ] \right)^2$, which is at the same order as $\frac{\log(d)}{d}$. 

A proof of this statement is similar to the proof of Lemma~\ref{lem:joint-eigen} in Appendix~\ref{app:pf-prop:joint-class} and is omitted here.

Theorem~\ref{thm:max-smallest-eigen} relies on the following Theorem~\ref{thm:eigen-bound-ent}, which may be of independent interest. It provides an upper bound on the smallest eigenvalue of the covariance matrix of the conditional expectation of a normal random vector given any discrete random variable. 

\begin{theorem}\label{thm:eigen-bound-ent}
Suppose	$\vZ\sim N(0, \bs{I}_{d})$ and $ W $ is a discrete random variable whose probability mass function is smaller than $1/2$. The entropy of $ W $ is defined as $\sum_{ w }\bbP( W = w ) \log \frac{1}{\bbP( W = w )}$ and is denote by $\Ent( W )$. 
It holds that 
$$
	\lambda_{\min}\left\{ \Cov\left[ \bbE \left( \vZ \mid  W  \right) \right] \right\}\leqslant 37 ~ d^{-1} \Ent ( W ). $$
 In particular, if the support of $W$ has $K$ elements, we have $
	\lambda_{\min}\left\{ \Cov\left[ \bbE \left( \vZ \mid  W  \right) \right] \right\}\leqslant 37 ~ d^{-1} \log K. $
\end{theorem}

To prove Theorem~\ref{thm:max-smallest-eigen}, we need the following proposition in addition to Theorem~\ref{thm:eigen-bound-ent}. 

\begin{proposition}\label{prop:ssc-discrete-approx}
Suppose $\vZ$ is a $d$-random vector, $Y$ is a random element. 
Suppose $\{S_{h} : h=1,\ldots, H\}$ is a partition of the range of $Y$ such that for some $\gamma >0$ and $\tau> 1+\gamma $, $\bbP(Y\in S_{h})\leqslant (1+\gamma)  H^{-1}$ ,  and for any $\bbeta\in \mb S^{d-1}$, 
\begin{align}
\frac{1}{H}\sum_{h=1}^{H}\var\left(\bbeta\tp \bbE \left( \vZ \mid Y \right) \big| Y\in S_h \right)\leqslant \frac{1}{\tau} \var\left(\bbeta\tp \bbE \left( \vZ \mid Y \right) \right).
\end{align}
Let $W=\sum_{h=1}^{H} h \one_{Y\in S_h}$. It holds that for any $\bbeta\in  \mb S^{d-1}$, 
\begin{align}
\left(1- \frac{1+\gamma}{\tau} \right) \var\left(\bbeta\tp \bbE \left( \vZ \mid Y \right)  \right)\leqslant \var\left(\bbeta\tp \bbE \left( \vZ \mid W \right) \right). 
\end{align}
\end{proposition}

\begin{proof}[Proof of Theorem~\ref{thm:max-smallest-eigen}]
Without loss of generality, we assume $K\geq 4$. 
Let $\vZ=\vB^{\top}\vX$. It is clear that $\vZ\sim N(0,\bs{I}_{d})$ and $\bbE \left( \vX \mid  Y  \right) =\bbE \left[\bbE \left( \vX \mid  Y,Z  \right)\mid Y\right] = \bbE \left( \vB\vZ \mid  Y  \right)=$. Therefore, we have $\lambda_{d}( \Cov\left[ \bbE \left( \vX \mid  Y  \right) \right] ) =\lambda_{d}( \vB\Cov\left[  \bbE \left( \vZ \mid  Y  \right) \right] \vB^\top )=\lambda_{d}( \Cov\left[  \bbE \left( \vZ \mid  Y  \right) \right])$ since $\vB^{\top}\vB=\bs{I}_d$. Therefore, we can focus on $\Cov\left[  \bbE \left( \vZ \mid  Y  \right) \right]$.

Since $\vm(y)$ satisfies the weak $(K,\tau)$-sliced stable condition, then Proposition~\ref{prop:ssc-discrete-approx} indicates that the smallest eigenvalue of $\Cov \left[\bbE \left( \vZ \mid Y \right)\right]$ can be bounded by that of $\Cov \left[\bbE \left( \vZ \mid W \right)\right]$ for some discrete random variable $W$ with $K$ outcomes and each outcome has a probability mass no larger than $(1+\gamma)/K\leq 1/2$ . 
Then by Theorem~\ref{thm:eigen-bound-ent}, the smallest eigenvalue of $\Cov \left[\bbE \left( \vZ \mid W \right)\right]$ is bounded above by $O(d^{-1}\log K)$. Since $K=O(d)$, this bound becomes $O(\frac{\log (d)}{d})$. 
\end{proof}

\begin{proof}[Proof of Proposition~\ref{prop:ssc-discrete-approx}]
	Fix $\bbeta\in \R^d$ and let $U=\bbeta\tp \vZ$. By the law of total variance, we have
	\begin{align*}
		\var \left( \bbE \left( U \mid Y \right) \right) & = 		\bbE \left\{ \var \left[ \bbE \left( U \mid Y \right) \mid W \right] \right\} + \var \left\{ \bbE \left[  \bbE \left( U \mid Y \right)  \mid W  \right] \right\} \\ 
		&= \sum_{h=1}^{H} \bbP(W=h) \var \left[ \bbE \left( U \mid Y \right) \mid W=h \right] + \var\left[  \bbE \left( U \mid W \right)   \right] \\
		&\leqs \frac{1+\gamma }{H}\sum_{h=1}^{H}  \var \left[ \bbE \left( U \mid Y \right) \mid Y\in S_h \right] + \var\left[  \bbE \left( U \mid W \right) \right] \\
		& \leqs \frac{1+\gamma}{\tau} 	\var \left( \bbE \left( U \mid Y \right) \right) + \var\left[  \bbE \left( U \mid W \right) \right],
	\end{align*}
	where the two inequalities are due to the premises. It then follows that 
$$\left(1- \frac{1+\gamma}{\tau}  \right)	\var\left( \bbE \left( U \mid Y \right)  \right)\leqs \var\left( \bbE \left( U \mid W \right) \right)$$
and the proposition is proved. 
\end{proof}

The smallest eigenvalue of the covariance matrix of the conditional mean is not easy to characterize. Our proof of Theorem~\ref{thm:eigen-bound-ent} will make use of the following useful lemma, which provides an upper bound on the eigenvalue in terms of the average of the $\bbE\left[\bbE(Z_i |W)^2\right]$ across the coordinates. The upper bound given by this lemma is easier to characterize and will serve as an important step in proving the fast decay rate in Section~\ref{sec:gSNR-fast-decay}.

\begin{lemma}\label{lem:gsnr-to-ave-sq-cond-mean}
Suppose $Y$ is a random element and $\vZ$ is a $d$-dimensional random vector with mean $\bs{0}$. Then
$$ \lambda_{\min }( \Cov\left[ \bbE \left( \vZ \mid  Y  \right) \right] )  \leqslant \frac{1}{d} \sum_{i=1}^{d}\bbE\left(  \bbE \left( Z_i \mid  Y  \right)^{2}\right). 
$$	
\end{lemma}
\begin{proof}[Proof of Lemma~\ref{lem:gsnr-to-ave-sq-cond-mean}]
Suppose $ \bbeta \sim N(0, \bs{I}_{d})$ is independent with $(\vZ, Y )$. 
For any symmetric matrix $\vM$, by the definition of the smallest eigenvalue, it holds that 
\begin{align*}
	\lambda_{\min}(\vM) & \leqs \bbE \left( \frac{ \bbeta \tp}{\| \bbeta \|} \vM  \frac{ \bbeta }{\| \bbeta \|} \right)\\
	& = d^{-1} \bbE \| \bbeta \|^{2}\bbE \left( \frac{ \bbeta \tp}{\| \bbeta \|} \vM  \frac{ \bbeta }{\| \bbeta \|} \right)\\
	& = d^{-1}\bbE \left( \| \bbeta \|^{2} \frac{ \bbeta \tp}{\| \bbeta \|} \vM  \frac{ \bbeta }{\| \bbeta \|} \right)\\
	& = d^{-1}  \bbE \left(  \bbeta \tp \vM  \bbeta  \right), 
\end{align*}
where the first equality is implied by the normality of $\bbeta$, and the second equality is because $\|\bbeta\| $ is independent of $ \bbeta /\| \bbeta \|$. 
For any $\vu\in \R^d$, by the law of total expectation, we have $\bbE\left[ \bbE \left( \vu\tp \vZ \mid  Y  \right) \right]=\bbE \left( \vu\tp \vZ\right)=0$ and 
\begin{align*}
& \vu\tp \Cov\left[ \bbE \left( \vZ \mid  Y  \right) \right]  \vu \\
=	&  \Cov\left[ \bbE \left( \vu\tp \vZ \mid  Y  \right)  \right]   \\
=	& \bbE \left(\left[ \bbE \left( \vu\tp \vZ \mid  Y  \right) \right]^{2} \right) \\
=	& \bbE \left( \sum_{i=1}^{d}\sum_{j=1}^{d}\bbE \left( u_{i} Z_i \mid  Y  \right)   \bbE \left( u_j Z_j \mid  Y   \right)   \right) \\
=&  \sum_{i=1}^{d}\sum_{j=1}^{d} u_i u_j \bbE \left[ \bbE \left( Z_i \mid  Y  \right)   \bbE \left( Z_j \mid  Y   \right]   \right). 
\end{align*}
Therefore, using the equation that $\bbE( \bbeta _i  \bbeta _j)=1_{i=j}$ and the independence between $ \bbeta $ and $(\vZ, Y )$, we have
$$
\bbE \left(  \bbeta \tp \Cov\left[ \bbE \left( \vZ \mid  Y  \right) \right]   \bbeta  \right)=\sum_{i=1}^{d}\bbE\left(  \bbE \left( Z_i \mid  Y  \right)^{2}    \right). 
$$
Combining the two inequalities, we complete the proof. 
\end{proof}

\begin{proof}[Proof of Theorem~\ref{thm:eigen-bound-ent}]

By Lemma~\ref{lem:gsnr-to-ave-sq-cond-mean}, we have 
$$ \lambda_{\min }( \Cov\left[ \bbE \left( \vZ \mid  W  \right) \right] )  \leqslant \frac{1}{d} \sum_{i=1}^{d}\bbE\left(  \bbE \left( Z_i \mid  W  \right)^{2}\right), 
$$
so it is sufficient to show 
$$
\label{eq:upper-eigen-goal}
\sum_{i=1}^{d}\bbE\left(  \bbE \left( Z_i \mid  W  \right)^{2}\right) \leqs 37 ~ \Ent( W ). 
$$  
Fixed any $ w $ in the support of $ W $. Lemma~\ref{lem:decay-lambda-optimal-variation} shows that $$ \sum_{i=1}^{d} \bbE \left( Z_i \mid  W = w  \right)^{2} \leqs \bbP( W = w )^{-2}\min_{ \theta}\left[  \theta \bbP( W = w ) + \bbE (Z_1 - \theta)_{+}  \right]^{2}. $$

We now choose $ \theta$ such that $\bbP( W = w )=\bbP(Z_1> \theta)$. Here $ \theta>0$ because $\bbP( W = w )<1/2$. 
Since  $\bbE(Z_1-\theta)_{+}=\bbE (Z_1-\theta) 1_{Z_1>\theta}=\bbE \left(  Z_1 1_{Z_1>\theta}  \right) - \theta \bbP(Z_1>\theta)$,  we have 
$$ \sum_{i=1}^{d} \bbE \left( Z_i \mid  W = w  \right)^{2} \leqs  \left[ \bbE \left( Z_1 \mid Z_1> \theta  \right) \right]^{2}. $$
By Lemma~\ref{lem:decay-lambda-tail-moment-bound}, the right hand side of the last inequality is bounded by $37 ~ \log \frac{1}{\bbP( W = w )}$. Hence, we have 
$$
\sum_{i=1}^{d}\bbE\left(  \bbE \left( Z_i \mid  W  \right)^{2}\right)  \leqs  \sum_{ w } \bbP(  W  =  w  ) \cdot 37 ~ \log \frac{1}{\bbP( W = w )} =  37 ~ \Ent ( W ),
$$
and thus complete the proof of Equation~\ref{eq:upper-eigen-goal}. 

In particular, when the support of $W$ has $K$ elements, the maximum entropy of $W$ is maximized by a uniform distribution on these $K$ elements and $\max_{\mc{L}(W)}\left\{ \Ent(W) \right\}= \log K$. 
\end{proof}
\begin{lemma}\label{lem:decay-lambda-optimal-variation}
	Suppose the distribution of $\vZ$ is invariant to orthogonal transformations, and $ W $ is a discrete random variable. For any $w$ in the support of $W$ and any number $ \theta$, it holds that $$\| \bbE(\vZ 1_{ W = w }) \|^{2}\leqs \left[  \theta \bbP( W = w ) + \bbE (Z_1 - \theta)_{+}  \right]^{2}. $$
\end{lemma}
\begin{proof}[Proof of Lemma~\ref{lem:decay-lambda-optimal-variation}]
Let $f(\vz)=\bbP( W = w \mid \vZ=\vz)$ be the conditional probability of $ W = w $ given $\vZ=\vz$ and $\bbP( W = w ) \in (0,1)$. 
Then $f(\vz)\in [0,1]$ and $\bbE f(\vZ)= \bbP( W = w )$. 

Let $\bs{\alpha}=\bbE \left( \vZ 1_{ W = w } \right)$, which equals to $\bbE \left( \vZ \bbE[1_{ W = w }\mid \vZ] \right) = \bbE \left( \vZ f(\vZ) \right)$ by the law of total expectation. If $\bs{\alpha}=\bs{0}$, the lemma holds trivially.

Assume $\bs{\alpha}\neq \bs{0}$. 
Let $\vV=[\vV_1,\dots,\vV_d]$ be a $d$-dimensional orthogonal matrix such that $\vV_1=\bs{\alpha} /\|\bs{\alpha}\|$. Let $\vU=\vV\tp \vZ$.  $\vU$ has the same distribution as $\vZ$ because its distribution is invariant to orthogonal transformations. So $\bbE f(\vU)=\bbE f(\vZ)=\bbP(W=w)$. 

Note that $\bs{\alpha}\tp \vZ=\|\bs{\alpha}\|  U_1$. We have $\|\bs{\alpha}\|^{2} = \bbE\left[ \bs{\alpha} \tp \vZ f(\vZ) \right]=\|\bs{\alpha}\| \bbE \left[  U_1 f(\vV \vU) \right]$. 

For any number $ \theta$, we have 
\begin{align*}
&	\|\bs{\alpha}\| - \theta \bbP( W = w ) \\
= &	\bbE \left[  U_1 f(\vV \vU) \right]- \theta \bbE f(\vU) \\
=	&  \bbE \left[ Z_1 f(\vV \vZ) \right] - \bbE \left[  \theta f(\vV \vZ) \right] \\
=	&  \bbE \left[ (Z_1- \theta) f(\vV \vZ) \right]  \\
\leqs &  \bbE \left[ (Z_1- \theta)_{+}  \right],
\end{align*}
where the last inequality is because $f(\vV \vZ)\in [0,1]$ and $ a b \leqs \max(0,a)$ for any $b\in [0,1]$. Therefore, $\|\bs{\alpha}\| \leqs  \theta \bbP( W = w ) + \bbE \left[ (Z_1- \theta)_{+}  \right]$. 
\end{proof}
\begin{lemma}\label{lem:decay-lambda-tail-moment-bound}

	If $Z\sim N(0,1)$ and $ \theta>0$, then 
	$$\bbE(Z\mid Z> \theta)^{2} \leqs 37 ~ \log \frac{1}{\bbP(Z> \theta)}.$$
\end{lemma}
\begin{proof}[Proof of Lemma~\ref{lem:decay-lambda-tail-moment-bound}]
It is well known that for any $t>0$, 
\begin{equation}
\label{eq:basic-normal-tail}
(2\pi)^{-1/2}\frac{t}{t^2+1}e^{-t^2/2} \leqs \bbP(Z>t) \leqs  e^{-t^2/2}. 
\end{equation}
By direct calculation,  
\begin{equation}\label{eq:basic-normal-tail-moment}
    \bbE \left(  Z 1_{Z> \theta}  \right) = (2\pi)^{-1/2} e^{- \theta^2/2}.
\end{equation}

We consider the value of $ \theta$ separately in two cases: 
\begin{enumerate}
    \item 
 $ \theta\geqslant 1$: Using the first inequality in Equation~\eqref{eq:basic-normal-tail} and  Equation~\eqref{eq:basic-normal-tail-moment}, we have
$
\bbE(Z\mid Z> \theta)^{2} \leqs ( \theta+1/ \theta)^{2}\leqs 4  \theta^{2}
$. 
Using the second inequality in Equation~\eqref{eq:basic-normal-tail}, we have 
$
  \log \frac{1}{\bbP(Z> \theta)}\geqslant   \theta^2/2 
$. 
\item 
 $ \theta\in (0,1)$: Then $\bbE \left(  Z 1_{Z> \theta}  \right) \leqs \bbE \left(  Z 1_{Z>0}  \right)$ and $ \bbP(Z> \theta)>\bbP(Z>1)$. 
We have 
$
\bbE(Z\mid Z> \theta)^{2}\leqs 4 e$. 
Also note that $\bbP(Z> \theta)<\bbP(Z>0)=1/2$, we have $\log \frac{1}{\bbP(Z> \theta)}\geqslant \log 2$. 
\end{enumerate}

Since $4< 37/2$ and $ 4e  < 37 ~ \log 2$, we  conclude the desired inequality for both cases. 
\end{proof}

\section{Proof of Theorem~\ref{thm:GP-gSNR-fast-decay} }\label{sec:gSNR-fast-decay}

Let $\vZ=\vB\tp\vX$. It is clear that $\vZ\sim N(0,\bs{I}_{d})$ and $\bbE \left( \vX \mid  Y  \right) =\bbE \left[\bbE \left( \vX \mid  Y,Z  \right)\mid Y\right] = \bbE \left( \vB\vZ \mid  Y  \right)$.  
Therefore, 
\begin{align*}
     \lambda_{\min }( \Cov\left[ \bbE \left( \vX \mid  Y  \right) \right] ) &=\lambda_{d}( \vB\Cov\left[  \bbE \left( \vZ \mid  Y  \right) \right] \vB^\top ) \\
     &\leq \frac{1}{d} \Tr\left(\vB\Cov\left[   \bbE \left( \vZ \mid  Y  \right) \right] \vB^\top \right) \\ 
     &= \frac{1}{d} \sum_{i=1}^{d}\bbE\left(  \bbE \left( Z_i \mid  Y  \right)^{2}\right).
\end{align*}
In the following, we focus on deriving bounds for each dimension separately.

We begin with bounding the supremum of the GP over a compact set. 
Let $K=48(72\alpha)^{1/4}+\sqrt{2}$ and define an event $A \triangleq \{ \sup_{\|\vt\|\leqslant 3\sqrt{d} } |f (\vt) | \leqslant K d  \}$. 
By the classical theory on the supremum of GPs, we have the following result, whose proof is deferred to Section~\ref{sec:pf-lem-GP-bound}. 
\begin{lemma}\label{lem:GP-bound}
The event $A$ happens with probability at least $1-e^{-d}$. 
\end{lemma}

\subsection{Main proof}

In the following, we assume the event $A$ happens. 
We introduce the shorthand  $\Gamma=\sigma+K d$, and we consider bounding $\bbE \left(Z_{i} \mid Y\right)^{2}$ for $|Y|>\Gamma$ and $|Y|\leq \Gamma$ separately. 

For any given $f$, the inequality $|Y|>\Gamma$ implies $\|\vZ\|>3\sqrt{d}$. As a result, we have the following:
\begin{equation}\label{eq:GP-Y-large}
\begin{aligned} \bbE \left[\bbE \left(Z_{i} \mid Y\right)^{2} \cdot I_{|Y|>\Gamma}\right] 
\leqslant &  
\sqrt{ \bbE \left[\bbE \left(Z_{i} \mid Y\right)^{4} \right] \cdot \bbP\left( |Y|>\Gamma\right)   } \\ 
 \leqslant &  \sqrt{\bbE  Z_{i}^{4} \cdot P(|Y|>\Gamma)} \\
 \leqslant &  \sqrt{3 \times P(\| \vZ\|> 3 \sqrt{d})}\\
 \leqslant &   \sqrt{3} e^{-(9d-d)/16} = \sqrt{3} e^{-d/2} ,
\end{aligned}
\end{equation}
where the first inequality is due to the Cauchy-Schwartz inequality, the second is due to Jensen inequality and the law of total expectation, the third is because the fourth moment of the standard normal distribution is 3 and $\{|Y|>\Gamma \}\subset \{\|\vZ\|> 3\sqrt{d} \}$, and the last inequality is because the tail probability of $\chi^2$ random variable can be bounded as 
\begin{equation*}
	\label{eq:trivial-chisq-dev}
\bbP(\|\vZ\|>r)= \bbP( \chi^2_{d} >d+(r^2-d))\leqslant e^{-(r^2-d)^2/(4r^2)} \leqslant e^{-(r^2-d)/8}
\end{equation*}
as long as $r^2>2d$.

The rest of our proof is to bound 
$$
d^{-1}\sum_{i=1}^{d} \bbE \left[\bbE \left(Z_{i} \mid Y\right)^{2} I_{|Y|\leq \Gamma}\right].
$$

Let $\psi_{d}(\cdot)$ be the density of a $d$-dimensional standard normal random vector. 
For any given $f$, the joint Lebesgue density of $(\vZ,Y)$ can be obtained by the product of the density of $\vZ$ and the conditional density of $Y$ given $\vZ$, which is 
$$\psi_{d}(\vz) \cdot \frac{1}{\sigma} \varphi\left(\frac{y-f(\vz)}{\sigma}\right), 
$$
and the Lebesgue density of $Y$ is 
$$f_Y(y)=\int_{\mathbb{R}} \psi_{d}(\vz) \cdot \frac{1}{\sigma} \varphi\left(\frac{y-f(\vz)}{\sigma}\right) \diff \vz.
$$
From these densities, we can explicitly write out the conditional expectation 
$$\bbE \left(Z_{i} \mid Y=y\right)=\frac{1}{f_{Y}(y)} \int z_{i} \psi_{d}(\vz) \frac{1}{\sigma} \varphi\left(\frac{y-f(\vz)}{\sigma}\right) \diff \vz. $$
For ease of notation, we denote the numerator by
$$
G_{y}^{f,i}=\int z_{i} \psi_{d}(\vz) \frac{1}{\sigma} \varphi\left(\frac{y-f(\vz)}{\sigma}\right) \diff \vz.
$$ 

The denominator will be handled by discussing the value of $y$ in $[-\Gamma,\Gamma]$. More concretely, define $H_{f}=\{y\in [-\Gamma,\Gamma]: f_{Y}(y) > e^{-\beta d} \}$ and  $L_{f}=[-\Gamma,\Gamma]\setminus H_{f}$. We have $\bbP(Y\in L_f)=\int_{L_f} f_{Y}(y) \diff y \leqslant 2 \Gamma e^{-\beta d} $. By Cauchy-Schwartz inequality and the law of total expectation, it holds that
\begin{equation}\label{eq:gSNR-Y-Lf}
\begin{aligned}  \bbE \left[\bbE\left(Z_{i} \mid Y\right)^{2} I_{Y\in L_{f}}\right] 
&\leqslant \sqrt{ \bbE Z_{i}^{4} \cdot \P(Y\in L_{f})} \\ 
&\leqslant \sqrt{6 \Gamma } e^{-\beta d/2}.
\end{aligned}
\end{equation}

For any $y\in H_f$, we have 
\begin{equation}\label{eq:gSNR-Y-Hf}
\bbE\left(Z_{i} \mid Y=y\right)^{2}\leqslant e^{2\beta d} \times \left(G_{y}^{f,i}\right)^2.
\end{equation}
To bound $G_{y}^{f,i}$, we have the following lemma. 

\begin{lemma}\label{lem:Gyf-bound}
For any $i\in [d]$, any $L>0$ and any positive integer $M$, it holds that 
$$
\bbP_{f} \left( \sup_{| y| \leq \Gamma} |G_{y}^{f,i}| > L + \frac{ \Gamma C_1}{M \sigma^2}  \right) \leqslant (1+M)  \frac{C_{0}^{2}}{L^2 \pi^2\sigma^{2}} (4+\frac{1}{2\alpha}) (1+4\alpha)^{- \frac{d-1}{2}}. 
$$
\end{lemma}
The proof of Lemma~\ref{lem:Gyf-bound} is deferred to Section~\ref{sec:pf-lem:Gyf-bound}. 

Denote by $\tilde{A}$ the event that 
$$
\max_{i\in [d]} \sup_{|y|\leq \Gamma} |G_{y}^{f,i}| \leq L + \frac{ \Gamma C_1}{M \sigma^2}.
$$
Using the Boole's inequality and Lemma~\ref{lem:Gyf-bound}, we see that the event $\tilde{A}$ happens with probability at least $1- d(1+M)  \frac{C_{0}^{2}}{L^2 \pi^2\sigma^{2}} (4+\frac{1}{2\alpha}) (1+4\alpha)^{- \frac{d-1}{2}}$.

In particular, we choose 
$L=(1+4\alpha)^{-\frac{d-1}{8}} C_1/\sigma^2$, $M=\lceil d(1+4\alpha)^{\frac{d-1}{8}}\rceil$, and $\beta=\frac{1}{9}\log(1+4\alpha)$. When $\tilde{A}$ happens, \eqref{eq:gSNR-Y-Lf} and \eqref{eq:gSNR-Y-Hf} together yield
\begin{equation}\label{eq:GP-Y-small}
  \begin{aligned}
  	d^{-1}\sum_{i=1}^{d} \bbE \left[\bbE \left(Z_{i} \mid Y\right)^{2} I_{|Y|\leq \Gamma}\right] & = d^{-1}\sum_{i=1}^{d} \bbE \left[\bbE \left(Z_{i} \mid Y\right)^{2} I_{Y\in L_{f}} \right] + d^{-1}\sum_{i=1}^{d} \bbE \left[\bbE \left(Z_{i} \mid Y\right)^{2} I_{Y\in H_{f}}\right] \\
  	& \leqslant \sqrt{6 \Gamma } e^{-\beta d/2} +  e^{2\beta d} \left( L + \frac{ \Gamma C_1}{M \sigma^2} \right)^2 \\
  	& \leqslant {C}^{\prime} \left( d^{1/2} e^{-\beta d/2} + e^{2\beta d} e^{-9 \beta d / 4}   \right) \\
  	& \leqslant {C}^{\prime\prime} e^{-\beta d/4},
  \end{aligned}
\end{equation}
where the constants ${C}^{\prime}$ and ${C}^{\prime\prime}$ only depend on $\alpha$, $C_1$, and $\sigma$, and the last inequality is because $d^{1/2} e ^{ - \beta d/4}$ is uniformly bounded for every positive integer $d$. 

In view of \eqref{eq:GP-Y-large}, \eqref{eq:GP-Y-small}, and  Lemma~\ref{lem:gsnr-to-ave-sq-cond-mean}, when both $\tilde{A}$ and ${A}$ happen, we have 
\begin{equation}
  \begin{aligned}
 \lambda_{\min }( \Cov\left[ \bbE \left( \vZ \mid  Y  \right) \right] )  & \leqslant  	d^{-1}\sum_{i=1}^{d} \bbE \left[\bbE \left(Z_{i} \mid Y\right)^{2} \right] \\
 & =   	d^{-1}\sum_{i=1}^{d} \bbE \left[\bbE \left(Z_{i} \mid Y\right)^{2}I_{|Y|\leq \Gamma} \right]  +   	d^{-1}\sum_{i=1}^{d} \bbE \left[\bbE \left(Z_{i} \mid Y\right)^{2}I_{|Y|> \Gamma} \right] 
  	 \\
  	& \leqslant {C}^{\prime\prime} e^{-\beta d/4} + \sqrt{3} e^{-d/2}.
  \end{aligned}
\end{equation}
Furthermore, we have 
\begin{equation}
  \begin{aligned}
  	\bbP\left(  (A\cap \tilde{A})^{c} \right) & \leqslant  e^{-d} + 2 d^2  (1+4\alpha)^{\frac{3(d-1)}{8} - \frac{d-1}{2} }  \frac{C_0^2}{\pi^2 \sigma^2}(4+\frac{1}{2\alpha}) \\ 
&  	\leqslant  e^{-d} +  \tilde{C}^{\prime} d^2 e^{-9\beta d /8} \\
&  	\leqslant  e^{-d} +  \tilde{C}^{\prime\prime}  e^{-\beta d/8}. 
  \end{aligned}
\end{equation}
Choosing $\theta=\min(\beta/8, 1/2)$ and $\tilde{C}=\max(C^{\prime\prime} +\sqrt{3} ,  1 +\tilde{C}^{\prime\prime})$ will complete the proof of the theorem.

\subsection{Proof of Lemma~\ref{lem:GP-bound}}\label{sec:pf-lem-GP-bound}
The cornerstone of our proof is an upper bound on the expected supremum of a Gaussian process (Theorem~\ref{thm:mean-sup-GP}) and an concentration inequality for the sample supremum (Theorem~\ref{thm:sup-GP-concentration}). 
The bound on the expected supremum of a Gaussian process is given by the Dudley integral \citep{dudley1967sizes} with respect to the canonical pseudometric defined as $d(\vx, \vy) = \sqrt{\operatorname{Var}_{f}\left(f(\vx)-f(\vy) \right)}$. 
Interested readers are referred to Section~3.4 in \citet{massart2007concentration} for more details.

\begin{theorem}[Theorem 3.18 in \citet{massart2007concentration}]\label{thm:mean-sup-GP}
 Let $(f_t)_{t \in \Omega}$ be some centered Gaussian process and $d$ be the covariance pseudometric of $(f_t)_{t \in \Omega}$. Assume that $(\Omega, d)$ is totally bounded and denote by $N(\Omega, d, \epsilon)$ the $\delta$-packing number of $(\Omega, d)$, for all positive $\epsilon$. If $\sqrt{\log N(\Omega, d, \epsilon)}$ is integrable at 0, then $(f_t)_{t \in \Omega}$ admits a version which is almost surely uniformly continuous on $(\Omega, d)$. Moreover, if $(f_t)_{t \in \Omega}$ is almost surely continuous on $(\Omega, d)$, then
$$
\mathbb{E}\left[\sup _{t \in \Omega} f_t \right] \leqslant 12 \int_0^{\sigma_\Omega} \sqrt{\log N(\Omega, d, \epsilon)}  \diff \epsilon
$$
where $\sigma_{\Omega}=\left(\sup _{t \in \Omega} \mathbb{E}\left[f^2_t\right]\right)^{1 / 2}$.

\end{theorem}

\begin{theorem}[Proposition 3.19  in \citet{massart2007concentration}]\label{thm:sup-GP-concentration}
Suppose $(f_t )_{t \in \Omega}$ is some almost surely continuous centered Gaussian process on the totally bounded set $(\Omega, d)$.  Define $\sigma_{\Omega}^2=\sup _{t \in \Omega}\left(\mathbb{E}\left[f^2_t\right]\right)$
and $Z= \sup _{t \in \Omega}|f_t |$. 
Then,

$$
\mathbb{P}[Z-\mathbb{E}[Z] \geq \sigma_{\Omega} \sqrt{2 x}] \leqslant \exp (-x)
$$

for all positive $x$. 

\end{theorem}

Let $r=\sqrt{9d}$ and let $\Omega=\{\vz\in \R^{d}: \|\vz\| \leqslant r \}$.  
By construction of the GP with the squared exponential kernel, we have $\sigma_\Omega=1$. 

It is straightforward to compute the covariance pseudometric of the Gaussian process as 
$$
\begin{aligned} d(\vz_1, \vz_2)^{2} &= 2- 2 \Cov \left(f(\vz_1), f(\vz_2)  \right)\\
& =2-2 e^{-\alpha\|\vz_1- \vz_2\|^{2}} \\ & \leqslant 2 \alpha\|\vz_1-\vz_2\|^{2} . \end{aligned} 
$$
The above inequality establishes a relationship between the canonical pseudometric and the Euclidean metric:

\begin{align*}
 & N(\Omega, d, \epsilon)\\
  \leqslant  & 
  N\left(\Omega,\|\cdot\|, \frac{\epsilon}{\sqrt{2 \alpha}}\right)\\
  \leqslant & \left( 1 + \frac{r}{\epsilon/\sqrt{8\alpha}}\right)^d, 
\end{align*}
where the second inequality is due to the standard volume argument and because that the balls centered at each point of the maximum packing set with radii of $\epsilon/\sqrt{8\alpha}$ are disjoint and are contained in the ball centered at the origin with a radius of $r+\epsilon/\sqrt{8\alpha}$. 

We can then bound the integral in Theorem~\ref{thm:mean-sup-GP} as follows:
$$
\begin{aligned} 
\int_0^1 \sqrt{\log N(\Omega, d, \epsilon)}  \diff \epsilon \leqslant & \int_{0}^{1} \sqrt{d \cdot \log \left(1+\frac{\sqrt{8 \alpha} r}{\epsilon}\right)} \diff \epsilon  \\ 
= & \sqrt{d} \int_{1}^{\infty} \sqrt{\lg (1+\sqrt{8 \alpha} r \cdot t)} \frac{1}{t^{2}} \diff t \\
\leqslant & 2 \sqrt{r d} ( 8 \alpha)^{1/4},
\end{aligned}
$$
where the last inequality is due to the following straightforward calculation with $a=\sqrt{8 \alpha \cdot r^{2}} $: 
$$
\begin{aligned} & \int_{1}^{\infty} \sqrt{\lg(1+a t)} \frac{1}{t^{2}} \diff t \quad \\ = & a \cdot \int_{a}^{\infty} \sqrt{\lg (1+x)} \frac{d x}{x^{2}} \\ \leqslant & a \cdot \int_{a}^{\infty} \frac{1}{x^{3 / 2}} \diff x=\left.\frac{a}{-\frac{1}{2}} \quad x^{-\frac{1}{2}}\right|_{a} ^{\infty}=2 a / \sqrt{a}=2 \sqrt{a}. \end{aligned}
$$

By Theorem~\ref{thm:mean-sup-GP}, we have 
$$
\begin{aligned}
\mathbb{E}\left[\sup _{t \in \Omega} | f_t| \right] 
& \leqslant \mathbb{E}\left[\sup _{t \in \Omega}  f_t \right] +\mathbb{E}\left[\sup _{t \in \Omega} - f_t \right] \\
& \leqslant 48 \left( 9 \cdot 8 \cdot \alpha \right)^{1/4} d^{3/4}
\end{aligned}
$$

In Theorem~\ref{thm:sup-GP-concentration}, we 
choose $x=d$ and 
conclude that $\sup _{t \in \Omega} | f_t|\leq 48 \left( 9 \cdot 8 \cdot \alpha \right)^{1/4} d^{3/4} +\sqrt{2d}$ with probability at least $1-e^{-d}$. Note that $d\geq 1$, the upper bound is no greater than $Kd$, which concludes Lemma~\ref{lem:GP-bound}.

\subsection{Proof of Lemma~\ref{lem:Gyf-bound}}\label{sec:pf-lem:Gyf-bound}~

Without loss of generality, we consider $i=1$ and drop the superscript $i$ in this proof. 

Given any $y$, we can show that the mean of $G_{y}^{f}$ is zero and the variance can be bounded as $$\operatorname{Var}_{f}\left(G_{y}^{f}\right) \leqslant \frac{C_{0}^{2}}{\pi^2\sigma^{2}} (4+\frac{1}{2\alpha}) (1+4\alpha)^{- \frac{d-1}{2}};
$$
the detailed calculation is deferred to the end of this subsection. 
It then follows from the Chebyshev's inequality that 
\begin{equation}\label{eq:Gyf-single}
	\bbP_{f} \left( |G_{y}^{f}| > L \right) \leqslant \frac{\operatorname{Var}_{f}(G_{y}^{f}) }{L^2} \leqslant   \frac{C_{0}^{2}}{L^2 \pi^2\sigma^{2}} (4+\frac{1}{2\alpha}) (1+4\alpha)^{- \frac{d-1}{2}}. 
\end{equation}

The length of the interval $[-\Gamma,\Gamma]$ is $2\Gamma$ and can be equalled divided by $1+M$ points $\{y_0=-\Gamma, \ldots, y_M=\Gamma\}$, such that each interval has length $2\Gamma/M$. 
For each of these points, we use \eqref{eq:Gyf-single} and the Boole's inequality to obtain
\begin{equation}\label{eq:Gyf-M}
\bbP\left( \max_{0\leqslant j\leqslant M}|G_{y_j}^{f}| > L \right)\leqslant \frac{(1+M)}{ L^2}  \frac{ C_{0}^{2}}{\pi^2 \sigma^{2}} (4+\frac{1}{2\alpha}) (1+4\alpha)^{- \frac{d-1}{2}}. 
\end{equation}

Note that the gradient of $G_y^f$ is 
$$\frac{\partial}{\partial y} G_{y}^{f}=\int z_{1} \psi_{d}(\vz) \frac{1}{\sigma^{2}} \varphi^{\prime}\left(\frac{y-f(\vz)}{\sigma}\right) \diff \vz,
$$
which, in view of the assumption that$|\varphi^{\prime}|\leqs C_1 $, can be bounded as 
\begin{equation}\label{eq:Gyf-grad-bound}
 \left|\frac{\partial}{\partial y} G_{y}^{f}\right| \leqslant \frac{C_{1}}{\sigma^{2}} \int\left|z_{1}\right| \psi_{d}(\vz) \cdot \diff \vz \leqslant \frac{ C_{1}}{\sigma^{2}}. 
\end{equation}

For any $y\in [-\Gamma,\Gamma]$, there is some $i$ between $0$ and $M$ such that $|y-y_i|\leq \Gamma/M$. By the mean value theorem, there is some $\tilde{y}$ lying between $y$ and $y_i$ such that 
$$
G_{y}^{f} - G_{y_i}^{f} = \left. \frac{\partial}{\partial y} G_{y}^{f}\right|_{y=\tilde{y}} (y-y_i).
$$
In view of \eqref{eq:Gyf-grad-bound}, we have
\begin{equation}\label{eq:Gyf-interval}
\sup_{y\in [-\Gamma,\Gamma]}|G_y^f|\leqslant \max_{0\leqslant j\leqslant M}|G_{y_j}^{f}|+\frac{C_1}{\sigma^2} \frac{\Gamma}{M}. 
\end{equation}
The lemma follows from the combination of \eqref{eq:Gyf-M} and \eqref{eq:Gyf-interval}.

\paragraph{Calculation of the mean and the variance of $G_{y}^{f}$: } 
To compute the expectation, we divide the domain of the integral into two parts by considering $z_1\geq 0$ and $z_1<0$, and use the change of variable $\tilde{\vz}=-\vz$ for the second resultant integral. Then we have 
\begin{align*}
   \bbE_{f}\left[ G_{y}^{f} \right] 
   & = \bbE_{f} \int_{0}^{\infty} z_1 \cdot \psi_{1}(z_1) \frac{1}{\sigma} \int_{\R^{d-1}}  \psi_{d-1}(\vz_{-1}) \left[\varphi\left(\frac{y-f(\vz)}{\sigma}\right)-\varphi\left(\frac{y-f(-\vz)}{\sigma}\right) \right] \diff \vz_{2:d} \\
& = \bbE_{f} \int_{0}^{\infty} z_1 \cdot \psi_{1}(z_1) \frac{1}{\sigma} \int_{\R^{d-1}}  \psi_{d-1}(\vz_{-1}) \left[\varphi\left(\frac{y-f(-\vz)}{\sigma}\right)-\varphi\left(\frac{y-f(\vz)}{\sigma}\right) \right] \diff \vz_{2:d}\\
& = 0, 
\end{align*}
where the second equation is because $(f(t))_{t\in \R^d} \overset{d}{=} (f(-t))_{t\in \R^d}$. 
Furthermore, we have
\begin{equation}
\operatorname{Var}_{f}\left(G_{y}^{f}\right)=\iint_{\R^d\times \R^d} x_{1} z_{1} \psi_{d}(\vx) \psi_{d}(\vz) \frac{1}{\sigma^{2}} \operatorname{Cov}_{f}\left(\varphi\left(\frac{y-f(\vx)}{\sigma}\right), \varphi\left(\frac{y-f(\vz)}{\sigma}\right)\right) \diff \vx \diff \vz. 
\label{eq:GP-var}
\end{equation}

In the following, we compute the covariance inside the above integral for any fixed $\vx$ and $\vz$ in $\R^{d}$. 
By definition of the GP, it holds that $[U,V]\triangleq [f(\vx), f(\vz)] \sim N\left( \bs{0},\left(\begin{array}{cc}1 & \rho \\ \rho & 1\end{array}\right)\right)$, where $\rho=k(\vx, \vz)=e^{-\alpha\|\vx-\vz\|^{2}}$. 
Let $\phi_{\rho}(u,v)$ be the density of the aforementioned bivariate normal distribution and let $F(\rho) = \bbE\left[ \varphi\left(\frac{y- U }{\sigma}\right)\cdot  \varphi\left(\frac{y- V }{\sigma} \right) \right]$. 
We have 
$$\operatorname{Cov}_{f}\left(\varphi\left(\frac{y-f(\vx)}{\sigma}\right), \varphi\left(\frac{y-f(\vz)}{\sigma}\right)\right) =F(\rho)-F(0)$$
and  
$$
\frac{\partial \phi_{\rho}(u,v)}{\partial \rho }=\frac{\partial^{2} \phi_{\rho}(u,v)}{\partial u ~ \partial v}.
$$
Following the proof of the celebrated Berman inequality (see, for example, \cite[Theorem 4.2.1]{leadbetter2012extremes}), we have 
\begin{equation}\label{eq:berman-ineq}
\begin{aligned}
	F(\rho)-F(0) &=\rho  \cdot \int_{0}^{1} F'(\rho h)  \diff h \\
	& =\rho \cdot \int_{0}^{1} \iint_{\R^2} \frac{\partial \phi_{\rho h}}{\partial (\rho h)} \cdot \varphi\left(\frac{y-u}{\sigma}\right) \cdot \varphi\left(\frac{y-v}{\sigma}\right)  \diff u \diff v  \diff h \\ 
	& =\rho \cdot \int_{0}^{1} \iint_{\R^2} \frac{\partial^{2} \phi_{\rho h} }{\partial u \partial v} \cdot \varphi\left(\frac{y-u}{\sigma}\right) \cdot \varphi\left(\frac{y-v}{\sigma}\right)  \diff u \diff v  \diff h \\ 
	& \leqslant \rho \cdot C_{0}^{2} \int_{0 }^{1} \iint_{\R^2}  \frac{\partial^{2} \phi_{\rho h}}{\partial u \partial v} \cdot 1_{|u-y|\leqslant \delta} \cdot 1_{|v-y| \leqslant \delta } \diff u \diff v \diff h \\ 
& = \rho \cdot C_{0}^{2} \int_{0}^{1} \left[ \phi_{\rho h}(y+\delta, y+\delta)- \phi_{\rho h}(y-\delta, y+\delta)- \phi_{\rho h}(y+\delta, y-\delta)+ \phi_{\rho h}(y-\delta, y-\delta)\right]  \diff h \\
	& \leqslant \rho \cdot C_{0}^{2} \int_{0}^{1} \left[ \phi_{\rho h}(y+\delta, y+\delta)+\phi_{\rho h}(y-\delta, y-\delta)\right]  \diff h \\
	& \leqslant  \frac{ 2 C_{0}^{2} \rho  }{2\pi(1-\rho^2)}, 
	\end{aligned}
\end{equation}
where the first inequality is due to the assumption about the support and the boundedness of $|\varphi|$ and the last inequality is because $\left(2\pi (1-\rho^2)\right)^{-1}$ is an upper bound on $\phi_{\rho h}$. 

\medskip

In view of \eqref{eq:GP-var} and \eqref{eq:berman-ineq}, we have
\begin{equation}\label{eq:Gyf-var-sum}
\begin{aligned} & \operatorname{Var}_{f}\left(G_{y}^{f}\right) \leqslant \iint_{\R^d\times \R^d} \left|x_{1} z_{1}\right| \psi_{d}(\vx) \psi_{d}(\vz) \frac{2 C_{0}^{2}}{2\pi \sigma^{2}}  \cdot \frac{\rho_{x  z}}{1-\rho_{x z}^{2}} \diff \vx \diff \vz \\ & =\frac{C_{0}^{2}}{\pi \sigma^{2}} \iint_{\R^d\times \R^d}  \left|x_{1} z_{1}\right| \psi_{1}\left(x_{1}\right) \psi_{1}\left(z_{1}\right) \cdot \psi_{d-1}\left( \vx_{-1}\right) \psi_{d-1}( \vz_{-1}) \rho_{x z} \sum_{k=0}^{\infty} \rho_{x z}^{2 k} \diff \vx \diff \vz \\ 
& =\frac{ C_{0}^{2}}{ \pi \sigma^{2}} \sum_{k=0}^{\infty} \iint_{\R^d\times \R^d} \left|x_{1} z_{1}\right| \psi_{1}\left(x_{1}\right) \psi_{1}\left(z_{1}\right) \cdot \psi_{d-1}\left( \vx_{-1}\right) \psi_{d-1}( \vz_{-1}) \rho_{x z}^{1+ 2 k} \diff \vx \diff \vz , 
\end{aligned}
\end{equation}
where $\rho_{x z}=e^{-\alpha \|\vx-\vz\|^2}$. 
By Fubini's theorem and the fact that the integrand is a product of different terms, the integral in \eqref{eq:Gyf-var-sum} can be taken separately with respect to $(x_1,z_1)\in \R\times \R$ and with respect to $(\vx_{-1},\vz_{-1})\in \R^{d-1}\times \R^{d-1}$.

Fix any natural number $k$ and let $\beta=\alpha \cdot(1+2 k)$. 
The integral w.r.t. $(x_1,z_1)$ is simple: since $e^{-\beta (x_1-z_1)^2}\leqslant 1$, we conclude that 
 $$ 
 \iint_{\R\times \R} \left|x_{1} z_{1}\right| \psi_{1}\left(x_{1}\right) \psi_{1}\left(z_{1}\right) \cdot e^{-\beta (x_1-z_1)^2} \diff x_1 \diff z_1 
 \leqslant (\bbE|X|)^2 = \frac{2}{\pi}. 
 $$

Furthermore, let $\vu=\vx_{-1}$ and $\vv=\vz_{-1}$, the integral w.r.t. $(\vx_{-1},\vz_{-1})\in \R^{d-1}\times \R^{d-1}$ can be written as
$$\begin{aligned}
&\iint_{\R^{d-1}\times \R^{d-1}} \psi_{d_{-1}}\left( \vx_{-1}\right) \psi_{d-1}( \vz_{-1})  \exp\left( -\beta \cdot\left\| \vx_{-1}-\vz_{-1}\right\| ^{2}\right) \diff \vx_{-1} \diff \vz_{-1}   \\
= & \left( 2\pi \right) ^{-\left( d-1\right) }\cdot \iint _{R^{d-1}\times R^{d-1}}\exp \left( -\dfrac{1}{2}\left( \vu, \vv\right) \begin{bmatrix} \left( 1+2\beta \right) I_{d-1} & -2\beta I_{d-1} \\
-2\beta I_{d-1} & (1+2\beta) I_{d-1} \end{bmatrix}\begin{pmatrix} \vu \\ \vv \end{pmatrix}\right) \diff \vu \diff \vv \\
=& \left[\operatorname{det}\left[\begin{array}{cc} 1+2\beta, & - 2 \beta \\ -2\beta, & 1+2 \beta \end{array}\right]\right]^{-\frac{d-1}{2}} \\
= & \left(1+4 \beta+4 \beta^{2}-4 \beta^{2}\right)^{-\frac{d-1}{2}}   = \left(1+4 \beta\right)^{-\frac{d-1}{2}}  
\end{aligned}$$

So 
$$
\begin{aligned} & \operatorname{Var}\left(G_{y}^{f}\right) \leqslant \frac{2 C_{0}^{2}}{\pi^2 \sigma^{2}} \sum_{k=0}^\infty (1+4\alpha(1+2k))^{-(d-1)/2}. \end{aligned}
$$

Since $d\geq 4$, the above summation can be bounded by $(2+\frac{1}{4\alpha}) (1+4\alpha)^{- \frac{d-1}{2}} $. A detailed derivation is given below:

Let $g(x)=(1+x)^{-(d-1)/2}$. Since $g''>0$, $g$ is convex.  
Therefore, for any $k\geq 1$ and any $x\in (-1,1)$, it holds that $g(4\alpha(1+2k) \leqslant \frac{1}{2}\left( g(4\alpha(1+2k-x))+g(4\alpha(1+2k+x)) \right)$. 
The summation can then be bounded as 
$$
\begin{aligned} & g(4 \alpha)+\sum_{k=1}^{\infty} g(4 \alpha(1+2 k)) \\ & \leqslant g(4 \alpha)+\sum_{k=1}^{\infty} \int_{0}^{1} \frac{1}{2}[g(4 \alpha(1+2 k+x))+g(4 \alpha(1+2 k-x))] \diff x \\ &= g(4 \alpha)+\frac{1}{2} \int_{2}^{\infty} g(4 \alpha \cdot x) \diff x\\ & =\left(\frac{1}{1+4 \alpha}\right)^{\frac{d-1}{2}}+\frac{1}{2} \frac{1}{4 \alpha} \frac{1}{\frac{d-3}{2}}(1+8 \alpha)^{\frac{-(d-3)}{2}}. 
\end{aligned}
$$
Using $(d-3) \geq 1$ and $8\alpha\geq 4\alpha$, the right hand side of the last display can be bounded by  $(2+\frac{1}{4\alpha}) (1+4\alpha)^{- \frac{d-1}{2}}$.

We conclude that 
$$\operatorname{Var}_{f}\left(G_{y}^{f}\right) \leqslant \frac{C_{0}^{2}}{\pi^2\sigma^{2}} (4+\frac{1}{2\alpha}) (1+4\alpha)^{- \frac{d-1}{2}}.
$$

\section{Comparison with minimax Lower bounds in \cite{lin2021optimality}}\label{sec:comparison-2021}
Although our minimax lower bound in Theorem~\ref{thm:lower-smallp} may look similar to the minimax lower bound developed in \cite{lin2021optimality}, the proofs for these two results are actually very different. 
Compared to the previous result, our proof, as presented in Appendix~\ref{sec:proof lower bound given K}, involves several highly nontrivial technical innovations.

Before we get into the details of our proof, we point out the differences between our construction of distributions in Section~\ref{sec:sketch-lower-bound} and the previous construction in \cite{lin2021optimality}.
The previous construction in the proof of Lemma 15 in \cite{lin2021optimalitysupplement} is summarized as follows: 

\begin{itemize}
    \item []
They first find a smooth $d$-variate function $h$ such that for $z \sim N_d(0, I)$, and $Y=h(z)$, then $\Cov(\mathbb{E}(z \mid Y))$ has eigenvalues bounded away from 0. 
They then set $g(\cdot)=Ah(\cdot)$ with some constant $A$ and define $p_{B,g}$ the joint distribution of $(\bs{X}, Y)$to be: 
\begin{equation}\label{Pxy-old}
\bs{X}\sim N(0,\bs I_p), \quad Y=g\left(\bs{B}^\top \bs{X}\right)+\varepsilon, \text{ where }\varepsilon\sim N(0,1).
\end{equation}
Finally, they make use of the following facts for deriving their minimax lower bound:
\begin{itemize}
    \item[(F1)] By choosing the constant $A$ sufficiently large, the eigenvalues of $\Cov(\mathbb{E}(\vX \mid Y))$ are not less than the parameter $\lambda$ that defines the model class. 
    \item[(F2)]
They derive the following inequality for pairwise KL-divergence:
\begin{equation}\label{kl-old}
    \kl\left(p_{\vB, g}, p_{\vB^{\prime}, g}\right) \leq A^2 \max |\nabla h|^2\left\|\vB-\vB^{\prime}\right\|_F^2, 
\end{equation}
\end{itemize}
The function $h$ is constructed as follows: 
For any $x\in\R$, let $\phi(x)$ be a smooth function which maps $(-\infty, 0]$ to 0 and $[1, \infty)$ to 1 and has a positive first derivative over $(0,1)$. For any  $\bs z=(z_i)_{i=1}^d\in\R^d$, let $h(\bs z):=\sum_{i \leqslant d} 2^{i-1} \phi\left(z_i /\zeta \right)$ where $\zeta$ is sufficiently small such that for $\vZ=(Z_i)_{i\in[d]}\sim N(0,\bs I_d)$, the following probability inequality holds: 
\begin{equation}\label{small-probability-tol-old}
    \mathbb{P}\left(\exists i, 0<Z_i<\zeta\right)\leqslant d\zeta/(\sqrt{2\pi}) <2^{-d}.
    \end{equation}

\end{itemize}

It is clear that our construction in \eqref{eq:joint-x-y-B} is very different from theirs in \eqref{Pxy-old}. 
Now we explain why their construction can only apply with bounded $d$. 
\begin{enumerate}
\item 
When $d$ is unbounded, the construction in \eqref{Pxy-old} fails to ensure $p_{B,g}$ satisfies the weak sliced stable condition with $K=O(d)$. In fact, we can show that the central curve spanned over all quadrants of $\mathbb{R}^d$ so the number of slices need to get stable slicing is at least at the order of $2^d$ (a rigorous statement has been proved in Lemma 3.3.8 in Chapter 3 of \cite{huang2020reliable}). 

\item 
As $d$ increases, (F1) and (F2) cannot hold at the same time. 
Note that $\max |\nabla h|=2^{d-1}/\zeta$. As $d$ increases, \eqref{small-probability-tol-old} requires $\zeta$ to decrease at the rate of $\frac{1}{d 2^d}$. Therefore, $\max |\nabla h|$ increases at the rate of $d 2^{2d-1}$. 
In order to make use of (F2), for any given $\lambda$, one has to match \eqref{eq:lower-KL-bound} with \eqref{kl-old}. Consequently, $A$ has to satisfy
$$
A \asymp \frac{\sqrt{\lambda} }{d 2^{2d-1}}.
$$
However, (F1) requires $A$ to be not too small; in fact, $A$ has to scale as $\sqrt{\lambda}$ to ensure eigenvalues of $\Cov(\mathbb{E}(\vX \mid Y))$ are not less than $\lambda$. 
This contradiction suggests that the construction in \cite{lin2021optimality} does not apply when $d$ is unbounded. 
\end{enumerate}

Next, we explain why our construction overcomes the aforementioned problems. 
\begin{enumerate}
    \item[(A)] The central curve of \eqref{eq:joint-x-y-B} is much simpler and takes only $2d+1$ different values. We can show the distribution satisfies the weak sliced stable condition with $K=O(d)$. 
    \item[(B)] The eigenvalues of $\Cov(\mathbb{E}(\vX \mid Y))$ are determined by the value of $\rho$ in the construction in \eqref{eq:joint-x-y-B}: they are all equal to $\rho^2 \lambda_{0,d}$ where $\lambda_{0,d}$ is a number that scales as $\frac{\log d}{d}$ as $d$ increases. 
    For any small $\lambda<d^{-8.1}$, we can choose $\rho=\sqrt{\lambda/\lambda_{0,d}}$ to make sure the eigenvalue condition is satisfied. 
    \item [(C)]
    Most importantly, the inequality \eqref{eq:lower-KL-bound} allows us to bound the KL-divergence using any small $\lambda$. 
\end{enumerate}
These properties of our novel construction distinguish our work from the previous work. 

However, to establish the sharp control on the KL-divergence in terms of the gSNR is very challenging. 
We consider a Taylor expansion of the KL-divergence and we go through a careful analysis to control the induced measures, which is highly nontrivial.

\section{Proof of Theorem~\ref{thm:lower-smallp}}\label{sec:proof lower bound given K}

Our proof makes use of the following lemma known as the generalized Fano method.

\begin{lemma}[\cite{yu1997assouad}]\label{lem:fano}
	Let $N\geqslant 2$ be an integer and $\{\theta_1,\dots, \theta_N\}\subset \Theta_{0}$ index a collection of probability measures $\bbP_{\theta_i}$ on  a measurable space $(\X, \mathcal{A})$. Let $\rho$ be a pseudometric on $\Theta_{0}$ and suppose that for all $i \neq j$ 
	$$
	\rho(\theta_i, \theta_j)\geqslant \alpha_N,\quad\text{and}\quad \kl(\bbP_{\theta_i}, \bbP_{\theta_j}) \leqslant \beta_N. 
	$$
	Then every $\mathcal{A}$-measurable estimator $\hat{\theta}$ satisfies 
	$$
	\max_{i} \E \rho(\hat{\theta}, \theta_{i}) \geqslant \frac{\alpha_N}{2}\left( 1-\frac{\beta_N +\log 2}{\log N} \right). 
	$$
\end{lemma}

To apply this lemma, we will analyze the class of distributions constructed in \eqref{eq:joint-x-y-B}. 
Recall the Definition~\ref{def:fn-lower-median} that $\psi(z_1,\ldots, z_d)$ equals to the index $i$ of the largest absolute values of the coordinates multiplied by $\sgn(z_i)$  if $\|\bs z\|^2$ is less than $m_d$ the median of $\chi_{d}^2$ the chi-squared distribution with $d$ degrees of freedom  and $A_{i}=\psi^{-1}(i)$ for all $i$ in $\{\pm 1, \ldots, \pm d\}$, or more explicitly,  
$$A_{i} \, =\, \{\vz\in \R^d: \|\vz\|^2 \leqs m_{d}, \sgn(z_{|i|}) = \sgn(i), \text{ and } |z_{|i|}| > |z_j|, \forall j\neq i\}. $$ 
Essentially, the ball centered at the original with radius $\sqrt{m_{d}}$ in  $\mathbb{R}^d$ is partitioned into $2d$ disjoint parts $A_i$'s that have the same shape. 
For our later convenience, we define $A_0 \, =\, \{z\in \R^d: \|z\|^2> m_{d}\}$ the complement of the ball. We have $\P(Z\in A_0)=\P(\psi(Z)=0)$. 

Define $$\lambda_{0,d}\,:=(2d)^{-1} \E\left( Z_1 \mid \vZ \in A_1 \right)^{2},~ \text{where } ~\vZ\sim N(0, \bs{I}_{d}).$$ 
This number will be used in the following two propositions, which show that the joint distribution $(\vX,Y)$ in \eqref{eq:joint-x-y-B}  enjoys the following desired properties.

\begin{proposition}\label{prop:joint-class}
	For any $\vB\in \mathbb{O}(p,d)$ and $\bbP_{\vB}$ constructed by Equation~\eqref{eq:joint-x-y-B}, $Y$ can be represented as $f(\vB\tp \vX, \epsilon)$ for $\epsilon\sim N(0,1)$. 
 Furthermore, the joint distribution of $(\vX, Y)$ belongs to $\mathfrak{M}\left(p,d,\lambda\right)$ where $\lambda = \rho^{2}\lambda_{0,d}$. 
 
\end{proposition}

\begin{proposition}\label{prop:joint-kl}
	Suppose $\vB$ and $\widetilde{\vB}$ are in $\mathbb{O}(p,d)$. 
	Let $\bbP_{\vB}$ and $\bbP_{\widetilde{\vB}}$ be defined by Equation~\eqref{eq:joint-x-y-B}. 
\begin{enumerate}
\item For any $\rho\in (0,1)$, it holds that \[
	\kl(\bbP_{\vB}, \bbP_{\widetilde{\vB}})\leqslant \frac{\rho^2}{2(1-\rho^2)}\|\vB - \widetilde{\vB}\|_{F}^{2}.
	\]
    \item 
 There exist a universal constant $C$ and a constant $\delta_{d}=\Theta(d^{-7.1/2})$ such that for any $\rho\in (0, \delta_d]$, it holds that
	\[
	\kl(\bbP_{\vB}, \bbP_{\widetilde{\vB}})\leqslant \frac{C \rho^2}{1-\rho^2} \lambda_{0,d} \|\vB - \widetilde{\vB}\|_{F}^{2}.
	\]
 
\end{enumerate}
\end{proposition}

To apply Lemma~\ref{lem:fano}, we also need a packing set from the following result. 
\begin{lemma}[Packing Set]\label{lem:packing}
	For any $\epsilon \in (0, \sqrt{ 2(d \wedge (p-d))}]$ and any $\alpha \in (0,1)$, there exists a subset $\Theta\subset \mathbb{O}(p,d)$ such that 
	\[
	|\Theta| \geqslant \left( \frac{c_0}{\alpha}  \right)^{d(p-d)}, 
	\]
	and for any $\vB, \widetilde{\vB}\in \Theta$, 
	\[
	\|\vB-\widetilde\vB\|_{F} \leqslant 2 \epsilon, \qquad
	\|\vB\vB\tp -\widetilde{\vB}\widetilde{\vB}\tp\|_F \geqslant  \alpha\epsilon, 
	\]
	where $c_0$ is an absolute constant. 
\end{lemma}

The proofs of Propositions~\ref{prop:joint-class}, \ref{prop:joint-kl}, and Lemma~\ref{lem:packing} will be given in the subsequent subsections. 
Note that Propositions~\ref{prop:joint-class} and \ref{prop:joint-kl} are highly technical and nontrivial while Lemma~\ref{lem:packing} can be obtained using results from the existing literature.

We now prove Theorem~\ref{thm:lower-smallp}. 
Fix any $\alpha \in (0,1)$ (e.g. $\alpha=1/2$) and take $\Theta$ to be the subset in Lemma~\ref{lem:packing}.
For each $\vB\in \Theta$, define $\bbP_{\vB}$ by Equation~\eqref{eq:joint-x-y-B}. 
Proposition~\ref{prop:joint-class} guarantees that $\bbP_{\vB}\in \mathfrak{M}\left(p,d,\lambda\right)$. 
Denoted by $\varpi_{d}:=\min(\delta_{d}^{2}, 1/2)\lambda_{0,d}$. 
Suppose $\lambda\leqslant \varpi_{d}$. Let $\rho=\sqrt{\lambda/\lambda_{0,d}}$. Then  $1/(1-\rho^{2})\leqslant 2$ and we can apply the second statement of Proposition~\ref{prop:joint-kl} to bound the KL-divergence between each pair of different populations $\bbP_{\vB}^{n}$ and $\bbP_{\widetilde{\vB}}^{n}$ for $\vB, \widetilde{\vB}\in \Theta$.

 Let $\epsilon^{2} = c_1 \frac{d (p-d)}{ C n \lambda}$, 
	where $C$ is the constant in Proposition~\ref{prop:joint-kl} and $c_1$ is a constant such that $c_1/(\log(c_0/\alpha)\leqslant 1/16$ for $c_0$ in Lemma \ref{lem:packing}. 
Then, using Lemma~\ref{lem:fano}, we have 
	\begin{align}
	\label{eq:fano:argument-new} & \inf_{\widehat{\vB}}  \sup_{\vB \in \Theta}\bbE\|\widehat{\vB}\widehat{\vB}\tp - \vB\vB\tp \|^{2}_{F}\\
	\geqslant& \min_{\vB, \widetilde{\vB}\in \Theta, \vB\neq \widetilde{\vB}} \|\vB\vB\tp-\widetilde{\vB}\wt{\vB}\tp\|^{2}_{F} \left(1-\frac{\max \kl(\bbP_{\vB}^{n}, \bbP_{\widetilde{\vB}}^{n}) +\log(2)}{\log(|\Theta|)}\right) \nonumber \\
	\geqslant & \alpha^2\epsilon^{2}\left(1-\frac{ 8 C n \rho^2 \lambda_{0,d}  \epsilon^{2}+\log2}{\log(|\Theta|)} \right) \nonumber \\
	\geqslant & \alpha^2 c_1 \,\cdot\,  \frac{d(p-d)}{ n\lambda } \,\cdot \, \left(1- \frac{8c_1 d(p-d)}{\log(|\Theta|)} - \frac{\log2}{\log(|\Theta|)} \right).  \nonumber
	\end{align} 
	Since $\log|\Theta| > d(p-d) \log( c_0 / \alpha)\geqslant 16 c_1 d (p-d)$, we have 
	\begin{align}\nonumber
	&\inf_{\widehat{\vB}} \sup_{\vB \in \Theta}\bbE_{\vB}\|\widehat{\vB}\widehat{\vB}\tp - \vB\vB\tp \|^{2}_{F}\gtrsim \frac{d(p-d)}{n\lambda}. 
	\end{align}
We complete the proof of Theorem~\ref{thm:lower-smallp}.

\subsection{Proof of Proposition~\ref{prop:joint-class}}\label{app:pf-prop:joint-class}

To prove Proposition~\ref{prop:joint-class}, we need the following lemma. 
\begin{lemma}\label{lem:joint-eigen}
	Suppose $(\vZ, Y)$ is constructed as in Equation~\eqref{eq:joint-x-y-B}. 
Define $\bs{e}_{0}=\bs{0}$ and for $i=1,\ldots, d$, define $\bs{e}_{-i}=-\bs{e}_{i}$. 
	\begin{enumerate}
		\item $P(W=i)=(4d)^{-1}$, for $i=\pm 1,\ldots, \pm d$ and 
 $\E \left[\bs{e}_W\bs{e}_{W}\tp \right]=\frac{1}{2d}\bs{I}$. 
		\item If $y \in (  i-\sigma,  i +\sigma)$, then $\vl(y)=\E[ \vZ \mid Y=y]=  \E(Z_1| \psi (\vZ)= 1 )  \ve_{i} = \sqrt{2 d \lambda_{0,d}} \ve_{i} $ for $i=\pm 1,\ldots, \pm d$. \\ 
  If $y \in (  -\sigma,  \sigma)$, then $\vl(y)= \bs{0} $ .
		\item $\Cov \left(\E[ \vZ \mid Y] \right)=\lambda_{0,d} \bs{I}_{d}$. 
		\item All eigenvalues of $\Cov( \E[ \vZ \mid Y] )$ equal to $\lambda_{0,d}$. 
		\item   $\vl(y)$ satisfies the weak $(K, \tau)$-sliced stable condition for any $K\geqslant  4d \max(1+2\gamma, \tau)$, where $\gamma\in (0,1)$ is defined as in Definition \ref{def:weak-sliced-stable}.  
  \item $\frac{1}{100 d}\leqs \lambda_{0,d}\leqs  \frac{4 \log(2d)}{d}$ and $\frac{1}{8d\sqrt{2}} \leqs \E\left( Z_1 \one_{ \vZ \in A_1} \right) \leqs \frac{\sqrt{2\log (2d)}}{2d}$. Furthermore, $\lambda_{0,d}\asymp \frac{\log(d)}{d}$ as $d\rightarrow \infty$. 
	\end{enumerate}
\end{lemma}
\begin{proof}[Proof of Lemma~\ref{lem:joint-eigen}]~
\begin{itemize}
    \item[1:] 	It is a direct corollary of the fact that $\bs Z\sim N(0,\vI_d)$ and $m_d$ is the median of $\|\vZ\|^2$. 
    \item[2:] Since $A_0$ is the complement of a ball centered at $\bs{0}$, it is rationally invariant. By the symmetry of standard normal random vectors, $\E(\vZ\mid  \psi(\vZ) = 0 ) \overset{a.s.}{=}  \E(\vZ\mid \vZ \in A_0)=\bs{0}$. \\  Fix any $i=1,\ldots, d$ and any $\Upsilon = \pm 1$. Under the condition that $y \in ( \Upsilon i-\sigma, \Upsilon i +\sigma)$, one  has
    \begin{align*}
	&\E[\vZ|Y=y]=\E[\vZ|\psi(\vZ)=\Upsilon i]=\E[\vZ|\Upsilon Z_i=\max_{j\in [d]}|Z_j|],\\
	&\E[Z_k|\Upsilon Z_i=\max_{j\in [d]}|Z_j|]=0(\forall k\neq i),\\
 &\E[Z_i|\Upsilon Z_i=\max_{j\in [d]}|Z_j|]= \Upsilon \E[Z_1 \mid Z_1 = \max_{j\in [d]} |Z_j|] = \Upsilon  \E[Z_1 \mid \psi(\vZ)=1].\\
	\end{align*}

   \item [3 \& 4:] By the first and the second statements, 
    \begin{align*}
    \Cov \left(\E[ \vZ \mid Y] \right)=&\E[\E[\vZ|Y]\E[\vZ\tp|Y]]\\
=&\frac{1}{4d}\sum_{i=-d}^d \E^2(Z_1| \psi (\vZ)= 1 )\bs e_i\bs e_i\tp \\
 =& \lambda_{0,d} \bs{I}_{d}. 
\end{align*}
    \item[5:] In the following, fixed $\bbeta\in \R^d$. Since $\ell(y)=\bs{0}$ for $y\in (-\sigma, \sigma)$, we can focus on the case where $|Y|>\sigma$. Fix any $\gamma\in (0,1)$ as defined in Definition \ref{def:weak-sliced-stable}.  
	
	Let $J_{ i}=(  i-\sigma,  i +\sigma)$ for each $i=\pm 1, \ldots, \pm d$. 
	Suppose $\{\S_h = [a_{h-1}, a_{h}) :h=1,\ldots, H\}$ is a partition of $[-d-\sigma,d+\sigma]$ such that 
	$$
	\frac{1-\gamma}{H} \leqslant \mathbb{P}(Y\in \S_{h})\leqslant \frac{1+\gamma}{H}, \qquad \forall h=1,\ldots, H. 
	$$
	Since $(1+\gamma)/H\leqslant (1+\gamma)/K < (4d)^{-1}=P(Y\in J_{i})$ for any $i=\pm 1, \ldots, \pm d$, we conclude that $\S_h$ can overlap with at most two $J_{i}$'s. If $\S_h$ is covered by some $J_{i}$, then $\mathrm{Cov}\left(\bbeta\tp\vl(Y)\big|  Y\in \S_h \right)=0$ because of the second statement.  If $\S_h$ overlaps with $J_{i}$ and $J_{k}$, then by the AM-GM inequality, it holds that 
	\begin{align*}
	\mathrm{Cov}\left(\bbeta\tp\vl(Y)\big|  Y\in \S_h \right) & = \bbP( Y\in J_i \mid Y\in \S_h)\bbP( Y \in J_k \mid Y\in \S_h) \left( \vl(i)\tp \bbeta  -  \vl(k)\tp \bbeta   \right)^{2}  \\
	& \leqslant 2^{-1} \left(  \left[\vl(i)\tp \bbeta \right]^{2} + \left[ \vl(k)\tp \bbeta \right]^{2}  \right). 
	\end{align*}
	Summing over all $h$, one has 
	\begin{align*}
	& \qquad \frac{1}{H} \sum_{h=1}^{H}	\mathrm{Cov}\left(\bbeta\tp\vl(Y)\big|  Y\in \S_h \right) \\
 & \leqslant \frac{1}{H} \sum_{h=1}^{H}\sum_{\substack{
  i\neq k,  \\ i,k \in \{\pm 1, \dots, \pm d\}
     }
 } \one_{\S_h \cap J_i\neq 0} \one_{\S_h \cap J_k  \neq 0}2^{-1} \left(  \left[\vl(i)\tp \bbeta \right]^{2} + \left[ \vl(k)\tp \bbeta \right]^{2}  \right) \\
& \leqslant  \frac{1}{H} \sum_{i \in \{\pm 1, \dots, \pm d\} } \left[\vl(i)\tp \bbeta \right]^{2}\\
	& = \frac{4 d}{H} \mathrm{Cov}\left(\bbeta\tp\vl(Y) \right), 
	\end{align*}
	where the last inequality is due to the fact that for each $i$, there are at most two values of $h$ such that $\S_h$ is overlapped with $J_i$ but not covered by $J_i$ and thus $\left[\vl(i)\tp \bbeta \right]^{2}$ appears at most twice in the summation. 
	Since $H/(4d)\geqslant K/(4d)\geqslant \tau$, we conclude that $\ell(y)$ satisfies the weak $(K,\tau)$-sliced stable condition w.r.t. $Y$. 

\item[6:] Suppose $\vZ\sim N(0, \bs{I}_{d})$. By symmetry, 
	\begin{align*}
	\E\left( Z_1 \one_{ \vZ \in A_1} \right) =&\E\left( Z_i \one_{ \vZ \in A_i} \right)  \quad (\text{for any }~ i\in[d]) \\
	=&  (2d)^{-1} \sum_{i\in \{\pm 1,\ldots, \pm d\} } \E\left( |Z_{|i|}|  \one_{ \vZ \in A_i} \right)\\
		=&  (2d)^{-1} \sum_{i\in \{\pm 1,\ldots, \pm d\} } \E\left(  \max_{i=1,\ldots, d} |Z_{i}|  \one_{ \vZ \in A_i} \right)\\
	=& (2d)^{-1}  \E\left( \max_{i=1,\ldots, d} |Z_{i}| \one_{\| \vZ \|^2 \leqs m_{d}} \right). 
\end{align*}

To obtain an upper bound, we note that 
\begin{align*}
	 \E\left( \max_{i=1,\ldots, d} |Z_{i}| \one_{\| \vZ \|^2 \leqs m_{d}} \right) 
	 \leqs &  \E\left( \max_{i=1,\ldots, d} |Z_{i}| \right) \\
\leqs & \sqrt{ 2\log (2d)},
\end{align*}
where the last inequality is due to the maximal inequality of Gaussian r.v.s. (See Section~2.5 in \cite{boucheron2013}).

To get a lower bound, we note that $\max_{i}|Z_i|\geqslant \sqrt{ \frac{1}{d} \sum_{i} Z_{i}^2}$. Therefore
\begin{align*}
 \E\left( \max_{i=1,\ldots, d} |Z_{i}| \one_{\| \vZ \|^2 \leqs m_{d}} \right)
& \geqslant d^{-1/2} \E\left( \|\vZ\| \one_{\| \vZ \|^2 \leqs m_{d}}\right) \\
& \geqslant d^{-1/2} \frac{1}{10} \E \|\vZ\| \\ 
& \geqslant \frac{1}{10}\sqrt{\frac{d-\frac{1}{2}}{d}} \geqslant \frac{1}{10\sqrt{2}},
\end{align*}
where the second inequality is by part 2 of Lemma~\ref{lem:chisq} and the third is due to a lower estimate used in the proof of that lemma. 
The two bounds together yield $$\frac{1}{20 d\sqrt{2}}\leqs \E\left( Z_1 \one_{ \vZ \in A_1} \right)\leqs \frac{ \sqrt{ 2\log (2d)}}{2d}.$$

Recall that $\lambda_{0,d}\,:=(2 d)^{-1} \E\left( Z_1 \mid \vZ \in A_1 \right)^{2}$.  
By part 1, $\P(\vZ \in A_1)=(4d)^{-1}$. Therefore, $\frac{1}{100 d} \leqs \lambda_{0,d}\leqs \frac{4 \log (2d)}{d}$.

For $d$ sufficiently large, we have $m_{d}>\sqrt{\log d}$. Following the same proof of Equation~(3.14) in \citet[Chapter 3.3]{ledoux1991probability}, there is some positive constant $c_0$ such that $$ \E\left( \max_{i=1,\ldots, d} |Z_{i}| \one_{\| \vZ \|^2 \leqs m_{d}} \right) \geqslant c_0 \sqrt{\log d}.$$ 
Therefore,  $\lambda_{0,d}\asymp \frac{\log d}{d}$.

\end{itemize}
\end{proof}

\begin{proof}[Proof of Proposition~\ref{prop:joint-class}]
	Note that any $k\geqslant 2$ independent uniform random variable sequence (r.v.s. for short) can be constructed from a single $U \sim \text{Unif}(0,1)$ as follows. Represent $U$ as $\sum_{j=1}^{\infty} k^{-j} a_{j}$ for $a_{j}\in \{0, 1,\dots, k-1\}$.	Let $U^{(i)}=\sum_{j=1}^{\infty} k^{-j} a_{(j-1)  k + i}$ for each $i= 1,\dots, k$. Since 
	$a_{j}$'s are independent and identically distributed, we conclude that $U^{(i)}$'s are independent and identically distributed, each following Unif($0,1$). 
	
Let $U=\Phi(\epsilon)$, where $\Phi$ is the cumulative density function (C.D.F.) for the standard normal distribution. Since $\epsilon\sim N(0,1)$, one has $U\sim \text{Unif}(0,1)$. 
Let $k=1+d$. 
Using the above construction of $U^{(i)}$'s and let $\xi_j=\Phi^{-1}(U^{(j)})$ for $1\leqslant j\leqslant d$ and $\eta=\sigma \cdot U^{(1+d)}$, we can represent $Y$ as a function of $\vB\tp \vX$ and $\epsilon$. 
	
	Since 
	\begin{align*}
\begin{pmatrix}
  \bs X \\
  \bs Z
\end{pmatrix}=\begin{pmatrix}
 \bs I_p & \bs 0\\
 \rho\bs B^\top &\sqrt{1-\rho^2}\bs{I}_d
\end{pmatrix}	 
	 \begin{pmatrix}
 \bs X \\
  \xi
\end{pmatrix},
	\end{align*}
the joint distribution of $(\bs X,\bs Z)$ is also normal. Thus by elementary results for normal distributions, one has 
\begin{align*}
\vX\mid \vZ \sim& N\left(\E[\bs X]+\Cov(\vX,\vZ)\mr{var}^{-1}(\bs Z)(\bs Z-\E[\bs Z]),  \mr{var}(\bs X)-\Cov(\vX,\vZ)\mr{var}^{-1}(\bs Z)\Cov(\vZ,\vX)\right)\\
=&N( \rho\vB \vZ, \bs{I}_{p} - \rho^2 \vB\vB\tp).
\end{align*}
 Hence
	\begin{align*}
	\E[\vX \mid Y]
	=\E[ \E[ \vX \mid \vZ, \eta] \mid Y]
	= \E[ \rho \vB \vZ \mid Y]
	=	\rho \vB \E[ \vZ \mid Y],
	\end{align*}
	and 
	\begin{align*}
	\Cov( \E[\vX \mid Y] )
	=	 \rho^2 \vB \Cov( \E[ \vZ \mid Y] ) \vB\tp.
	\end{align*}
	Lemma~\ref{lem:joint-eigen} shows that all eigenvalues of $\Cov( \E[ \vZ \mid Y] )$ equal to $\lambda_{0,d}$, and $\vl(y)=\E[ \vZ \mid Y=y]$ satisfies the weak $(K, \tau)$-sliced stable condition for any $K\geqslant  4d \max(2, \tau)$. 
 If we choose $\tau=2$ and $\gamma=1/2$, we conclude that the distribution of $(\vX, Y)$ belongs to $\mathfrak{M}\left(p,d,\rho^{2}\lambda_{0,d}\right)$. 
 
\end{proof}

\subsection{Proof of Proposition~\ref{prop:joint-kl}}\label{app:proof-prop-joint-kl}
The first statement is relatively simple to prove. 
By the construction in Equation~\eqref{eq:joint-x-y-B}, $\bbP_{\vB}( Y \mid \vX, \vZ)=\bbP_{\widetilde{\vB}}(Y\mid \vX, \vZ)$, a.s. 
By basic properties of KL-divergence, 
	\begin{align}
	\kl(\bbP_{\vB}, \bbP_{\widetilde{\vB}}) \leqslant  &\;  	\kl(\bbP_{\vB}, \bbP_{\widetilde{\vB}})  + \E_{\vX, Y\sim \bbP_{\vB} }\left(
	\kl( \bbP_{\vB}( \vZ \mid \vX, Y),  \bbP_{\widetilde{\vB}}(\vZ\mid \vX, Y) 	\right)  \nonumber \\
	=	&	\,  \kl (\bbP_{\vB}(\vX, \vZ, Y), \bbP_{\widetilde{\vB}}(\vX, \vZ, Y)) \nonumber \\
	=&\,  \kl (\bbP_{\vB}(\vX, \vZ), \bbP_{\widetilde{\vB}}(\vX, \vZ)).
	\end{align}
Furthermore, let $\phi_{p}(\vx)$ be the density function  for $N(0,\bs{I}_{p})$. Then we have
	\begin{align*}
	&\kl (\bbP_{\vB}(\vX, \vZ), \bbP_{\widetilde{\vB}}(\vX, \vZ)) \\
 = & \E_{\vB}\left[ \log \left( \frac{\phi_p(\vX) (1-\rho^2)^{-d/2} \phi_{d}(\vZ- \rho \vB\tp \vX) }{\phi_p(\vX) (1-\rho^2)^{-d/2}  \phi_{d}(\vZ- \rho  \widetilde{\vB}\tp \vX) } \right) \right] \nonumber \\ 
	 =& \E_{\vB}\left[\frac{1}{2(1-\rho^2)} \left( \|\vZ- \rho \vB\tp \vX\|^2-\|\vZ- \rho \widetilde{\vB}\tp \vX\|^2 \right)\right]\\
   =&\frac{\rho^2}{2(1-\rho^2)} \E\left[\| (\vB-\widetilde{\vB})\tp \vX\|^2\right]=\frac{\rho^2}{2(1-\rho^2)}\|\vB-\widetilde{\vB}\|_{F}^2. 
	\end{align*}
 The first statement is proved.

The rest of the proof is about the second statement. 
By the construction in Equation~\eqref{eq:joint-x-y-B}, $\bbP_{\vB}( Y \mid \vX, W)=\bbP_{\widetilde{\vB}}(Y\mid \vX, W)$, a.s. 
By basic properties of KL-divergence, 
	\begin{align}
	\kl(\bbP_{\vB}, \bbP_{\widetilde{\vB}}) \leqslant  &\;  	\kl(\bbP_{\vB}, \bbP_{\widetilde{\vB}})  + \E_{(\vX, Y)\sim \bbP_{\vB} }\left(
	\kl( \bbP_{\vB}( W \mid \vX, Y),  \bbP_{\widetilde{\vB}}(W\mid \vX, Y) 	\right)  \nonumber \\
	=	&	\,  \kl (\bbP_{\vB}(\vX, W, Y), \bbP_{\widetilde{\vB}}(\vX, W, Y)) \nonumber \\
	=&\,  \kl (\bbP_{\vB}(\vX, W), \bbP_{\widetilde{\vB}}(\vX, W)). \label{eq:lower-kl-1}
	\end{align}
Furthermore, let $\phi_{p}(\vx)$ be the density function  for $N(0,\bs{I}_{p})$. Then we have
	\begin{align*}
	&\kl (\bbP_{\vB}(\vX, W), \bbP_{\widetilde{\vB}}(\vX, W)) \\
 = & \E_{\vB}\left[ \log \left( \frac{\phi_p(\vX) \bbP_{\vB}(\vZ\in A_W \mid \vX) }{\phi_p(\vX) \bbP_{\widetilde{\vB}}(\vZ\in A_W \mid \vX) } \right) \right] \nonumber \\ 
	 =& \E_{\vB}\left[ \log \left( \frac{  \bbP_{\vB}(\vZ\in A_W \mid \vX) }{ \bbP_{\widetilde{\vB}}(\vZ\in A_W \mid \vX) } \right) \right]. 
	\end{align*}
 Since we need to analyze the probability $\bbP_{\widetilde{\vB}}(\vZ\in A_W \mid \vX)$, it is convenient to express it as functions of $(\widetilde{\vB}-\vB)\tp \vX$. 
In the following, $\bxi$ is a generic random vector that is independent with everything else and follows $N(0, \bs{I}_{d})$. 
 Let $\Delta \vB=\widetilde{\vB}- \vB$. 
For any fixed $\vmu\in \R^d$, $w\in \{-d,\ldots, d\}$, define $g_{w}^{\vmu}(\vt)=\P(\rho \vmu+\rho \vt + \sqrt{1-\rho^2} \bxi \in A_{w})$ for $\vt\in \R^{d}$. 
Now, we have $\bbP_{\vB}(\vZ\in A_W \mid \vX) =g_{W}^{\vB\tp \vX}(\bs{0})$ and  $\bbP_{\widetilde{\vB}}(\vZ\in A_W \mid \vX) =g_{W}^{ \vB\tp \vX}(\Delta \vB\tp \vX)$.  Furthermore, 
	\begin{equation}
 \kl (\bbP_{\vB}(\vX, W), \bbP_{\widetilde{\vB}}(\vX, W)) = - \E_{\vB} \left[\log g_{W}^{\vB\tp \vX}(\Delta \vB\tp \vX) -\log g_{W}^{\vB\tp \vX}(\bs{0})  \right]\label{eq:lower-kl-2}. 
	\end{equation}

 \smallskip 
 It is pedagogical to provide an overview of our argument in obtaining an upper bound of the KL-divergence. We will apply a second-order Taylor expansion to ${ \log g_{W}^{\vB\tp \vX}(\bs{t})}$ around $\bs{0}$.  Since the first-order derivative has a zero expectation, we only need a careful examination of the second derivative. At the end, we can show that the KL-divergence is close enough to $\frac{\rho^2}{1-\rho^2} 
  \Tr\left( \Delta \vB\tp \Delta \vB \E_{\vB}\left[   \E \left(\bxi \mid \psi(\bxi) = W \right)^{\otimes} \right]   \right)$ when $\rho$ is sufficiently small. 
By our construction in Equation~\eqref{eq:joint-x-y-B}, 
 $$\E_{\vB}\left[\E \left(\bxi \mid \psi(\bxi)=W \right)^{\otimes}\right]=\lambda_{0,d} \bs{I}_d,$$
 which implies that the KL-divergence is closed to $\frac{\rho^2 \lambda_{0,d} }{1-\rho^2} \Tr\left( \Delta \vB\tp \Delta \vB  \right)$. 
\medskip 

\textbf{I. Taylor expansion. }
 By Taylor expansion with an integral remainder, i.e., $f(\boldsymbol{t})=f(\boldsymbol{0})+\nabla f(\boldsymbol{0}) \boldsymbol{t}+ \int_{0}^{1}   \boldsymbol{t}^{\top}\nabla^{2} f( s  \boldsymbol{t}) \boldsymbol{t} (1-s) \diff s $, one has 
	\begin{align}
	& \log g_{w}^{\vB\tp \vx}(\Delta \vB\tp \vx) -\log g_{w}^{\vB\tp \vx}(\bs{0}) \nonumber \\
= &  \langle \Delta \vB\tp \vx, \nabla\log g_{w}^{\vB\tp \vx}(\bs{0}) \rangle \ldots \nonumber \\
 & + \int_0^1  \vx\tp \Delta\vB \left( \nabla^{2}\log g_{w}^{\vB\tp \vx}( \alpha   \Delta \vB\tp \vx )   \right)  \Delta \vB\tp \vx (1-\alpha) \diff \alpha . \label{eq:lower-kl-taylor}
	\end{align}

\begin{lemma}\label{lem:lower-derivatives}
  The derivative of $\log g_{w}^{\vmu}(\bs{t})$ is
\begin{equation}\label{eq:lower-kl-g-1}
		\nabla\log g_{w}^{\vmu}(\bs{t}) =\sqrt{\frac{\rho^{2}}{1-\rho^2} } \E \left( \bxi \mid \rho\vmu+\rho\vt +\sqrt{1-\rho^2} \bxi \in A_w \right),
		\end{equation} 
		and the second-order derivative is 
\begin{align}
		\nabla^{2}\log g_{w}^{\vmu}(\bs{t}) & =\frac{\rho^{2}}{1-\rho^2} \left\{ - \bs{I}_{d} 
 + \E\left( \bxi^{\otimes} \mid  \rho\vmu+\rho\vt +\sqrt{1-\rho^2} \bxi \in A_w \right)  \right. \nonumber \\
		&\qquad \qquad \quad \left. -   \E \left( \bxi \mid \rho\vmu+\rho\vt +\sqrt{1-\rho^2} \bxi \in A_w \right)^{\otimes}  \right\}.  \label{eq:lower-kl-g-2} 
		\end{align}
 \end{lemma}

	By Equation~\eqref{eq:lower-kl-g-1},  
	\begin{align*}
	&  \E_{\vB}  \langle \Delta \vB\tp \vX, \nabla\log g_{W}^{\vB\tp \vX}(\bs{0}) \rangle\\
	=\, & \sqrt{\frac{\rho^{2}}{1-\rho^2} }   \E_{\vB} \left\langle \Delta \vB\tp \vX, \E \left( \bxi \mid \rho\vB\tp \vX +\sqrt{1-\rho^2} \bxi \in A_W, \vX, W \right) \right\rangle. \\
		\end{align*}
		
We split the expectation into parts given by $\{W=w\}$ and use properties of conditional expectation to obtain
\begin{align*}
	&  \E_{\vB} \left\langle \Delta \vB\tp \vX, \E \left( \bxi \mid \rho\vB\tp \vX +\sqrt{1-\rho^2} \bxi \in A_W, \vX, W \right) \right\rangle \\
=\, &\sum_{w=-d}^{d}  \E_{\vB} \left( \one_{W=w} \left\langle \Delta \vB\tp \vX, \E \left( \bxi \mid \rho\vB\tp \vX +\sqrt{1-\rho^2} \bxi \in A_w, \vX \right) \right\rangle \right) \\
=\, &\sum_{w=-d}^{d}  \E_{\vB} \left(  \P\left( W=w \mid \vX \right) \left\langle \Delta \vB\tp \vX, \frac{\E \left( \bxi \one_{\rho\vB\tp \vX +\sqrt{1-\rho^2} \bxi \in A_w} \mid \vX \right)}{\P\left( \rho\vB\tp \vX +\sqrt{1-\rho^2} \bxi \in A_w\mid \vX \right) } \right\rangle \right) \\
=\, &\sum_{w=-d}^{d}  \E_{\vB} \left(  \left\langle \Delta \vB\tp \vX, \E \left( \bxi \one_{\rho\vB\tp \vX +\sqrt{1-\rho^2} \bxi \in A_w} \mid \vX \right)\right\rangle \right) \\
=\, &\sum_{w=-d}^{d}  \E_{\vB} \left(  \left\langle \Delta \vB\tp \vX, \bxi \one_{\rho\vB\tp \vX +\sqrt{1-\rho^2} \bxi \in A_w}\right\rangle \right) \\
=\, &  \E_{\vB} \left(  \left\langle \Delta \vB\tp \vX, \bxi\right\rangle \right)=0, \\
	\end{align*}
where the second equation is due to fact that the conditional distribution of $Z\mid \vX \overset{d}{=} \rho\vB\tp \vX +\sqrt{1-\rho^2} \bxi $ and the last equation is because $\xi$ is independent with $\vX$. 
We thus showed that 
\begin{align*}
&  \E_{\vB}  \langle \Delta \vB\tp \vX, \nabla\log g_{W}^{\vB\tp \vX}(\bs{0}) \rangle\\  =\, & \sqrt{\frac{\rho^{2}}{1-\rho^2} }   \E_{\vB} \left( \langle \Delta \vB\tp \vX, \bxi \rangle\right)= 0. 
\end{align*}

Therefore, it suffices to focus on the second-order term in Equation~\eqref{eq:lower-kl-taylor}. 

\textbf{II. Analysis of the second-order term. }
In the following, we fix any $\alpha \in (0,1)$.

As a shorthand, we write $J(w, \rho)\,:=\, (1-\rho^2)^{-1/2}(A_w- \rho\vB\tp \vX- \rho\alpha   \Delta \vB\tp \vX)$, which is a random set that depends on $\vX$ with parameters $w$ and $\rho$.

By Equation~\eqref{eq:lower-kl-g-2}, one has
	\begin{align}
	& - \E_{\vB} \left[ \vX\tp \Delta \vB ~ \nabla^{2}\log g_{W}^{\vB\tp \vX}(\alpha \Delta\vB\tp \vX) ~ \Delta \vB\tp \vX \right] \label{eq:lower-kl-2nd} \\
	=\, & \frac{\rho^{2}}{1-\rho^2} \Tr\left( \E_{\vB}  \left\{\Delta \vB\tp \vX \vX\tp \Delta \vB \left[ 
\bs{I}_{d}  -  \E\left( \bxi^{\otimes} \mid  \bxi \in J(W, \rho), \vX, W \right)   \right] \right\}  \right) \nonumber  \\
	& + \frac{\rho^{2}}{1-\rho^2} \Tr\left( \E_{\vB}  \left\{\Delta \vB\tp \vX \vX\tp \Delta \vB \left[     
 \E \left(\bxi \mid  \bxi \in J(W, \rho), \vX, W \right)^{\otimes} \right] \right\}  \right). \nonumber
	\end{align}

The rest of the proof is dedicated to bounding the two terms in \eqref{eq:lower-kl-2nd} by dropping the factor $\frac{\rho^{2}}{1-\rho^2}$.

\smallskip 
\textit{Some intuitions. \newline }
Before we move on, it is worth checking the limits of these two terms as $\rho\rightarrow 0$. In this case, $J(W, \rho)\rightarrow A_W$. 

\begin{enumerate}
    \item [(a).] The inner conditional expectation in the first term 
$$\E\left( \bxi^{\otimes} \mid  \bxi \in J(W, \rho ), \vX, W \right)\rightarrow \E\left( \bxi^{\otimes} \mid  \bxi \in A_{W} \right)=\frac{ \E\left( \bxi^{\otimes} \one_{\bxi \in A_{W}}\right)}{\bbP\left(\bxi \in A_{W}\right)}.$$
Furthermore, when $\rho=0$, $W$ and $\vX$ becomes independent, and the distribution of $W$ is the same as $\psi(\xi)$. Therefore, the expectation of the last equation w.r.t. $W$ equals to $\sum_{w} \E\left( \bxi^{\otimes} \one_{\bxi \in A_{w}}\right) =\E\left( \bxi^{\otimes}\right)=\bs{I}_{d}$, from which we conclude that the first term converges to $0$ as $\rho\rightarrow 0$. 
\item [(b).]
Similarly, the inner conditional expectation in the second term 
$$\E\left( \bxi \mid  \bxi \in J(W, \rho), \vX, W \right)\rightarrow \E\left( \bxi \mid  \bxi \in A_{W} \right)=\sqrt{2 d \lambda_{0,d} }\bs{e}_{W}, $$
where $\{\bs{e}_{1},\ldots,\bs{e}_{d}\}$ is the standard basis of $\mathbb{R}^d$, $\bs{e}_{0}=\bs{0}$, and $\bs{e}_{-i}=-\bs{e}_{i}$ for $i=1,\ldots, d$. 
Therefore, $$\E\left[ \E\left( \bxi \mid  \bxi \in J(W, \rho), \vX, W \right)^{\otimes} \right] \rightarrow \lambda_{0,d}  \bs{I}_{d}. $$
Thus, the second term converges to $\frac{\rho^{2}\lambda_{0,d}}{1-\rho^2} \Tr\left( \E_{\vB}  \left\{\Delta \vB\tp \vX \vX\tp \Delta \vB \right\} \right)=\frac{\rho^{2}\lambda_{0,d}}{1-\rho^2}\|\Delta \vB\|_{F}^2$. 
\end{enumerate}

\smallskip 

These intuitions can be justified rigorously by using the continuous dependence of the probability measure on $\rho$. We state the result in the next two lemmas, whose proofs are deferred.

\begin{lemma}\label{lem:lower-kl-2nd-small}
Let $\epsilon=100 \lambda_{0,d}$. 
There exists a constant $\delta_{d}^{(1)}$ such that for any $\rho\in (0, \delta_{d}^{(1)})$, any $\alpha\in (0,1)$, and any $\vB, \widetilde{\vB} \in \mathbb{O}(p,d)$, 
		$$\Tr\left( \E_{\vB}  \left\{ \Delta \vB\tp \vX \vX\tp \Delta \vB \left[  \bs{I}_{d} -   \E\left( \bxi^{\otimes} \mid  \bxi \in J(W, \rho), \vX, W \right) \right] \right\}\right)\leqslant   2 \epsilon  \| \Delta \vB\|_{F}^{2}.$$
		Furthermore, $\delta_{d}^{(1)}$ can be taken as $c^{\prime} d^{-5/2}$ where $c^{\prime}$ is the constant in Lemma~\ref{lem:lower-diff-bound}. 
\end{lemma}

 \begin{lemma}\label{lem:lower-kl-2nd-main}
Let $\epsilon=100 \lambda_{0,d}$. 
There exist a constant $\delta_{d}^{(2)}$ and a universal constant $C$  such that for any $\rho\in (0, \delta_{d}^{(2)})$, any $\alpha\in (0,1)$, and any $\vB, \widetilde{\vB} \in \mathbb{O}(p,d)$, 

$$ \Tr\left( \E_{\vB}  \left\{\Delta \vB\tp \vX \vX\tp \Delta \vB \left[   \E \left(\bxi \mid  \bxi \in J(W, \rho), \vX, W \right)^{\otimes} \right] \right\}  \right) <\left[\lambda_{0,d} + C\epsilon\right] \|\Delta \vB\|_{F}^{2}.$$
  
Furthermore, $\delta_{d}^{(2)}$ can be taken as $c^{\prime} d^{-7/2 - \varsigma}$ for some universal constant $c^{\prime}$ and any positive number $\varsigma$. 
\end{lemma}
	
We apply Lemmas~\ref{lem:lower-kl-2nd-small} and \ref{lem:lower-kl-2nd-main}, and 
let $\delta_{d}=\min \left(  \delta_{d}^{(1)},  \delta_{d}^{(2)} \right)$. 
In view of Equation~\eqref{eq:lower-kl-2nd}, we conclude that for any $\rho\in (0, \delta_{d} )$, any $\alpha\in (0,1)$, it holds that
	$$
	- \E_{\vB} \left[ \vX\tp \Delta \vB ~ \nabla^{2}\log g_{W}^{\vB\tp \vX}(\alpha \Delta\vB\tp \vX) ~ \Delta \vB\tp \vX \right] \leqslant C  \rho^2 \lambda_{0,d}  \|\Delta \vB\|_{F}^{2}.
	$$
Combining Equations~\eqref{eq:lower-kl-1}, \eqref{eq:lower-kl-2}, and \eqref{eq:lower-kl-taylor}, we conclude that 
	$$\kl(\bbP_{\vB}, \bbP_{\widetilde{\vB}}) \leqslant C \rho^2 \lambda_{0,d}   \|\Delta \vB\|_{F}^{2}.$$
	Therefore, we complete the proof of the proposition. 

\subsection{Lemmas for uniform controls on $d$-dimensional Gaussian measures}
This section collects some results about $d$-dimensional Gaussian random vectors. These results will be used in proving Lemmas~\ref{lem:lower-kl-2nd-small} and \ref{lem:lower-kl-2nd-main}. 

\begin{lemma}[Tail probability for Chi-square]\label{lem:chisq}
Suppose $X$ is a chi-squared random variable with $d$ degrees of freedom ($\chi^2_{d}$). 
\begin{enumerate}
  \item 
For any constant $k>0$, there exists a constant $C_{k}=O(k)$, such that $\P\left(X\geqslant C_{k} d \right)\leqs d^{-k}$ and $\E\left( X^{a} \one_{X>C_{k} d} \right)\leqs  C d^{-k }$ for $a=1,2$, where $C$ is a universal constant. 
\item 
Let $m_d$ be the median of $\chi^2_{d}$. It holds that 
 $$
 \E\left( \sqrt{X} \one_{X\leqs m_d}  \right) \geqslant  \frac{1}{10} \E\left( \sqrt{X}\right)  . 
 $$

\item 
There exists a universal constant $\pi_0>0$, such that 
$\P( X \geqslant  \frac{m_d}{ {1-\rho^2}} ) \geqslant \pi_0$ whenever $\rho^2 < \left( 3d e^{1/3}   \right)^{-1}$. 

\end{enumerate}
\end{lemma}
\begin{proof}[Proof of Lemma~\ref{lem:chisq}]
~\\ \newline
\textbf{Statement 1.}\newline

If $d=1$, it is trivial. Suppose $d\geqslant 2$. By Equation~(4.3) in \citet{laurent2000adaptive}, 
\begin{equation}\label{eq:basic-chisq-tail}
\mathbb{P}(X-d \geqslant 2 \sqrt{d x}+2 x) \leqs \exp (-x).
\end{equation}
Choosing $x=k\log(d)$, we have $\mathbb{P}(X \geqslant  d + 2 \sqrt{k d\log{d}}+2 k\log(d) ) \leqs d^{-k}$.

Let $C_{0}\geqslant k  (\inf_{d\geqslant 2}\frac{d}{\log d})^{-1}$. Then $2 \sqrt{k d\log{d}}+2 k\log(d) \leqs 2(\sqrt{C_0}+C_0) d$. For any $C\geqslant 1+2(\sqrt{C_0}+C_0)=O(k)$, it holds that $\mathbb{P}(X \geqslant  C d  ) \leqs d^{-k}$. 

The inequality~\eqref{eq:basic-chisq-tail} also implies that for any $t>d$,  we have 
$$\P(X>t)\leqs \exp\left( - (t-d)^2/(4t) \right).$$ 
If $t>2d$, then $t-d> t/2$ and $(t-d)^2/(4t)\geqslant t/16$. 
By Fubini's theorem, for any $r>2d$, 
\begin{align*}
  \E\left( X^{a} \one_{X>r} \right)
& = r^{a} \P(X>r)+\int_{r}^\infty a t^{a-1} \P(X>t) \diff t \\
& = r^{a} \P(X>r)+\int_{r}^\infty a t^{a-1} e^{-t/16} \diff t\\
& \leqs e^{-r/16}\left( r^a+16+ \one_{a=2}\left( 240+32 r \right) \right)\\
& \leqs C_{1} r^2 e^{-r/16}. 
\end{align*}
For $r=2d+16 k  \log (d)=d \cdot O(k)$, $C_{1} r^2 e^{-r/16} \leqs \frac{1}{d^k} C_{1} (2 d + 16 k\log (d) )^2 e^{-d/8}\leqs C_{2}\frac{1}{d^k}$ because $d e^{-d/16}$ is bounded. 

~\\ \newline
\textbf{Statement 2.}\newline

Since $m_d$ is the median of $\chi^2_{d}$,  $\P( X \leqs m_d) = 1/2$. 

Note that $\E\left( \sqrt{X} \one_{X\leqs m_d}  \right) = \E\left( \sqrt{X}\right) - \E\left( \sqrt{X} \one_{X> m_d}  \right)$ and by the Cauchy--Schwarz inequality, $\E\left( \sqrt{X} \one_{X> m_d}  \right) \leqs \sqrt{ \E X \P \left( X> m_d \right) }$. Since $m_d$ is the median and $\E X=d$, we have 
\begin{align*}
\frac{\E\left( \sqrt{X} \one_{X\leqs m_d}  \right) }{\E\left( \sqrt{X}\right)} \geqslant 1-\frac{\sqrt{ d / 2 }}{\E\left( \sqrt{X}\right)}. 
\end{align*}
For $d=1$, $\bbE=\sqrt{\frac{2}{\pi}}$ and thus the RHS is no less than $1-\sqrt{\pi}/2 > 0.1$. 
A direct calculation yields 
\begin{align*}
  \E\left( \sqrt{X}\right) 
& = \int_{0}^{\infty} \frac{x^{(d+1)/2 - 1 } e^{-x/2}}{ 2^{d/2} \Gamma(d/2)} dx \\
& = 2^{1/2} \frac{\Gamma( (1+d)/2)}{\Gamma(d/2)} \\
&  \geqslant \sqrt{2 \left( \frac{d}{2} -\frac{1}{4} \right)}, 
\end{align*}
where in the last inequality we have applied a bound on the ratio of gamma functions that $\sqrt{x-\frac{1}{4}}<\frac{\Gamma(x+1 / 2)}{\Gamma(x)}$ proved by \citet{watson1959note}. 
Since $d/(d-1/2)$ is decreasing in $d$ and $d\geqslant 2$,  we conclude that  $\frac{\E\left( \sqrt{X} \one_{X\leqs m_d}  \right) }{\E\left( \sqrt{X}\right)} \geqslant 1/10$. 

~\\ \newline
\textbf{Statement 3.}\newline
	
	By Corollary~3 in \citet{zhang2020non}, there exist uniform constants $C, c>0$ such that  $$\mathbb{P}( X - d  \geqslant x) \geqslant c \exp \left(-C x \wedge \frac{x^2}{k}\right), \quad \forall x>0.$$ 
By the continuity of measure, we have $\mathbb{P}( X  \geqslant  d) \geqslant c$. 
	It remains to check that $\frac{m_d}{{1-\rho^2}}<d$ for $\rho^2$ small.

\citet{berg2006chen} have proved that $m_d\leqs d e^{- 1/ (3d)}$. Note that  
\begin{align*}
\frac{e^{-1/(3d)}}{{1-\rho^2}} \leqs 1 \Leftrightarrow 
e^{-1/(3d)} \leqs 1-\rho^2  \Leftrightarrow 
\rho^2 \leqs 1-e^{-1/(3d)}. 
\end{align*}

By simple calculus, we have an elementary inequality that $1- e^{-x} > x e^{-1/3}$ for any $x\in (0,1/3]$. Therefore, $1-e^{-1/(3d)}>  e^{-1/3} /(3d)$, which is greater than $\rho^2$. In other words, we have $\frac{m_d}{{1-\rho^2}}< d$. 

\end{proof}

\begin{lemma}\label{lem:lower-diff-bound}
Fix any $w\in \{-d, \ldots, 0, 1, \ldots, d\}$ and any $\vmu\in \R^d$. 
Denote by $\mathcal{N}_1$ the distribution $N( \rho\vmu  ; (1-\rho^2) \bs{I}_d)$ and by $\mathcal{N}_0$ the distribution  $N(\bs{0} ;\bs{I}_d)$. Denote by $\vZ$ the element in the sample space. 
\begin{enumerate}
  \item For $\rho^2<1/2$, it holds that $\left|	\bbP_{\mathcal{N}_1}(\vZ\in A_{w})-
	\bbP_{\mathcal{N}_0}(\vZ\in A_{w}) \right| \leqs \rho\|\vmu\|$. 
\item Let $c_0=\log(2)/4 $. 
For any $k>0$, there is a positive constant $c_k$ such that for any $\rho< c_k d^{-k-3/2}$, it holds that if $\rho \|\vmu\|^2<c_0$ then $\left|	\E_{\mathcal{N}_1}( \vZ \one_{\vZ \in A_{w}} )-\E_{\mathcal{N}_0}( \vZ \one_{\vZ\in A_{w}} ) \right| \leqs C \left(d^{-k} +(\rho^{1/3} \|\vmu\|)^2 d^{-k/3}\right)+\rho\|\vmu\|$. 
\item 
In particular, if $\|\vmu\|\leqs C^{\prime} d^{1/2}$, then there exists a constant $c^{\prime}>0$ depending only on $k$, so that \begin{enumerate}
  \item for any $\rho \leqs  c^{\prime}  d^{-k-1/2}$, it holds that $\left|	\bbP_{\mathcal{N}_1}(\vZ\in A_{w})-
	\bbP_{\mathcal{N}_0}(\vZ\in A_{w}) \right| \leqs  8^{-1} d^{-k}$;
	\item for any $\rho\leqs c^{\prime} d^{-k-3/2}$,  it holds that  $\left\|	\E_{\mathcal{N}_1}( \vZ \one_{\vZ \in A_{w}} )-\E_{\mathcal{N}_0}( \vZ \one_{\vZ\in A_{w}} ) \right\| \leqs  O(d^{-k})$.
\end{enumerate}

\end{enumerate}

\end{lemma}

\begin{proof}[Proof of Lemma~\ref{lem:lower-diff-bound}]
~\newline 
\textbf{Statement 1.} \\  
Using the formula for the KL-divergence between Gaussian measures (see for example, Equation A.23 in \citet{rasmussen2006}), we have 
\begin{align*} \mathrm{KL}\left(\mathcal{N}_1|| \mathcal{N}_0\right)= & \frac{1}{2} \log \left|\Sigma_0 \Sigma_1^{-1}\right|+ \\ & \frac{1}{2} \operatorname{tr} \Sigma_0^{-1}\left(\left(\boldsymbol{\mu}_1-\boldsymbol{\mu}_0\right)\left(\boldsymbol{\mu}_1-\boldsymbol{\mu}_0\right)^{\top}+\Sigma_1-\Sigma_0\right) \\ 
=& \frac{1}{2}\left( -d \log(1-\rho^2) + \|\rho\vmu \|^2 + d(1- \rho^2) - d \right) \\
\leqs  & \frac{\rho^2}{2} \|\vmu\|^2, 
\end{align*}
where the last inequality is because $x+\log(1-x)\geqslant 0$ for any $x\in (0, 1/2)$. 

By the definition of total variation, we have
\begin{align*}
&\left|	\P(\rho \vmu+\sqrt{1-\rho^2} \bxi \in A_{w})-
	\P( \bxi \in A_{w}) \right|   \\ 
 \leqs  &  TV(\mathcal{N}_1, \mathcal{N}_0)\leqs \sqrt{ \frac{1}{2} \mathrm{KL}\left(\mathcal{N}_1|| \mathcal{N}_0\right)  } \leqs \frac{\rho\|\vmu\|}{2}, 
\end{align*}
where the second inequality is Pinsker's inequality (see for example Lemma~2.5 in \citet{tsybakov2008introduction}).

~\newline 
\textbf{Statement 2.} \\  

By Lemma~\ref{lem:chisq}, one can choose any $r\geqslant\sqrt{ C_{k} d}$ such that 
\begin{equation}
	\label{eq:lower-conv-z-1}
\E_{\mathcal{N}_0}( \| \vZ\| \one_{\|\vZ\|>r} ) \leqs \sqrt{\P(\|\vZ\|>r)\E_{\mathcal{N}_0}( \| \vZ\|^2 \one_{\|\vZ\|>r} )} \leqs C d^{-k}. 
\end{equation}
Similarly, we can choose any $r\geqslant c_0 + \sqrt{ 2 C_{k} d}$ such that 
\begin{align}
  & \E_{\mathcal{N}_1}(\| \vZ\| \one_{\|\vZ\|> r} ) \label{eq:lower-conv-z-2}\\ 
 \leqs & \E_{\mathcal{N}_0}\left(\| \sqrt{1-\rho^2} \vZ + \rho \vmu \| \one_{\|\sqrt{1-\rho^2} \vZ + \rho \vmu \|> r} \right) \nonumber\\
 \leqs & \rho\|\vmu\|+ \sqrt{1-\rho^2} \E_{\mathcal{N}_0}\left(\|  \vZ  \| \one_{\| \vZ  \|> ( r- \rho \|\vmu\|})/\sqrt{1-\rho^2} \right)\nonumber \\
  \leqs &  \rho\|\vmu\| + \E_{\mathcal{N}_0}\left(\|  \vZ  \| \one_{\| \vZ  \|> \sqrt{C_{k} d}} \right)  \nonumber \\ 
  \leqs &  \rho\|\vmu\| + C d^{- k}  \nonumber 
\end{align}
Below, we fix $r= c_0 + \sqrt{ 2 C_{k} d}$.

Denote by $\phi_{1}( \vz)$ and $\phi_{0}( \vz)$ the density functions of $\mathcal{N}_1$ and $\mathcal{N}_0$, respectively. 
Note that 
$$
\frac{\phi_{1}(\vz)}{\phi_{0}(\vz)}=\exp \left(\frac{-1}{2(1-\rho^2)}\left( \rho \|\vz\|^2-2\rho \vmu\tp \vz + \rho^2 \|\vmu\|^2  \right)  \right)
$$

By calculus, we have an elementary inequality that $t\leqs e^{t}-1\leqs 2 t$ for any $|t|<\log 2$. 
Note that $\left| \rho \|\vz\|^2-2\rho \vmu\tp \vz + \rho^2 \|\vmu\|^2  \right| \leqs 2\rho\|\vz\|^2+ (\rho+\rho^2) \|\vmu\|^2\leqs 2\rho(\|\vz\|^2+\|\vmu\|^2)$. 

Recall that $c_0=(\log 2)/4$. If  $\rho\|\vmu\|^2<c_0$ and $\rho r^2 <c_0$, then 
\begin{align}	\label{eq:lower-conv-z-3}
&\left\|	\E_{\mathcal{N}_1}( \vZ \one_{\vZ \in A_{w}, \|\vZ\|\leqs r} )-\E_{\mathcal{N}_0}( \vZ \one_{\vZ\in A_{w},  \|\vZ\|\leqs r }  ) \right\|\\
\leqs & \int  \vz \one_{\vz \in A_{w},  \|\vZ\|\leqs r} \phi_{0}(\vz) \left| \frac{\phi_{1}(\vz)}{\phi_{0}(\vz)} -1   \right| \diff \vz  \nonumber\\
\leqs & \int  \vz \one_{\vz \in A_{w},   \|\vZ\|\leqs r} \phi_{0}(\vz) 2 \left( 2\rho\|\vz\|^2+ (\rho+\rho^2) \|\vmu\|^2 \right) \diff \vz  \nonumber\\
\leqs  & 2\rho r (r^2+\|\vmu\|^2) \int  \phi_{0}(\vz)  \diff \vz.  \nonumber
\end{align}
Combining Equations~\eqref{eq:lower-conv-z-1}, \eqref{eq:lower-conv-z-2}, and \eqref{eq:lower-conv-z-3}, we have if $\rho r^2<c_0$, then 
$$\left\|	\E_{\mathcal{N}_1}( \vZ \one_{\vZ \in A_{w}} )-
	\E_{\mathcal{N}_0}( \vZ \one_{\vZ\in A_{w}} ) \right\| \leqs \rho \|\vmu\|+2C d^{-k}+ 2\rho r (r^2+\|\vmu\|^2).$$ 
	Recall that $r= c_0 + \sqrt{ 2 C_{k} d}$ and $r^2=O(C_k d)$, one can choose $c_k$ small enough to guarantee $c_k d^{-k-3/2} r^2<c_0$ and $c_k d^{-k-3/2} r^3 < Cd^{-k}$, from which the desired inequality follows. 

Statement 3 is a direct corollary of Statement 2. 
\end{proof}

\begin{lemma}\label{lem:lower-kl-mainterm-tail}
Suppose $\bxi\sim N(0, \bs{I}_{d})$ and $\rho^2<(3d)^{-1}$. 
For any $\vt\in \R^d$ and any $w\in \{\pm 1, \ldots, \pm d\}$, there exists a universal constant $C$
it holds that 
	$$
\|	\E \left( \bxi\  \mid \sqrt{1-\rho^2}\bxi + \vt\in A_w \right)\| \leqslant C ( \sqrt{d}+\|\vt\|).
	$$
\end{lemma}
\begin{proof}
For any  $w=\pm 1, \ldots, \pm d$, by definition of $A_w$, $\sqrt{1-\rho^2} \bxi + \vt\in A_w$ implies that $\|\sqrt{1-\rho^2}\bxi + \vt \|\leqs \sqrt{m_{d}}<\sqrt{d}$. Thus $\|	\E \left( \bxi\  \mid  \sqrt{1-\rho^2} \bxi + \vt\in A_w \right)\|\leqs (\sqrt{d}+\|\vt\|)/\sqrt{1-\rho^2}$. 
	
For $w=0$, by the famous Anderson inequality \citep{anderson1955integral}, we have 
$$
\P( \| \sqrt{1-\rho^2} \bxi + \vt \|\leqs \sqrt{ m_{d}})\leqs \P( \| \sqrt{1-\rho^2} \bxi  \|\leqs \sqrt{m_{d}})\leqs 1-\pi_0<1, 
$$
where the second inequality is due to Lemma~\ref{lem:chisq} and $\rho^2<(3d)^{-1}$. 
Therefore, $\P(\sqrt{1-\rho^2} \bxi + \vt\in A_0)\geqs \pi_0$. 
Note that $ \| \E \left( \bxi \one_{\sqrt{1-\rho^2} \bxi + \vt\in A_0} \right)\|\leqs \sqrt{\E \|\bxi\|^2}=\sqrt{d}$, we conclude that $$\|	\E \left( \bxi\  \mid  \sqrt{1-\rho^2} \bxi + \vt\in A_0 \right)\|\leqs \sqrt{d} / \pi_0.$$ 
\end{proof}

\subsection{Proofs of Lemmas in Section~\ref{app:proof-prop-joint-kl}}
We first prove Lemma~\ref{lem:lower-kl-2nd-small} and Lemma~\ref{lem:lower-kl-2nd-main}. The proof of Lemma~\ref{lem:lower-derivatives} is straightforward and can be found at the end of this section. 

Following the intuitions in Section~\ref{app:proof-prop-joint-kl}, 
we will make use of the continuous dependence of the probability measure on $\rho$. Such continuous dependence is uniform if the set $J(W,\rho)$ is bounded. Therefore, we consider dividing the sample space into two parts: one where $J(W,\rho)$ is bounded, and the other where the expectations are negligible. 

	Let $\bs{\Pi}$ be the projection matrix onto a $(2d)$-dimensional subspace of $\R^p$ formed by the columns of $\vB$ and $\widetilde{\vB}$; if the rank of $[\vB, \widetilde{\vB}]$ is smaller than $2d$, we can always add some extra columns provided that $2d<p$. Then $\vB\tp \vX=\vB\tp \bs{\Pi} \vX$, $\widetilde{\vB}\tp \vX=\widetilde{\vB}\tp \bs{\Pi} \vX$, and $\Delta\vB\tp \vX=\Delta\vB\tp \bs{\Pi} \vX$. 
	We also have $\| \vB\tp \vX +  \alpha   \Delta \vB\tp \vX\|\leqs   \| \bs{\Pi} \vX\|$.

In the proofs, we will fix some positive real number $R=O(d)$ and define two events $E_1\, :=\,  \{ \| \bs{\Pi} \vX \|\leqs R\}$ and $E_2 \, :=\,  \{ \| \bs{\Pi} \vX \|> R\}$. 
We will also make use of the following facts:
	\begin{enumerate}
	\item 
By Lemma~\ref{lem:chisq}, for any $k>0$, we can choose $R=\sqrt{ 2 C_{k}d}$ so that $\P(E_2)\leqs (2d)^{-k}$. 
\item By Lemma~\ref{lem:joint-eigen},  $\epsilon=100\lambda_{0,d}\geqslant d^{-1}$. 
 	\item On the event $E_1$, 
  \begin{equation}\label{eq:lower-kl-basic-2}
      \| \vB\tp \vX+\rho\alpha (\widetilde{\vB}-\vB)\tp \vX\|\leqslant R. 
  \end{equation}
  \item 
For any symmetric matrix $\vA$ of dimension $d$,  $\|\vA\|\bs{I}_{d} -\vA$ is positive definite. For any two positive definite matrices $\vB$ and $\vC$, $\Tr(\vB \vC)\geqslant 0$. Therefore, $\Tr(\vA \vB)\leqs \|\vA\| \Tr(\vB)$. 
\end{enumerate}

\begin{proof}[Proof of Lemma~\ref{lem:lower-kl-2nd-small}]

We write 
$$\E_{\vB}  \left\{ \Delta \vB\tp \vX \vX\tp \Delta \vB  \bs{I}_{d} \right\}=\E_{\vB}  \left\{ \one_{E_1} \Delta \vB\tp \vX \vX\tp \Delta \vB   \right\}+\E_{\vB}  \left\{ \one_{E_2} \Delta \vB\tp \vX \vX\tp \Delta \vB  \right\}$$
and
\begin{align*}
   &  -\Tr\left( \E_{\vB}  \left\{ \Delta \vB\tp \vX \vX\tp \Delta \vB   \E\left( \bxi^{\otimes} \mid  \bxi \in J(W, \rho), \vX, W \right) \right\} \right)\\
\leqs & -\Tr\left( \E_{\vB}  \left\{\one_{E_1} \Delta \vB\tp \vX \vX\tp \Delta \vB   \E\left( \bxi^{\otimes} \mid  \bxi \in J(W, \rho), \vX, W \right) \right\} \right).
\end{align*}

We first control the expectation on $E_2$. 
Since $\Delta\vB\tp \vX=\Delta\vB\tp \Pi \vX$, we have 
	\begin{align}
\label{eq:simple-trace}  \Tr\left[ \Delta \vB\tp \vX \vX\tp \Delta \vB \right]  &   \,=\, \Tr\left[  \left(\bs{\Pi}\vX \right)^{\otimes} \left(\Delta \vB\right)^{\otimes} \right] \\
& \,\leqs \, \|\left( \bs{\Pi}\vX \right)^{\otimes} \| ~ \Tr\left[  \left(\Delta \vB\right)^{\otimes} \right]. \nonumber
\end{align}
Thus
$$\Tr\left[ \E \left(\one_{E_2}\Delta \vB\tp \vX \vX\tp \Delta \vB \right) \right] \leqs  \|\Delta B\|_{F}^{2} \E \left( \| \bs{\Pi}\vX\|^2 \one_{E_2} \right). $$

We fix $R=2C_{4}d$ and obtain $\P( E_2)<(2d)^{-4}$. 
Note that the second moment of $\chi^2_{2d}$ is $4d(d+1)$. 
By Cauchy--Schwarz inequality, $\E \left( \| \bs{\Pi}\vX\|^2 \one_{E_2} \right)\leqs \E \left( \| \bs{\Pi}\vX\|^4\right)^{1/2} \P( E_2)^{1/2}  = 2 \sqrt{d(d+1)}\P( E_2)^{1/2}<d^{-1}$.
By Lemma~\ref{lem:joint-eigen},  $\epsilon=100\lambda_{0,d}\geqslant d^{-1}$, and thus 
\begin{equation}\label{eq:lem:lower-kl-2nd-small-target-1}
\Tr\left[ \E \left(\one_{E_2}\Delta \vB\tp \vX \vX\tp \Delta \vB \right) \right] \leqs \epsilon \|\Delta B\|_{F}^{2}.
\end{equation}

We next control the two expectations on $E_1$. 

We can split the expectation into parts given by $\{W=w\}$ and use properties of conditional expectation to obtain
\begin{align*}
&   \E_{\vB}  \left\{ \one_{E_1} \Delta \vB\tp \vX \vX\tp \Delta \vB \left[   -   \E\left( \bxi^{\otimes} \mid  \bxi \in J(W,\rho), \vX, W \right) \right] \right\}  \\
	= &  \sum_{w=-d}^{d}   \E_{\vB}  \left\{ \one_{E_1} \Delta \vB\tp \vX \vX\tp \Delta \vB \one_{W=w} \left[   -   \E\left( \bxi^{\otimes} \mid  \bxi \in J(w,\rho), \vX \right) \right] \right\} \\
	= & - \sum_{w=-d}^{d}   \E  \left\{ \one_{E_1} \Delta \vB\tp \vX \vX\tp \Delta \vB \bbP_{\vB}(W=w\mid \vX) \left[    \E\left( \bxi^{\otimes} \mid  \bxi \in J(w,\rho), \vX \right) \right] \right\}. 
	\end{align*}
In addition, we have 
	\begin{align*}
 \bs{I}_{d} = &   \E\left( \bxi^{\otimes} \sum_{w=-d}^{d} \one_{\bxi\in J(w,\rho)} \mid \vX \right)    \\
	= &  \E\left( \bxi^{\otimes}\one_{\bxi\in J(w,\rho)} \mid \vX \right)   \\
	= & \sum_{w=-d}^{d}  \bbP(\bxi \in J(w,\rho) \mid \vX) \E\left( \bxi^{\otimes} \mid \bxi \in J(w,\rho), \vX  \right). 
	\end{align*}
 From the last two expressions, we can write
	\begin{align*}
	&\Tr\left( \E_{\vB}  \left\{ \one_{E_1} \Delta \vB\tp \vX \vX\tp \Delta \vB \left[  \bs{I}_{d} -   \E\left( \bxi^{\otimes} \mid  \bxi \in J(W,\rho), \vX, W \right) \right] \right\}\right)\\
	= & \sum_{w=-d}^{d} \Tr\left( \E  \left\{ \one_{E_1} \Delta \vB\tp \vX \vX\tp \Delta \vB   \E\left( \bxi^{\otimes} \mid \bxi \in J(w,\rho), \vX \right)  \right. \right.  \\ 
	& \qquad\qquad \left. \left.  \left[\bbP( \bxi \in J(w,\rho) \mid \vX  ) - \bbP_{\vB}( W=w \mid \vX ) \right]   \right\}\right).\nonumber
	\end{align*}

On the event $E_1$, we apply Lemma~\ref{lem:lower-diff-bound}(3) together with the inequality in \eqref{eq:lower-kl-basic-2}, and conclude that for any $\rho \leqs  c^{\prime}  d^{-5/2}$ and for any 
 $w\in \{-d, \ldots, d\}$, 
 $$\left|\bbP( \bxi \in J(w,\rho) \mid \vX  ) - \bbP_{\vB}( W=w \mid \vX )\right| <  \frac{1}{8 d^2}. $$
 For any $w\neq 0$, since $\bbP_{\vB}( W=w \mid \vX )= (4d)^{-1}$, it then follows that $\bbP( \bxi \in J(w,\rho) \mid \vX  )\geqslant (8d)^{-1}$ and 
\begin{align*}
\left|\bbP( \bxi \in J(w,\rho) \mid \vX  ) - \bbP_{\vB}( W=w \mid \vX )\right| & < d^{-1}
	P( \bxi \in J(w,\rho) \mid \vX  ) \\
&	\leqs \epsilon \bbP( \bxi \in J(w,\rho) \mid \vX  ).
\end{align*}
The same holds for $w=0$. 

Using the last inequality, we have
	\begin{align*}
	&\Tr\left( \E_{\vB}  \left\{ \one_{E_1} \Delta \vB\tp \vX \vX\tp \Delta \vB \left[  \bs{I}_{d} -   \E\left( \bxi^{\otimes} \mid  \bxi \in J(W,\rho), \vX, W \right) \right] \right\}\right)\\
	\leqslant & \epsilon \sum_{w=-d}^{d} \Tr\left( \E  \left\{ \one_{E_1} \Delta \vB\tp \vX \vX\tp \Delta \vB   \bbP( \bxi \in J(W,\rho) \mid \vX  )   \E\left( \bxi^{\otimes} \mid \bxi \in J(W,\rho), \vX \right)  \right\}\right) \\
	= & \epsilon  \Tr\left\{ \E  \left[ \one_{E_1} \Delta \vB\tp \vX \vX\tp \Delta \vB    \E\left( \bxi^{\otimes} \sum_{w=-d}^{d} \one_{ \bxi \in J(W,\rho) } \mid \vX  \right) \right] \right\} \\
	= & \epsilon  \Tr\left\{ \E  \left[ \one_{E_1} \Delta \vB\tp \vX \vX\tp \Delta \vB  \bs{I}_{d} \right] \right\}\\
	\leqslant & \epsilon  \Tr\left\{ \E  \left[\Delta\vB\tp \vX \vX\tp \Delta \vB \right]  \right\}\\
	= &  \epsilon  \| \Delta \vB\|_{F}^{2}.
	\end{align*}
Combining the last inequality with \eqref{eq:lem:lower-kl-2nd-small-target-1}, we complete the proof. 
\end{proof}

\begin{proof}[Proof of Lemma~\ref{lem:lower-kl-2nd-main}]
In the following, we use $C$ to denote any universal constant, whose value may change from line to line. 

We first control the expectation on $E_2$. 
Without loss of generality, we can assume $\rho^2<(3d)^{-1}$. 
By an elementary trace inequality that $\Tr \left( \vu \vu\tp \vv\vv\tp \right)\leqslant \|\vv\|^{2} \|\vu\|^{2}$, one has
	\begin{align}
	\label{eq:lower-kl-mainterm-control-1}
 & \Tr\left( \Delta \vB\tp \vX \vX\tp \Delta \vB \left[   \E \left(\bxi \mid  \bxi \in J(W,\rho), \vX, W \right)^{\otimes} \right]  \right) \\	
	\leqslant & \| \Delta \vB\tp \vX\|^{2} \left\|   \E \left(\bxi \mid  \bxi \in J(W,\rho), \vX, W \right)^{\otimes} \right\|^{2} \nonumber \\
\leqslant & \|\ \Delta \vB\tp \vX\|^{2}  C(d+ \rho^2\| \bs{\Pi} \vX\|^2) \nonumber   \\
\leqslant & \|\Delta\vB\|_{F}^2 \| \bs{\Pi} \vX\|^{2} C(d+ \rho^2\| \bs{\Pi} \vX\|^2) \nonumber   , 	
	\end{align}
	where the second inequality is due to Lemma~\ref{lem:lower-kl-mainterm-tail} and the inequality in \eqref{eq:lower-kl-basic-2}, and the third is due to the inequality in \eqref{eq:simple-trace}.

	Denote by  $U:= \| \bs{\Pi} \vX\|$. Then $U^2 \sim \chi^{2}_{2d}$. We can pick $R= \sqrt{ 2 C_{1} d} $ so that $\E \left(\one_{ U >R} U^2(d+U^2)  \right) \leqs C d^{-1} $. 
	Then \eqref{eq:lower-kl-mainterm-control-1} implies that
	\begin{align*}
	& \Tr\left( \E_{\vB}  \left\{ \one_{E_2 } \Delta \vB\tp \vX \vX\tp \Delta \vB \left[   \E \left(\bxi \mid  \bxi \in J(W,\rho), \vX, W \right)^{\otimes} \right] \right\}  \right)\\	
	\leqslant & C \| \Delta \vB\|_{F}^{2}   \E   \left\{  \one_{U>R} \left( U^2(d+\rho^2 U^2 \right) \right\} \\
	\leqslant & C d^{-1} \| \Delta \vB\|_{F}^{2}. 
	\end{align*}

 We next control the expectation on $E_1$. 
 
	\begin{align*}
	& \E_{\vB} \left[   \E \left(\bxi \mid  \bxi \in J(W,\rho), \vX, W \right)^{\otimes} \mid \vX \right]   \\	
	=&	\sum_{w=-d}^{d} \E_{\vB} \left[ ~ \one_{W=w} ~   \E \left(\bxi \mid  \bxi \in J(w,\rho), \vX \right)^{\otimes}  \mid \vX \right] \\	
	=&	\sum_{w=-d}^{d}  \E_{\vB} \left[ ~ \bbP_{\vB}(W=w\mid \vX)  \E \left(\bxi \mid  \bxi \in J(w,\rho), \vX \right)^{\otimes}  \mid \vX \right] \\	
	=&\sum_{w=-d}^{d}  \E_{\vB} \left[ ~ \bbP_{\vB}(W=w\mid \vX) \bbP\left( \bxi\in J(W,\rho) \mid \vX \right)^{-2}   \E \left(\bxi \one_{\bxi \in J(w,\rho)} \mid \vX \right)^{\otimes}  \mid \vX \right] .
	\end{align*}

Note that as $\rho\rightarrow 0$, the limit of $\bbP\left( \bxi\in J(W,\rho) \mid \vX \right)$ is $\bbP(W=w \mid \vX)$ and the limit of $\E \left(\bxi \one_{\bxi \in J(w,\rho)} \mid \vX \right)$ is $\E \left(\bxi \one_{\bxi \in A_{w}} \right)$. Below we shall bound the deviation of these terms from their respective limits.

We write $\vmu=\vB\tp \vX+\rho\alpha (\widetilde{\vB}-\vB)\tp \vX$. On the event $E_1$, the inequality in \eqref{eq:lower-kl-basic-2} shows that $\|\vmu\|\leqs R$. 
Denote by $\mathcal{N}_1$ the distribution $N( \rho\vmu  ; (1-\rho^2) \bs{I}_d)$ and by $\mathcal{N}_0$ the distribution $N( \bs{0};\bs{I}_d)$. 
We have
\begin{align*}
  \E \left(\bxi \one_{\bxi \in J(w,\rho)} \mid \vX \right)  & = \left[  \E \left( \rho \vmu  + \sqrt{1-\rho^2} \bxi \one_{\bxi \in J(w,\rho)} \mid \vX \right) - \rho  \vmu \right] (1-\rho^2)^{-1/2}  \\
  & = \left[ \E_{\mathcal{N}_1}( \vZ \one_{\vZ \in A_{w}} ) - \rho  \vmu \right] (1-\rho^2)^{-1/2}. 
\end{align*}

We apply Lemma~\ref{lem:lower-diff-bound} and conclude that for any $\rho \leqs  c^{\prime}  d^{-4}$ and for any $w\in \{\pm 1, \ldots, \pm d\}$,  
 
 $$|\bbP( \bxi \in J(w,\rho) \mid \vX  ) - \bbP_{\vB}( W=w \mid \vX ) |<  \frac{1}{8 d^{7/2}}$$
and 
\begin{align*}
&  \| \E \left(\bxi \one_{\bxi \in J(w,\rho)} \mid \vX \right) - \E( \bxi  \one_{\bxi \in A_w})\|\\
\leqs   &  (1-\rho^2)^{-1/2}\left|	\E_{\mathcal{N}_1}( \vZ \one_{\vZ \in A_{w}} )-\E_{\mathcal{N}_0}( \vZ \one_{\vZ\in A_{w}} ) \right| ~~\dots \\
& \qquad +\rho(1-\rho^2)^{-1/2}\|\vmu\| +(1-(1-\rho^2)^{-1/2})\|\E( \bxi  \one_{\bxi \in A_w})\| \\
\leqs & O(d^{-5/2})+ 2\rho \|\vmu\| + \rho^2 \sqrt{\log(2 d) }/d = O(d^{-5/2}), 
\end{align*}
where we have used $R=O(\sqrt{d})$ and  $\|\E( \bxi  \one_{\bxi \in A_w})\|=\E( \bxi_{1}  \one_{\bxi \in A_1})\leqs \frac{\sqrt{2\log (2d)}}{2d}$ given in Lemma~\ref{lem:joint-eigen}. 
It also follows that 
$$\|\E \left(\bxi \one_{\bxi \in J(W,\rho)} \mid \vX \right)\|\leqs C \sqrt{\log(2 d) }/d \text{~and~} (8d)^{-1}\leqs \bbP( \bxi \in J(w,\rho) \mid \vX )\leqs (2d)^{-1}. $$

To utilize the bounds that have been established, we present the following elementary facts. 
For any positive numbers $b$, $B$ and any vectors $\vu$, $\vU$, we have 
\begin{enumerate}
  \item $  b^{-2}\vu^{\otimes} -   B^{-2}\vU^{\otimes} 
   =   b^{-2}\left( \vu^{\otimes} -  \vU^{\otimes} \right)+(  b^{-2}- B^{-2})\vU^{\otimes} $
   \item $\|\vu^{\otimes} -  \vU^{\otimes} \|_{F}\leqs \|\vu (\vu-  \vU)\tp \|_{F} + \| (\vu - \vU) \vU\tp  \|_{F}\leqs (\|\vu\|+\|\vU\|) \|\vu-\vU\|$
   \item $|b^{-2}-B^{-2}|\leqs b^{-2}B^{-2} (b+B) |b-B|.$
\end{enumerate}

Now let $B=\bbP_{\vB}( W=w \mid \vX )$, $b=\P( \bxi \in J(w,\rho) \mid \vX )$, $\vU=\E( \bxi  \one_{\bxi \in A_w})$, and $\vu=\E \left(\bxi \one_{\bxi \in J(W,\rho)} \mid \vX \right)$. We conclude that on $E_1$,  
\begin{align*}
&	\left\| \P\left( \bxi\in J(w,\rho) \mid \vX \right)^{-2} \left[   \E \left(\bxi \one_{\bxi \in J(w,\rho)} \mid \vX \right)^{\otimes} \right]  - \bbP_{\vB}(W=w\mid \vX)^{-2}\left[   \E( \bxi ~ \one_{\bxi \in A_w})^{\otimes} \right] \right\|_{F} \\
	\leqs & O ( d\sqrt{\log (2d)} d^{-5/2}+ d\log(2d)  d^{-7/2}) = O ( d^{-1}).
	\end{align*}
 A similar result can be obtained for $w=0$. 
 
Since $\epsilon\geqslant d^{-1}$, there is some universal constant $C$ such that
	\begin{align*}
	& \Tr\left( \E_{\vB}  \left\{ \one_{E_1} \Delta \vB\tp \vX \vX\tp \Delta \vB \left[   \E \left(\bxi \mid  \bxi \in J(W,\rho), \vX, W \right)^{\otimes} \right] \right\}  \right)\\	
	\leqslant & \sum_{w=-d}^{d} \Tr\left( \E_{\vB}  \left\{ \one_{E_1} \Delta \vB\tp \vX \vX\tp \Delta \vB  \bbP_{\vB}(W=w\mid \vX)  \left[   \E( \bxi ~ \one_{\bxi \in A_w})^{\otimes} \right] \right\}  \right) ~~ \ldots \\  
 & \qquad + C \epsilon  \E_{\vB}  \left\{ \one_{E_1}\| \Delta \vB\tp \vX\|^{2} \right\}\\
	\leqslant & C \epsilon   \|\Delta \vB\|_{F}^{2} + \Tr\left( \E_{\vB}  \left\{ \one_{E_1} \Delta \vB\tp \vX \vX\tp \Delta \vB  \Sigma_{0,d} \right\}  \right) \\
	\leqslant & ( C \epsilon + \|\Sigma_{0,d} \| )  \|\Delta \vB\|_{F}^{2}, 
	\end{align*}
	where $\Sigma_{0,d}=\sum_{w=-d}^{d} \bbP_{\vB}(W=w\mid \vX)   \left[   \E( \bxi ~ \one_{\bxi \in A_w}) \right]^{\otimes}= \Cov\left[ \E \left( \bxi \mid \psi (\bxi)  \right) \right]$. By Lemma~\ref{lem:joint-eigen}, $\| \Sigma_{0,d}\|=\lambda_{0,d}$. 
 \begin{remark}
    In the proof, we choose $\rho \leqs  c^{\prime}  d^{-4}$ to simplify the result. In fact, it is sufficient to choose $\rho \leqs  c^{\prime}  d^{-7/2-\xi'}$ for any small $\xi'>0$. 
\end{remark}

\end{proof}

 \begin{proof}[Proof of Lemma~\ref{lem:lower-derivatives}]
 It is straightforward to compute the derivative and Hessian of $\log g_{w}^{\vmu}(\bs{t})$ as follows:
 
\begin{list}{\textbullet}{}  
\item Let $\phi_d( \vz ; \va ; \vM)$ be the p.d.f. of $N(\va, \vM)$ for $\va\in \R^{d}$ and $\vM$ be a $d\times d$ positive definite matrix. Then $\frac{\diff }{ \diff \va} \phi_d( \vz ; \va ; \vM)= \vM^{-1}(\vz-\va)\phi_d( \vz ; \va ; \vM).$
		\item $g_{w}^{\vmu}(\bs{t})=\int \one_{\vz\in A_w} \phi_{d}\left( \vz; \rho \vmu + \rho \vt ; (1-\rho^2)\bs{I}_{d}\right) \diff \vz $. This is obtained by transforming $\vZ=\rho\vmu+\rho \vt +\sqrt{1-\rho^2} \bxi$.  
		\item Then $\nabla g_{w}^{\vmu}(\bs{t}) =\int \one_{\vz\in A_w}\frac{\rho}{1-\rho^2} (\vz-\rho\vmu-\rho \vt) \phi_{d}\left( \vz; \rho \vmu + \rho \vt ; (1-\rho^2)\bs{I}_{d}\right) \diff \vz $. This can also be written as $\nabla g_{w}^{\vmu}(\bs{t}) = \sqrt{\frac{\rho^{2}}{1-\rho^2}}  \E \left( \bxi \one_{\rho\vmu+\rho \vt +\sqrt{1-\rho^2} \bxi\in A_w} \right)$.
		\item \begin{align*}
		\nabla^{2} g_{w}^{\vmu}(\bs{t}) & = \int \one_{\vz\in A_w} \left(-\frac{\rho^{2}}{1-\rho^2} \bs{I}_{d} + \frac{\rho^{2}}{(1-\rho^2)^{2}} (\vz-\rho\vmu-\rho \vt)^{\otimes} \right) \ldots \\ 
		& \qquad \times \phi_{d}\left( \vz; \rho \vmu + \rho \vt ; (1-\rho^2)\bs{I}_{d}\right) \diff \vz \\
		& = -\frac{\rho^{2}}{1-\rho^2} \bs{I}_{d} g_{w}^{\vmu}(\bs{t})  \ldots \\ 
		& \quad +\frac{\rho^{2}}{(1-\rho^2)^{2}}  \int_{ A_w}(\vz-\rho\vmu-\rho \vt)^{\otimes}\phi_{d}\left( \vz; \rho \vmu + \rho \vt ; (1-\rho^2)\bs{I}_{d}\right) \diff \vz.
		\end{align*}
		\item $\nabla\log g_{w}^{\vmu}(\bs{t}) = \frac{1}{g_{w}^{\vmu}(\bs{t}) } \nabla g_{w}^{\vmu}(\bs{t})$. 
		\item  $\nabla^{2}\log g_{w}^{\vmu}(\bs{t})  = \frac{1}{g_{w}^{\vmu}(\bs{t}) } \nabla^{2} g_{w}^{\vmu}(\bs{t}) - \left( \frac{1}{g_{w}^{\vmu}(\bs{t}) } \nabla g_{w}^{\vmu}(\bs{t})  \right)^{\otimes}$. 
		\item By transforming $ \bxi=(\vZ-\rho\vmu-\rho \vt )/\sqrt{1-\rho^2}$, one has 
\begin{equation*} 
		\nabla\log g_{w}^{\vmu}(\bs{t}) =\sqrt{\frac{\rho^{2}}{1-\rho^2} } \E \left( \bxi \mid \rho\vmu+\rho\vt +\sqrt{1-\rho^2} \bxi \in A_w \right)
		\end{equation*} 
  and 
		$$\frac{1}{g_{w}^{\vmu}(\bs{t}) } \nabla^{2} g_{w}^{\vmu}(\bs{t})=-\frac{\rho^{2}}{1-\rho^2} \bs{I}_{d}+ \frac{\rho^{2}}{1-\rho^2}  \E\left( \bxi^{\otimes} \mid  \rho\vmu+\rho\vt +\sqrt{1-\rho^2} \bxi \in A_w \right).
		$$
		Therefore
		\begin{align}
		\nabla^{2}\log g_{w}^{\vmu}(\bs{t}) & =\frac{\rho^{2}}{1-\rho^2} \left\{ - \bs{I}_{d} 
 + \E\left( \bxi^{\otimes} \mid  \rho\vmu+\rho\vt +\sqrt{1-\rho^2} \bxi \in A_w \right)  \right. \nonumber \\
		&\qquad \qquad \quad \left. -   \E \left( \bxi \mid \rho\vmu+\rho\vt +\sqrt{1-\rho^2} \bxi \in A_w \right)^{\otimes}  \right\}. 
		\end{align}
	\end{list}
 
\end{proof}

\subsection{Proof of Lemma~\ref{lem:packing}}
We first state two lemmas from the literature. 
	\begin{lemma}[{\cite[Lemma 1]{cai2013sparse}}]\label{lem:cai2013}
		For any $\varepsilon \in (0, \sqrt{ 2(d \wedge (p-d))}]$, any $\alpha \in (0,1)$, and any $\vA\in\mathbb{O}(p,d)$, there exists a subset $\Theta\subset \mathbb{O}(p,d)$ such that 
		\[
		|\Theta| \geqslant \left( \frac{c_0}{\alpha}  \right)^{d(p-d)}, 
		\]
		and for any $\vB, \widetilde{\vB}\in \Theta$, 
		\[
		\rho (\vA\vA\tp ,\vB\vB\tp ) \leqslant \varepsilon, \qquad
		\rho (\vB\vB\tp ,\widetilde{\vB}\widetilde{\vB}\tp ) \geqslant  \alpha\varepsilon, 
		\]
		where $c_0$ is an absolute constant. 
	\end{lemma}
	\begin{lemma}[{\cite[Lemma 6.5]{ma2020subspace}}]\label{lem:linear-alg}
		For any matrices $\bs{A}_1, \bs{A}_2 \in  \mathbb{O}(p,d)$, there exists some $\bs{Q}\in \mathbb{O}(d,d)$ such that 
		\begin{align}
		\| \bs{A}_1-\bs{A}_2 \bs{Q}\|_{F}\leqslant \|\bs{A}_{1}\bs{A}\tp_{1}-\bs{A}_{2}\bs{A}\tp_{2}\|_F.
		\end{align}
	\end{lemma}
\begin{proof}[Proof of Lemma~\ref{lem:packing}]	
	Pick any $\bs{A}\in \mathbb{O}(p,d)$. 
 Let $\Theta_0$ be the subset in Lemma~\ref{lem:cai2013}. 
	For each $\vB_0\in \Theta_0$, we can find $\bs{Q}_{\vB_0}$ such that $\|\vA-\vB_0 \bs{Q}_{\vB_0}\|_{F}\leqslant \epsilon$ due to  Lemma~\ref{lem:linear-alg}. 
	Define $\Theta=\{\vB_0 \bs{Q}_{\vB_0}:\vB_0\in \Theta_0\}$. By the triangle inequality, it is easy to see that $\Theta$ satisfies the requirement of the lemma. Thus we complete the proof of Lemma \ref{lem:packing}.
\end{proof}

\section{Proof of Theorem \ref{thm:lower-largep}}\label{app:rate-largep}
\begin{proof}
We only need to prove the lower bounds 
	\begin{align}\label{eq:high-exact-ds}
	\inf_{\widehat{\vB}} \sup_{\mathcal{M}\in \mathfrak{M}_{s }\left( p,d,\lambda\right) }\bbE_{\mathcal{M}} \left\| \widehat{\vB} \widehat{\vB}\tp -\vB\vB\tp \right\|_{\mathrm{F}}^{2}&\gtrsim \frac{d (s-d)}{n\lambda} .
	\end{align}
	and 
	\begin{align}\label{eq:high-exact-slogp}
	\inf_{\widehat{\vB}} \sup_{\mathcal{M}\in \mathfrak{M}_{s }\left( p,d,\lambda\right) } \bbE_{\mathcal{M}} \left\| \widehat{\vB} \widehat{\vB}\tp -\vB\vB\tp \right\|_{\mathrm{F}}^{2}&\gtrsim \frac{ (s-d)\log\frac{e(p-d)}{s-d}}{n\lambda} .
	\end{align}

To prove the inequality~\eqref{eq:high-exact-ds}, consider the following sub-model which assumes the support of the indices matrix is $\{1,2,\ldots,s\}$.

 \begin{align}\label{model:high:dim:oracle}
\widetilde{\mathfrak{M}}_{s }\left( p,d,\lambda\right) 
 :=\mathfrak{M}_{s}\left(p,d,\lambda\right)\cap \left\{\begin{aligned}
&\text{\quad distribution of } \\
&\left(\vX,Y=f(\vB\tp\vX,\epsilon)\right)
\end{aligned}
\quad\vline
\begin{aligned}
&\quad 
\vX\sim N(0,\bs{I}_{p}),\\ 
&\vB\tp\vB=\vI_d, ~~\supp(\vB)=[s] \\
\end{aligned}\right\}, 
\end{align}

	A sufficient statistic for estimating $\vB$ in this submodel is the data of $Y$ and $\vX_{1:s}$. Because $\widetilde{\mathfrak{M}}_{s }\left( p,d,\lambda \right)\subset \mathfrak{M}_{s }\left( p,d,\lambda\right)$ and $\widetilde{\mathfrak{M}}_{s }\left( p,d,\lambda\right)$ is essentially the same as a submodel of $\mathfrak{M}\left( s,d,\lambda\right)$ for $(\vX_{1:s}, Y)$ with $\Cov(X_{1:s})=\bs{I}_{s}$ assumed known. Following the exact same proof of Theorem~\ref{thm:lower-smallp}, we obtain the inequality~\eqref{eq:high-exact-ds}. 
	
To prove the inequality~\eqref{eq:high-exact-slogp}, we apply the Fano method (Lemma~\ref{lem:fano}) and the construction of $\bbP_{\vB}$ in Equation~ \eqref{eq:joint-x-y-B}. The main challenge is to construct a rich packing set $\Theta\subset \mathbb{O}_{s }(p,d)$. 
The rest of the proof is nearly identical to the proof of Theorem~4 in \cite{lin2021optimality}, which in turn follows from the argument in  \cite[Theorem~2.1]{vu2012minimax} and \cite[Theorem~3]{cai2013sparse}. We present it here for the sake of completeness.
For any $\varepsilon\in (0,1]$, 
	\cite[Lemma 3.1.2]{vu2012minimax} have constructed a set $\Theta_{0} \subset \mathbb{O}_{s-d+1}(p-d+1,1)$, such that
	\begin{itemize}
		\item[1.] $\varepsilon / \sqrt{2}< \|\theta_{1}-\theta_{2}\|\leqslant \sqrt{2}\varepsilon$ for all distinct pairs $\theta_{1},\theta_{2} \in \Theta_{0}$,
		\item[2.] $\log|\Theta_{0}|\geqslant c (s-d) [\log(p-d)-\log(s-d)]$, where $c$ is a positive constant. 
	\end{itemize}
	
	For each $\theta\in \Theta_{0}$, define 
	$$
	\vB_{\theta}=\left(\begin{array}{cc}
	\theta & \bs0_{(p-d+1)\times (d-1)} \\
	\bs0_{(d-1)\times 1}& \vI_{d-1}
	\end{array} \right). 
	$$
	Then $\vB_{\theta} \in \mathbb{O}_{s }(p,d)$. Let $\Theta=\{\vB_{\theta}  : \theta\in \Theta_{0} \}$. 
	We then use the construction in Equation~\eqref{eq:joint-x-y-B} to obtain a family of distributions with the indices matrix  $\vB_{\theta}$ for each $\theta\in \Theta_{0}$. 
Applying the argument in Equation~\eqref{eq:fano:argument-new} with $\varepsilon^2=c_1' \frac{ (s-d) [\log(p-d)-\log(s-d)]}{n\lambda}$,  we obtain Equation~\eqref{eq:high-exact-slogp}. 

Since $2d<s$, the right hand sides on the inequalities~\eqref{eq:high-exact-ds} and \eqref{eq:high-exact-slogp} can be reduced to $\frac{ds}{n\lambda}$ and $\frac{s\log(ep/s) }{n\lambda}$, respectively. 
\end{proof}

\section{Proofs of upper bounds with a known covariance matrix}\label{app:ub-identity}
Here we present the proofs of the upper bounds  in Section~\ref{sec:minimax rate} with $\bS$ known. 
These proofs are adapted from \cite{lin2021optimality}. 
Without loss of generality, we can assume $\bS= \vI_p$. 
The proof for general cases with unknown $\bS$ is
presented in Appendix~\ref{app:ub-general}, which makes use of the results here.

\subsection{Proof of Theorem~\ref{thm:risk:oracle:upper:d}}\label{app:sir:low-d}

\subsubsection*{Preliminary}

Before we start proving the theorems, we need some preparations. 

{\bf Notations:} 
Suppose that we have  $n=Hc$ samples $(\vX_{i},Y_{i})$ from a distribution $\mathcal{M}\in \mathfrak{M}(p,d,\lambda)$. 
Throughout this section, $H$ is taken to be an integer such that $H$ satisfies the inequality $H>K\vee Cd$ in Lemma~\ref{lem:key lemma in Lin under moment condition} and $H\leqslant H_{0}d$ for some constant $H_{0}>K_0\vee C$. In this way, we can apply the result of Lemma~\ref{lem:key lemma in Lin under moment condition} and we will implicitly use $H=O(d)$.

Since $\bS=\vI_p$, the SIR estimator in \eqref{eqn:Bhat} is defined directly as $\widehat{\vB}=\left[\widehat{\vB}_{1},...,\widehat{\vB}_{d}\right]$,  where $\widehat{\vB}_{i}$ is the $i$-th leading eigenvector of $\widehat{\bLambda}_{H}$.  

Let $\vB_{\perp}$ be a $p\times (p-d)$ orthogonal matrix such that $\vB\tp \vB_{\perp}=0$. 

For any pair of $(\vX, Y)$ sampled from $\mc M$, let $\vZ=\vB\tp  \vX$ and $\vE=\vB_{\perp}\tp  \vX$.  Since $\vB\tp\vB=\vI_d$, $\vB_\perp\tp\vB_\perp=\vI_{p-d}$, one has $\vZ\sim N(0, \bs{I}_{d})$ and $\vE\sim N(0, \bs{I}_{p-d})$. Furthermore, $\vZ\indp \vE$ since $\mathrm{Cov}(\vZ,\vE)=\vB\tp \vB_{\perp}=0$. Besides, we have  
\begin{equation*}
\begin{aligned}
\vX&=\mc{P}_{\mathcal{S}}\vX+\mc{P}_{\mathcal{S}^{\perp}}\vX
= \vB\vZ+\vB_{\perp}\vE\quad (\because \vB\vB\tp+\vB_\perp\vB_\perp\tp=\vI_p)
\end{aligned}
\end{equation*}
where $\mathcal{S}=\col(\vB)$ is the central space.

Let $\vV=\vB \vZ$ and $\vW=\vB_{\perp}\vE$. Then $\vV\tp \vW=0$. We introduce the notation $\overline{\vV}_{h,\cdot}$, 
$\overline{\vZ}_{h,\cdot}$, 
$\overline{\vW}_{h,\cdot}$, 
and $\overline{\vE}_{h,\cdot}$ similar to
the definition of the sample mean in the $h$-th slice $\overline{\vX}_{h,\cdot}$ near Equation~\eqref{eqn:lambda}.

Let
$\mathbf{\mathcal{V}}=\frac{1}{\sqrt{H}}
\left[\overline{\vV}_{1,\cdot} ~,~ \overline{\vV}_{2,\cdot},...,~ 
\overline{\vV}_{H,\cdot}\right]$,
$\mathbf{\mathcal{Z}}=\frac{1}{\sqrt{H}}
\left[\overline{\vZ}_{1,\cdot} ~,~ \overline{\vZ}_{2,\cdot},...,~ 
\overline{\vZ}_{H,\cdot}\right]$, 
$\mathbf{\mathcal{W}}=\frac{1}{\sqrt{H}}
\left[\overline{\vW}_{1,\cdot} ~,~ \overline{\vW}_{2,\cdot},...,~ \overline{\vW}_{H,\cdot}\right]$,
and $\mathbf{\mathcal{E}}=\frac{1}{\sqrt{H}}
\left[\overline{\vE}_{1,\cdot} ~,~ \overline{\vE}_{2,\cdot},...,~ \overline{\vE}_{H,\cdot}\right]$
be four matrices formed by the $p$-dimensional vectors 
 $\frac{1}{\sqrt{H}}\overline{\vV}_{h,\cdot}$,  $\frac{1}{\sqrt{H}}\overline{\vZ}_{h,\cdot}$, $\frac{1}{\sqrt{H}}\overline{\vW}_{h,\cdot}$,  and $\frac{1}{\sqrt{H}}\overline{\vE}_{h,\cdot}$. 
 We have 
 $\mathbf{\mathcal{V}}=\vB\mathbf{\mathcal{Z}}$,
 $\mathbf{\mathcal{W}}=\vB_{\perp}\mathbf{\mathcal{E}}$, 
and $\mathbf{\mathcal{V}}\tp \mathbf{\mathcal{W}}=0$.

Define $\widehat{\bLambda}_{z}=\mathbf{\mathcal{Z}}\mathbf{\mathcal{Z}}\tp $ and $\widehat{\bLambda}_{V}=\mathbf{\mathcal{V}}\mathbf{\mathcal{V}}\tp =\vB \widehat{\bLambda}_{z}\vB\tp$. Then we have the following decomposition
\begin{equation}
\begin{aligned} \label{eqn:estimator:decomposition}
\widehat{\bLambda}_{H}&={\mathcal{V}}{\mathcal{V}}\tp 
+{\mathcal{V}}{\mathcal{W}}\tp
+{\mathcal{W}}{\mathcal{V}}\tp
+{\mathcal{W}}{\mathcal{W}}\tp \\
&=\widehat{\bLambda}_{V}  
+\vB \mathbf{\mathcal{Z}}\mathbf{\mathcal{E}}\tp \vB_{\perp}\tp 
+ \vB_{\perp}\mathbf{\mathcal{E}}\mathbf{\mathcal{Z}}\tp \vB\tp 
+\vB_{\perp}\mathbf{\mathcal{E}}\mathbf{\mathcal{E}}\tp \vB_{\perp}\tp.
\end{aligned}
\end{equation}
Since $\vE \sim N(0,\vI_{p-d})$ and is independent of $Y$, we know that the entries $\mathcal{E}_{i,j}$ of $\mathcal{E}$ are $i.i.d.$ samples of $N(0,\frac{1}{n})$. 

\subsubsection*{Main part of the proof}

First, we have the following lemma.
\begin{lemma} \label{lem:minimx:key}
	
 Suppose the distribution of data is in $\overline{\mathfrak{M}}\left(p,d,\lambda\right)$
 and  $H\geqslant \max\{K,Cd\}$ for a sufficiently large constant $C$.  We have the following statements. 
	
	\begin{itemize}
		\item[(a)]
		\begin{align*}
		\bbP\left(\|\mathcal{W}\mathcal{W}\tp \|   > 6\frac{p\vee H +t }{n} \right)\leqslant  2 \exp\left(- t\right).
		\end{align*}
		\item[(b)] For $\nu\in (\kappa, 2\kappa]$, 
		\begin{align*}
		\bbP\left( \exists \bbeta\in \mathbb{S}^{p-1}, \text{ s.t. } \Big| \bbeta\tp \left(\widehat{\bLambda}_{V}-\bLambda\right)\bbeta\Big| > \frac{2}{3\nu}\bbeta\tp \bLambda\bbeta \right)\leqslant C_{1}\exp\left(-C_{2}\frac{n\lambda}{H^{2}\nu^{2}} +C_{3}\log(nH)+C_{4}d\right).
		\end{align*} 
	\end{itemize}
\end{lemma}
\begin{proof}
	$(a)$: We apply Lemma~\ref{random:nonasymptotic} to $\sqrt{n} \cdot \mathbf{\mathcal{E}}$ and note that
	
	$$
	\left( \sqrt{p-d}+\sqrt{H}+\sqrt{2t}  \right)^{2}\leqslant 3 \left( p-d + H + 2t \right) \leqslant 6 (p \vee H + t). 
	$$

	$(b)$: 
	Since $\col(\widehat{\bLambda}_{V})=\col(\bLambda)=\col(\vB)$, we only need to consider the vector $\bbeta$ that lies in $\col(\bLambda)$. 
	Let $\bLambda=\vV \vD \vV^{\top}$ be the eigen-decomposition of $\bLambda$ where $\vV$ is a $p\times d$ orthogonal matrix and $\vD$ is a $d\times d$ invertible diagonal matrix. 
	Let $\mathbf{\Omega}:= \vD^{-\frac{1}{2} }\vV\tp (\widehat{\bLambda}_{H}-\bLambda)\vV \vD^{-\frac{1}{2} }$. For any unit vector $\bbeta\in \col(\bLambda)$, consider the transformed vector $\vU=\vD^{1/2}\vV\tp \bbeta$. 
	Since $\col(\vB)=\col(\bLambda)$, one has $\vB_\perp\tp\bbeta=0$. Then from \eqref{eqn:estimator:decomposition} we  turn to prove that
	\begin{align*}
	&\bbP\left( \exists \bbeta\in \mathbb{S}^{p-1}, \text{ s.t. } \Big| \bbeta\tp \left(\widehat{\bLambda}_{H}-\bLambda\right)\bbeta\Big| > \frac{2}{3\nu}\bbeta\tp \bLambda\bbeta \right)\\
	=&\bbP\left( \exists \vU\in \mathbb{R}^{p}, \text{ s.t. } \Big| \vU\tp \bO \vU\Big| > \frac{2}{3\nu}\vU^{\tp}\vU \right)\\
	\leqslant& C_{1}\exp\left(-C_{2}\frac{n\lambda}{H^{2}\nu^{2}} +C_{3}\log(nH)+C_{4}d\right).
	\end{align*}
	Lemma \ref{lem:key lemma in Lin under moment condition} yields 
	\begin{align}\label{eq:deviation-fixed-u}
	\bbP\left( \Big| \vU\tp \bO \vU\Big|> \frac{1}{2\nu} \vU\tp \vU \right) \leqslant   C_{1}\exp\left(-C_{2}\frac{n \lambda}{H^{2}\nu^{2}} +C_{3}\log(nH)\right),
	\end{align} 
	where we have used $\vU\tp \vU = \bbeta\tp \bLambda \bbeta \geqslant \lambda$. 
	We then use the standard $\epsilon$-net argument (see, e.g., \cite[Chapter 2.3.1]{tao2012topics}) to bound 
	\begin{align*}
	\bbP\left( \exists \vU\in \mathbb{R}^{p}, \text{ s.t. } \Big| \vU\tp \bO \vU\Big| > \frac{2}{3\nu} \vU\tp \vU \right)= \bbP\left(\|\bO\|   > \frac{2}{3\nu} \right). 
	\end{align*} 
	Let $\N$ be a $\frac18$-net in $\mathbb{S}^{d-1}$: a minimal set of points  in $\mathbb{S}^{d-1}$  such that  for any $\vu \in \mathbb{S}^{d-1}$, one can find find $\widetilde{\vu} \in \N$ such that $\| \vu-\widetilde{\vu}\|\leqslant 1/8$. This implies that $\vu\tp \bO \vu = \widetilde{\vu} \tp \bO \widetilde{\vu} + (\vu - \widetilde{\vu})\tp \bO \widetilde{\vu} + \vu\tp \bO (\vu-\widetilde{\vu})\leqslant \widetilde{\vu} \tp \bO \widetilde{\vu} +  \|\bO\|  /4$.  Taking the maximum of $\vu\in\mb{S}^{d-1}$, one has $\|\bO\|   \leqslant 4/3 \cdot  \max_{\widetilde{\vu} \in \N} \widetilde{\vu} \tp \bO \widetilde{\vu}$. 
	Therefore
	\begin{align*}
	\bbP\left( \|\bO\|   > \frac{2}{3\nu} \right) &\leqslant\bbP\left(  \max_{\widetilde{\vu} \in \N} \widetilde{\vu} \tp \bO \widetilde{\vu} > \frac{1}{2\nu} \right)\leqslant
	\sum_{\widetilde{\vu}\in \N} \bbP\left(  \Big| \widetilde{\vu} \tp \bO \widetilde{\vu}\Big| > \frac{1}{2\nu} \right)\\
	&\leqslant C_{1}\exp\left(-C_{2}\frac{n\lambda}{H^{2}\nu^{2}} +C_{3}\log(nH)+C_{4}d\right), 
	\end{align*}
	where in the last inequality we use the the fact that $| \N | \leqslant 17^{d}$ (See e.g., Lemma 5.2 in \cite{vershynin2010introduction}) and insert $\bbeta= \vV \vD^{-1/2}  \widetilde{\vu}/ \| \vV \vD^{-1/2} \widetilde{\vu}\|$ and $\vU=\vD^{1/2}\vV\tp \bbeta$ into Equation~\eqref{eq:deviation-fixed-u}.
\end{proof}

To proceed, we define some events: $\ttE_{1}=\Big\{~ \|\mathbf{\mathcal{W}}\mathbf{\mathcal{W}}\tp \|   \leqslant 6 \frac{p \vee H +\log(n\lambda)}{n}~\Big\}$, 
$\ttE_{2}=\Big\{~ \|\widehat{\bLambda}_{V}-\bLambda\|   \leqslant \frac{2}{3 \nu}\kappa\lambda~\Big\}$ 
and $\ttE=\ttE_{1}\cap \ttE_{2}$. 

\begin{corollary}\label{cor:minimax:key} 
	For any $\nu \in (\kappa, 2\kappa]$, we can find constants $C$ and $\widetilde{C}$, such that $\bbP\left( \ttE^{c}\right) \leqslant \frac{ \widetilde{C} }{n\lambda}$ holds if \\ 
	\begin{equation}
 \label{eq:cor-key-require-1,known covariance,low dim}
	\kappa^2 H^2\left( \log (nH\kappa) + d \right)<C n \lambda.
	\end{equation}
If further $\kappa \left( p\vee H + \log (n\lambda) \right)< 1800^{-1} n\lambda$, then on the event $ \ttE$, the followings hold
	\begin{itemize}
				\item[$a)$] 
		$ \frac{1}{3}\lambda\leqslant \lambda_{d}(\widehat{\bLambda}_{V}) \leqslant \lambda_{1}(\widehat{\bLambda}_{V})\leqslant 2\kappa \lambda$.
		\item[$b)$]
		$\|\widehat{\bLambda}_{H}-\widehat{\bLambda}_{V}\|  
		\leqslant \lambda \sqrt{18 \kappa \frac{p \vee H +\log(n\lambda )}{n \lambda}}< \frac{1}{4}\lambda$. 
		\item[$c)$] 
		$\lambda_{d+1}(\widehat{\bLambda}_{H}) < \frac{1}{4}\lambda$.
	\end{itemize}
\end{corollary}
\begin{proof}
	From Lemma \ref{lem:minimx:key}, one has
\begin{align*}
\bbP(\ttE_1^c)&=\bbP\left(\|\mathcal{W}\mathcal{W}\tp \|   > 6\frac{p\vee H +\log(n\lambda) }{n} \right)\leqslant  2 \exp\left(- \log(n\lambda)\right)=\frac{2}{n\lambda}\\
\bbP(\ttE_2^c)&= \bbP\left( \exists \bbeta\in\mb{S}^{p-1}, \text{ s.t. } \Big| \bbeta\tp \left(\widehat{\bLambda}_{V}-\bLambda\right)\bbeta\Big| > \frac{2}{3\nu}\kappa\lambda\right)\\
&\leqslant \bbP\left( \exists \bbeta, \text{ s.t. } \Big| \bbeta\tp \left(\widehat{\bLambda}_{V}-\bLambda\right)\bbeta\Big| > \frac{2}{3\nu}\bbeta\tp \bLambda\bbeta \right)\leqslant C_{1}\exp\left(-C_{2}\frac{n\lambda}{H^{2}\nu^{2}} +C_{3}\log(nH)+C_{4}d\right).
\end{align*}
Thus to show 
\begin{align*}
\bbP(\ttE^c)&=\bbP(\ttE_1^c\cup \ttE_2^c)\leqslant \bbP(\ttE_1^c)+\bbP(\ttE_2^c) \leqslant \frac{ \widetilde{C} }{n\lambda}
\end{align*}
for some $\widetilde{C}$, one only need  $\exp\left(-C_{2}\frac{n\lambda}{H^{2}\nu^{2}} +C_{3}\log(nH)+C_{4}d\right) \lesssim \frac{1}{ n \lambda}$. 
This will be  true if the followings are bounded from below by some positive constant
\begin{align*}
\frac{n \lambda}{ H^2 \nu^2} / \log (nH), ~~ 
\frac{n \lambda}{ H^2 \nu^2} /  d, ~~ 
\frac{n \lambda}{ H^2 \nu^2} / \log(n\lambda). 
\end{align*}
Since $\nu\leqslant 2\kappa$, the first two are bounded by choosing a small $C$ in Equation~\eqref{eq:cor-key-require-1,known covariance,low dim}. Since $x / \log(x)$ is increasing for $x>e$, one has 
\begin{align*}
n\lambda / \log (n \lambda) > C^{-1} H^2\kappa^{2} (\log (nH \kappa)+d) / \log \left[ C^{-1} H^2\kappa^{2} (\log (nH \kappa)+d)  \right] \gtrsim H^2 \nu^2    
\end{align*}
and the last one is also bounded.

On the event $ \ttE_{2}$, Weyl's inequality implies that $\lambda_{d}(\widehat{\bLambda}_{V}) \geqslant \lambda_{d}( \bLambda) -  \frac{2}{3 \nu}\kappa\lambda > \frac{1}{3}\lambda$ and $ \lambda_{1}(\widehat{\bLambda}_{V})\leqslant \lambda_{1}(\bLambda) + 2\kappa \lambda/(3\nu)  \leqslant 2\kappa \lambda$.

From Equation~\eqref{eqn:estimator:decomposition}, 
\begin{equation}
\begin{aligned} 
\|\widehat{\bLambda}_{H}- \widehat{\bLambda}_{V} \|  
& \leqslant \|{\mathcal{V}}{\mathcal{W}}\tp\|  
+\|{\mathcal{W}}{\mathcal{V}}\tp\|  
+\|{\mathcal{W}}{\mathcal{W}}\tp\|   \\
& \leqslant 2 \sqrt{ \| \widehat{\bLambda}_{V} \|   \| {\mathcal{W}}{\mathcal{W}}\tp \|   } + \| {\mathcal{W}}{\mathcal{W}}\tp \|  . 
\end{aligned}
\end{equation}
On the event $\ttE$, if $6 \frac{p \vee H +\log(n\lambda)}{n \lambda}\leqslant 2\kappa$, the last display is further bounded by $\sqrt { 6 \cdot 3^2 \cdot 2 \kappa \lambda   \frac{p \vee H +\log(n\lambda)}{n} }$. 
If $\kappa \frac{p \vee H +\log(n\lambda)}{n \lambda}<2^{-6} 3^{-3}$, the bound is smaller than $\frac{1}{4} \lambda$. By Lemma \ref{lem:weyl} (Weyl's inequality) and $\lambda_{d+1}( \widehat{\bLambda}_{V} )=0$, one has $\lambda_{d+1}(\widehat{\bLambda}_{H})<\frac{1}{4}\lambda$. 
\end{proof}

Now we start the proof of Theorem~\ref{thm:risk:oracle:upper:d}. 
Throughout the proof $C$ is a constant independent of$(p,d,H,n,\lambda)$ whose value may very from line to line. 
Note that
\begin{equation*}
\begin{aligned}
\bbE\|\widehat{\vB}\widehat{\vB}\tp &-\vB\vB\tp \|^{2}_{F}\\
&=\underbrace{\bbE\|\widehat{\vB}\widehat{\vB}\tp -\vB\vB\tp \|^{2}_{F}\mathbf{1}_{\ttE^{c}}}_{I}+\underbrace{\bbE\|\widehat{\vB}\widehat{\vB}\tp -\vB\vB\tp \|^{2}_{F}\mathbf{1}_{\ttE}}_{II}.
\end{aligned}
\end{equation*}
\textbf{For $I$:}
By the triangle inequality and the fact that the Frobenius norm of a projection matrix equals to its rank, we have
\begin{align}
I\leqslant 2d\bbP(\ttE^{c})\leqslant  C \frac{d}{n \lambda}.  \label{eq:sir-upper-low-I}
\end{align}

\textbf{For $II$:}  
Let 
$\widehat{\bLambda}_{V}=
\widetilde{\vB}\vD_{H}\widetilde{\vB}\tp $
be the eigen-decomposition of $\widehat{\bLambda}_{V}$, where $\widetilde{\vB}$ is a $p\times d$ orthogonal matrix and $\vD_{H}$ is a $d\times d$ diagonal matrix. 
Note that  $\widetilde{\vB}$ and $\vB$ are sharing the same column space (i.e., $\widetilde{\vB}\widetilde{\vB}\tp =\vB\vB\tp $) since $\col(\bLambda)=\col(\wh\bLambda_{V})$.

Corollary~\ref{cor:minimax:key} states that 
on $\ttE$, one has $\lambda_{d}(\widehat{\bLambda}_{V})=\lambda_{d}(\vD_{H})\geqslant \frac{\lambda}{3}$, 
$\|\widehat{\bLambda}_{V}\|\leqslant 2\kappa\lambda$, 
and  $\lambda_{d+1}(\widehat{\bLambda}_{H})\leqslant \frac{1}{4}\lambda$. 

Let $\vQ=\widehat{\bLambda}_{H}-\widehat{\bLambda}_{V}$ 
 and let $\widehat{\vB}_{\perp}$ be a $p\times (p-d)$ orthogonal matrix whose columns are the last $(p-d)$ eigenvectors of $\widehat{\bLambda}_{H}$.
Applying  the Sin-Theta theorem (e.g., Lemma \ref{lem:sin_theta}) to  the pair of symmetric matrices
$(\widehat{\bLambda}_{V}, \widehat{\bLambda}_{H}$),
one has
\begin{align*}
II=\, &\bbE\|\vB\vB\tp -\widehat{\vB}\widehat{\vB}\tp \|^{2}_{F}\mathbf{1}_{\ttE}=\bbE\|\widetilde{\vB}\widetilde{\vB}\tp -\widehat{\vB}\widehat{\vB}\tp \|^{2}_{F}\mathbf{1}_{\ttE}\\
\leqslant\, &\frac{288}{\lambda^{2}}
\min
\left(
\bbE\|\widetilde{\vB}_{\perp}\tp  \vQ \widehat{\vB}\|_{F}^{2}\mathbf{1}_{\ttE}, \bbE\|\widetilde{\vB}\tp \vQ \widehat{\vB}_{\perp}\|_{F}^{2}\mathbf{1}_{\ttE}
\right)\\
\leqslant\, &\frac{288}{\lambda^{2}}\min\left(\bbE\|\widetilde{\vB}_{\perp}\tp\vQ\|_{F}^{2}\mathbf{1}_{\ttE}, \bbE\|\widetilde{\vB}\tp\vQ\|_{F}^{2}\mathbf{1}_{\ttE}\right)\quad(\text{Lemma \ref{lem:elementary:trivial2} and } \|\widetilde{\vB}\|\leqs 1, \|\widehat{\vB}_{\perp}\| \leqs 1 )\\
\leqslant&\frac{288}{\lambda^{2}}\bbE\|\widetilde{\vB}\tp\vQ\|_{F}^{2}\mathbf{1}_{\ttE}
\end{align*}
Since $\widetilde{\vB}$ and $\vB$ share the same column space, one has
$
\widetilde{\vB}\tp \mathcal{W}= \bs{0}$. 
Thus, one has
\begin{align*}
\widetilde{\vB}\tp \vQ
&=\widetilde{\vB}\tp \mathcal{V}\mathcal{W}\tp.
\end{align*}

By Lemma~\ref{lem:elementary:trivial2} and note that $\|\widetilde{\vB}\tp \mathcal{V}\|^2 \mathbf{1}_{\ttE} \leqslant \|\widehat{\bLambda}_{V}\|\mathbf{1}_{\ttE}\leqslant 2\kappa\lambda$, one has
\begin{align}\label{eq:bound of II}
\bbE\|\widetilde{\vB}\tp \mathcal{V}\mathcal{W}\tp \|^{2}_{F}\mathbf{1}_{\ttE}\leqslant 2\kappa\lambda \bbE\|\mathcal{W}\tp \|^{2}_{F}\leqslant \frac{2\kappa\lambda}{n} H(p-d)
\end{align}
where in the last inequality we apply Lemma~\ref{inline:trivial1} to $\sqrt{n} \cdot \mathbf{\mathcal{E}}$.
Since $\kappa$ is assumed to be fixed, one has
\begin{align}
II \leqslant\, &\frac{576\kappa}{n \lambda} H (p-d) \nonumber\\
\leqslant \, & C^{''}\frac{H (p-d) }{n\lambda}. \label{eq:sir-upper-low-II} 
\end{align}
Combining \eqref{eq:sir-upper-low-I} and \eqref{eq:sir-upper-low-II}, we conclude that 
\begin{align*}
\sup_{\mathcal{M} \in \mathfrak{M}(p,d,\lambda)}\bbE\|\widehat{\vB}\widehat{\vB}\tp -\vB\vB\tp \|^{2}_{F} \lesssim \frac{d+ H(p-d)}{n\lambda} \lesssim \frac{d p }{n\lambda}. 
\end{align*}

\subsection{Proof of Theorem \ref{thm:risk:sparse:upper:d}} \label{app:sparse:risk:proof}

\subsubsection*{Preliminaries} 

Since we have assumed that $\bS=\vI_p$ in this section, the two-fold estimator $\wh\vB$ defined near Theorem \ref{thm:risk:sparse:upper:d} can be simplified. Specifically,  we  first divide the samples into two equal sets of samples and have the corresponding decomposition 
\eqref{eqn:estimator:decomposition}
\begin{equation*}
\begin{aligned}
\widehat{\bLambda}_{H}&={\mathcal{V}}{\mathcal{V}}\tp 
+{\mathcal{V}}{\mathcal{W}}\tp
+{\mathcal{W}}{\mathcal{V}}\tp
+{\mathcal{W}}{\mathcal{W}}\tp \\
&=\widehat{\bLambda}_{V}  
+\vB \mathbf{\mathcal{Z}}\mathbf{\mathcal{E}}\tp \vB_{\perp}\tp 
+ \vB_{\perp}\mathbf{\mathcal{E}}\mathbf{\mathcal{Z}}\tp \vB\tp 
+\vB_{\perp}\mathbf{\mathcal{E}}\mathbf{\mathcal{E}}\tp \vB_{\perp}\tp.
\end{aligned}
\end{equation*}
for these two sets of samples.   That is, 
for $i=1,2,$ we  define $\bLambda_{H}^{(i)}$, $\bLambda_{V}^{(i)}$, $\mathbf{\mathcal{Z}}^{(i)},\mathcal{W}^{(i)}$   $\mathcal{V}^{(i)},\wh{\bLambda}_{z}^{(i)}$  and $\mathcal{E}^{(i)}$  for the first and second set of samples respectively according to the decomposition \eqref{eqn:estimator:decomposition}. 
Then the two-fold aggregation  estimator  $\widehat{\vB}$ can be defined as:

{\it Two-fold Aggregation Estimator with identity covariance: }
\begin{itemize}
	\item[$\mathbf{(i)}$]  	For each $L \in \mathcal{L}( s)$ (the set of all subsets of $[p]$ with size $s$), let 
	\begin{equation}\label{estimator:sample}
	\begin{aligned}
	&\widehat{\vB}_{L}:= 
	\arg\max_{\vB} ~~~\Tr(\vB\tp
	{\bLambda}_{H}^{(1)}\vB)\\
	&\mbox{ s.t. }\vB\tp\vB=\vI_d \mbox{ and } \supp(\vB)\subset L.
	\end{aligned}
	\end{equation}	
	\item[$\mathbf{(ii)}$]  Our aggregation estimator $\widehat{\vB}$ is defined to be $\widehat{\vB}_{L^{*}}$ where 
	\begin{align*}
	L^{*}:= 	\arg\max_{L\in \mathcal{L}(s)}~~~\Tr(\widehat{\vB}_{L}\tp
	{\bLambda}_{H}^{(2)}\widehat{\vB}_{L}).
	\end{align*}
\end{itemize}	
In addition, we introduce an ``oracle estimator'' $\widehat{\vB}_{O}$ that utilizes the information about $S$, the support of $\vB$. 

{\it Oracle Estimator: }
\begin{equation}\label{estimator:oracle}
\begin{aligned}
&\widehat{\vB}_{O}:= \arg\max_{\vA} \langle\bLambda^{(1)}_{H},\vA\vA\tp\rangle
=\arg\max_{\vA}\Tr(\vA^{\top}\bLambda^{(1)}_{H}\vA)\\
&\mbox{ s.t. }  \vA^{\top}\vA=\vI_{d} \mbox{ and } supp(\vA) = S.
\end{aligned}
\end{equation}

Let us first introduce some notations.  For $i=1,2$, let $\bLambda_V^{(i)}=\vB^{(i)}\vD^{(i)}\vB^{(i),\top}$ where $\vB^{(i)}$ is $p\times d$ orthogonal matrix and  $\vD^{(i)}:=\{\lambda_1^{(i)},\dots,\lambda_d^{(i)}\}$ is a diagonal matrix.  For any subset $\tilde{S}$ of $[p]$, let $\vJ_{\tilde{S}}$ be the diagonal matrix such that $\vJ_{\tilde{S}}(i,i)=1$ if $i\in[\tilde{S}]$ and $\vJ_{\tilde{S}}(i,i)=0$ otherwise.

For $i=1,2$, let $\ttE^{(i)}_{2}$ be the event defined similarly as  $\ttE_{2}$ (which is introduced near Corollary \ref{cor:minimax:key}). Let $\bar{\ttE}_{2}=\ttE^{(1)}_{2}\cap \ttE^{(2)}_{2}$ and $\vQ_{S}=\vJ_{S} \left(\bLambda^{(1)}_{H}-\bLambda^{(1)}_{V}\right)\vJ_{S}$.

Let $\mathtt{F}$ consist of the events such that $\|\vJ_{S}\mathcal{\vW}^{(1)}\mathcal{\vW}^{(1),\top}\vJ_{S}\|\leqslant  6 \frac{s \vee H +\log(n\lambda)}{n}$ and define $\ttE:=\bar{\ttE}_{2}\cap\mathtt{F}$.
Following the reasoning of Corollary~\ref{cor:minimax:key}, if $\nu\in (\kappa, 2\kappa]$,  $\kappa^2 H^2\left( \log(nH) + \log \kappa + d \right)/ (n \lambda) $ and $\kappa \left( s\vee H + \log (n\lambda) \right)/ (n\lambda)$ are sufficiently small, one has $\bbP\left( \ttE^{c}\right) \leqslant \frac{C }{n\lambda}$ and the followings hold on $\ttE$:
\begin{enumerate}
\item
\begin{align}\label{fact:key}
\frac\lambda3\leqslant \lambda_{d}^{(i)}\leqslant ... \leqslant \lambda_{1}^{(i)}\leqslant2\kappa\lambda.
\end{align}
\item
By Weyl's inequality, one has
\begin{align}\label{fact:key2}
\|\vQ_{S}\| < \frac{1}{4}\lambda,\quad\lambda_{d+1}\left(\vJ_S\bLambda_{H}^{(1)}\vJ_{S}\right)\leqslant \frac{\lambda}{4}.
\end{align} 
\end{enumerate}

Let $\widehat{\vB}_{O}^{\top}\vB=\vU_{1}\Delta \vU_{2}^{\top}$ be the singular value decomposition of $\widehat{\vB}_{O}^{\top}\vB$ such that the entries of $\Delta$ are non-negative and  $\vM:= \vU_{2}^{\top}\wh{\bLambda}_z^{(2)}\vU_{2}$.  

\subsubsection*{Main part of the proof} 
Now, we start our proof of Theorem \ref{thm:risk:sparse:upper:d}. It is easy to verify that
\begin{align*}
\|\widehat{\vB}\widehat{\vB}^{\top}-\vB\vB^{\top}\|_{F}^{2}
\leqslant C\left( 
\|\widehat{\vB}\widehat{\vB}^{\top}-\widehat{\vB}_{O}\widehat{\vB}_{O}^{\top}\|_{F}^{2}
+\|\widehat{\vB}_{O}\widehat{\vB}_{O}^{\top}-\vB\vB^{\top}\|_{F}^{2}\right).
\end{align*}
For the first term $\|\widehat{\vB}\widehat{\vB}^{\top}-\widehat{\vB}_{O}\widehat{\vB}_{O}^{\top}\|^{2}_{F}$, conditioning on $\ttE$, we know 
\begin{align}
\|\widehat{\vB}\widehat{\vB}^{\top}-\widehat{\vB}_{O}\widehat{\vB}_{O}^{\top}\|^{2}_{F} 
\leqslant & \frac{2}{\lambda_{d}(\wh{\bLambda}_z^{(2)})}\langle\widehat{\vB}_{O}\vU_{1}\vM\vU_{1}^{\top}\widehat{\vB}_{O}^{\top},\widehat{\vB}_{O}\widehat{\vB}_{O}^{\top}-\widehat{\vB}\widehat{\vB}^{\top} \rangle \label{inline:E-1:temp}\\
\leqslant & \frac{C}{\lambda}\langle\widehat{\vB}_{O}\vU_{1}\vM\vU_{1}^{\top}\widehat{\vB}_{O}^{\top}-\bLambda^{(2)}_{H},\widehat{\vB}_{O}\widehat{\vB}_{O}^{\top}-\widehat{\vB}\widehat{\vB}^{\top} \rangle \label{inline:E-O:temp} \\
\nonumber :=& I+II.
\end{align}
where
\begin{align}
I\nonumber=&\frac{C}{\lambda}\langle\widehat{\vB}_{O}\vU_{1}\vM\vU_{1}^{\top}\widehat{\vB}_{O}^{\top}
-\bLambda^{(2)}_{V},\widehat{\vB}_{O}\widehat{\vB}_{O}^{\top}-\widehat{\vB}\widehat{\vB}^{\top} \rangle \\
II\nonumber =&\frac{C}{\lambda}\langle\bLambda^{(2)}_{V}-\bLambda^{(2)}_{H},\widehat{\vB}_{O}\widehat{\vB}_{O}^{\top}-\widehat{\vB}\widehat{\vB}^{\top} \rangle. &
\end{align}
Inequality \eqref{inline:E-1:temp} follows from applying Lemma \ref{cor:elemetary:temp} with the positive definite matrix $\vU_{1}\vM\vU_{1}^{\top}$.
The inequality \eqref{inline:E-O:temp} follows from the definition of $\widehat{\vB}$  and the fact that the eigenvalues of $\wh{\bLambda}_z^{(2)}$ are in $(\lambda/3,2\kappa\lambda)$ (See fact 1). To simplify the notation, we let
\begin{equation*}
\begin{aligned}
\delta=\|\widehat{\vB}_{O}\widehat{\vB}_{O}^{\top}
-\vB\vB^{\top}\|_{F}.
\end{aligned}
\end{equation*}
\textbf{For I:} 
First, $\widehat{\vB}_{O}\vU_{1}$ and $\vB \vU_{2}$   satisfy the condition that $\vU_{1}^{\top}\widehat{\vB}_{O}^{\top}\vB \vU_{2}=\Delta$ is a diagonal matrix with non-negative entries.
Second, the eigenvalues of $\wh{\bLambda}_{z}^{(2)}$ $\in (\frac{1}{3}\lambda, 2\kappa\lambda)$  thus $\vM:= \vU_{2}^{\top}\wh{\bLambda}_{z}^{(2)}\vU_{2}$ has eigenvalues in $(\lambda/3,2\kappa\lambda)$.  By Lemma \ref{lem:elementary:trivial3}, there exists a constant $C$ such that
\begin{equation}\nonumber
\begin{aligned}
\| \widehat{\vB}_{O}\vU_{1}\vM\vU_{1}^{\top}\widehat{\vB}_{O}^{\top}
-\bLambda_V^{(2)}\|_{F}
\leqslant C\lambda \|\widehat{\vB}_{O}\widehat{\vB}^{\top}_{O}-\vB\vB^{\top}\|_{F}.
\end{aligned}
\end{equation}
Thus, conditioning on $\ttE$, one has
\begin{equation}\label{inlince:ttt}
\begin{aligned}
\big|I\big| \leqslant& C \|\widehat{\vB}_{O}\widehat{\vB}_{O}^{\top}-
\vB\vB^{\top}\|_{F}\|\widehat{\vB}_{O}\widehat{\vB}_{O}^{\top}-\widehat{\vB}\widehat{\vB}^{\top}\|_{F}\\
= & C\delta \|\widehat{\vB}_{O}\widehat{\vB}_{O}^{\top}-\widehat{\vB}\widehat{\vB}^{\top}\|_{F}.
\end{aligned}
\end{equation} 
\textbf{For II:}  
Define $\vK_{L}=\|\widehat{\vB}_{O}\widehat{\vB}_{O}^{\top}-\widehat{\vB}_{L}\widehat{\vB}_{L}^{\top}\|^{-1}_{F}\left(\widehat{\vB}_{O}\widehat{\vB}_{O}^{\top}-\widehat{\vB}_{L}\widehat{\vB}_{L}^{\top}\right)$. ( For any $L \in \mathcal{L}(s)$, $\widehat{\vB}_{L}$ is introduced in \eqref{estimator:sample} ). Then
\begin{align*}
|II|=& \frac{C}{\lambda}\langle\bLambda^{(2)}_{V}-\bLambda^{(2)}_{H},(\widehat{\vB}_{O}\widehat{\vB}_{O}^{\top}-\widehat{\vB}\widehat{\vB}^{\top})\|\widehat{\vB}_{O}\widehat{\vB}_{O}^{\top}-\widehat{\vB}\widehat{\vB}^{\top}\|_F^{-1} \rangle\|\widehat{\vB}_{O}\widehat{\vB}_{O}^{\top}-\widehat{\vB}\widehat{\vB}^{\top}\|_F\\
\leqslant& \frac{C}{\lambda}\max_{L\in\mathcal{L}(s)} \left| \langle\bLambda^{(2)}_{V}-\bLambda^{(2)}_{H},\vK_L\rangle\right|  \|\widehat{\vB}_{O}\widehat{\vB}_{O}^{\top}-\widehat{\vB}\widehat{\vB}^{\top}\|_F.
\end{align*}

From the equation \eqref{eqn:estimator:decomposition}, one has
\begin{equation}\label{inlince:tt}
\begin{aligned}
\big|II \big| 
\leqslant& \frac{C}{\lambda}\|\widehat{\vB}_{O}\widehat{\vB}_{O}^{\top}-\widehat{\vB}\widehat{\vB}^{\top}\|_{F}\left( 2T_{2}+T_{1}\right)
\end{aligned}
\end{equation}
where 
$T_{1}=\max_{L\in\mathcal{L}(s)}\Big|\Big< \mathcal{W}^{(2)}\mathcal{W}^{(2),\top},\vK_{L}\Big>\Big|$,
$T_{2}=\max_{L\in\mathcal{L}(s)}\Big|\Big< \mathcal{V}^{(2)}\mathcal{W}^{(2),\top},\vK_{L}\Big>\Big|$.
\vspace*{3mm}
To summarize, conditioning on $\ttE$, one has
\begin{align}
\|\widehat{\vB}\widehat{\vB}^{\top}-
\widehat{\vB}_{O}\widehat{\vB}_{O}^{\top}\|_{F}\leqslant C\left(\delta+\frac{1}{\lambda}(2T_{2}+T_{1})\right).
\end{align}
Thus, one has
\begin{align*}
\|\widehat{\vB}\widehat{\vB}^{\top}- \vB\vB^{\top}\|_{F}^2\bold{1}_{\ttE} 
&\leqslant 
C\left( \delta^2+\|\widehat{\vB}_{O}\widehat{\vB}_{O}^{\top}-\widehat{\vB}\vB^{\top}\|^{2}_{F}\right)\bold{1}_{\ttE}\\
&\leqslant 
C\left( \delta^2+C\left( \delta+\frac{1}{\lambda}\left(2T_{1}+T_{2} \right)\right)^2\right)\bold{1}_{\ttE}\\
&\leqslant C\left( \delta^2+\left(\frac{1}{\lambda}\left(2T_{1}+T_{2} \right)\right)^2\right)\bold{1}_{\ttE}.
\end{align*}
Note that $\P( \ttE^c )\leqslant \frac{C }{n\lambda}$. If we can prove
\begin{align}
\bbE\delta^{2}\bold{1}_{\ttE}\leqslant C\epsilon_{n}^{2} \quad  
\mbox{ and } \bbE( 2T_{1}+T_{2})^{2}\bold{1}_{\ttE}\leqslant \lambda^{2} \epsilon^{2}_{n},
\end{align}   
then one has $\bbE \|\widehat{\vB}\widehat{\vB}^{\top}- \vB\vB^{\top}\|_{F}^2\bold{1}_{\ttE}\leqslant C\epsilon_{n}^2$ and then 
$$
\E\left[ \|\widehat{\vB}\widehat{\vB}\tp -\vB\vB\tp \|_{F}^{2} \one_{\ttE^c}\right] \leqslant 2 d   \frac{C   }{n\lambda}\lesssim \epsilon_n^2. 
$$
Therefore, we conclude that $\E\left[ \|\widehat{\vB} \widehat{\vB} \tp - \vB\vB\tp \|_{F}^{2}  \right]\lesssim \epsilon_{n}^{2}$. 
Thus, it is suffice to prove the following two lemmas.
\begin{lemma} If $n\lambda\leqslant e^{s\vee H}$, then
	\begin{align}
	\bbE\delta^{2}\bold{1}_{\ttE}\leqslant C\epsilon_{n}^{2}.  
	\end{align}

\end{lemma}
\begin{proof}

By \eqref{fact:key} and \eqref{fact:key2}, we know that the eigenvalues of $\vJ_{S}\bLambda^{(1)}_{V}\vJ_{S}=\bLambda^{(1)}_{V}$ is in  $(\frac{1}{3}\lambda,2\kappa\lambda)$ and  the $(d+1)$-th largest eigenvalues of $\vJ_{S}\bLambda^{(1)}_{H}\vJ_{S}$ is less than $\frac{\lambda}{4}$. 
Let $\widehat{\vB}^{\perp}_O$ be a $p\times (p-d)$ orthogonal matrix whose columns are the last $(p-d)$ eigenvectors of $\vJ_S\bLambda_{H}^{(1)}\vJ_S$. 
 After applying the Sin-Theta Theorem (Lemma \ref{lem:sin_theta}) to  the pair of symmetric matrices $(\bLambda^{(1)}_{V}=\vJ_{S}\bLambda^{(1)}_{V}\vJ_{S}, \vJ_{S}\bLambda^{(1)}_{H}\vJ_{S})$ , one has
\begin{equation*}\label{use:sintheta}
\begin{aligned}
\delta^2 =\|\widehat{\vB}_{O}\widehat{\vB}_{O}^{\top}
-\vB^{(1)}\vB^{(1),\top}\|_{F}
\leqslant  \frac{C}{\lambda^2} \| \widehat{\vB}_{O}^{\perp,\top } \vQ_{S}  \vB^{(1)}\|_{F}^2
\leqslant  \frac{C}{\lambda^{2}} \|  \vQ_{S}  \vB^{(1)}\|_{F}^{2}.
\end{aligned}
\end{equation*} 
Note that $\vB^{(1),\top}  \vJ_{S} \mathcal{\vW} = \vB^{(1),\top}\vB_{\perp} \mathcal{\vE}=0$ because $\vJ_{S}\vB^{(1)}=\vB^{(1)}$ and $\vB^{(1),\top} \vB_{\perp}=0$. Hence 
$$
\|  \vQ_{S}  \vB^{(1)} \|_{F}^{2}= \|   \vB \tp  \mathcal{\vV}^{(1)} \mathcal{\vW}^{(1),\top} \vJ_{S} \|_{F}^{2} \leqslant  \|  \mathcal{\vV}^{(1)}\|^{2} \|\mathcal{\vW}^{(1),\top} \vJ_{S} \|_{F}^{2}. 
$$
On the event $\ttE$,$\|\mathcal{\vV}^{(1)}\|^{2}\leqs \|\bLambda^{(1)}_{V}\| \leqslant 2 \kappa \lambda$ and $\left\{ \| \vJ_{S}\mathcal{\vW}^{(1)}\mathcal{\vW}^{(1),\top}\vJ_{S}  \| \leqslant  6 \frac{s \vee H +\log(n\lambda)}{n} \right\}$.
Therefore,
\begin{equation*}
\begin{aligned}
\bbE[\delta^{2}\mathbf{1}_{ \ttE}] &\leqslant \frac{C}{\lambda}  \bbE[ \| \vJ_{S} \mathcal{\vW}^{(1)}  \|_{F}^{2}\mathbf{1}_{ \ttE}]\leqslant \frac{C(s\wedge  H)}{\lambda}\E[\| \vJ_{S}\mathcal{\vW}^{(1)}\mathcal{\vW}^{(1),\top}\vJ_{S}\|\mathbf{1}_{ \ttE}]\leqslant C(s\wedge  H) \frac{s \vee H}{n\lambda}\leqslant C\epsilon_n^{2}
\end{aligned}
\end{equation*} 
where the second inequalities follows from the basic inequality that  $\Tr(\vA)\leqslant \mr{rank}(\vA)\|\vA\|$ and the fact that $$\mr{rank}(\vJ_{S}\mathcal{\vW}^{(1)}\mathcal{\vW}^{(1),\top}\vJ_{S})\leqslant s\wedge H$$ and the third inequality follows from $n\lambda<e^{s\vee H}$.
\end{proof}

\begin{lemma} \label{lem:bound T1 T2}
There exists positive constant $C$ such that
	\[
	\bbE(2T_{1}+T_{2})^{2} \bold{1}_{\ttE}\leqslant C\lambda^{2}\epsilon^{2}_{n}.
	\]
\end{lemma}
	\begin{proof}
	 Since $(2T_{1}+T_{2})^{2} \leqslant C(T_{1}^{2}+T_{2}^{2})$, we only need to bound $\bbE T_{1}^{2}$ and $\bbE T_{2}^{2}$ separately.
\newline\textbf{For $T_{1}$. }
Recall that $\mathcal{W}^{(2)}=\vB_{\perp}\mathcal{E}^{(2)}$ (See notation near \eqref{eqn:estimator:decomposition}.) and for each fixed $L \in \mathcal{L}_{s}$, $\vK_{L}\independent \mathcal{W}^{(2)}$, hence 
\begin{align}
\langle\mathcal{W}^{(2)}\mathcal{W}^{(2),\top},\vK_{L} \rangle=\langle\mathcal{E}^{(2)}\mathcal{E}^{(2),\top},\vB_{\perp}^{\top}\vK_{L}\vB_{\perp} \rangle
\end{align} 
and $\vB_{\perp}^{\top}\vK_{L}\vB_{\perp}  \independent \mathcal{W}^{(2)}$.

By Lemma \ref{lem:elementary:trivial2}, $\|\vB_{\perp}^{\top}\vK_{L}\vB_{\perp}\|_{F}\leqslant 1$.  
For any $m\times m$ symmetric matrix $\vA$, $\|\vA - \frac{\Tr(\vA)}{m} \bs{I}_{m}\|_{F}^2=\Tr(\vA\tp \vA)- \frac{1}{m}\Tr(\vA)^2\leqslant \|\vA\|_{F}^2$. Therefore, 
\begin{align*}
\|\vB_{\perp}^{\top}\vK_{L}\vB_{\perp}-\frac{\Tr(\vB_{\perp}^{\top}\vK_{L}\vB_{\perp})}{p-d}\vI_{p-d}\|_F\leqslant\|\vB_{\perp}^{\top}\vK_{L}\vB_{\perp}\|_F\leqslant1.
\end{align*}
Note that 
$\mathcal{E}^{(2)}$ is a $(p-d)\times H$ matrix and  $\sqrt{n}\mathcal{E}_{i,j}^{(2)}\sim N(0,1)$, we can apply  Lemma \ref{lem:Cai:lem4} with $\vZ=\sqrt n\mathcal{E}^{(2)}$ and $\vK=\vB_{\perp}^{\top}\vK_{L}\vB_{\perp}-\frac{\Tr(\vB_{\perp}^{\top}\vK_{L}\vB_{\perp})}{p-d}\vI_{p-d}$ to derive that
\begin{align}
\bbP\left(\Big|\Big\langle\mathcal{E}^{(2)}\mathcal{E}^{(2),\top}, \vB_{\perp}^{\top}\vK_{L}\vB_{\perp}-\frac{\Tr(\vB_{\perp}^{\top}\vK_{L}\vB_{\perp})}{p-d}\vI_{p-d} \Big\rangle\Big
|\geqslant \frac{2\sqrt{H}}{ n}t +\frac{2}{ n}t^{2} \right) \leqslant 2\exp\left(-t^{2} \right).
\end{align}	

After applying Lemma \ref{lem:Cai:lem5} with $N=|\mathcal{L}(s)|\leqslant \left(\frac{ep}{s} \right)^{s}$, $a=\frac{2\sqrt{H}}{n}$, $b=\frac{2}{n}$, $c=2$ and $X_i=\Big\langle \mathcal{E}^{(2)}\mathcal{E}^{(2),\top}, \vB_{\perp}^{\top}\vK_{L}\vB_{\perp}-\frac{\Tr(\vB_{\perp}^{\top}\vK_{L}\vB_{\perp})}{p-d}\vI_{p-d} \Big\rangle$, one has
\begin{align*}
&\E{\max_{L\in\mathcal L(s)}}\Big|\Big\langle \mathcal{E}^{(2)}\mathcal{E}^{(2),\top}, \vB_{\perp}^{\top}\vK_{L}\vB_{\perp}-\frac{\Tr(\vB_{\perp}^{\top}\vK_{L}\vB_{\perp})}{p-d}\vI_{p-d} \Big\rangle\Big|^2\\
\leqslant& \frac{2H+32}{n^2}\log(2eN)+\frac{8}{n^2}\log^2(2N).
\end{align*}
Note that 
\begin{align*}
&\E{\max_{L\in\mathcal L(s)}}\Big|\Big\langle \mathcal{E}^{(2)}\mathcal{E}^{(2),\top}, \vB_{\perp}^{\top}\vK_{L}\vB_{\perp} \Big\rangle\Big|^2\leqs 2\E{\max_{L\in\mathcal L(s)}}\Big|\Big\langle \mathcal{E}^{(2)}\mathcal{E}^{(2),\top}, \frac{\Tr(\vB_{\perp}^{\top}\vK_{L}\vB_{\perp})}{p-d}\vI_{p-d} \Big\rangle\Big|^2\\
+&2\E{\max_{L\in\mathcal L(s)}}\Big|\Big\langle \mathcal{E}^{(2)}\mathcal{E}^{(2),\top}, \vB_{\perp}^{\top}\vK_{L}\vB_{\perp}-\frac{\Tr(\vB_{\perp}^{\top}\vK_{L}\vB_{\perp})}{p-d}\vI_{p-d} \Big\rangle\Big|^2
\end{align*}
and
\begin{align*}
\E{\max_{L\in\mathcal L(s)}}\Big|\Big\langle \mathcal{E}^{(2)}\mathcal{E}^{(2),\top}, \frac{\Tr(\vB_{\perp}^{\top}\vK_{L}\vB_{\perp})}{p-d}\vI_{p-d} \Big\rangle\Big|^2&=\E{\max_{L\in\mathcal L(s)}}(\frac{\Tr(\vB_{\perp}^{\top}\vK_{L}\vB_{\perp})}{p-d})^2\|\mathcal{E}^{(2)}\|_F^4\overset{(a)}{\leqslant}\frac{2s}{(p-d)^2}\E\|\mathcal{E}^{(2)}\|_F^4\\
&\overset{(b)}{=}\frac{2s}{(p-d)^2}\frac{(p-d)^2H^2+2H(p-d)}{n^2}\asymp \frac{H^2s}{n^2}.
\end{align*}
Here in $(a)$ we used the inequality that 
$|\Tr(\vA)|\leqs \sqrt{\mr{rank}(\vA) } \|\vA\|_{F}$ and the facts that $\mr{rank}(\vB_{\perp}^{\top}\vK_{L}\vB_{\perp})\leqs 2s$ and $\|\vB_{\perp}^{\top}\vK_{L}\vB_{\perp}\|_{F}\leqs 1$, 
and in $(b)$ we used Lemma \ref{inline:trivial1}. 
Then we have
\begin{align*}
&\E{\max_{L\in\mathcal L(s)}}\Big|\Big\langle \mathcal{E}^{(2)}\mathcal{E}^{(2),\top}, \vB_{\perp}^{\top}\vK_{L}\vB_{\perp} \Big\rangle\Big|^2\\
\lesssim&  \frac{4H+64}{n^2}\log(2eN)+\frac{16}{n^2}\log^2(2N)+\frac{H^2s}{n^2}\\
\lesssim& \frac{H}{n^2}\log(2eN)+\frac{\log^2(2N)}{n^2}+\frac{H^2s}{n^2}\lesssim\lambda^2\epsilon_n^4\lesssim\lambda^2\epsilon_n^2.
\end{align*}

\textbf{For $T_{2}$. }
Fix $L \in \mathcal{L}(s)$. Since $\mathcal{V}^{(2)}\independent \mathcal{W}^{(2)}$, $\vK_{L}\independent \mathcal{W}^{(2)}$ and  $\vK_{L}\independent \mathcal{V}^{(2)}$, conditioned on the $\mathcal{V}^{(2)}$ and $\vK_{L}$,  we know that 
\[
\sqrt{n}\langle \mathcal{V}^{(2)}\mathcal{W}^{(2),\top},\vK_{L}\rangle=\langle\vB_{\perp}^{\top}\vK_{L}\mathcal{V}^{(2)},\sqrt{n}\mathcal{E}^{(2)} \rangle
\] is distributed according to $N(0, \|\vB_{\perp}^{\top}\vK_{L}\mathcal{V}^{(2)}\|^{2}_{F})$. 
Therefore 
\begin{align*}
\sqrt{n}\langle \mathcal{V}^{(2)}\mathcal{W}^{(2),\top},\vK_{L}\rangle\overset{d}{=}\|\vB_{\perp}^{\top}\vK_{L}\mathcal{V}^{(2)}\|_{F}W
\end{align*}
for some $W \sim N(0,1)$ independent of $\mathcal{V}^{(2)}$ and $\vK_{L}$. 
For simplicity of notation, we denote $ \sqrt{n}\langle \mathcal{V}^{(2)}\mathcal{W}^{(2),\top},\vK_{L} \rangle$ by $F_{L}$. 
Define the event $\tilde{\ttE}=\{ \|\mathcal{V}^{(2)}\mathcal{V}^{(2),\top}\| \leqslant 2\kappa \lambda \}$. Then $\ttE\subset \tilde{\ttE}$ and $\tilde{\ttE}$ only depends on $\mathcal{V}^{(2)}$. 
Note that $\|\vB_{\perp}^{\top}\vK_{L}\mathcal{V}^{(2)}\|_{F}\leqs \|\vB_{\perp}^{\top}\vK_{L}\|_{F} \|\mathcal{V}^{(2)}\|\leqs \sqrt{\|\mathcal{V}^{(2)}\mathcal{V}^{(2),\top}\|}$. 
Consequently, 
\begin{align}
\bbP\left(|F_{L}| >t \mid \tilde{\ttE} \right) \leqslant \bbP\left(\sqrt{2\kappa\lambda}|W| >t\right) \leqslant 2 \exp\left(-\frac{t^{2}}{2\kappa\lambda} \right).
\end{align}
In other words, conditioning on $\tilde{\ttE}$, each of $F_{L}$ ($L\in \mc{L}(s)$) satisfies the premise in Lemma~\ref{lem:Cai:lem5} with $(a,b,c)=(\sqrt{2\kappa \lambda},0,2)$. Therefore, 
$$\bbE \left( T_{2}^{2} \bold{1}_{\ttE}\right)\leqslant  \frac{1}{n} \bbE \max_{L\in \mc{L}(s)}\left( F_{L}^{2} \bold{1}_{\tilde{\ttE}}\right)\leqslant \frac{4 \kappa\lambda}{n}\log \left( 2e N \right) \leqs C\lambda^{2}\epsilon^{2}_{n}.$$
\end{proof}

\section{Proofs of upper bounds with a unknown  covariance matrix}\label{app:ub-general}

In this section, we prove the upper bounds  in Section~\ref{sec:minimax rate} for the general cases where $\bS$ is unknown. 
As in Appendix~\ref{app:ub-identity}, we take $H$ to be an integer such that $H\leqslant H_{0} d$ for some constant $ H_{0}>K_0\vee C$ and the inequality in Lemma~\ref{lem:key lemma in Lin under moment condition} holds.

\subsection{Proof of Theorem~\ref{thm:risk:oracle:upper:d}}\label{app:ub-ld-general}

Let $\bS^{1/2}$ be a square root of $\bS$,   $\widetilde{\vB}=\bS^{1/2}\vB$ and $\widetilde{\vX}_{i}=\bS^{-1/2}\vX_i$. Then $\widetilde{\vX}_{i}\sim N(0, \vI_p)$ and $\vB\tp  \vX_{i}=\widetilde{\vB}\tp  \widetilde{\vX}_{i}$. 
Let $\widetilde{\bLambda}=\Cov \left( \bbE[ \widetilde{\vX} \mid Y] \right)$.

Recall in Definition~\ref{def:class-upper-bound} that $\lambda\leqslant\lambda_{d}(\bLambda) \leqslant\lambda_{1}(\bLambda)\leqslant \kappa \lambda.$
Since $\widetilde{\bLambda}=\bS^{-1/2}\bLambda\bS^{-1/2}$, we can obtain $$
\lambda_i(\widetilde{\bLambda})\leqs \lambda_{\max}(\bS^{-1}) \lambda_i(\bLambda) \text{ and } \lambda_i(\bLambda)\leq \lambda_{\max}(\bS) \lambda_i(\widetilde{\bLambda}),
$$
and thus $\lambda_i(\widetilde{\bLambda})\in [\lambda/M,M\kappa\lambda]$. 
To apply results from Section~\ref{app:ub-identity} for $\widetilde{\bLambda}$, 
we may define $\widetilde{\lambda}=\lambda/M$ and $\widetilde{\kappa}=M^2\kappa$, so that the eigenvalues of $\widetilde{\bLambda}$ lies in $[\widetilde{\lambda},\widetilde{\kappa}\widetilde{\lambda}]$. 
However, to ease the notation, we abuse the notation and still use $(\lambda,\kappa)$ in place of $(\widetilde{\lambda},\widetilde{\kappa})$ in the following.

We will use the results in Section~\ref{app:ub-identity} for the pairs $(\widetilde{\vX}_{i},Y_i)$'s by defining $\widetilde{\vB}_{\perp}$,  $\widetilde{\vE}$, $\widetilde{\vV}$, $\widetilde{\vW}$ accordingly.

Let $\widetilde{\vB}_{\perp}$ be a $p\times (p-d)$ orthogonal matrix such that $\widetilde{\vB}\tp  \widetilde{\vB}_{\perp}=0$. 

For any pair of $(\vX, Y)$ sampled from $\mc M$, let $\vZ=\vB\tp  \vX=\widetilde{\vB}\tp \widetilde{\vX}$ and $\widetilde{\vE}=\widetilde{\vB}_{\perp}\tp  \widetilde{\vX}$. We  have the decomposition that
\begin{equation*}
\begin{aligned}
\widetilde{\vX}&=\widetilde{\vB}\widetilde{\vB}\tp\widetilde{\vX}+\widetilde{\vB}_{\perp}\widetilde{\vB}_{\perp}\tp\widetilde{\vX}
= \widetilde{\vB}\vZ+\widetilde{\vB}_{\perp}\widetilde{\vE}. 
\end{aligned}
\end{equation*}

Let $\widetilde{\vV}=\widetilde{\vB}\vZ$ and $\widetilde{\vW}=\widetilde{\vB}_{\perp}\widetilde{\vE}$. Then $\widetilde{\vV}\tp \widetilde{\vW}=0$.

We introduce the notation $\overline{\widetilde{\vX}}_{h,\cdot}$ similar to
the definition of $\overline{\vX}_{h,\cdot}$ in Equation~\ref{eq:sliced-xy} and 
let  $\widetilde{\mathbf{\mathcal{X}}}=\frac{1}{\sqrt{H}}
\left[\overline{\widetilde{\vX}}_{1,\cdot} ~,~ \overline{\widetilde{\vX}}_{2,\cdot},...,~ 
\overline{\widetilde{\vX}}_{H,\cdot}\right]$. Similarly we define $\overline{\widetilde{\vV}}_{h,\cdot}$, $\overline{\vZ}_{h,\cdot}$, $\overline{\widetilde{\vW}}_{h,\cdot}$, $\overline{\widetilde{\vE}}_{h,\cdot}$ and $\widetilde{\mathbf{\mathcal{V}}}$, $\mathbf{\mathcal{Z}}$, $\widetilde{\mathbf{\mathcal{W}}}$, $\widetilde{\mathbf{\mathcal{E}}}$.  We have $\widetilde{\mathbf{\mathcal{V}}}= \widetilde{\vB} \mathbf{\mathcal{Z}}$ and $\widetilde{\mathcal{W}}=\widetilde{\vB}_{\perp}\widetilde{\mathcal{E}}$. Since $\widetilde{\vE} \sim N(0,\vI_{p-d})$ and is independent of $Y$, we know that the entries $\widetilde{\mathcal{E}}_{i,j}$ of $\widetilde{\mathcal{E}}$ are $i.i.d.$ samples of $N(0,\frac{1}{n})$.

Define $\widehat{\widetilde{\bLambda}}_{H}=\widetilde{\mathbf{\mathcal{X}}}\widetilde{\mathbf{\mathcal{X}}}\tp$, $\widehat{\bLambda}_{z}=\mathbf{\mathcal{Z}}\mathbf{\mathcal{Z}}\tp $ and $\widehat{\widetilde{\bLambda}}_{V}=\widetilde{\mathbf{\mathcal{V}}} \widetilde{\mathbf{\mathcal{V}}}\tp $. Then $\widehat{\widetilde{\bLambda}}_{V}=\widetilde{\vB} \widehat{\bLambda}_{z}\widetilde{\vB}\tp$, 
and $\widehat{\bLambda}_{H}=\bS^{1/2}\widehat{\widetilde{\bLambda}}_{H}\bS^{1/2}$.

We have the following decomposition
\begin{equation}
\begin{aligned}
\widehat{\widetilde{\bLambda}}_{H}&=\wt{\mathcal{V}}\wt{\mathcal{V}}\tp 
+ \widetilde{\mathbf{\mathcal{V}}}\widetilde{\mathbf{\mathcal{W}}}\tp
+\widetilde{\mathbf{\mathcal{W}}}\widetilde{\mathbf{\mathcal{V}}}\tp
+\widetilde{\mathbf{\mathcal{W}}}\widetilde{\mathbf{\mathcal{W}}}\tp. 
\end{aligned}
\end{equation}

We define some events: $\widetilde{\ttE}_{1}=\Big\{~ \| \widetilde{\mathbf{\mathcal{W}}}\widetilde{\mathbf{\mathcal{W}}}\tp \|\leqslant 6 \frac{p \vee H +\log(n\lambda)}{n}~\Big\}$, 
$\widetilde{\ttE}_{2}=\Big\{~ \| \widehat{\widetilde{\bLambda}}_{V}-\widetilde{\bLambda} \|\leqslant \frac{2}{3 \nu}\kappa\lambda~\Big\}$ 
and $\widetilde{\ttE}=\widetilde{\ttE}_{1}\cap \widetilde{\ttE}_{2}$. 
In view of Corollary~\ref{cor:minimax:key}, we have the following result.

\begin{corollary}\label{cor:minimax:key-general}
For $\nu \in (\kappa, 2\kappa]$, we can find constants $C$ and $\widetilde{C}$ , such that if \\ 
\[\label{eq:cor-key-require-1:unknown cov low}
\kappa^2 H^2\left( \log(n\kappa H)  + d \right)<C n \lambda
\] and $\kappa \left( p\vee H + \log (n\lambda) \right)< C n\lambda$,
then 
$\bbP\left( \widetilde{\ttE}^{c}\right) \leqslant \frac{ \widetilde{C} }{n\lambda}$
and on the event $ \widetilde{\ttE}$, the followings hold
\begin{itemize}
\item[$a)$] 
$ \frac{1}{3}\lambda\leqslant \lambda_{d}(\widehat{ \widetilde{\bLambda}}_{V}) \leqslant \lambda_{1}(\widehat{\widetilde{\bLambda}}_{V})\leqslant 2\kappa \lambda$.
\item[$b)$]
$\|\widehat{\widetilde{\bLambda}}_{H}-\widehat{\widetilde{\bLambda}}_{V}\|
\leqslant \lambda \sqrt{18 \kappa \frac{p \vee H +\log(n\lambda )}{n \lambda}}< \frac{1}{4}\lambda$. 
\item[$c)$] 
$\lambda_{d+1}(\widehat{\widetilde{\bLambda}}_{H}) < \frac{1}{4}\lambda$.
\end{itemize}
	
\end{corollary}

Recall the SIR estimator $\widehat{\vB}$ in \eqref{eqn:Bhat}. 
Our goal is to bound the expectation of $\| \widehat{\vB}^{\otimes} -  \vB^{\otimes} \|^{2}_{F}$. Under the assumption that $\|\bS^{-1}\|< M$, one has 
\begin{align}
\|   \widehat{\vB}^{\otimes} -  \vB^{\otimes}\|_{F} & < M \| \bS^{1/2}\left(  \widehat{\vB}^{\otimes} -  \vB^{\otimes} \right) \bS^{1/2} \|_{F} \nonumber \\
& \leqslant M \left( \| \bS^{1/2}  \widehat{\vB}^{\otimes} \bS^{1/2} -  (\widehat{\bS}^{1/2} \wh{\vB})^{\otimes}  \|_{F} + \|  (\widehat{\bS}^{1/2}\widehat{\vB})^{\otimes}  - \bS^{1/2}\vB^{\otimes}\bS^{1/2} \|_{F}  \right). \label{eq:ub-ld-dec-general}
\end{align}

\textbf{Step 1: Bounding the first term in \eqref{eq:ub-ld-dec-general}. }

For any matrices $\vA \in \R^{ p \times p}$ and $ \vL\in \R^{p \times p}$, one has the identity that $\vA\tp  \vL \vA - \vL= (\vA - \vI_p)\tp \vL \vA+ \vL(\vA-\vI_p)$ and the inequality 
\begin{align*}
  \|\vA\tp  \vL \vA -\vL\|_{F} & \leqslant \|\vA-\vI_p\| \|\vL\|_{F} \| \vA\| + \|\vL\|_{F} \|\vA-\vI_p\|  \\
  & = (\| \vA\| +1) \|\vA-\vI_p\| \|\vL\|_{F}. 
\end{align*}
Substitute $\vA=\widehat{\vI}= \widehat{\bS}^{-1/2}\bS^{1/2}$ and $\vL=(\widehat{\bS}^{1/2}\widehat{\vB})^{\otimes}$. Then
\begin{align*}
\| \bS^{1/2}  \widehat{\vB}^{\otimes} \bS^{1/2} -  (\widehat{\bS}^{1/2} \wh{\vB})^{\otimes}  \|_{F}\leqslant (\| \widehat{\vI}\| +1) \|\widehat{\vI}-\vI_p\| \|(\widehat{\bS}^{1/2} \wh{\vB})^{\otimes}\|_{F}\leqslant d (\| \widehat{\vI}\| +1) \|\widehat{\vI}-\vI_p\|
\end{align*}
since $\widehat{\bS}^{1/2}\widehat{\vB}$ is a $p\times d$ orthogonal matrix. 
In order to obtain an upper bound for the right hand side in the last inequality, we present the following lemma, whose proof will be provided at the end of this subsection.

\begin{lemma}\label{lem:Sigma-control}
There exist constants $C$ and $\widetilde{C}$, such that if $p+\log(n\lambda)< C n$, then 
$\|\widehat{\vI}\|^{2}<2 M^2$, 	$ \|\widehat{\bS}^{-1/2}\|^{2}<2M$ and $\|\widehat{\vI}-\vI_p \|^{2}< \widetilde{C}
	\frac{p+ \log( n\lambda) }{n}$ hold with probability at least $1-\widetilde{C}/n\lambda$. 
\end{lemma}

Let $\ttE$ be the intersection of the events in Corollary~\ref{cor:minimax:key-general} and Lemma~\ref{lem:Sigma-control}. 

 By Lemma~\ref{lem:Sigma-control}, one has 
\begin{align}\label{eq:ub-ld-intermed-general}
\one_{\ttE}  \| \bS^{1/2}  \widehat{\vB}^{\otimes} \bS^{1/2} -  \widehat{\bS}^{1/2} \widehat{\vB}^{\otimes}\widehat{\bS}^{1/2}  \|_{F}^{2} \lesssim \frac{d[ p+ \log( n\lambda) ] }{n}. 
\end{align}

\textbf{Step 2: Bounding the second term in \eqref{eq:ub-ld-dec-general}.}

We know that $\widehat{\bS}^{1/2}\widehat{\vB}$ is formed by the first $d$ leading eigenvector of $\widehat{\widetilde{\bLambda}}_{H}=\widehat{\bS}^{-1/2}\widehat{\bLambda}_{H}\widehat{\bS}^{-1/2}$.  
Since $\widetilde{\vB}$ is formed by the first $d$ leading eigenvectors of $\widehat{\widetilde{\bLambda}}_{V}$, 
by Lemma~\ref{lem:sin_theta} and Corollary~\ref{cor:minimax:key-general}, on the event $\ttE$, it holds that
\begin{align}\label{eq:ld-main-general}
\| (\widehat{\bS}^{1/2}\widehat{\vB})^{\otimes} - \widetilde{\vB}^{\otimes} \|^{2}_{F}\, & 
\lesssim 
\frac{1}{\lambda^2}\|\widehat{\bS}^{-1/2}\widehat{\bLambda}_{H}\widehat{\bS}^{-1/2}-\widehat{\widetilde{\bLambda}}_{V}\|_{F}^{2}. 
\end{align}
Let $\Delta=\widehat{\bS}^{-1/2}\widehat{\bLambda}_{H}\widehat{\bS}^{-1/2}-\widehat{\widetilde{\bLambda}}_{V}$ and $\widehat{\vI}= \widehat{\bS}^{-1/2}\bS^{1/2}$. Then $\Delta=\widehat{\vI} \left( \widehat{\widetilde{\bLambda}}_{H}-\widehat{\widetilde{\bLambda}}_{V} \right) \widehat{\vI}\tp + (\widehat{\vI}-\vI_p)\widehat{\widetilde{\bLambda}}_{V}\widehat{\vI}\tp + \widehat{\widetilde{\bLambda}}_{V}\left( \widehat{\vI}-\vI_p \right)\tp $.

We then have
\begin{equation}
\label{eq:ld-main-delta-general}
1_{\ttE} \|\Delta\|_{F}^{2}\lesssim  1_{\ttE}\left( \|\widehat{\widetilde{\bLambda}}_{H}-\widehat{\widetilde{\bLambda}}_{V}\|_{F}^{2} + \|\widehat{\widetilde{\bLambda}}_{V}\|_{F}^{2}  \frac{p+ \log(n\lambda)}{n} \right). 
\end{equation}
Note that $\|\widehat{\widetilde{\bLambda}}_{H}-\widehat{\widetilde{\bLambda}}_{V}\|_{F}\leqslant 2 \|\widetilde{\mathbf{\mathcal{V}}}\widetilde{\mathbf{\mathcal{W}}}\tp\|_{F}+\|\widetilde{\mathbf{\mathcal{W}}}\widetilde{\mathbf{\mathcal{W}}}\tp\|_{F}$. By Lemma~\ref{inline:trivial1} and $\widetilde{\mathbf{\mathcal{W}}}=\widetilde{\vB}_{\perp}\widetilde{\mathbf{\mathcal{E}}}$, we have 
\begin{align}\label{eq:ld-ww-general}
\bbE \|\widetilde{\mathbf{\mathcal{W}}}\widetilde{\mathbf{\mathcal{W}}}\tp\|_{F}^{2} \, & \leqslant \bbE \|\widetilde{\mathbf{\mathcal{E}}}\widetilde{\mathbf{\mathcal{E}}}\tp\|_{F}^{2}\nonumber \\
\, & = \frac{1}{n^2} ( p-d ) H (p-d+H+1) \nonumber \\
\, & \lesssim \frac{ p H \left( p+H \right) }{n^2}, 
\end{align}
and 
\begin{align*}
\bbE \Tr \left( \widetilde{\mathbf{\mathcal{W}}}\tp\widetilde{\mathbf{\mathcal{W}}} \right) \, & \leqslant  \frac{1}{n} ( p-d ) H \\
\, & \lesssim \frac{ p H  }{n}. 
\end{align*}

Since on the event $\ttE$, $\| \widetilde{\mathbf{\mathcal{V}}}\tp \widetilde{\mathbf{\mathcal{V}}}\|< 2 \kappa \lambda$, 
\begin{align*}
1_{\ttE} \|\widetilde{\mathbf{\mathcal{V}}}\widetilde{\mathbf{\mathcal{W}}}\tp\|_{F}^{2}\, & =1_{\ttE}\Tr \left( \widetilde{\mathbf{\mathcal{V}}}\widetilde{\mathbf{\mathcal{W}}}\tp \widetilde{\mathbf{\mathcal{W}}}\widetilde{\mathbf{\mathcal{V}}}\tp \right)    \\
 \, & \leqslant 1_{\ttE} \| \widetilde{\mathbf{\mathcal{V}}}\tp \widetilde{\mathbf{\mathcal{V}}}\| \Tr \left(\widetilde{\mathbf{\mathcal{W}}}\tp \widetilde{\mathbf{\mathcal{W}}}  \right) \\
 \,  & \leqslant 2\kappa \lambda \Tr \left(\widetilde{\mathbf{\mathcal{W}}}\tp \widetilde{\mathbf{\mathcal{W}}}  \right), 
\end{align*}
where the second inequality is due to the fact that $\Tr(\vA\vL)\leqslant \|\vA\| \Tr(\vL)$ holds for any positive semi-definite matrices $\vL$ and $\vA$. 
It then follows that $\bbE 1_{\ttE} \|\widetilde{\mathbf{\mathcal{V}}}\widetilde{\mathbf{\mathcal{W}}}\tp\|_{F}^{2} \lesssim \frac{\lambda p H}{n}$. Combining this with Equation~\eqref{eq:ld-ww-general}, we have 
\begin{align}
 \nonumber \bbE 1_{\ttE} \|\widehat{\widetilde{\bLambda}}_{H}-\widehat{\widetilde{\bLambda}}_{V}\|_{F}^{2} & \lesssim  \frac{\lambda p H}{n} +  \frac{ p H \left( p+H \right) }{n^2}\\
& \lesssim \frac{\lambda p H}{n}, \label{eq:ld-Lambda-HV}
\end{align}
because $p\vee H / \left( n\lambda \right)$ is small.

Note that $\widetilde{\mathbf{\mathcal{V}}}= \widetilde{\vB} \mathbf{\mathcal{Z}}$ and $ \mathbf{\mathcal{Z}}=\widetilde{\vB}\tp \widetilde{\mathbf{\mathcal{V}}}$, we have $ \Tr \left( \widetilde{\mathbf{\mathcal{V}}} \widetilde{\mathbf{\mathcal{V}}}\tp   \right) = \Tr \left( {\mathbf{\mathcal{Z}}} {\mathbf{\mathcal{Z}}}\tp   \right) \leqslant d \| \widetilde{\mathbf{\mathcal{V}}} \widetilde{\mathbf{\mathcal{V}}}\tp \| $ because $ \widetilde{\vB}\tp \widetilde{\vB}   =\vI_d$. Therefore, 
\begin{align*}
\|\widehat{\widetilde{\bLambda}}_{V}\|_{F}^{2} \, & = \Tr \left(\widetilde{\mathbf{\mathcal{V}}}\widetilde{\mathbf{\mathcal{V}}}\tp \widetilde{\mathbf{\mathcal{V}}} \widetilde{\mathbf{\mathcal{V}}}\tp  \right) \\  
\, & \leqslant    \|  \widetilde{\mathbf{\mathcal{V}}} \widetilde{\mathbf{\mathcal{V}}}\tp \| \Tr \left( \widetilde{\mathbf{\mathcal{V}}} \widetilde{\mathbf{\mathcal{V}}}\tp   \right) \\
\, & \leqslant d    \|  \widetilde{\mathbf{\mathcal{V}}} \widetilde{\mathbf{\mathcal{V}}}\tp \| ^{2},
\end{align*}
which implies that $1_{\ttE}\|\widehat{\widetilde{\bLambda}}_{V}\|_{F}^{2} \leqslant 4\kappa^2 \lambda^2 d$.

In view of inequalities~\eqref{eq:ld-main-delta-general},\eqref{eq:ld-Lambda-HV}, and  \eqref{eq:ld-main-general}, we have
\begin{align}
  \bbE \left( 1_{\ttE} \| (\widehat{\bS}^{1/2}\widehat{\vB})^{\otimes} - \widetilde{\vB}^{\otimes} \|^{2}_{F} \right)
  & \lesssim
\frac{1}{\lambda^2} \bbE \left( 1_{\ttE} \|\Delta\|_{F}^{2} \right)\nonumber \\
&  \lesssim \frac{1}{\lambda^2} \left(  \frac{\lambda p H}{n } + \lambda^2 d \frac{p+\log (n \lambda) }{n}  \right)\nonumber  \\
& \lesssim  \frac{ H [ p + \log(n\lambda) ] }{n \lambda}.\label{eq:ld-main-F}
\end{align}

\bigskip 

Combining Equations~\eqref{eq:ub-ld-dec-general}, \eqref{eq:ld-main-F} and \eqref{eq:ub-ld-intermed-general}, one has
\begin{equation}\label{eq:ub-ld-B-general}
\bbE\left( \one_{\ttE} \|   \widehat{\vB}^{\otimes} -  \vB^{\otimes}\|_{F}^{2}  \right)\lesssim \frac{H [ p+ \log( n\lambda) ] }{n \lambda}. 
\end{equation}

\textbf{Step 3: Synthesis. }

In addition, one has $\|\widehat{\vB}^{\otimes} -  \vB^{\otimes}\|_{F}^{2} \leqslant 2 d $ and $\bbP(\ttE^{c})\lesssim \frac{1}{n\lambda}$, one has $ \bbE \left( \one_{\ttE^{c}} \|\widehat{\vB}^{\otimes} -  \vB^{\otimes}\|_{F}^{2} \right)\lesssim \frac{d}{n\lambda}$. 
We conclude that
$$
\bbE\|\widehat{\vB}^{\otimes} -  \vB^{\otimes}\|_{F}^{2}\lesssim \frac{H [ p+ \log( n\lambda) ] }{n \lambda}. 
$$ \qed

\begin{proof}[Proof of Lemma~\ref{lem:Sigma-control}]
Note that $\widehat{\widetilde{\bS}}:=\bS^{-1/2}\widehat{\bS}\bS^{-1/2}$ is the empirical covariance matrix of $\widetilde{\vX}_{i}$. 
	By Lemma~\ref{random:nonasymptotic}, with probability at least $1-2\exp(-n t^2/2)$, the eigenvalues of $
\widehat{\widetilde{\bS}}$ lie between $(1-\sqrt{p/n}-t)^{2}$ and $(1+\sqrt{p/n}+t)^{2}$. 

Fix $t \asymp \sqrt{2\log(n\lambda)/n}$ and assume the event happens. 

If $\frac{p+\log(n\lambda)}{n}$ is sufficiently small ($<1$), the eigenvalues of $\widehat{\widetilde{\bS}}$ lie between $(1/2,2)$. In this case, $\|\widehat{\bS}^{-1}\|\leqslant \|  \bS^{1/2}\widehat{\bS}^{-1} \bS^{1/2}\| M < 2M$ because $\|\bS^{-1}\|<M$ by assumption. We can also easily see that $\|\widehat{\vI}\|\leqslant \|\widehat{\bS}^{-1}\|^{1/2}\|\bS\|^{1/2}<\sqrt{2}M$. 

Furthermore, $\widehat{\vI}-\vI_p=\widehat{\bS}^{-1/2}\left( \bS^{1/2}-\widehat{\bS}^{1/2} \right)$. 
By \citet[Lemma~2.2]{schmitt1992perturbation} and the fact that $\bS\gtrsim M^{-1} \vI$ and $\widehat{\bS}\gtrsim (2M)^{-1} \vI$, we have 
\begin{align*}
	\| \bS^{1/2}-\widehat{\bS}^{1/2} \| &\, < 3\sqrt{M} \| \bS-\widehat{\bS}\| \\
		&\, = 3\sqrt{M}  \|\bS^{1/2} ( \vI-\widehat{\widetilde{\bS}} ) \bS^{1/2}\|\\
	&\, \leqslant 3\sqrt{M} M \|\vI-\widehat{\widetilde{\bS}}\|,
\end{align*}
where the last inequality is because $\|\bS\|<M$ by assumption. 
Since $\|\vI-\widehat{\widetilde{\bS}}\|=\max_{1\leqslant i\leqslant p}|1- \sigma_{i}(\widehat{\widetilde{\bS}}) | \lesssim \sqrt{ \frac{p+\log(n\lambda)}{n}}$,  we have $\|\widehat{\vI}-\vI_p\|<5 M^2 \|\vI-\widehat{\widetilde{\bS}}\| \lesssim \sqrt{\frac{p+\log(n\lambda)}{n}}$. 
\end{proof}

\subsection{Proof of Theorem \ref{thm:risk:sparse:upper:d}}\label{app:high dimension upper bound general cov}

\subsubsection*{Preliminaries}

Let $\mathcal{L}=\{T\subset [p]: S\subset T, |T|\leqslant 2s\}$. 
For any $T\in \mathcal{L}$, we define some notations. 

Let $\vJ_{T}$ be the matrix formed by the rows of  $\vI_p$ in $T$. 
Let $\bS_{TT}$ be the sub-matrix of $\bS$ with row indices and column indices both equal to $T$, i.e., $\vJ_{T} \bS \vJ_{T}\tp$. 
Let $\bS_{TT}^{1/2}$ be a square root of $\bS_{TT}$. Note that it is different from the sub-matrix of a square root of $\bS$. 
Let $\widetilde{\vB}_{(T)}=\bS_{TT}^{1/2}\vJ_{T}\vB$. Then 
$\widetilde{\vB}_{(T)}\tp \widetilde{\vB}_{(T)}=\vB\tp \vJ_{T}\tp  \bS_{TT}\vJ_{T}\vB=\vB\tp\vJ_T\tp\vJ_T\bS\vJ_T\tp\vJ_T\vB=\vB\tp\bS\vB=\vI_{d}$ because $S\subset T$. 
Let $\widetilde{\vB}_{(T),\perp}$ is a $|T|\times (|T|-d)$ orthogonal matrix such that $\widetilde{\vB}_{(T)}\tp \widetilde{\vB}_{(T),\perp}=0$. 

\medskip

For a pair of $(\vX, Y)$ that is sampled from the distribution $\mc M \in \mathfrak{M}_{s}\left( p,d,\lambda\right)$, we introduce the following notations. 

Let 
 $\widetilde{\vX}_{(T)}=\bS_{TT}^{-1/2}\vJ_{T}\vX$. Then $\widetilde{\vX}_{(T)}\sim N(0, \vI_{|T|} )$ because $\bS_{TT}^{-1/2}\vJ_{T}\bS \vJ_{T}\tp \bS_{TT}^{-1/2}=\vI_{|T|}$. 

Let $\vZ=\vB\tp  \vX$.Note that $\vZ=\vB\tp \vJ_{T}\tp \vJ_{T}  \vX=\widetilde{\vB}_{(T)}\tp  \widetilde{\vX}_{(T)}$ because $S\subset T$. 
Let  $\widetilde{\vE}_{(T)}=\widetilde{\vB}_{(T),\perp}\tp  \widetilde{\vX}_{(T)}$. 
Since $\widetilde{\vX}_{(T)}\sim N(0, \vI_{|T|} )$, one has $\vZ\sim N(0, \bs{I}_{d})$ and $\widetilde{\vE}_{(T)}\sim N(0, \bs{I}_{|T|-d})$. Furthermore, $\vZ\indp \widetilde{\vE}_{(T)}$ and 
\begin{equation*}
\begin{aligned}
\widetilde{\vX}_{(T)}&=\widetilde{\vB}_{(T)}\widetilde{\vB}_{(T)}\tp\widetilde{\vX}_{(T)}+\widetilde{\vB}_{(T),\perp}\widetilde{\vB}_{(T),\perp}\tp\widetilde{\vX}_{(T)}
= \widetilde{\vB}_{(T)}\vZ+\widetilde{\vB}_{(T),\perp}\widetilde{\vE}_{(T)}. 
\end{aligned}
\end{equation*}
Let $\widetilde{\vV}_{(T)}=\widetilde{\vB}_{(T)}\vZ$ and $\widetilde{\vW}_{(T)}=\widetilde{\vB}_{(T),\perp}\widetilde{\vE}_{(T)}$. Then $\widetilde{\vV}_{(T)}\tp \widetilde{\vW}_{(T)}=0$.
Let  $\bLambda_{z}=\Cov \left( \bbE[ \vZ \mid Y] \right)$ and let $\widetilde{\bLambda}_{(T)}=
\widetilde{\vB}_{(T)}  \bLambda_{z} \widetilde{\vB}_{(T)}\tp$. Then $\widetilde{\bLambda}_{(T)}=\Cov \left( \bbE[ \widetilde{\vX}_{(T)} \mid Y] \right)$.

\medskip
We next introduce the notation for the sliced samples. For example, we define $\overline{\widetilde{\vX}}_{(T),h,\cdot}$  similarly to the definition of $\overline{\vX}_{h,\cdot}$ in \eqref{eq:sliced-xy} and  $\widetilde{\mathbf{\mathcal{X}}}_{(T)}=\frac{1}{\sqrt{H}}
\left[\overline{\widetilde{\vX}}_{(T),1,\cdot} ~,~ \overline{\widetilde{\vX}}_{(T),2,\cdot},...,~ 
\overline{\widetilde{\vX}}_{(T),H,\cdot}\right]$. 

 Similarly, we define $\overline{\widetilde{\vV}}_{(T),h,\cdot}$, $\overline{\widetilde{\vW}}_{(T),h,\cdot}$, $\overline{\widetilde{\vE}}_{(T),h,\cdot}$ and $\widetilde{\mathbf{\mathcal{V}}}_{(T)}$, $\widetilde{\mathbf{\mathcal{W}}}_{(T)}$, $\widetilde{\mathbf{\mathcal{E}}}_{(T)}$.  
Then $\widetilde{\mathbf{\mathcal{V}}}_{(T)}= \widetilde{\vB}_{(T)} \mathbf{\mathcal{Z}}$ and $\widetilde{\mathcal{W}}_{(T)}=\widetilde{\vB}_{(T),\perp}\widetilde{\mathcal{E}}_{(T)}$. We see that $\widetilde{\mathbf{\mathcal{V}}}_{(T)}\tp \widetilde{\mathcal{W}}_{(T)}=0$.
 Since $\widetilde{\vE}_{(T)} \sim N(0,\vI_{|T|-d})$ and is independent of $Y$, we know that the entries $\widetilde{\mathcal{E}}_{(T),i,j}$ of $\widetilde{\mathcal{E}}_{(T)}$ are $i.i.d.$ samples of $N(0,\frac{1}{n})$.

Define 
$\widehat{\widetilde{\bLambda}}_{(T)}=\widetilde{\mathbf{\mathcal{X}}}_{(T)}\widetilde{\mathbf{\mathcal{X}}}_{(T)}\tp$, $\widehat{\widetilde{\bLambda}}_{V,(T)}=\widetilde{\mathbf{\mathcal{V}}}_{(T)} \widetilde{\mathbf{\mathcal{V}}}_{(T)} \tp $ and  $\widehat{\bLambda}_{z}=\mathbf{\mathcal{Z}}\mathbf{\mathcal{Z}}\tp $. 
Recall the definition of  the SIR estimate for $\Cov(\bbE(\vX\mid Y))$: $\widehat{\bLambda}_{H}=\mathbf{\mathcal{X}}\mathbf{\mathcal{X}}\tp$.  
 Then $\vJ_{T}\widehat{\bLambda}_H\vJ_{T}\tp=\bS_{TT}^{1/2}\widehat{\widetilde{\bLambda}}_{(T)}\bS_{TT}^{1/2}$ and $\widehat{\widetilde{\bLambda}}_{V,(T)}=\widetilde{\vB}_{(T)}  \widehat{\bLambda}_{z}\widetilde{\vB}_{(T)} \tp$.

We have the following decomposition
\begin{equation}
\begin{aligned}
\widehat{\widetilde{\bLambda}}_{(T)}&=\widetilde{\mathbf{\mathcal{V}}}_{(T)}\widetilde{\mathbf{\mathcal{V}}}_{(T)}\tp 
+ \widetilde{\mathbf{\mathcal{V}}}_{(T)}\widetilde{\mathbf{\mathcal{W}}}_{(T)}\tp
+\widetilde{\mathbf{\mathcal{W}}}_{(T)}\widetilde{\mathbf{\mathcal{V}}}_{(T)}\tp
+\widetilde{\mathbf{\mathcal{W}}}_{(T)}\widetilde{\mathbf{\mathcal{W}}}_{(T)}\tp
\end{aligned}
\end{equation}
Since we have randomly divided the samples into two equal sets of samples, we have the corresponding statistics 
 $\bLambda_H^{(i)},{\widetilde{\bLambda}}_{(T)}^{(i)}$, ${\widetilde{\bLambda}}_{V,(T)}^{(i)}$, $\widetilde{\mathbf{\mathcal{V}}}_{(T)}^{(i)},\widetilde{\mathcal{E}}_{(T)}^{(i)},\widetilde{\mathbf{\mathcal{W}}}_{(T)}^{(i)}$, and $\widehat{\bLambda}_{z}^{(i)}$ for the $i$th set of samples ($i=1,2$) similar to the definition of $\wh\bLambda_H,\wh{\widetilde{\bLambda}}_{(T)}$, $\wh{\widetilde{\bLambda}}_{V,(T)},\widetilde{\mathbf{\mathcal{V}}}_{(T)},\widetilde{\mathcal{E}}_{(T)},\widetilde{\mathbf{\mathcal{W}}}_{(T)}$, and $\wh{\bLambda}_{z}$ respectively. 
\medskip 

Finally, we  introduce an ``oracle estimator'', where the word \textit{oracle} suggests this is  an estimator only if we know $S$. 
\begin{equation}\label{estimator:bs-oracle-general}
\begin{aligned}
&\widehat{\vB}_{O}:= 
\arg\max_{\vB}\Tr(\vB\tp \bLambda_{H}^{(1)}\vB)\\
&\mbox{ s.t. } ~~~ \vB\tp \widehat{\bS}^{(1)} \vB = \vI_d,~~ \supp(\vB) = S. 
\end{aligned}
\end{equation}

\subsubsection*{Main part of the proof}
For $i=1,2$, we define two events:
\begin{itemize}
    \item [(i)]$\wt\ttE^{(i)}_2:=\left\{\max_{T\in\mc L}\left(\|{\widetilde{\bLambda}}_{V,(T)}^{(i)}-\wt{\bLambda}_{(T)}\|\right)\leqslant\frac{2\kappa\lambda}{3\nu}\right\}$, 
    \item [(ii)]$\widetilde{\ttE}_{1}^{(1)}=\Big\{~ \max_{T\in \mathcal{L}} \left(  \| \widetilde{\mathbf{\mathcal{W}}}_{(T)}^{(1)}\widetilde{\mathbf{\mathcal{W}}}_{(T)}^{(1),\top} \| \right) \leqslant 6 \frac{ 2 s \vee H + s \log(ep/s) +\log(n\lambda)}{n}~\Big\}$. 
\end{itemize}
Furthermore, define $\widetilde{\ttE}=\widetilde{\ttE}_{1}^{(1)} \cap\widetilde{\ttE}_{2}^{(1)}\cap\widetilde{\ttE}^{(2)}_{2}$.

We apply Lemma~\ref{random:nonasymptotic} to $\sqrt{n} \cdot \widetilde{\mathcal{E}}_{(T)}^{(i)}$ to conclude that 
\begin{align*}
\bbP\left(\|   \widetilde{\mathbf{\mathcal{W}}}_{(T)}^{(i)}\widetilde{\mathbf{\mathcal{W}}}_{(T)}^{(i),\top} \|> 6\frac{\max( |T|, H) +t }{n} \right)\leqslant  2 \exp\left(- t\right), 
\end{align*}
which implies that 
\begin{align*}
\bbP\left( \exists T \in \mathcal{L}, \|   \widetilde{\mathbf{\mathcal{W}}}_{(T)}^{(i)}\widetilde{\mathbf{\mathcal{W}}}_{(T)}^{(i),\top} \|> 6\frac{\max( |T|, H) +t }{n} \right)\leqslant  2 |\mathcal{L}|  \exp\left( - t\right).
\end{align*}

In view of Corollary~\ref{cor:minimax:key}, one has the following analogy. 
\begin{corollary}\label{cor:minimax:key-general-T}
For $\nu \in (\kappa, 2\kappa]$, we can find constants $C$ and $\widetilde{C}$ , such that if \\ 
\[\label{eq:cor-key-require-1:unknown cov high}
\kappa^2 H^2\left( \log(nH) + \log \kappa + d \right)<C n \lambda
\] and $\kappa \left( 2s\vee H + s \log(ep/s) + \log (n\lambda) \right)< C n\lambda$,
then 
$\bbP\left( \widetilde{\ttE}^{c}\right) \leqslant \frac{ \widetilde{C} }{n\lambda}$
and on the event $ \widetilde{\ttE}$, the followings hold
\begin{itemize}
\item[$a)$] 
$ \frac{1}{3}\lambda\leqslant \lambda_{d}  (\wh{\bLambda}_{z}^{(i)}) \leqslant \lambda_{1}( {\wh\bLambda}_{z}^{(i)})\leqslant 2\kappa \lambda$.
\item[$b)$]
$\|{\widetilde{\bLambda}}_{(T)}^{(i)}-{\widetilde{\bLambda}}_{V,(T)}^{(i)}\|
\leqslant \lambda \sqrt{18 \kappa \frac{2s \vee H + s\log(ep/s) +\log(n\lambda )}{n \lambda}}< \frac{1}{4}\lambda$, for any $T\in \mathcal{L}$.
\item[$c)$] 
$\lambda_{d+1}({\widetilde{\bLambda}}_{(T)}^{(i)}) < \frac{1}{4}\lambda$, for any $T\in \mathcal{L}$.
\item[$d)$]
$\|{\widetilde{\bLambda}}_{(T)}^{(i)} \|\leqslant 3\kappa \lambda$.
\end{itemize}
\textbf{Proof of $d)$:}
\begin{proof}
Since ${\widetilde{\bLambda}}_{V,(T)}^{(i)}=\widetilde{\vB}_{(T)}^{(i)}  {\wh\bLambda}_{z}^{(i)}\widetilde{\vB}_{(T)} \tp$ and $\widetilde{\vB}_{(T)} \tp \widetilde{\vB}_{(T)} =\vI_{d}$, one  has:  
$$
\lambda_1({\widetilde{\bLambda}}_{V,(T)}^{(i)})=\lambda_1(\widetilde{\vB}_{(T)}  {\wh\bLambda}_{z}^{(i)}\widetilde{\vB}_{(T)} \tp)=\lambda_1(\widetilde{\vB}_{(T)} \tp\widetilde{\vB}_{(T)}  {\wh\bLambda}_{z}^{(i)})=\lambda_1({\wh\bLambda}_{z}^{(i)}).
$$
Combining a) and b) leads to that
$$
\|{\widetilde{\bLambda}}_{(T)}^{(i)} \|\leqslant \|{\widetilde{\bLambda}}_{(T)}^{(i)}-{\widetilde{\bLambda}}_{V,(T)}^{(i)}\|+\lambda_{1}( {\wh\bLambda}_{z}^{(i)})\leqslant 3\kappa\lambda.
$$
\end{proof}	
\end{corollary}
Let  $\widehat{\bS}_{TT}=\vJ_{T} \wh\bS^{(1)} \vJ_{T}\tp$, 
 and $\widehat{\vI}_{T}= \widehat{\bS}_{TT}^{-1/2}\bS_{TT}^{1/2}$. 
\begin{lemma}\label{lem:Sigma-control-T}
There exist constants $C$ and $\widetilde{C}$, such that if $s\log(ep/s)+\log(n\lambda)< C n$, then it holds with probability at least $1-\widetilde{C}/n\lambda$ that for all $T\in \mathcal{L}$, 
$\|\widehat{\vI}_T\|^{2}<2 M^2$, 	$ \|\widehat{\bS}_{TT}^{-1/2}\|^{2}<2M$ and $\|\widehat{\vI}_{T}-\vI_{|T|} \|^{2}< \widetilde{C}
	\frac{s\log(ep/s)+ \log( n\lambda) }{n}$. 
\end{lemma}
\begin{proof}
The proof follows the same argument in Lemma~\ref{lem:Sigma-control} by choosing $t \asymp \sqrt{2[ s\log(ep/s)+\log(n\lambda)]/n}$ and is omitted. 
\end{proof}

Let $\ttE$ be the intersection of $\tilde{\ttE}$ and the events in Lemma~\ref{lem:Sigma-control-T}. 
On the event $\ttE$, the results stated in Corollary~\ref{cor:minimax:key-general-T} and Lemma~\ref{lem:Sigma-control-T} uniformly hold for $T\in \mathcal{L}$, in particular for $\supp(\wh\vB)\cup S$. 

In the following, we set $T$ to be the random element $\supp(\wh\vB)\cup S$. 
Furthermore, we abbreviate $\widetilde{\vB}_{(T)}$ by $\vF$, i.e. 
$\vF:=\bS_{TT}^{1/2}\vJ_{T} \vB$. Similarly, we define  $\widehat{\vF}_{O}=\widehat{\bS}_{TT}^{1/2} \vJ_{T}\widehat{\vB}_{O}$ and $\widehat{\vF}=\widehat{\bS}_{TT}^{1/2} \vJ_{T}\widehat{\vB}$. Then $\vF$, $\widehat{\vF}_{O}$ and $\widehat{\vF}$ are all in $\mathbb{O}(|T|,d)$.

Let $\widehat{\vF}_{O}\tp \vF =\vU_{1}\Delta \vU_{2}\tp $ be the singular value decomposition of $\widehat{\vF}_{O}\tp \vF$ such that $\vU_{i}\in \mathbb{O}(d,d)$ and $\Delta$ is a $d\times d$ diagonal matrix with non-negative entries. 
Let $\vM:= \vU_{2}\tp {\wh\bLambda}_{z}^{(2)}\vU_{2}$.

By the definition of $\wh\vB$ and $\bLambda^{(2)}_{H}$ (the SIR estimator of  $\bLambda$ based on the second set of samples), one has
$0\geqslant\langle {\bLambda}_{H}^{(2)},\widehat{\vB}_{O}\widehat{\vB}_{O}\tp -\widehat{\vB}\widehat{\vB}\tp  \rangle =\langle {\bLambda}_{H}^{(2)},\vJ_{T}\tp \vJ_{T}(\widehat{\vB}_{O}\widehat{\vB}_{O}\tp -\widehat{\vB}\widehat{\vB}\tp)\vJ_{T}\tp\vJ_{T}  \rangle=
 \langle  \vJ_{T}{\bLambda}_{H}^{(2)}\vJ_{T}\tp ,\vJ_{T}(\widehat{\vB}_{O}\widehat{\vB}_{O}\tp -\widehat{\vB}\widehat{\vB}\tp  )\vJ_{T}\tp\rangle$ where the first equality comes from the fact that  $\supp(\widehat{\vB})\cup\supp(\widehat{\vB}_{O})=T$.

Applying the Lemma \ref{cor:elemetary:temp} with the positive definite matrix $\vU_{1}\vM\vU_{1}\tp $, one has 
\begin{align}
\frac{\lambda_{d}(  {\wh\bLambda}_{z}^{(2)} )}{2}\|\widehat{\vF}\widehat{\vF}\tp -\widehat{\vF}_{O}\widehat{\vF}_{O}\tp \|^{2}_{F} 
\leqslant & \langle  \widehat{ \vF}_{O} \vU_{1}\vM\vU_{1}\tp \widehat{\vF}_{O}\tp ,\widehat{\vF}_{O}\widehat{\vF}_{O}\tp -\widehat{\vF}\widehat{\vF}\tp  \rangle  \nonumber \\
= & \langle \widehat{\bS}_{TT}^{1/2}  \widehat{ \vF}_{O} \vU_{1}\vM\vU_{1}\tp \widehat{\vF}_{O}\tp  \widehat{\bS}_{TT}^{1/2} , \vJ_{T}( \widehat{\vB}_{O}\widehat{\vB}_{O}\tp  -\widehat{\vB}\widehat{\vB}\tp)  \vJ_{T}\tp   \rangle  \nonumber \\
\leqslant & \langle \widehat{\bS}_{TT}^{1/2}  \widehat{ \vF}_{O} \vU_{1}\vM\vU_{1}\tp \widehat{\vF}_{O}\tp  \widehat{\bS}_{TT}^{1/2}- \vJ_{T}{\bLambda}_{H}^{(2)} \vJ_{T}\tp, \vJ_{T}( \widehat{\vB}_{O}\widehat{\vB}_{O}\tp  -\widehat{\vB}\widehat{\vB}\tp)  \vJ_{T}\tp   \rangle  \nonumber \\
 :=&   \mathrm{I} + \mathrm{II} \label{eq:ub-hd-general-main}
\end{align}
where
\begin{align}
 \mathrm{I} \nonumber=& \langle \widehat{\bS}_{TT}^{1/2}  \wh{ \vF}_{O} \vU_{1}\vM\vU_{1}\tp \widehat{\vF}_{O}\tp  \widehat{\bS}_{TT}^{1/2}-  \widehat{\bS}_{TT}^{1/2}  {\widetilde{\bLambda}}_{V,(T)}^{(2)}  \widehat{\bS}_{TT}^{1/2}, \vJ_{T}( \widehat{\vB}_{O}\widehat{\vB}_{O}\tp  -\widehat{\vB}\widehat{\vB}\tp)  \vJ_{T}\tp   \rangle  
&\\
 \mathrm{II} \nonumber =&\langle  \widehat{\bS}_{TT}^{1/2}  {\widetilde{\bLambda}}_{V,(T)}^{(2)}\widehat{\bS}_{TT}^{1/2} -\bS_{TT}^{1/2}{\widetilde{\bLambda}}_{(T)}^{(2)}\bS_{TT}^{1/2}, \vJ_{T}( \widehat{\vB}_{O}\widehat{\vB}_{O}\tp  -\widehat{\vB}\widehat{\vB}\tp)  \vJ_{T}\tp   \rangle.  &
\end{align}
The last inequality holds because
$\vJ_{T}{\bLambda}_{H}^{(2)} \vJ_{T}\tp=\bS_{TT}^{1/2}{\widetilde{\bLambda}}_{(T)}^{(2)}\bS_{TT}^{1/2}$.

\noindent\textbf{For I:}

We first rewrite $\mathrm{I}$:
\begin{align*}
 \mathrm{I} =& \langle \widehat{\bS}_{TT}^{1/2}  \widehat{ \vF}_{O} \vU_{1}\vM\vU_{1}\tp \widehat{\vF}_{O}\tp  \widehat{\bS}_{TT}^{1/2}-  \widehat{\bS}_{TT}^{1/2}  {\widetilde{\bLambda}}_{V,(T)}^{(2)} \widehat{\bS}_{TT}^{1/2}, \vJ_{T}( \widehat{\vB}_{O}\widehat{\vB}_{O}\tp  -\widehat{\vB}\widehat{\vB}\tp)  \vJ_{T}\tp   \rangle  \\
=   &  \langle   \widehat{ \vF}_{O} \vU_{1}\vM\vU_{1}\tp \widehat{\vF}_{O}\tp    -{\widetilde{\bLambda}}_{V,(T)}^{(2)} ,\widehat{\vF}_{O}\widehat{\vF}_{O}\tp -\widehat{\vF}\widehat{\vF}\tp  \rangle.
\end{align*}
Note that $\vM=\vU_{2}\tp  \widehat{\bLambda}_{z}^{(2)}\vU_{2}$, so on the event $\ttE$, the eigenvalues of $\vM$ are in $(\frac{1}{3}\lambda, 2\kappa \lambda)$. 
Following the same proof of \eqref{inlince:ttt}, one has $1_{\ttE} | \mathrm{I}|\leqslant C\lambda \|\widehat{\vF}_{O}\widehat{\vF}\tp _{O}-\vF  \vF\tp\|_{F}   \|\widehat{\vF}_{O}\widehat{\vF}_{O}\tp -\widehat{\vF}\widehat{\vF}\tp \|_{F}$.

\textbf{For II:}  
Recall that $\widehat{\vI}_{T}=\widehat{\bS}_{TT}^{-1/2}\bS_{TT}^{1/2}$. 
\begin{align*}
\mathrm{II}= &\,   \langle  \widehat{\bS}_{TT}^{1/2}  {\widetilde{\bLambda}}_{V,(T)}^{(2)}\widehat{\bS}_{TT}^{1/2} -\bS_{TT}^{1/2}{\widetilde{\bLambda}}_{(T)}^{(2)}\bS_{TT}^{1/2}, \vJ_{T}( \widehat{\vB}_{O}\widehat{\vB}_{O}\tp  -\widehat{\vB}\widehat{\vB}\tp)  \vJ_{T}\tp   \rangle \\
=  &\,   \langle {\widetilde{\bLambda}}_{V,(T)}^{(2)} -\widehat{\bS}_{TT}^{-1/2}\bS_{TT}^{1/2}{\widetilde{\bLambda}}_{(T)}^{(2)}\bS_{TT}^{1/2}\widehat{\bS}_{TT}^{-1/2},  \widehat{\bS}_{TT}^{1/2}  \vJ_{T}( \widehat{\vB}_{O}\widehat{\vB}_{O}\tp  -\widehat{\vB}\widehat{\vB}\tp)  \vJ_{T}\tp \widehat{\bS}_{TT}^{1/2}     \rangle \\
=  &\,   \langle {\widetilde{\bLambda}}_{V,(T)}^{(2)} -\widehat{\vI}_{T}{\widetilde{\bLambda}}_{(T)}^{(2)}\widehat{\vI}_{T}\tp,   \widehat{\vF}_{O}\widehat{\vF}_{O}\tp  -\widehat{\vF}\widehat{\vF}\tp  \rangle.
\end{align*}

We can bound the last expression using the next lemma, whose proof is deferred to the end of this section. 

\begin{lemma}\label{lem:decomposition upper spare general bS}
If $\lambda\leqs\varpi_d\leqs\frac1d$ and $n\lambda\leqs e^s$, then on the event $\ttE$, we have:
  \begin{align}\label{eq:decomposition upper spare general bS}
 \langle {\widetilde{\bLambda}}_{V,(T)}^{(2)} -\widehat{\vI}_{T}{\widetilde{\bLambda}}_{(T)}^{(2)}\widehat{\vI}_{T}\tp,   \widehat{\vF}_{O}\widehat{\vF}_{O}\tp  -\widehat{\vF}\widehat{\vF}\tp  \rangle\lesssim&\lambda\epsilon_n\|\widehat{\vF}_{O}\widehat{\vF}_{O}\tp  -\widehat{\vF}\widehat{\vF}\tp\|_F.
\end{align}  
\end{lemma}

On the event $\ttE$, $ \lambda_{d}  ({\wh\bLambda}_{z}^{(2)})  \geqslant\lambda/3$. Equation~\eqref{eq:ub-hd-general-main} leads to 
\begin{align*}
  \|\widehat{\vF}\widehat{\vF}\tp -\widehat{\vF}_{O}\widehat{\vF}_{O}\tp \|^{2}_{F} 
\lesssim  \left(  \|\widehat{\vF}_{O}\widehat{\vF}\tp _{O}-\vF  \vF\tp\|_{F}+\sqrt{\epsilon_{n}^{2}}  \right)  \|\widehat{\vF}\widehat{\vF}\tp -\widehat{\vF}_{O}\widehat{\vF}_{O}\tp \|_{F}, 
\end{align*}
which yields 
\begin{equation}
      \|\widehat{\vF}\widehat{\vF}\tp -\widehat{\vF}_{O}\widehat{\vF}_{O}\tp \|_{F} \lesssim  \|\widehat{\vF}_{O}\widehat{\vF}\tp _{O}-\vF  \vF\tp\|_{F}+\sqrt{\epsilon_{n}^{2}}.\label{eq:ub-hd-general-inline1}
  \end{equation}
  By triangle inequality, 
\begin{align*}
\|\widehat{\vF}\widehat{\vF}\tp -\vF\vF \|_{F} \leqslant  \|\widehat{\vF}\widehat{\vF}\tp -\widehat{\vF}_{O}\widehat{\vF}_{O}\tp \|_{F}+ \|\widehat{\vF}_{O}\widehat{\vF}\tp _{O}-\vF  \vF\tp\|_{F}, 
\end{align*}
and thus on the event $\ttE$, 
\begin{align*}
  \|\widehat{\vF}\widehat{\vF}\tp -\vF\vF \|_{F}^{2}\lesssim \|\widehat{\vF}_{O}\widehat{\vF}\tp _{O}-\vF  \vF\tp\|_{F}^{2} + \epsilon_{n}^{2}.
  \end{align*}
Following the proof of Equation \eqref{eq:ub-ld-B-general} in Section~\ref{app:ub-ld-general}, there exists a set $\ttE_{O}$ such that $\bbP(\ttE_{O}^{c})\lesssim \frac{1}{n\lambda}$ and 
\begin{align}
  \bbE \left( \one_{\ttE_{O}}  \|\widehat{\vF}_{O}\widehat{\vF}\tp _{O}-\vF  \vF\tp\|_{F}^{2} \right)
& \lesssim  \epsilon_{n}^{2},\label{eq:ub-hd-general-Eo-F}\\
  \bbE \left( \one_{\ttE_{O}}  \|\widehat{\vB}_{O}\widehat{\vB}\tp _{O}-\vB  \vB\tp\|_{F}^{2} \right)
& \lesssim \epsilon_{n}^{2}. \label{eq:ub-hd-general-Eo-B}
\end{align}

Therefore
\begin{align*}
\bbE \left( \one_{\ttE\cap \ttE_{O}}  \|\widehat{\vF}\widehat{\vF}\tp -\vF  \vF\tp\|_{F}^{2} \right)
& \lesssim \epsilon_{n}^{2}. 
\end{align*}

On the event $\ttE$, by Lemma~\ref{lem:Sigma-control-T} and   $\|\widehat{\bS}_{TT}^{-1/2}\|^{2}< 2M$, 
one has
\begin{align}
& \|\widehat{\vB}\widehat{\vB}\tp - \widehat{\vB}_{O}\widehat{\vB}_{O}\tp\|_{F}  \nonumber  \\
 = \,  &  \|\vJ_{T} \left(\widehat{\vB}\widehat{\vB}\tp-\widehat{\vB}_{O}\widehat{\vB}_{O}\tp\right) \vJ_{T}\tp\|_{F}  \nonumber  \\
  \leqslant  \,  & \|\widehat{\bS}_{TT}^{-1/2}\|^{2}  \| \widehat{\bS}_{TT}^{1/2} \vJ_{T} \left( \widehat{\vB}\widehat{\vB}\tp - \widehat{\vB}_{O}\widehat{\vB}_{O}\tp\right) \vJ_{T}\tp\widehat{\bS}_{TT}^{1/2} \|_{F} \nonumber  \\
= \,  & \|\widehat{\bS}_{TT}^{-1/2}\|^{2}  \|   \widehat{\vF}  \widehat{\vF}\tp -  \widehat{\vF}_{O}  \widehat{\vF}_{O}\tp \|_{F} \nonumber \\
<\, & 2M  \left(  \|\widehat{\vF}_{O}\widehat{\vF}\tp _{O}-\vF  \vF\tp\|_{F}+\sqrt{\epsilon_{n}^{2}} \right), \label{eq:ub-hd-general-BtoF}
\end{align}
where the first equation is due to the definition of $T$, the first inequality is due to Lemma~\ref{lem:elementary:trivial2} and the last inequality is because  Equation~\eqref{eq:ub-hd-general-inline1}.

By triangle inequality, one has $
 \|\widehat{\vB}\widehat{\vB}\tp-\vB  \vB\tp\|_{F}\leqslant  \|\widehat{\vB}\widehat{\vB}\tp - \widehat{\vB}_{O}\widehat{\vB}_{O}\tp\|_{F} +  \|\widehat{\vB}_{O}\widehat{\vB}\tp _{O}-\vB  \vB\tp\|_{F}$. This, together with Equations~\eqref{eq:ub-hd-general-Eo-F} to \eqref{eq:ub-hd-general-BtoF}, yields
 \begin{align*}
  \bbE \left( \one_{\ttE\cap \ttE_{O}}   \|\widehat{\vB}\widehat{\vB}\tp-\vB  \vB\tp\|_{F}^{2} \right)
& \lesssim  \epsilon_{n}^{2}.
\end{align*}

Since $\bbP( \ttE^c \cup \ttE_{O}^{c})\lesssim \frac{1}{n\lambda}$, one has 
 \begin{align*}
  \bbE \|\widehat{\vB}\widehat{\vB}\tp-\vB  \vB\tp\|_{F}^{2}
& \lesssim  \epsilon_{n}^{2}.
\end{align*}

\medskip 

\begin{proof}[Proof of Lemma~\ref{lem:decomposition upper spare general bS}]
First, we have
\begin{align*}
 {\widetilde{\bLambda}}_{V,(T)}^{(2)} -\widehat{\vI}_{T}{\widetilde{\bLambda}}_{(T)}^{(2)}\widehat{\vI}_{T}\tp 
 =& \widetilde{\mathbf{\mathcal{V}}}_{(T)}^{(2)}\widetilde{\mathbf{\mathcal{V}}}_{(T)}^{(2),\top}  -  \wh\vI_T\widetilde{\mathbf{\mathcal{V}}}_{(T)}^{(2)}\widetilde{\mathbf{\mathcal{V}}}_{(T)}^{(2),\top}\wh\vI_T^\top - 
 \\ 
& \left(  \wh\vI_T\widetilde{\mathbf{\mathcal{V}}}_{(T)}^{(2)}\widetilde{\mathbf{\mathcal{W}}}_{(T)}^{(2),\top}\wh\vI_T^\top + \wh\vI_T \widetilde{\mathbf{\mathcal{W}}}_{(T)}^{(2)}\widetilde{\mathbf{\mathcal{V}}}_{(T)}^{(2),\top}\wh\vI_T^\top + \wh\vI_T\widetilde{\mathbf{\mathcal{W}}}_{(T)}^{(2)}\widetilde{\mathbf{\mathcal{W}}}_{(T)}^{(2),\top}\wh\vI_T^\top \right).
\end{align*}
 Following the same proof of Lemma \ref{lem:bound T1 T2} and notice that 
 $1_{\ttE}\|\wh\vI_T\|^2<2M^2$, we have:
\begin{align*}
\langle \wh\vI_T\widetilde{\mathbf{\mathcal{V}}}_{(T)}^{(2)}\widetilde{\mathbf{\mathcal{W}}}_{(T)}^{(2),\top}\wh\vI_T^\top + \wh\vI_T \widetilde{\mathbf{\mathcal{W}}}_{(T)}^{(2)}\widetilde{\mathbf{\mathcal{V}}}_{(T)}^{(2),\top}\wh\vI_T^\top + \wh\vI_T\widetilde{\mathbf{\mathcal{W}}}_{(T)}^{(2)}\widetilde{\mathbf{\mathcal{W}}}_{(T)}^{(2),\top}\wh\vI_T^\top,  \widehat{\vF}_{O}\widehat{\vF}_{O}\tp  -\widehat{\vF}\widehat{\vF}\tp    \rangle\lesssim&\lambda\epsilon_n\|\widehat{\vF}_{O}\widehat{\vF}_{O}\tp  -\widehat{\vF}\widehat{\vF}\tp\|_F.
\end{align*}
Then we only need to show that
\begin{align*}
\langle \widetilde{\mathbf{\mathcal{V}}}_{(T)}^{(2)}\widetilde{\mathbf{\mathcal{V}}}_{(T)}^{(2),\top}  -  \wh\vI_T\widetilde{\mathbf{\mathcal{V}}}_{(T)}^{(2)}\widetilde{\mathbf{\mathcal{V}}}_{(T)}^{(2),\top}\wh\vI_T^\top ,   \widehat{\vF}_{O}\widehat{\vF}_{O}\tp  -\widehat{\vF}\widehat{\vF}\tp  \rangle\lesssim&\lambda\epsilon_n\|\widehat{\vF}_{O}\widehat{\vF}_{O}\tp  -\widehat{\vF}\widehat{\vF}\tp\|_F.
\end{align*}
Using the inequality that $|\Tr(\vA)|\leqs \sqrt{\mr{rank}(\vA) } \|\vA\|_{F}$ and  Lemma \ref{lem:elementary:trivial2}, we have
\begin{equation}\label{eq: inequality general sigma sparse}
\begin{aligned}
&\langle \widetilde{\mathbf{\mathcal{V}}}_{(T)}^{(2)}\widetilde{\mathbf{\mathcal{V}}}_{(T)}^{(2),\top}  -  \wh\vI_T\widetilde{\mathbf{\mathcal{V}}}_{(T)}^{(2)}\widetilde{\mathbf{\mathcal{V}}}_{(T)}^{(2),\top}\wh\vI_T^\top ,   \widehat{\vF}_{O}\widehat{\vF}_{O}\tp  -\widehat{\vF}\widehat{\vF}\tp  \rangle^2\\
\leqs & d \|  \left( \widehat{\vF}_{O}\widehat{\vF}_{O}\tp  -\widehat{\vF}\widehat{\vF}\tp \right) \left( \widetilde{\mathbf{\mathcal{V}}}_{(T)}^{(2)}\widetilde{\mathbf{\mathcal{V}}}_{(T)}^{(2),\top}  -  \wh\vI_T\widetilde{\mathbf{\mathcal{V}}}_{(T)}^{(2)}\widetilde{\mathbf{\mathcal{V}}}_{(T)}^{(2),\top}\wh\vI_T^\top\right) \|_F^2 \\
\leqs & d \|  \widehat{\vF}_{O}\widehat{\vF}_{O}\tp  -\widehat{\vF}\widehat{\vF}\tp \|_{F}^2 \|\widetilde{\mathbf{\mathcal{V}}}_{(T)}^{(2)}\widetilde{\mathbf{\mathcal{V}}}_{(T)}^{(2),\top}  -  \wh\vI_T\widetilde{\mathbf{\mathcal{V}}}_{(T)}^{(2)}\widetilde{\mathbf{\mathcal{V}}}_{(T)}^{(2),\top}\wh\vI_T^\top\|^2 \\ 
\leqs& d \| \widehat{\vF}_{O}\widehat{\vF}_{O}\tp  -\widehat{\vF}\widehat{\vF}\tp\|_F^2(\|\wh\vI_T\|+1)^2\|\widetilde{\mathbf{\mathcal{V}}}_{(T)}^{(2)}\widetilde{\mathbf{\mathcal{V}}}_{(T)}^{(2),\top}\|^2\|\wh\vI_T-\vI_{|T|}\|^2,
\end{aligned}
\end{equation}
where the last inequality is due to the following fact: 
for any two $m\times m$ matrices $\vA$ and $\vB$, we can write $\vA-\vB \vA \vB^\top=(\bs{I}-\vB) \vA + \vB \vA (\bs{I}-\vB)^\top$ and conclude the inequality that $\|\vA-\vB \vA \vB^\top\|\leqslant \|\bs{I}-\vB\| \|\vA\| + \|\vB\| \|\vA\| \|\bs{I}-\vB\| = (\|\vB\|+1 )\|\vA\| \|\bs{I}-\vB\|$. 

By Corollary \ref{cor:minimax:key-general-T} and Lemma \ref{lem:Sigma-control-T}, we know that
\begin{align*}
\|\wh\vI_T\|\lesssim 1;~~\|\widetilde{\mathbf{\mathcal{V}}}_{(T)}^{(2)}\widetilde{\mathbf{\mathcal{V}}}_{(T)}^{(2),\top}\|\lesssim \lambda;~~\|\wh\vI_T-\vI_{|T|}\|^2\lesssim \frac{s\log(ep/s)+\log(n\lambda)}{n}.
\end{align*}
Insert these equations into \eqref{eq: inequality general sigma sparse} and use the conditions that $\lambda\leqs\frac1d$ and $n\lambda\leqs e^s$, we have 
\begin{align*}
&\langle \widetilde{\mathbf{\mathcal{V}}}_{(T)}^{(2)}\widetilde{\mathbf{\mathcal{V}}}_{(T)}^{(2),\top}  -  \wh\vI_T\widetilde{\mathbf{\mathcal{V}}}_{(T)}^{(2)}\widetilde{\mathbf{\mathcal{V}}}_{(T)}^{(2),\top}\wh\vI_T^\top ,   \widehat{\vF}_{O}\widehat{\vF}_{O}\tp  -\widehat{\vF}\widehat{\vF}\tp  \rangle^2\\
\lesssim& d\lambda^2\frac{s\log(ep/s)+\log(n\lambda)}{n} \| \widehat{\vF}_{O}\widehat{\vF}_{O}\tp  -\widehat{\vF}\widehat{\vF}\tp\|_F^2\\
\lesssim& \lambda^2\epsilon_n^2\| \widehat{\vF}_{O}\widehat{\vF}_{O}\tp  -\widehat{\vF}\widehat{\vF}\tp\|_F^2.
\end{align*}
\end{proof}
\section{Assisting Lemmas}
\begin{lemma}\label{inline:trivial1}
	Let $\vK$ be an $a\times b$ matrix with each entry being i.i.d. standard normal random variables. Then
	$
	\bbE[\|\vK\vK\tp \|^{2}_{F}]=ab(a+b+1)$, 
	 $
	\bbE[\|\vK\|^{2}_{F}]=ab 
	$ and  $
	\bbE[\|\vK\|^{4}_{F}]=a^2b^2+2ab. 
	$
\end{lemma}

\begin{lemma}\label{lem:elementary:trivial2}
	Let $\bs{A}$, $\bs{B}$ be  $l\times m$ and  $m\times n$ matrices, respectively. Then one has
	$
	\|\bs{A}\bs{B}\|_{F} \leqslant \|\bs{A}\|\|\bs{B}\|_{F},
	$ where $\|\bs{A}\|$ denotes the largest singular value of $\bs{A}$. 
\end{lemma}

\begin{lemma}[Weyl's Inequality]\label{lem:weyl}
	Let $\bs{A}$, $\bs{B}$ be $m \times n$   matrices for some $1\leqslant m\leqslant n$. 
	Then for all $1\leqslant i\leqslant m$, 
	$$|\sigma_i( \bs{A}+\bs{B})-\sigma_i(\bs{A})|\leqslant \|\bs{B}\|_{\text{op}}. $$
	If $\bs{A}$, $\bs{B}$ are $m \times m$ symmetric matrices, then  for all $1\leqslant i\leqslant m$, 
	$$|\lambda_i( \bs{A}+\bs{B})-\lambda_i(\bs{A})|\leqslant \|\bs{B}\|_{\text{op}}. $$
	\begin{proof}
	See, e.g., \cite[Chapter~1.3]{tao2012topics}. 
	\end{proof}
\end{lemma}
\begin{lemma}[\cite{vershynin2010introduction}]\label{random:nonasymptotic}
	Let $\vA$ be a $p\times H$ matrix ($p\geqslant H$), whose entries are independent standard normal random variables. 
	Then for every $ t\geqslant 0$, with probability at least $1-2\exp(-t^{2}/2)$, one has that 
	\[ \nonumber
		\sqrt{p}-\sqrt{H}-t \leqslant\,  \sigma_{H}(\vA) \leqslant \, 
	\sigma_{1}(\vA) \leqslant\,  \sqrt{p}+\sqrt{H}+t.\]
	
\end{lemma} 
\begin{lemma}[Sin-Theta Theorem, \cite{cai2013sparse}] \label{lem:sin_theta} 
	Let $\vA$ and $\vA+\vE$ be symmetric matrices satisfying
	\begin{equation*}
	\vA=[\vF_{0},\vF_{1}]\left[ \begin{array}{cc}
	\vA_{0} & 0\\
	0 & \vA_{1}
	\end{array} \right]
	\left[ \begin{array}{c}
	\vF\tp _{0}\\
	\vF\tp _{1}
	\end{array} \right]
	\quad
	\vA+\vE=[\vG_{0},\vG_{1}]\left[ \begin{array}{cc}
	\bLambda_{0} & 0\\
	0 & \bLambda_{1}
	\end{array} \right]
	\left[ \begin{array}{c}
	\vG\tp _{0}\\
	\vG\tp _{1}
	\end{array} \right]
	\end{equation*}
	where $[\vF_{0},\vF_{1}]$ and $[\vG_{0},\vG_{1}]$ are orthogonal matrices. If the eigenvalues of $\vA_{0}$ lie in an interval $(a,b)$ and the eigenvalues of $\bLambda_{1}$ are excluded from the interval $(a-\delta,b+\delta)$ for some $\delta>0$, then
	\[
	\|\vF_{0}\vF_{0}\tp -\vG_{0}\vG_{0}\tp \|\leqslant \frac{\min(\|\vF_{1}\tp \vE\vG_{0}\|, \|\vF_{0}\tp \vE\vG_{1}\|)}{\delta},
	\]
	and 
	\[
	\frac{1}{\sqrt{2}}\|\vF_{0}\vF_{0}\tp -\vG_{0}\vG_{0}\tp \|_{F} \leqslant \frac{\min(\|\vF_{1}\tp \vE\vG_{0}\|_{F}, \|\vF_{0}\tp \vE\vG_{1}\|_{F})}{\delta}.
	\]
\end{lemma}

\begin{lemma}[\cite{cai2013sparse}]\label{lem:Cai:lem4}
	Let $\vK \in \bbR^{p\times p}$ be symmetric such that $\Tr(\vK)=0$ and $\|\vK\|_{F}\leqslant1$. Let $\vZ$ be an $H \times p$  matrix consisting of independent standard normal entries. Then for any $t>0$, one has
	\begin{align}
	\bbP\left(\Big|\Big\langle \vZ\tp \vZ, \vK \Big\rangle\Big
	|\geqslant 2\sqrt{H}t +2t^{2} \right) \leqslant 2\exp\left(-t^{2} \right).
	\end{align}
\end{lemma}
We remind that this lemma is a trivial modification of  Lemma 4 in \cite{cai2013sparse}, where they assumed $\|\vK\|_{F}=1$. 

\begin{lemma}[\cite{cai2013sparse}]\label{lem:Cai:lem5}
	Let $X_{1},...,X_{N}$ be random variables such that each satisfies
	\begin{align}
	\bbP(|X_{i}|\geqslant at+bt^{2})\leqslant c\exp\left(-t^{2} \right)
	\end{align}
	where $a,b,c>0$. Then
	\begin{align}
	\bbE\max|X_{i}|^{2} \leqslant (2a^{2}+8b^{2})\log(ecN)+2b^{2}\log^{2}(cN).
	\end{align}
\end{lemma}
\begin{lemma}\label{lem:elementary:trivial3}
	Let $\vA$, $\vB$ be $m\times l$ orthogonal matrices, i.e., $\vA\tp \vA= \vI_{l}= \vB\tp \vB$ and  $\vM$ be an $l\times l$ positive definite matrix with eigenvalues $d_j$ such as $0<\lambda\leqslant d_{l} \leqslant d_{l-1} \leqslant ... \leqslant d_{1} \leqslant \kappa \lambda$. If $ \vA\tp \vB$ is a diagonal matrix with non-negative entries, then there exists a constant $C$ which only depends on $\kappa$ such that 
	$
	\| \vA \vM \vA\tp - \vB \vM \vB\tp \|_{F}\leqslant C\lambda \|\vA\vA\tp -\vB\vB\tp \|_{F}.$
	\begin{proof}
We first show that for any symmetric matrix $\vC$ and positive semi-definite matrix $\vD$, one has
\begin{align*}
\lambda_{\min}(\vC)\lambda_i(\vD)\leqslant	\lambda_i(\vC\vD)\leqslant\lambda_{\max}(\vC)\lambda_i(\vD).   
\end{align*}
This is because the Courant--Fischer min-max theorem \citep[Theorem 1.3.2]{tao2012topics}:
\begin{align}\label{eq:elementary:trivial3-eigen-prod}
\lambda_i(\vC\vD)=\lambda_i(\sqrt{\vD}\vC\sqrt{\vD})=\inf_{F\subset\R^n,  \atop\mr{dim}(F)=n-i+1}\sup_{x\in F\backslash\{0\}}\frac{(\vC\sqrt{\vD}x,\sqrt{\vD}x)}{(\sqrt{\vD}x,\sqrt{\vD}x)}\frac{(\vD x,x)}{(x,x)}.
\end{align}
 
	Let $\Delta= \vI_{l}- \vB\tp \vA$, then $0 \leqslant \Delta_{ii}\leqslant 1$ for $1\leqslant i \leqslant l$, so $\Tr(\Delta^2)\leqslant \Tr(\Delta)$ and  $4\Tr(\Delta)-2\Tr(\Delta^2)=\| \vA\vA\tp - \vB\vB\tp\|_{F}^2$. 
 
 If $C^{2}>2\kappa^{2}-1$, then $(C^2-1)\Tr(\Delta^2)\leqslant (C^2-1)\Tr(\Delta) \leqslant 2(C^2-\kappa^2)\Tr(\Delta)$, that is, 
 \begin{equation}\label{eq:elementary:trivial3-C}
  2\kappa^2\Tr(\Delta)  - \Tr(\Delta^2)\leqslant  2 C^2 \Tr(\Delta) - 2 C^2  \Tr(\Delta^2).
 \end{equation}
 We have
\begin{equation}\nonumber
\begin{aligned}
\|\vA \vM\vA\tp - \vB \vM \vB\tp \|_{F}^{2}&=4\Tr(\vM^{2}\Delta)-2\Tr(\vM\Delta \vM\Delta)\overset{(a)}{\leqslant}  4\kappa^{2}\lambda^{2} \Tr(\Delta)-2\lambda^{2}\Tr(\Delta^{2})\\ 
& \overset{(b)}{\leqslant} 2C^{2}\lambda^{2}(2\Tr(\Delta)-\Tr(\Delta^{2}))=C^{2}\lambda^{2}\|\vA\vA\tp - \vB\vB\tp \|_{F}^{2},
\end{aligned}
\end{equation}
where $(a)$ is obtained by applying  \eqref{eq:elementary:trivial3-eigen-prod} for three times with $(\vC, \vD)=(\vM^2, \Delta)$, $(\vC, \vD)=(\vM, \Delta\vM\Delta)$, and $(\vC, \vD)=(\vM, \Delta^2)$, respectively, while $(b)$ comes from \eqref{eq:elementary:trivial3-C}. 
\end{proof}

\end{lemma}

\begin{lemma}\label{cor:elemetary:temp}
	For a positive definite matrix $\vM$ with eigenvalue  $\lambda_{1}\geqslant ... \geqslant \lambda_{d}>0$ and orthogonal matrices $\vA,\vB,\vE,\vF$, i.e., $ \vA\tp \vA= \vB\tp \vB= \vE\tp \vE= \vF\tp \vF= \vI_{d}$, 
	one has
	\[
	\frac{\lambda_{d}}{2}\| \vA \vB\tp - \vE \vF\tp \|^{2}_{F} \leqslant ~\langle \vA \vM \vB\tp , \vA \vB\tp -\vE\vF\tp  \rangle ~\leqslant \frac{\lambda_{1}}{2}\| \vA \vB\tp - \vE \vF\tp \|^{2}_{F}.
	\]
	\begin{proof}
	It is a direct corollary of the Lemma 8 in \cite{gao2014minimax}.
	\end{proof}
\end{lemma}

\section{Additional simulation results}\label{app:assitional-simulation}
The section contains (i) 
the detailed procedures of sampling from a GP  and the additional simulation results in the `Gaussian process' part of Section \ref{sec:simulation small gSNR}; (ii) the additional simulation results in the `dependence of error w.r.t. $d$ and $\lambda$' part of Section \ref{sec:minimax rate}.
\subsection{Detailed procedures of sampling from a GP}
\begin{itemize}
    \item[(i)] generate $\bs X_i\overset{iid}{\sim}N(0,I_p)$ and take $\bs x_i=\vB\tp \bs X_i,i\in[n]$;
    \item[(ii)] generate $(f(\bs x_1),\dots,f(\bs x_n))$ from the $n$-dimensional normal distribution 
    $$N\left(\{\mu(\bs x_i)\}_{i\in[n]},\{\Sigma(\bs x_i,\bs x_j)\}_{i,j\in[n]} \right);$$
    \item[(iii)] generate $\epsilon_i\overset{iid}{\sim}N(0,1),i\in[n]$ and take $Y_i=f(\bs x_i)+\epsilon_i$.
\end{itemize}

\subsection{Average logarithm of gSNR for $H=10,20,30,50$}
This subsection contains average logarithm of  estimated  gSNR  
  as a function of $n$  for various values of $d$ and  as a function of $d$  for various values of $n$. The $H$ is chosen to be  $10,20,30,50$ respectively.
\begin{figure}[H] 
	\centering
	\begin{minipage}{0.5\textwidth}
		\includegraphics[width=\textwidth]{"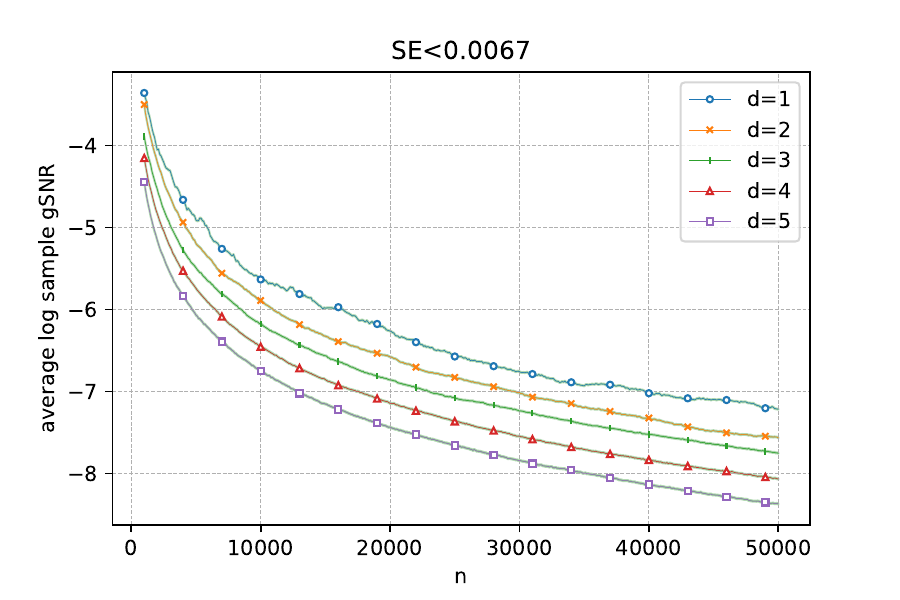"}
	\end{minipage}\hfill
	\begin{minipage}{0.5\textwidth}
		\includegraphics[width=\textwidth]{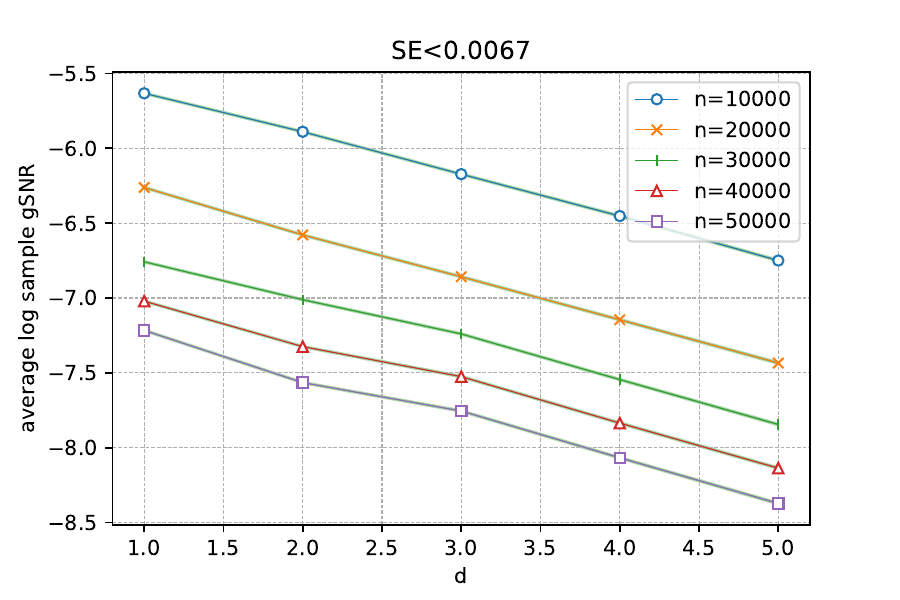}
	\end{minipage}\hfill
	\caption{Average logarithm of    gSNR with increasing $n$ (left) and increasing $d$ (right)  for $H=10$.}
\end{figure}

\begin{figure}[H] 
	\centering
	\begin{minipage}{0.5\textwidth}
		\includegraphics[width=\textwidth]{"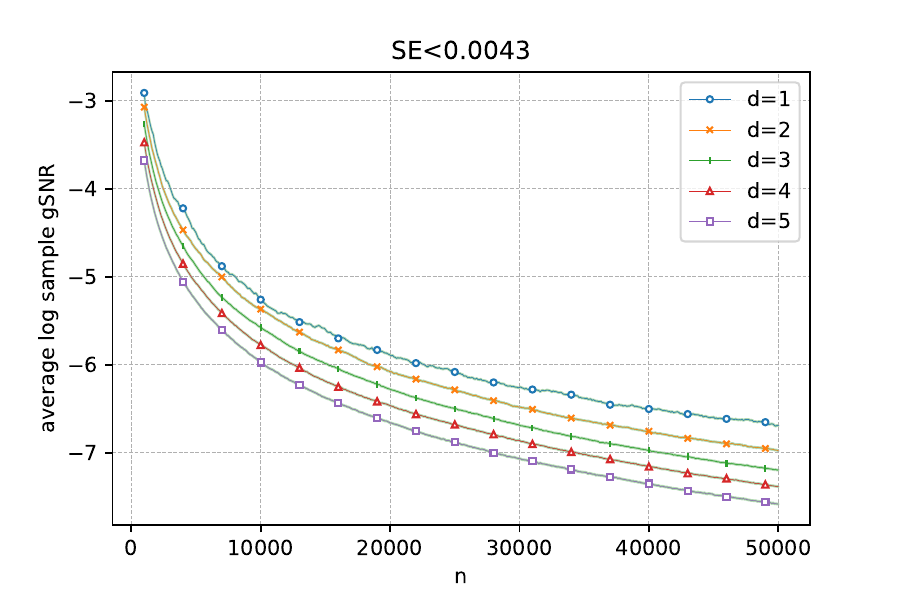"}
	\end{minipage}\hfill
	\begin{minipage}{0.5\textwidth}
		\includegraphics[width=\textwidth]{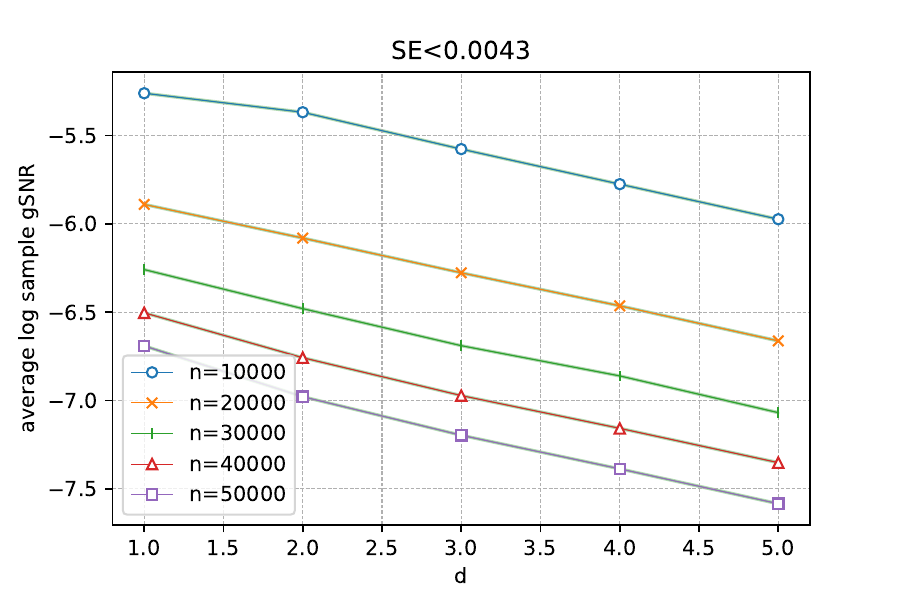}
	\end{minipage}\hfill
	\caption{Average logarithm of    gSNR with increasing $n$ (left) and increasing $d$ (right)  for $H=20$.}
\end{figure}

\begin{figure}[H] 
	\centering
	\begin{minipage}{0.5\textwidth}
		\includegraphics[width=\textwidth]{"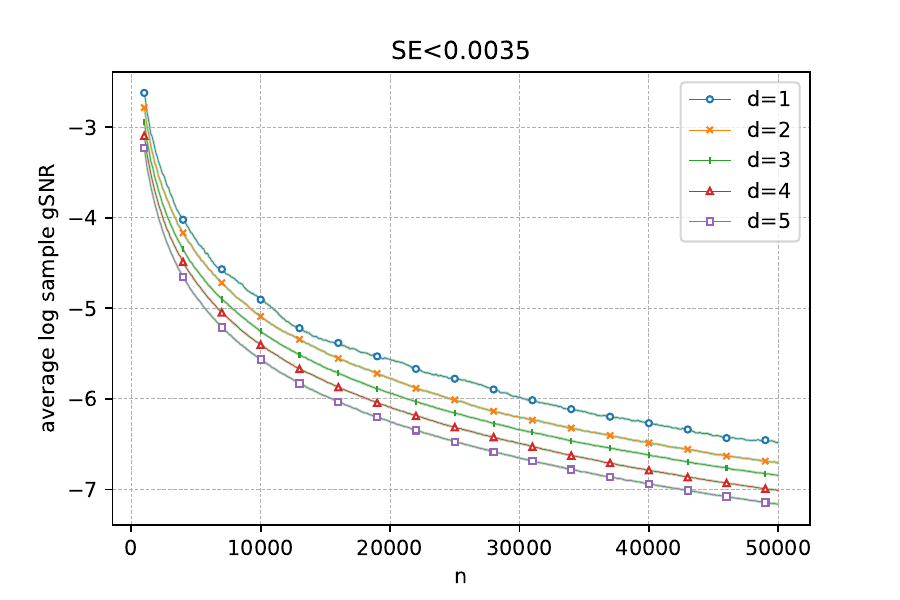"}
	\end{minipage}\hfill
	\begin{minipage}{0.5\textwidth}
		\includegraphics[width=\textwidth]{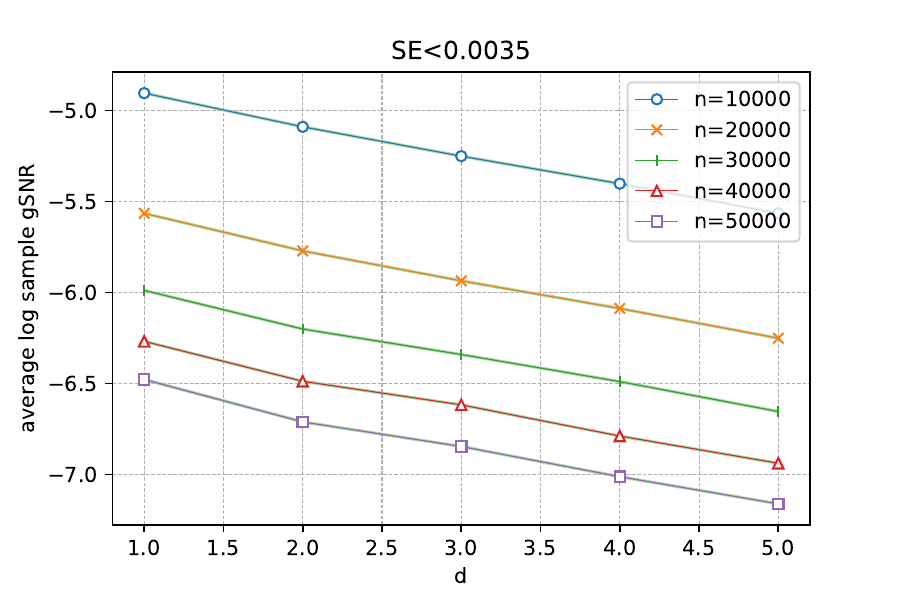}
	\end{minipage}\hfill
	\caption{Average logarithm of    gSNR with increasing $n$ (left) and increasing $d$ (right) for $H=30$.}
\end{figure}

\begin{figure}[H] 
	\centering
	\begin{minipage}{0.5\textwidth}
		\includegraphics[width=\textwidth]{"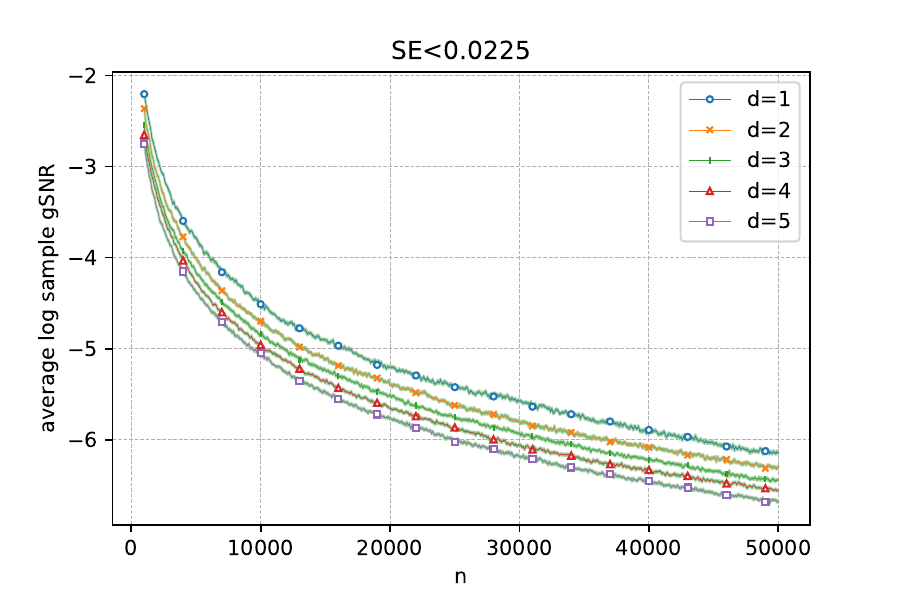"}
	\end{minipage}\hfill
	\begin{minipage}{0.5\textwidth}
		\includegraphics[width=\textwidth]{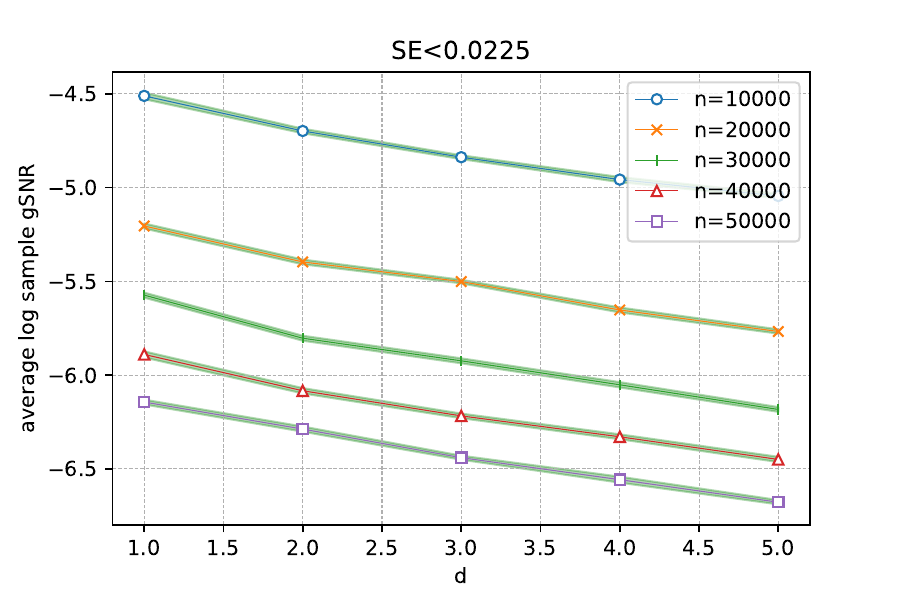}
	\end{minipage}\hfill
	\caption{Average logarithm of    gSNR with increasing $n$ (left) and increasing $d$ (right)  for $H=50$.}
\end{figure}

\subsection{Histogram of  gSNR of GP and pure noise}\label{app:Histogram of sample gSNR of GP}

This subsection contains  the histogram of   gSNR over
$1,000$ random functions with different $d$ and $n$ in
Figure \ref{hist, gSNR}.
\begin{figure}[H] 
	\centering
	\begin{minipage}{0.25\textwidth}
		\includegraphics[width=\textwidth]{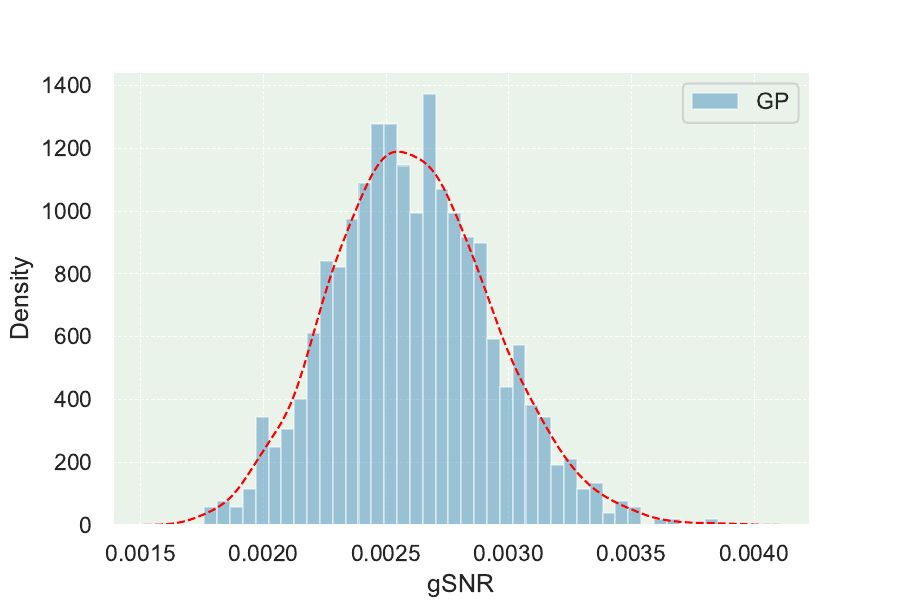}
		\caption*{$d=1,n=20000$}
	\end{minipage}\hfill
	\begin{minipage}{0.25\textwidth}
		\includegraphics[width=\textwidth]{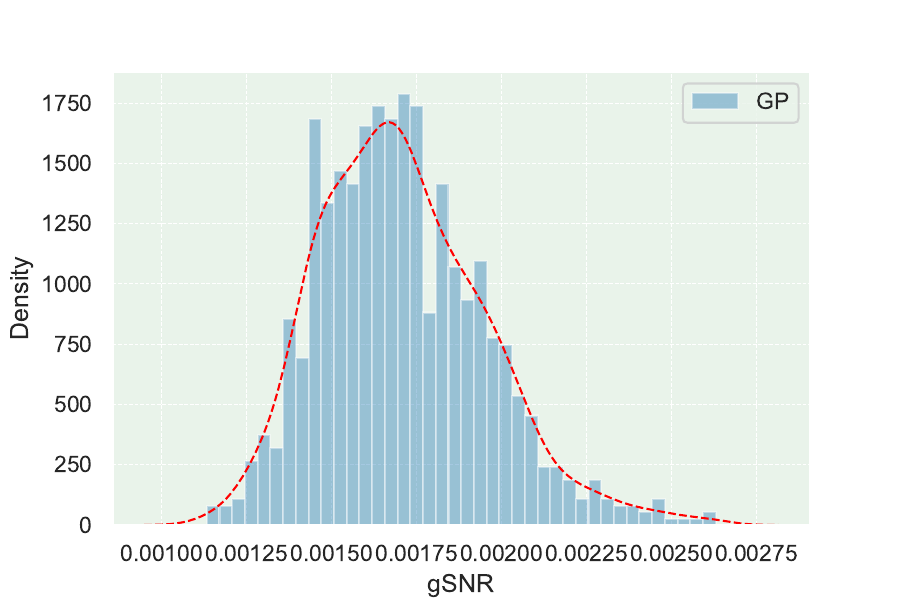}
		\caption*{$d=1,n=30000$}
	\end{minipage}\hfill
	\begin{minipage}{0.25\textwidth}
		\includegraphics[width=\textwidth]{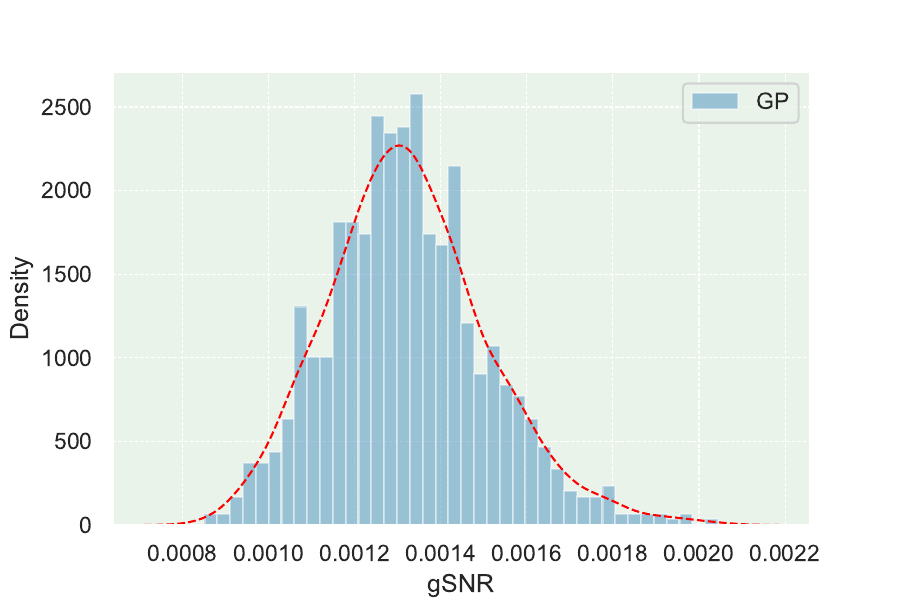}
		\caption*{$d=1,n=40000$}
	\end{minipage}\hfill
	\begin{minipage}{0.25\textwidth}
		\includegraphics[width=\textwidth]{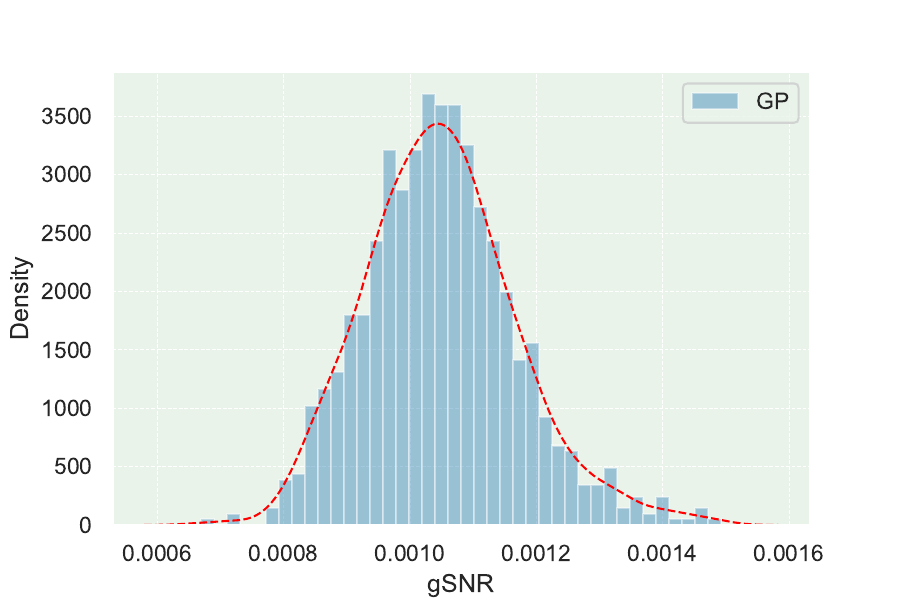}
		\caption*{$d=1,n=50000$}
		 
	\end{minipage}\hfill
	\begin{minipage}{0.25\textwidth}
		\includegraphics[width=\textwidth]{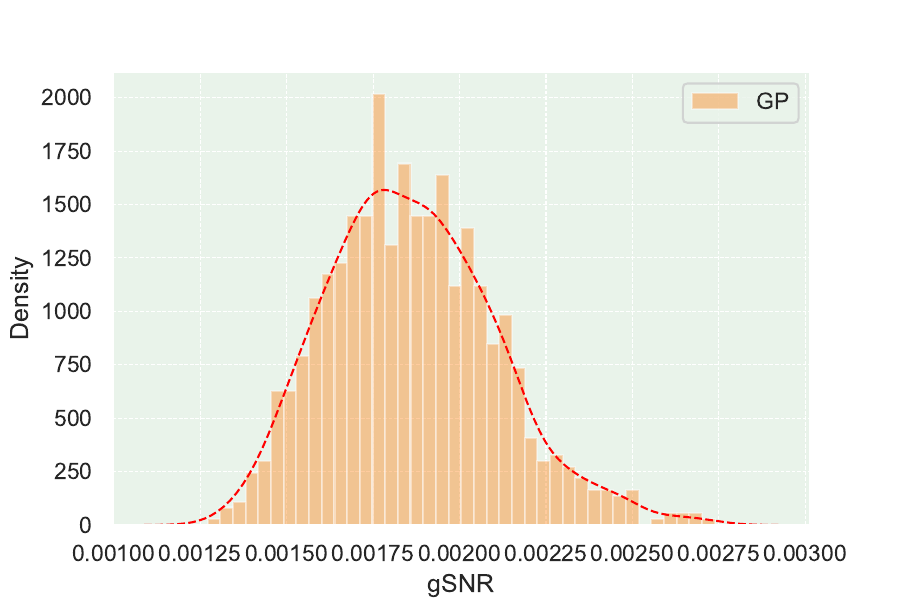}
		\caption*{$d=2,n=20000$}
		 
	\end{minipage}\hfill
	\begin{minipage}{0.25\textwidth}
		\includegraphics[width=\textwidth]{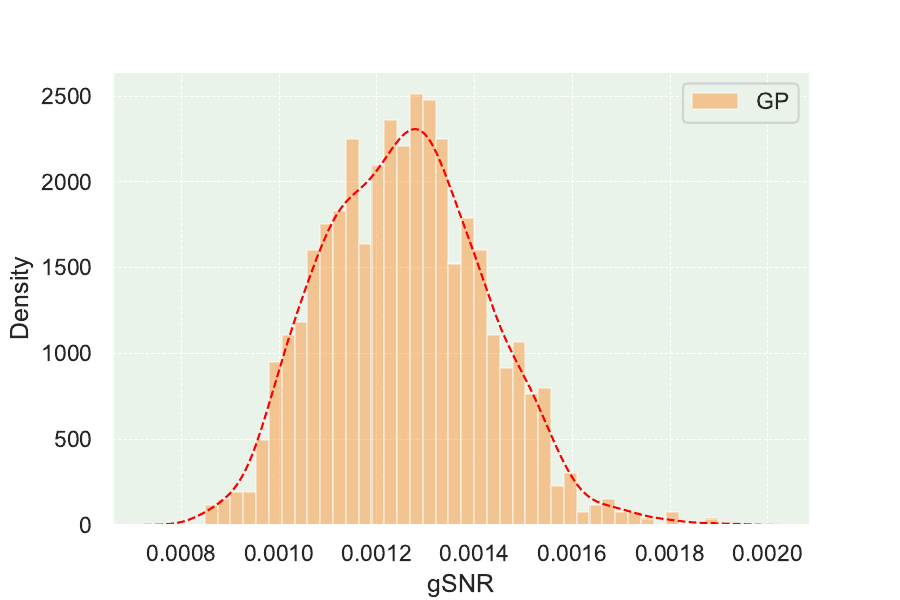}
		\caption*{$d=2,n=30000$}
		 
	\end{minipage}\hfill
	\begin{minipage}{0.25\textwidth}
		\includegraphics[width=\textwidth]{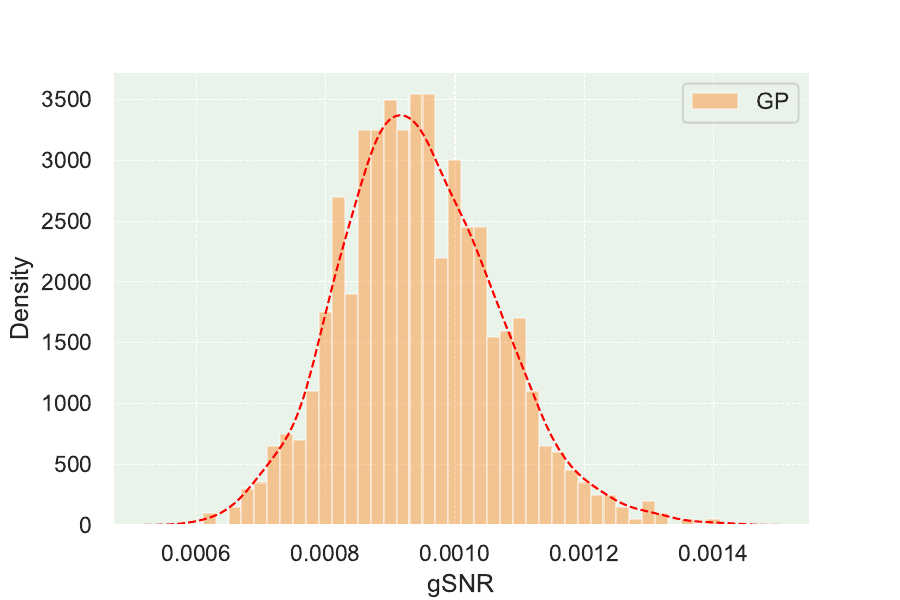}
		\caption*{$d=2,n=40000$}
		 
	\end{minipage}\hfill
	\begin{minipage}{0.25\textwidth}
		\includegraphics[width=\textwidth]{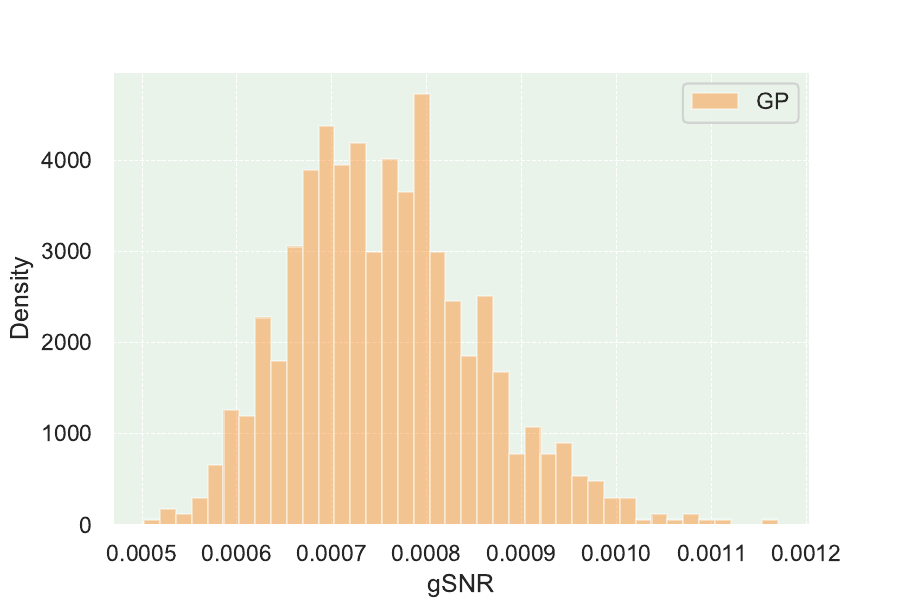}
		\caption*{$d=2,n=50000$}
		 
	\end{minipage}\hfill
	\begin{minipage}{0.25\textwidth}
	\includegraphics[width=\textwidth]{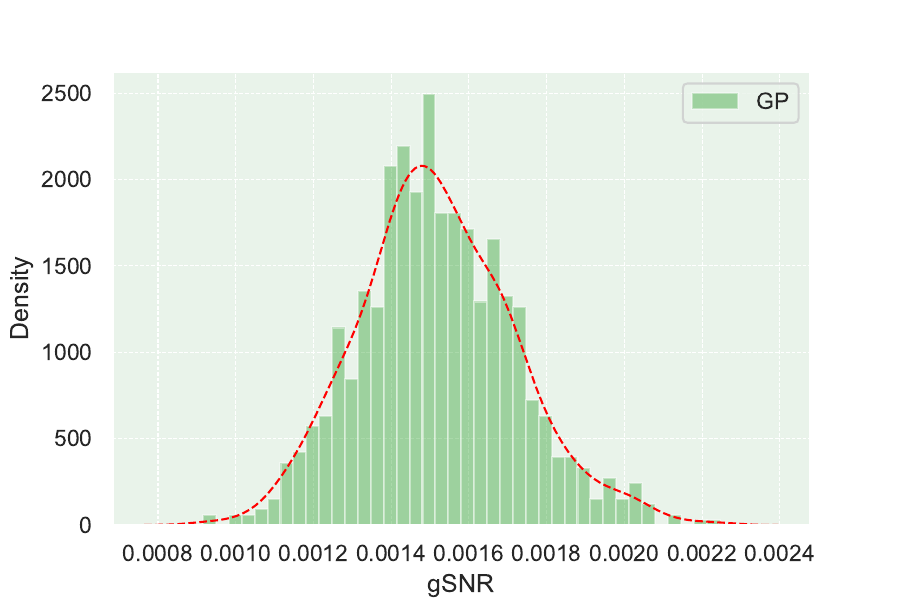}
	\caption*{$d=3,n=20000$}
	 
\end{minipage}\hfill
\begin{minipage}{0.25\textwidth}
	\includegraphics[width=\textwidth]{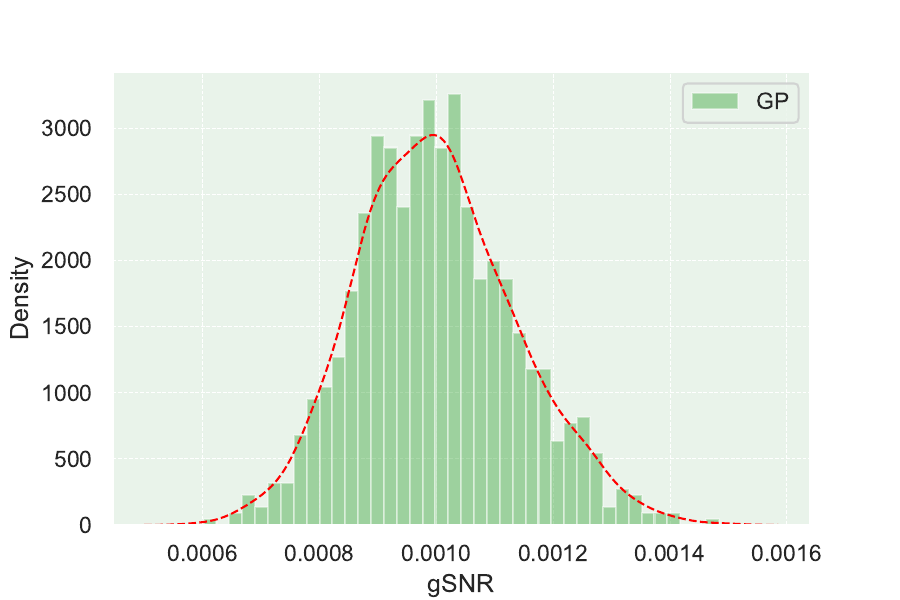}
	\caption*{$d=3,n=30000$}
\end{minipage}\hfill
\begin{minipage}{0.25\textwidth}
	\includegraphics[width=\textwidth]{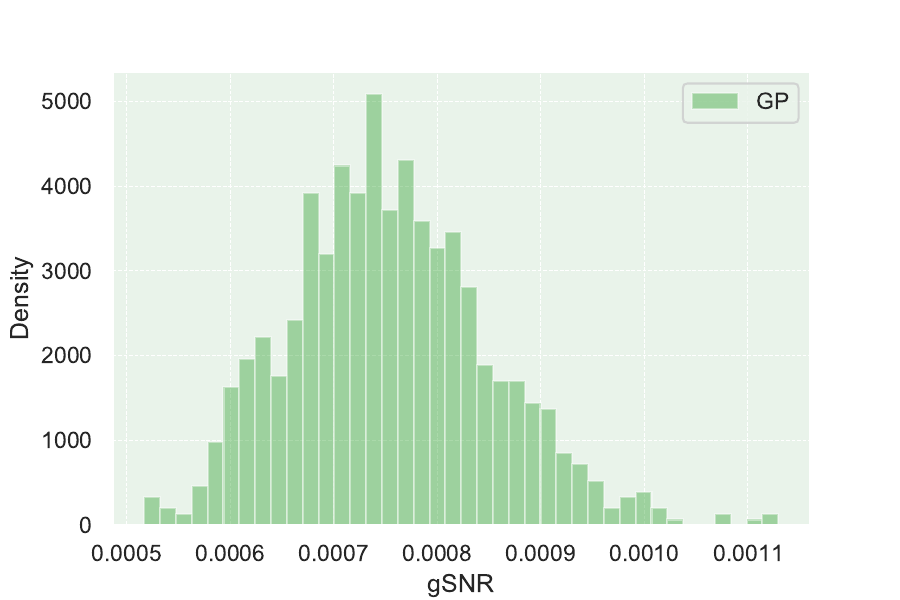}
	\caption*{$d=3,n=40000$}
	 
\end{minipage}\hfill
\begin{minipage}{0.25\textwidth}
	\includegraphics[width=\textwidth]{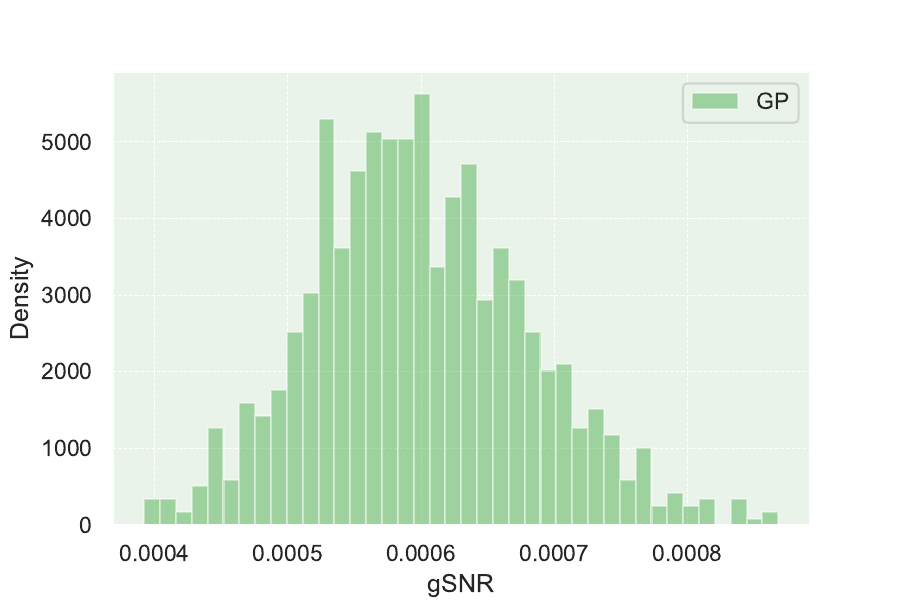}
	\caption*{$d=3,n=50000$}
	 
\end{minipage}\hfill
\begin{minipage}{0.25\textwidth}
	\includegraphics[width=\textwidth]{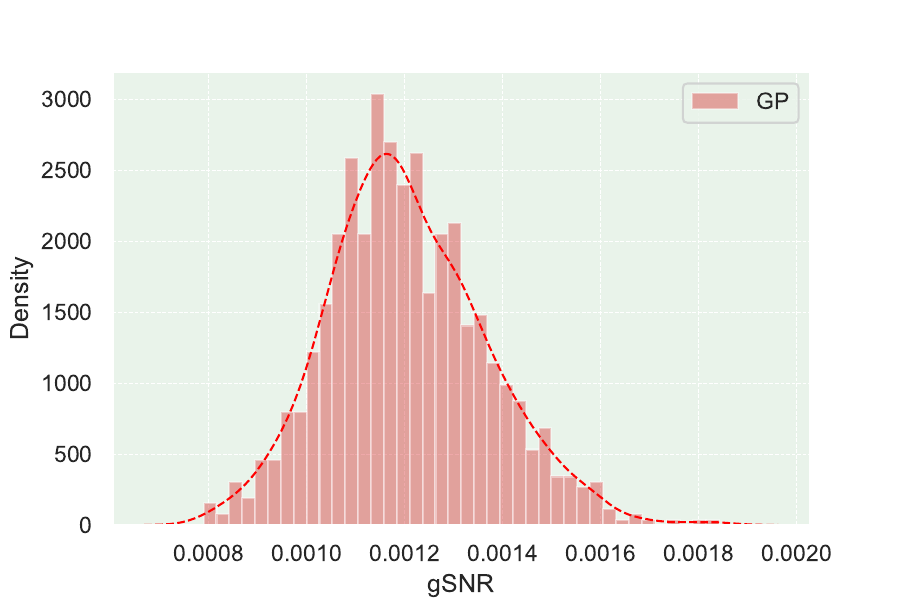}
	\caption*{$d=4,n=20000$}
	 
\end{minipage}\hfill
\begin{minipage}{0.25\textwidth}
	\includegraphics[width=\textwidth]{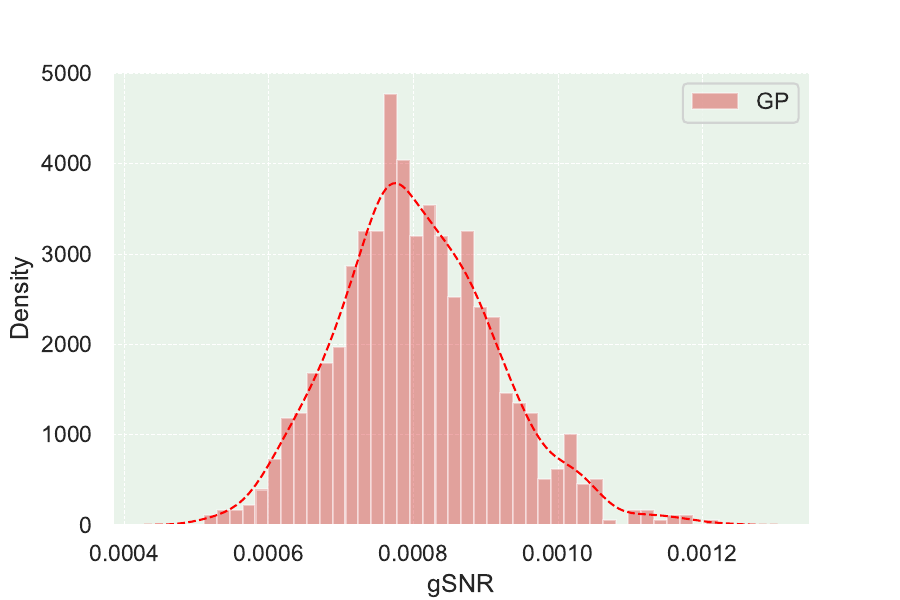}
	\caption*{$d=4,n=30000$}
	 
\end{minipage}\hfill
\begin{minipage}{0.25\textwidth}
	\includegraphics[width=\textwidth]{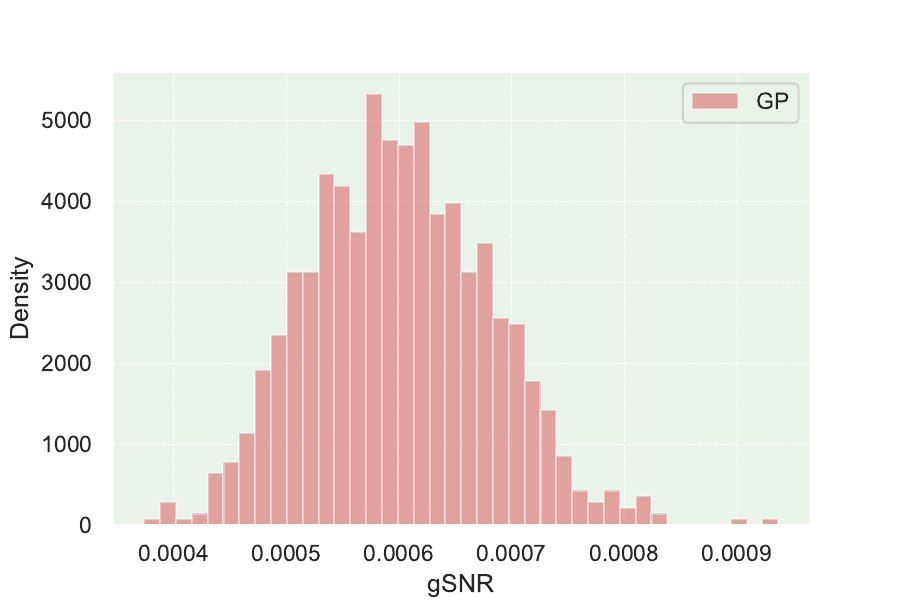}
	\caption*{$d=4,n=40000$}
	 
\end{minipage}\hfill
\begin{minipage}{0.25\textwidth}
	\includegraphics[width=\textwidth]{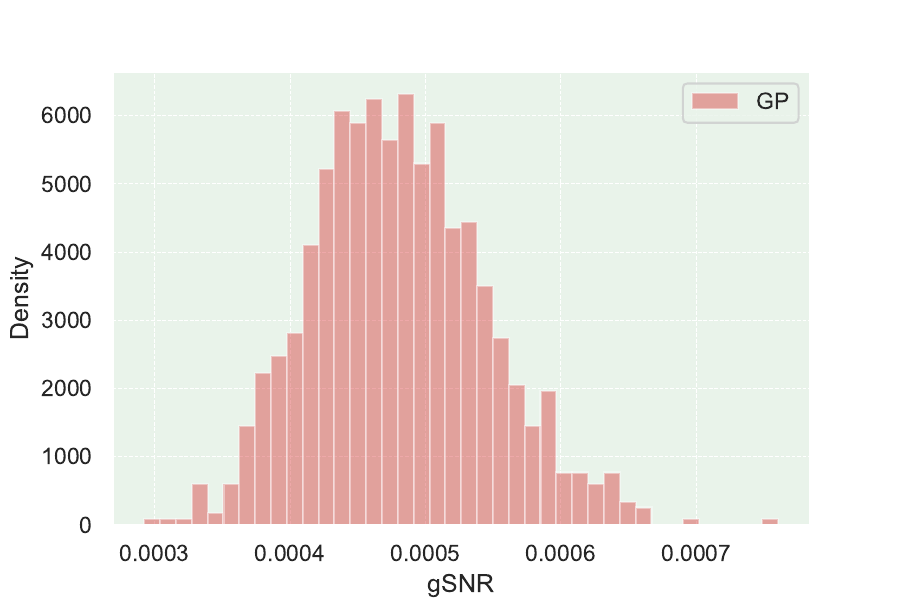}
	\caption*{$d=4,n=50000$}
	 
\end{minipage}\hfill
\begin{minipage}{0.25\textwidth}
	\includegraphics[width=\textwidth]{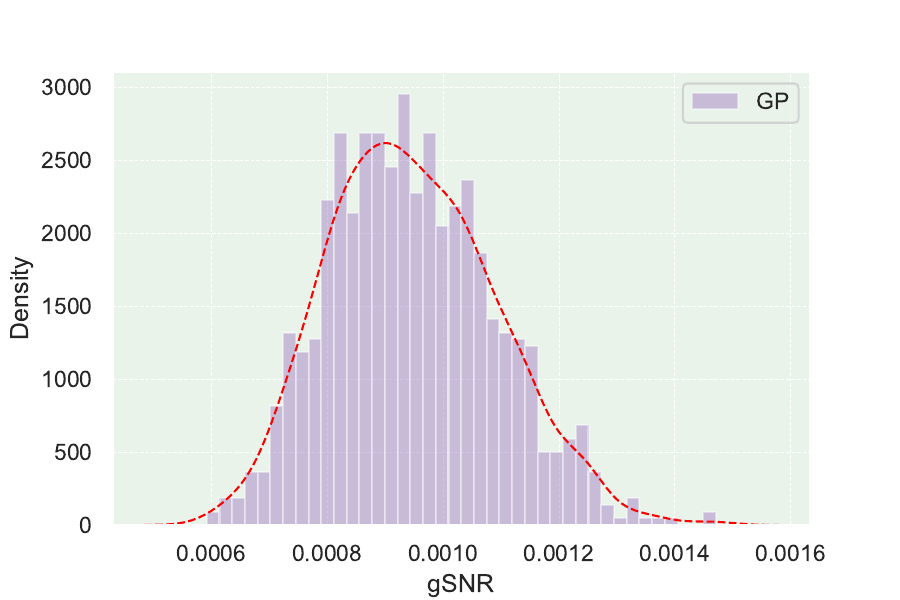}
	\caption*{$d=5,n=20000$}
	 
\end{minipage}\hfill
\begin{minipage}{0.25\textwidth}
	\includegraphics[width=\textwidth]{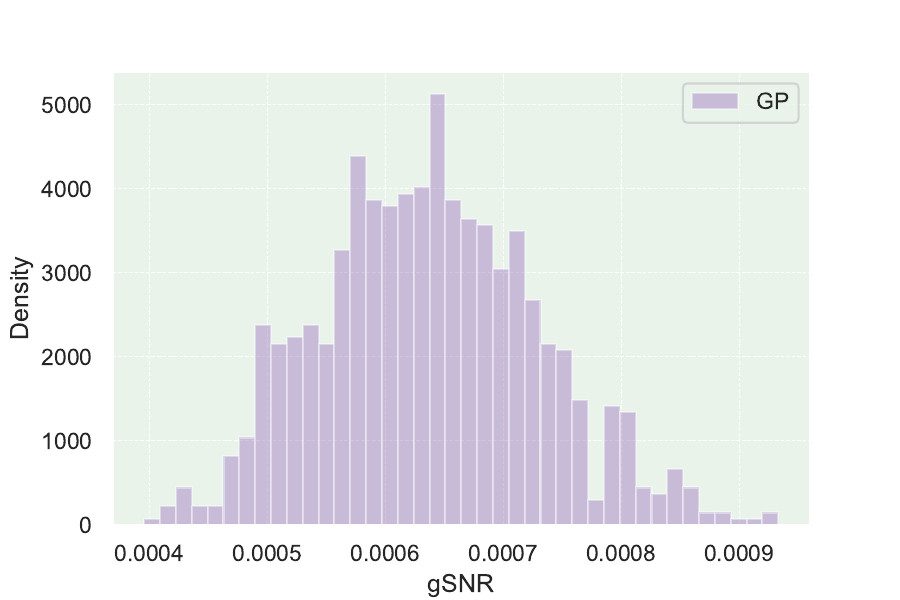}
	\caption*{$d=5,n=30000$}
	 
\end{minipage}\hfill
\begin{minipage}{0.25\textwidth}
	\includegraphics[width=\textwidth]{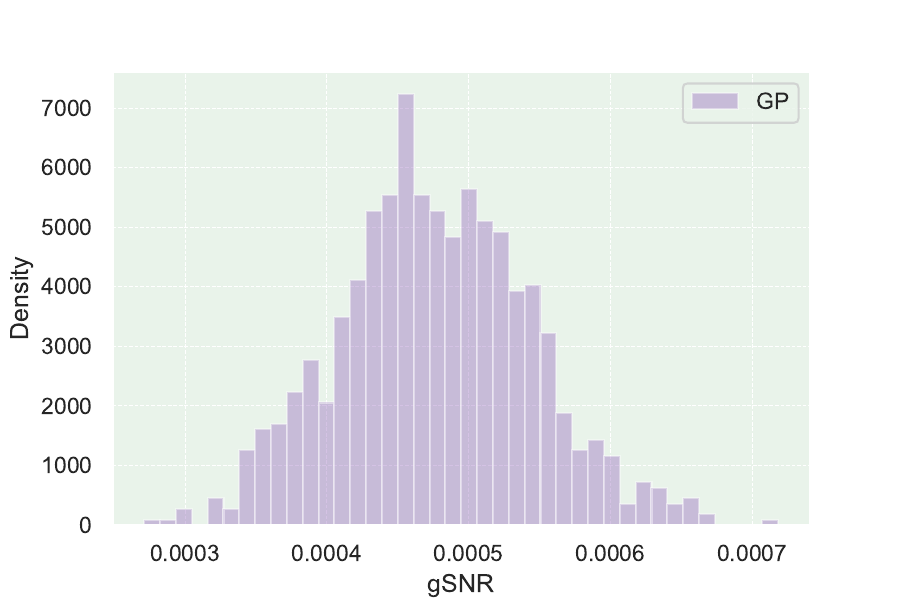}
	\caption*{$d=5,n=40000$}
	 
\end{minipage}\hfill
\begin{minipage}{0.25\textwidth}
	\includegraphics[width=\textwidth]{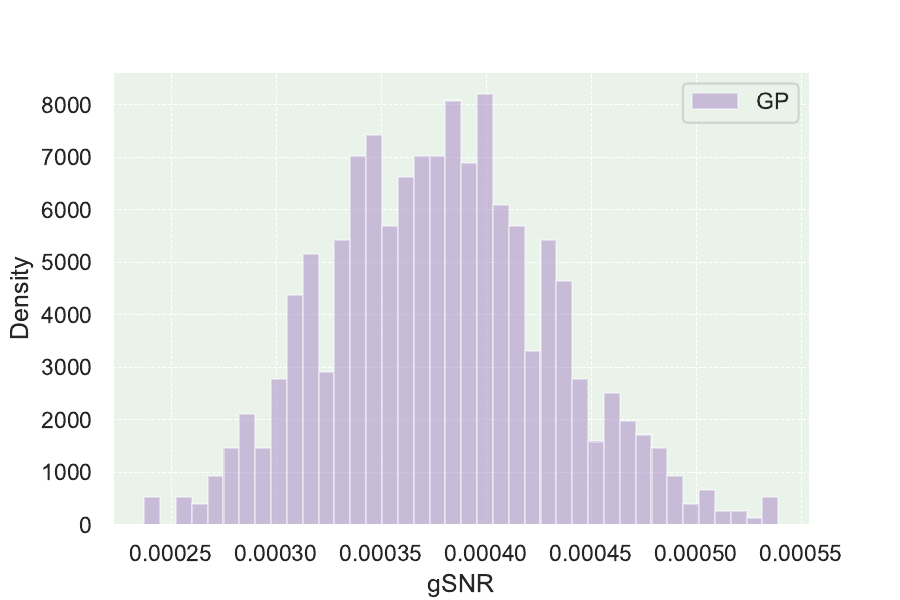}
	\caption*{$d=5,n=50000$}
	 
\end{minipage}\hfill
	\caption{Histogram of   gSNR of GP}
	\label{hist, gSNR}
\end{figure}
\newpage
\subsection{Additional simulation results in Section \ref{sec:minimax rate}}~
\begin{figure}[H] 
	\centering
	\begin{minipage}{0.5\textwidth}
		\includegraphics[width=\textwidth]{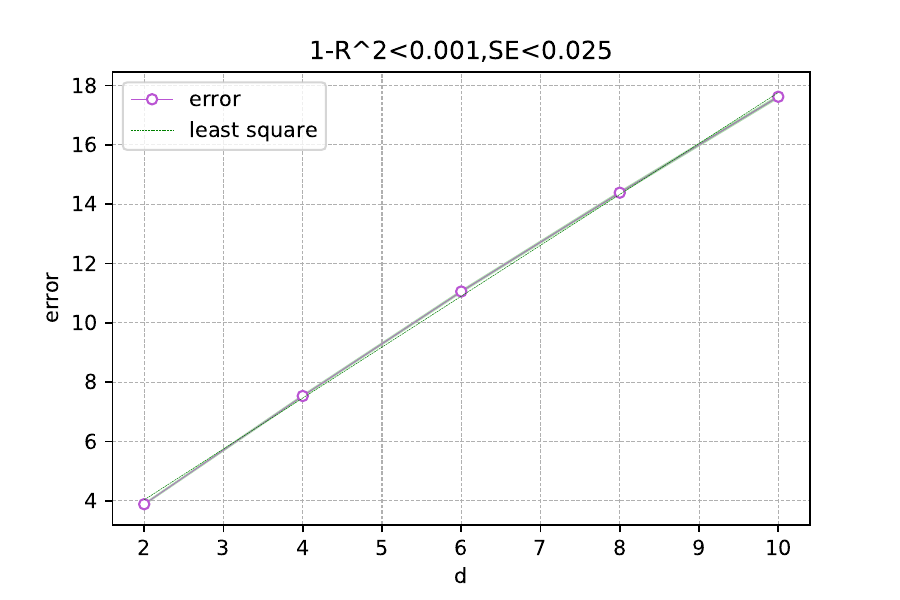}
	\end{minipage}\hfill
	\begin{minipage}{0.5\textwidth}
		\includegraphics[width=\textwidth]{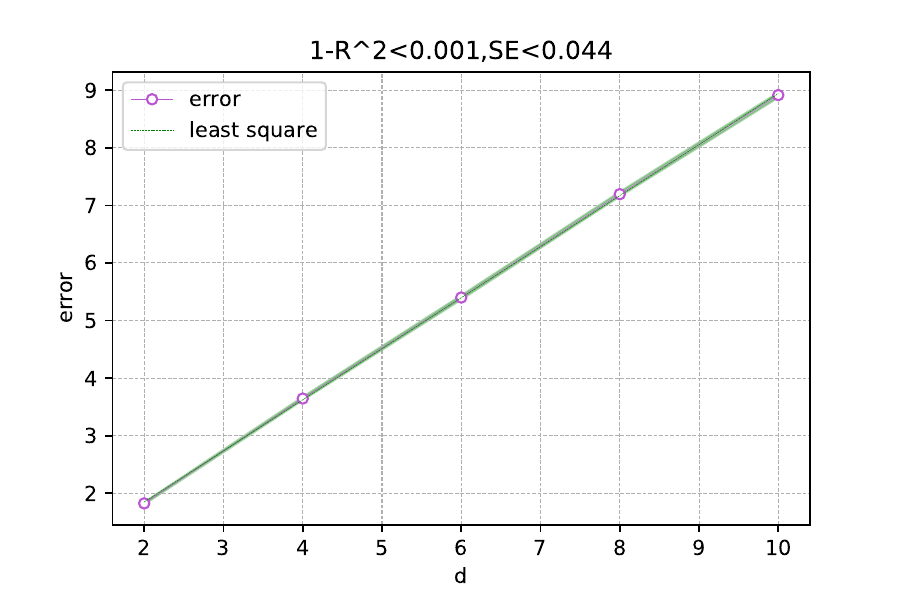}
	\end{minipage}\hfill
	\begin{minipage}{0.5\textwidth}
		\includegraphics[width=\textwidth]{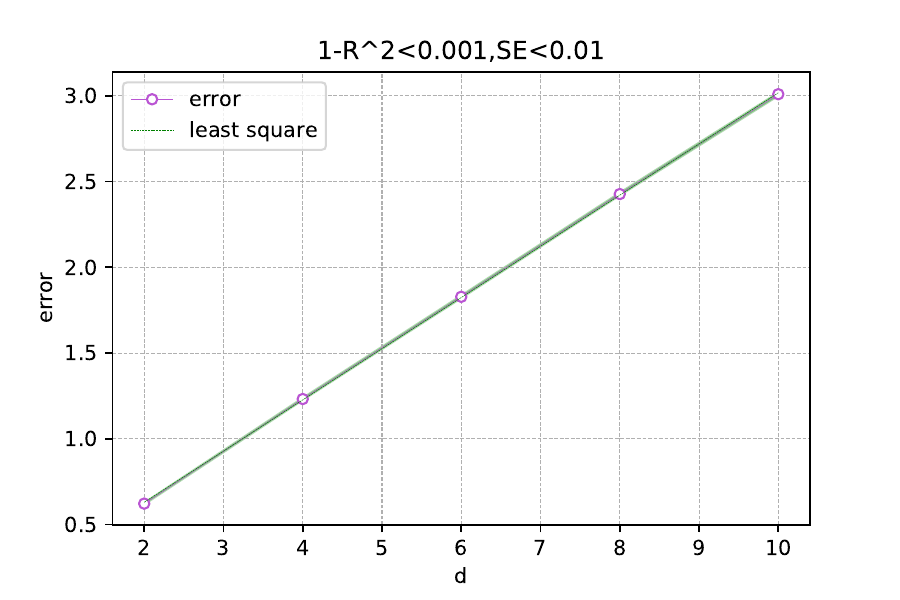}
	\end{minipage}\hfill
	\begin{minipage}{0.5\textwidth}
		\includegraphics[width=\textwidth]{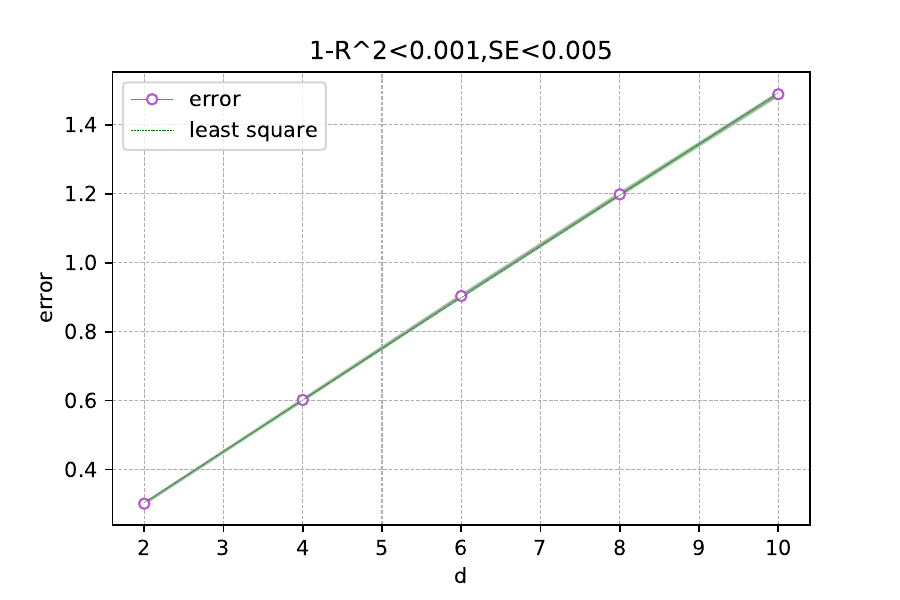}
	\end{minipage}\hfill
 	\begin{minipage}{0.5\textwidth}
		\includegraphics[width=\textwidth]{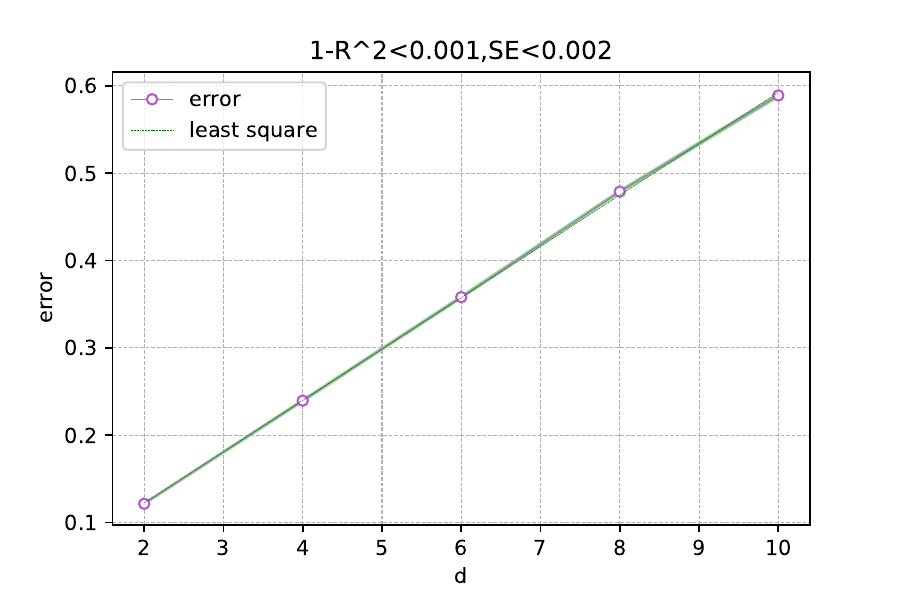}
	\end{minipage}\hfill
	\begin{minipage}{0.5\textwidth}
		\includegraphics[width=\textwidth]{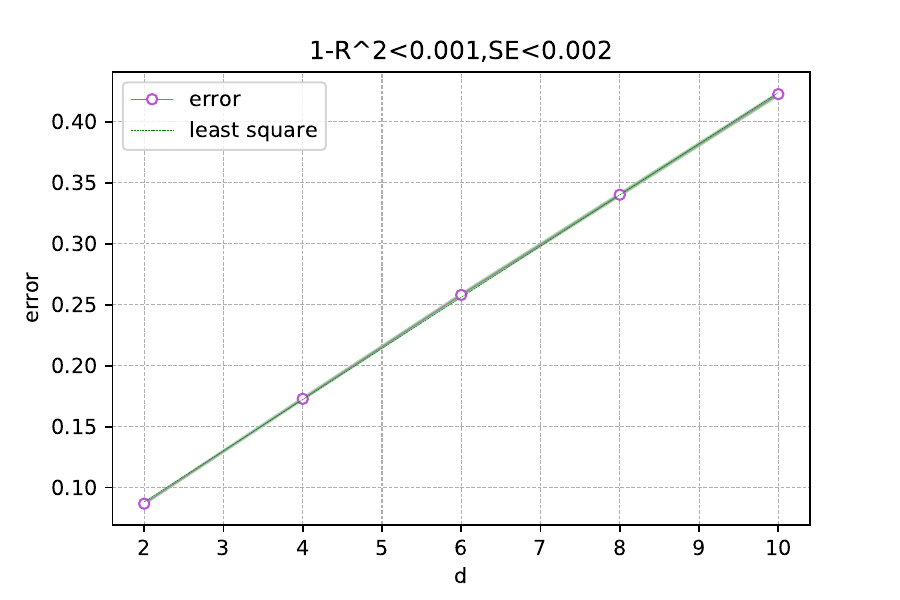}
	\end{minipage}\hfill
	\caption{Error with increasing $d$ for $   \delta\in\{0.01,0.02,0.03,0.04,0.06,0.07\}$}
\end{figure}

\begin{figure}[H] 
	\centering
	\begin{minipage}{0.5\textwidth}
		\includegraphics[width=\textwidth]{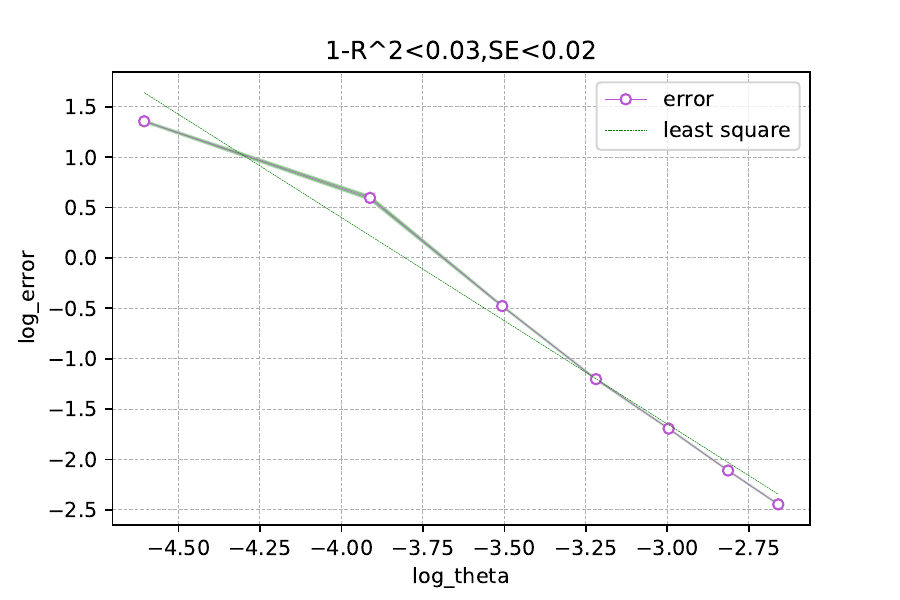}
	\end{minipage}\hfill
	\begin{minipage}{0.5\textwidth}
		\includegraphics[width=\textwidth]{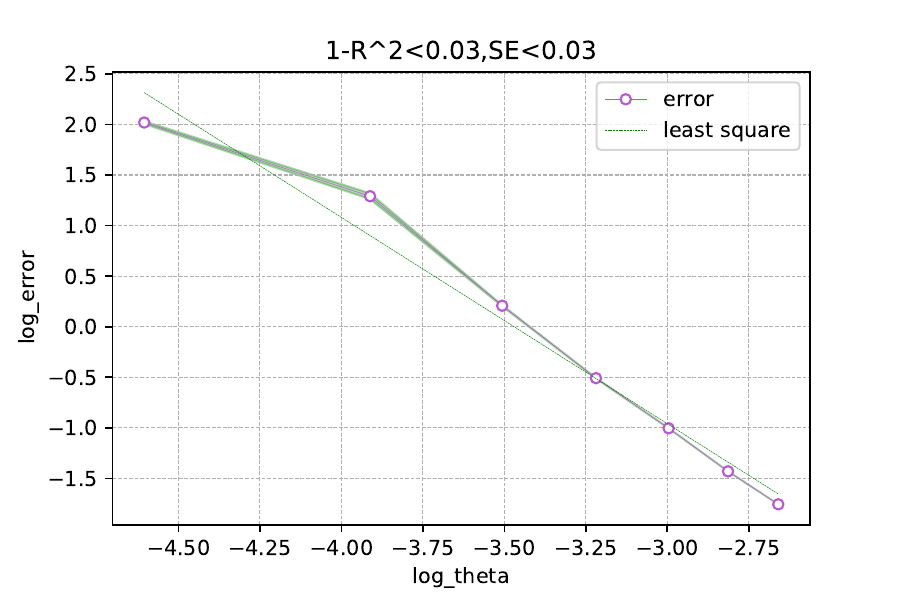}
	\end{minipage}\hfill
	\begin{minipage}{0.5\textwidth}
		\includegraphics[width=\textwidth]{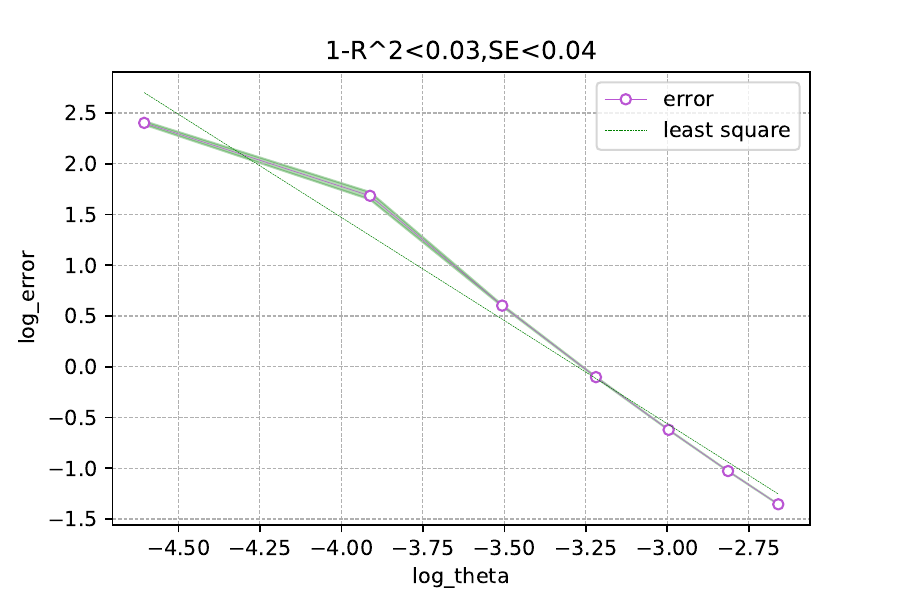}
	\end{minipage}\hfill
	\begin{minipage}{0.5\textwidth}
		\includegraphics[width=\textwidth]{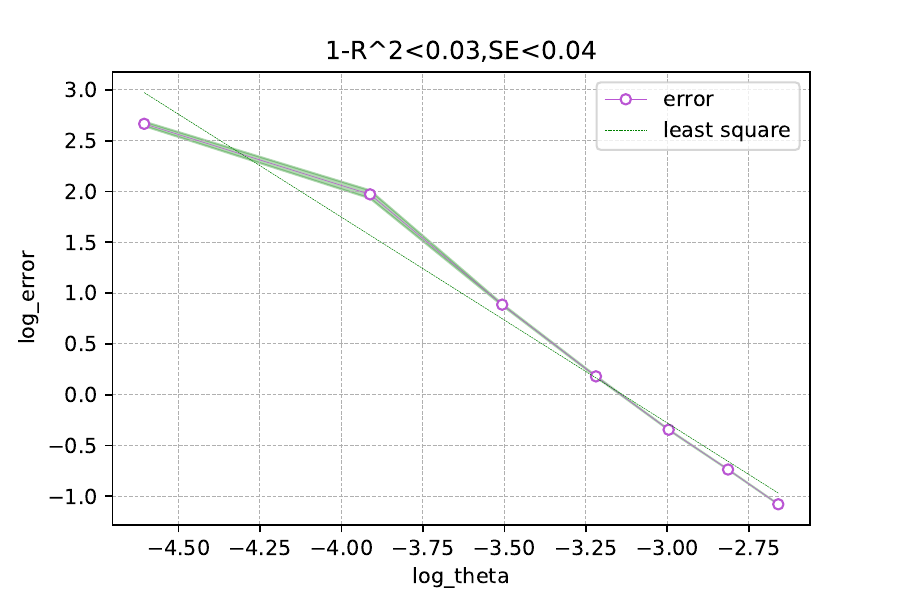}
	\end{minipage}\hfill
	\caption{Estimation error with increasing $\delta$ for $d\in\{2,4,6,8\}$. The slopes are $-2.049,-2.04,-2.032,-2.026$ respectively.}
\end{figure}